%% file: Main.tex
\documentclass[reqno]{amsart}
\usepackage{lmodern}
\usepackage[T1]{fontenc}
\usepackage[latin9]{inputenc}
\input{Preamble.tex}
\usepackage[margin=1.5in]{geometry}
\usepackage{quiver} 
\newtoggle{draft}
\togglefalse{draft}

\iftoggle{draft}{
\usepackage[notcite]{showkeys}

\newcommand{\NB}[1]{\todo[color=gray!40]{#1}}
\newcommand{\TODO}[1]{\todo[color=red]{#1}}
}{ 
\newcommand{\NB}[1]{}
\newcommand{\TODO}[1]{}
\renewcommand{\todo}[1]{}
\renewcommand{\todo}[1]{}
}

\title{$K$-theoretic counterexamples to Ravenel's telescope conjecture}
\date{\today}

\author{Robert Burklund}
\address{Department of Mathematical Sciences, University of Copenhagen, Denmark}
\email{rb@math.ku.dk}

\author{Jeremy  Hahn}
\address{Department of Mathematics, MIT}
\email{jhahn01@mit.edu}

\author{Ishan Levy}
\address{Department of Mathematics, MIT}
\email{ishanl@mit.edu}

\author{Tomer M. Schlank}
\address{Einstein Institute of Mathematics, The Hebrew University of Jerusalem}
\email{tomer.schlank@gmail.com}

\begin{document}
\maketitle
\begin{abstract}
At each prime $p$ and height $n+1 \ge 2$, we prove that the telescopic and chromatic localizations of spectra differ. Specifically, for $\mathbb{Z}$ acting by Adams operations on $\mathrm{BP}\langle n \rangle$, we prove that the $T(n+1)$-localized algebraic $K$-theory of $\mathrm{BP}\langle n \rangle^{h\mathbb{Z}}$ is not $K(n+1)$-local. We also show that Galois hyperdescent, $\A^1$-invariance, and nil-invariance fail for the $K(n+1)$-localized algebraic $K$-theory of $K(n)$-local $\E_{\infty}$-rings. In the case $n=1$ and $p \ge 7$ we make complete computations of $T(2)_*\mathrm{K}(R)$, for $R$ certain finite Galois extensions of the $K(1)$-local sphere. We show for $p\geq 5$ that the algebraic $K$-theory of the $K(1)$-local sphere is asymptotically $L_2^{f}$-local.

\end{abstract}

\begin{figure}[H]
  \centering{}
  \setlength{\fboxsep}{-5pt}
  \setlength{\fboxsep}{5pt}
  \frame{\includegraphics[scale=0.460]{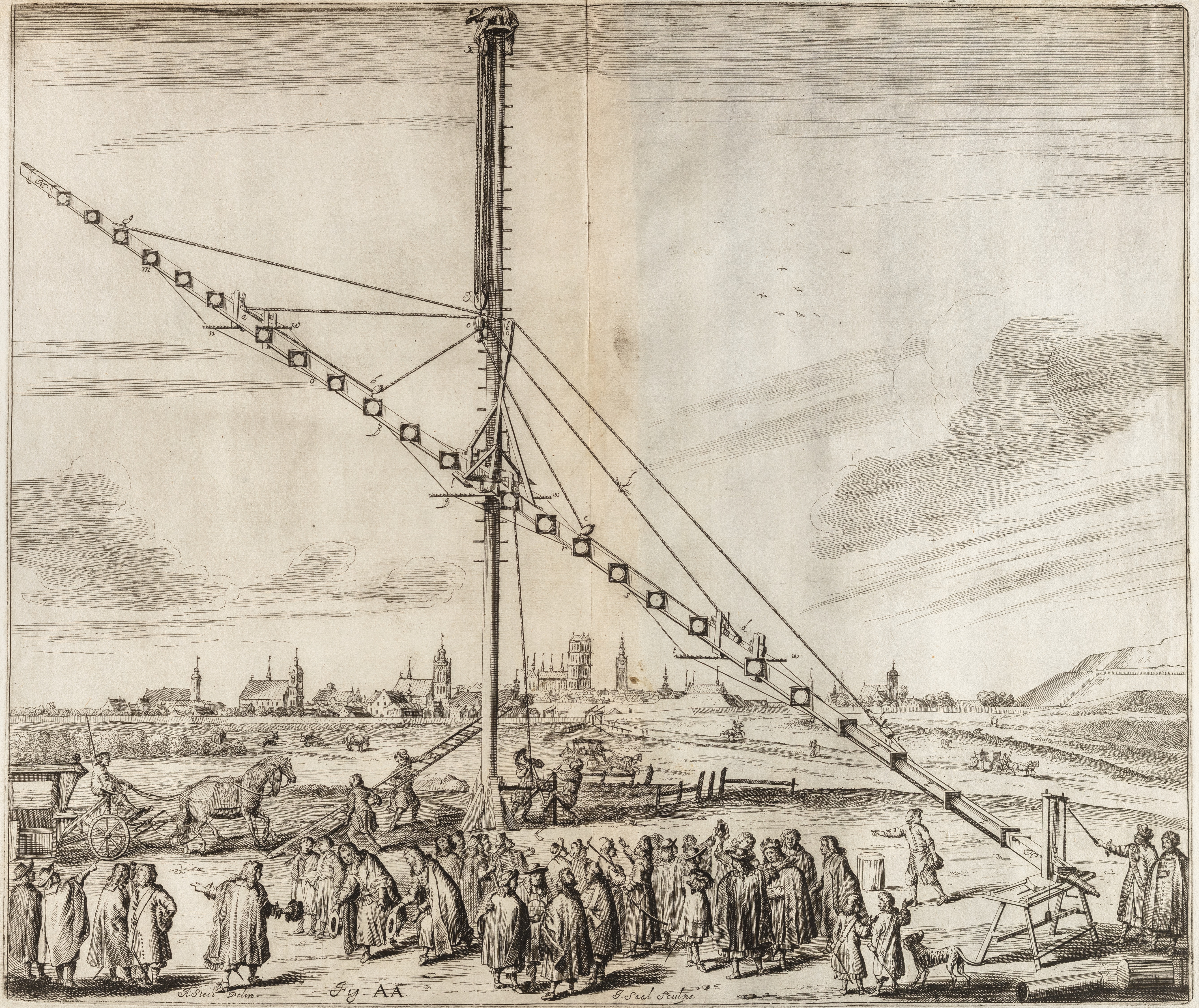}}
  \caption{\footnotesize Johannes Hevelius telescope by Isaac Saal after Andreas Stech, 1673 [The Royal Society \textcopyright]}
\end{figure}
\newpage

\tableofcontents
\vbadness 5000

\addtocontents{toc}{\protect\setcounter{tocdepth}{1}}
\section{Introduction}
\label{sec:intro}
\input{Introduction.tex}
\addtocontents{toc}{\protect\setcounter{tocdepth}{2}}

\section{Cyclotomic spectra}
\label{sec:bounded}
\input{bounded-short.tex}

\section{\texorpdfstring{$\THH$}{THH} of cochains on the circle}
\label{section:sphere}
\label{sec:cochaincircle}
\input{zeta3.tex}

\input{knlocal.tex}

\section{Locally unipotent \texorpdfstring{$\Z$}{Z}-actions and Lichtenbaum--Quillen}
\label{sec:tame}
\input{tame-action-short.tex}
\section{Adams operations on \texorpdfstring{$\BP\langle n\rangle$}{BPn}}
\input{adamsop.tex}

\section{Disproving the telescope conjecture}
\label{sec:mainthm}
\input{disproof.tex}

\section{Explicit \texorpdfstring{$T(2)$}{T(2)}-local \texorpdfstring{$\TC$}{TC} of some height \texorpdfstring{$1$}{1} local fields}
\input{TCcomputation2.tex}




\appendix
\section{}\label{sec:app}

\subsection{Homotopy ring structures}
\input{hring.tex}

\subsection{Almost compact objects}
\label{subsubsec:almost-compact}
\input{acompact.tex}

\subsection{Locally unipotent \texorpdfstring{$\Z$}{Z}-actions}
\input{loc-unip-basic.tex}

\subsection{Connectivity for \texorpdfstring{$\cO$}{O}-algebras}

\input{O-alg.tex}

\bibliographystyle{alpha}
\bibliography{bibliography}

\end{document}

%% file: Preamble.tex
\usepackage{verbatim}
\usepackage[textsize=scriptsize]{todonotes}
\usepackage{tikz-cd}
\usepackage{etoolbox}
\usepackage{etex}
\usepackage{wasysym}
\usepackage{mathrsfs}
\usepackage[T1]{fontenc}
\usepackage{chemarr}
\usepackage{amssymb}
\usepackage[leqno]{amsmath}
\usepackage{comment}
\usepackage{mathtools}
\usepackage{stmaryrd}
\usepackage{rotating}
\usepackage{enumitem}
\usepackage{wrapfig}
\usepackage{outlines}
\usepackage{graphicx}
\usepackage{scalerel}
\usepackage{amssymb}
\usepackage{bbm}
\usepackage{multicol}
\usepackage{float} 
\usepackage{amsthm}
\usepackage[all,arc]{xy}
\usepackage{stackrel}

\usepackage{spectralsequences}

\let\oldwidetilde\widetilde
\protected\def\widetilde{\oldwidetilde}

\newtheorem{theorem}{Theorem}

\usetikzlibrary{decorations.pathreplacing, calc}

\tikzcdset{
	diagrams={>={Straight Barb[scale=0.5]}}
}
\tikzset{
	altstackar/.style={decorate, decoration={show path construction,
			lineto code={
				\path (\tikzinputsegmentfirst); \pgfgetlastxy{\xstart}{\ystart}
				\path (\tikzinputsegmentlast); \pgfgetlastxy{\xend}{\yend}
				\path ($(0,0)!1.5pt!(\ystart-\yend,\xend-\xstart)$); \pgfgetlastxy{\xperp}{\yperp}
				\foreach \n[evaluate=\n as \k using .5*#1-\n+.5] in {1,...,#1}{
					\ifodd\n{\draw[->, shorten <=2pt, shift={($\k*(\xperp,\yperp)$)}](\xstart,\ystart)--(\xend,\yend);}
					\else{\draw[<-, shorten >=2pt, shift={($\k*(\xperp,\yperp)$)}](\xstart,\ystart)--(\xend,\yend);}\fi
				}
			}
	}}, altstackar/.default={1}
}

\usepackage{hyperref}
\usepackage{cleveref}

\hypersetup{
   colorlinks,
   linkcolor={red},
   citecolor={green!30!black},
   urlcolor={blue}
}

\usepackage{tikz}
\usetikzlibrary{matrix,arrows,decorations}
\usepackage{tikz-cd}

\usepackage{adjustbox}

\let\oldtocsection=\tocsection
 
\let\oldtocsubsection=\tocsubsection
 
\let\oldtocsubsubsection=\tocsubsubsection
 
\renewcommand{\tocsection}[2]{\hspace{0em}\oldtocsection{#1}{#2}}
\renewcommand{\tocsubsection}[2]{\hspace{1em}\oldtocsubsection{#1}{#2}}
\renewcommand{\tocsubsubsection}[2]{\hspace{2em}\oldtocsubsubsection{#1}{#2}}

\makeatletter
\def\subsection{\@startsection{subsection}{2}%
  \z@{.5\linespacing\@plus.7\linespacing}{.1\linespacing}%
  {\normalfont\bfseries}}
\def\subsubsection{\@startsection{subsubsection}{3}%
  \z@{.5\linespacing\@plus.7\linespacing}{.1\linespacing}%
  {\normalfont\itshape}}
\makeatother

\makeatletter
\newcommand{\leqnomode}{\tagsleft@true}
\newcommand{\reqnomode}{\tagsleft@false}
\makeatother

\theoremstyle{definition}

\newtheorem{nul}{}[section]
\newtheorem{dfn}[nul]{Definition}

\newtheorem{rmk}[nul]{Remark}

\newtheorem{cnstr}[nul]{Construction}
\newtheorem{cnv}[nul]{Convention}
\newtheorem{ntn}[nul]{Notation}
\newtheorem{exm}[nul]{Example}

\newtheorem{rec}[nul]{Recollection}

\newtheorem{wrn}[nul]{Warning}

\newtheorem*{warn*}{Warning}
\newtheorem*{dfn*}{Definition}
\newtheorem*{axm*}{Axiom}
\newtheorem*{ntn*}{Notation}
\newtheorem*{exm*}{Example}
\newtheorem*{exr*}{Exercise}
\newtheorem*{int*}{Intuition}
\newtheorem*{qst*}{Question}
\newtheorem*{rmk*}{Remark}

\AtBeginEnvironment{rmk}{\pushQED{\qed}}
\AtEndEnvironment{rmk}{\popQED\endrmk}
\AtBeginEnvironment{dfn}{\pushQED{\qed}}
\AtEndEnvironment{dfn}{\popQED\enddfn}
\AtBeginEnvironment{cnstr}{\pushQED{\qed}}
\AtEndEnvironment{cnstr}{\popQED\endcnstr}
\AtBeginEnvironment{rec}{\pushQED{\qed}}
\AtEndEnvironment{rec}{\popQED\endrec}
\AtBeginEnvironment{cnv}{\pushQED{\qed}}
\AtEndEnvironment{cnv}{\popQED\endcnv}
\AtBeginEnvironment{ntn}{\pushQED{\qed}}
\AtEndEnvironment{ntn}{\popQED\endntn}
\AtBeginEnvironment{exm}{\pushQED{\qed}}
\AtEndEnvironment{exm}{\popQED\endexm}
\AtBeginEnvironment{wrn}{\pushQED{\qed}}
\AtEndEnvironment{wrn}{\popQED\endwrn}

\DeclareMathOperator{\hex}{\pentagon}

\theoremstyle{plain}

\newtheorem{thm}[nul]{Theorem}
\newtheorem{prop}[nul]{Proposition}

\newtheorem{lem}[nul]{Lemma}

\newtheorem{cor}[nul]{Corollary}

\newtheorem*{thm*}{Theorem}
\newtheorem*{prop*}{Proposition}
\newtheorem*{cor*}{Corollary}
\newtheorem*{lem*}{Lemma}
\newtheorem*{cnj*}{Conjecture}

\DeclareMathOperator*{\colim}{colim}

\DeclareMathOperator{\cof}{cof}

\DeclareMathOperator{\fib}{fib}

\DeclareMathOperator{\Hom}{Hom}
\DeclareMathOperator{\Map}{Map}
\DeclareMathOperator{\TC}{TC}

\DeclareMathOperator{\End}{End} 
\DeclareMathOperator{\Aut}{Aut}

\DeclareMathOperator{\Fun}{Fun}

\DeclareMathOperator{\SP}{\mathbb{S}}

\DeclareMathOperator{\Pic}{Pic}
\DeclareMathOperator{\THH}{\mathrm{THH}}
\DeclareMathOperator{\TR}{\mathrm{TR}}
\DeclareMathOperator{\TP}{\mathrm{TP}}
\DeclareMathOperator{\BU}{\mathrm{BU}}
\DeclareMathOperator{\rB}{\mathrm{B}}
\DeclareMathOperator{\triv}{\mathrm{triv}}

\DeclareMathOperator{\Mod}{Mod}
\DeclareMathOperator{\Bimod}{Bimod}

\DeclareMathOperator{\CycSp}{CycSp}

\newcommand{\BPn}{\BP\langle n\rangle}

\def\PrL{\mathrm{Pr}^{\mathrm{L}}}

\def\Gal{\mathrm{Gal}}
\newcommand{\EQ}{\mathrm{EQ}}
\newcommand{\BZgg}{B\ZZ^\rhd}
\newcommand\noloc{%
  \nobreak
  \mspace{6mu plus 1mu}
  {:}
  \nonscript\mkern-\thinmuskip
  \mathpunct{}
  \mspace{2mu}
}

\def\Nm{\mathrm{Nm}}

\def\can{\mathrm{can}}

\newcommand{\crTR}{L_{\langle p^\infty\rangle}\Ss}

\newcommand{\pgnsp}[1]{\mathrm{PgcSp}_{#1}}

\newcommand{\UAlg}{\mathrm{UAlg}}
\newcommand{\lpr}{\langle p \rangle}
\def\A{\mathbb{A}}
\def\C{\mathbb{C}}
\def\E{\mathbb{E}}
\def\F{\mathbb{F}}
\def\G{\mathbb{G}}
\def\H{\mathbb{H}}

\def\Q{\mathbb{Q}}

\def\Ss{\mathbb{S}}
\def\T{\mathbb{T}}
\def\W{\mathbb{W}}
\def\Z{\mathbb{Z}}

\def\cC{\mathcal{C}}
\def\CC{\mathcal{C}}
\def\DD{\mathcal{D}}
\def\EE{\mathbb{E}}
\def\FF{\mathbb{F}}

\def\cL{\mathcal{L}}
\def\cO{\mathcal{O}}

\DeclareMathOperator{\CAlg}{CAlg}

\def\one{\mathbbm{1}}

\def\Fp{\mathbb{F}_p}
\def\Fpbar{\overline{\mathbb{F}}_p}

\def\KU{\mathrm{KU}}

\def\MU{\mathrm{MU}} 

\def\BP{\mathrm{BP}}
\def\BPn{\mathrm{BP}\langle n \rangle}
\def\map{\mathrm{map}}
\def\bu{\mathrm{bu}}

\def\mm{/\!\!/}

\def\bm{\mathrm{bm}}
\def\cn{\mathrm{cn}}

\newcommand{\SPp}{\SP}

\def\res{\mathrm{res}}
\def\triv{\mathrm{triv}}

\usepackage{babel}

\usepackage{cjhebrew}
\DeclareFontFamily{U}{rcjhbltx}{}
\DeclareFontShape{U}{rcjhbltx}{m}{n}{<->s*[1.2]rcjhbltx}{}
\DeclareSymbolFont{hebrewletters}{U}{rcjhbltx}{m}{n}
\DeclareMathSymbol{\pretsadi}{\mathord}{hebrewletters}{118}

\makeatother

\newcommand{\Perf}{\mathrm{Perf}_{\F_p}}

\newcommand{\TCm}{\TC^{-}}

\newcommand{\pflat}{\pi_0^{\flat}}

\def\Alg{\mathrm{Alg}}
\def\CAlg{\mathrm{CAlg}}

\def\Mod{\mathrm{Mod}}
\def\heart{\heartsuit}

\def\Sp{\mathrm{Sp}}

\def\Spaces{\mathrm{Spc}}
\def\Spc{\mathrm{Spc}}

\def\op{\mathrm{op}}
\def\Gr{\mathrm{Gr}}

\newcommand{\pushout}{\arrow[ul, phantom, "\ulcorner", very near start]}
\newcommand{\pullback}{\arrow[dr, phantom, "\lrcorner", very near start]}

\newcommand{\QQ}{\mathbb{Q}}
\newcommand{\ZZ}{\mathbb{Z}}

\newcommand{\cycl}{^{\wedge}_{\mathrm{cyc}}}

\newcommand{\dualz}{\diamondsuit}

\definecolor{DefColor}{rgb}{0.6,0.15,0.25}

\newcommand{\mdef}[1]{\textcolor{DefColor}{#1}} 
\newcommand{\tdef}[1]{\textit{\mdef{#1}}}

\def\id{\mathrm{id}}
\def\Id{\mathrm{Id}}

\def\Perf{\mathrm{Perf}}

\usepackage{amssymb,graphicx}

\newcommand{\WcanV}[1]{\mathrm{WCV}(\leq #1)}
\newcommand{\ScanVn}[1]{{\mathrm{SCV}}(\leq #1)}

\newcommand{\Seg}[1]{\mathrm{Segal}(\leq #1)}

\newcommand{\LCF}[1]{C^0(#1)}
\newcommand{\on}{\mathbbm{1}}

\def\repTR{L\Ss_{\langle p^\infty \rangle}}


%% file: Introduction.tex
Chromatic homotopy theory can be described as a surprising and intimate relationship between stable homotopy theory and the theory of $1$-dimensional commutative formal groups.
Morava and Ravenel laid out a vision for this relationship which can be summarized as giving a natural bijection between the moduli stack of formal groups and the ``primes'' in stable homotopy theory.

The key ideas of this point of view were summarized by the Ravenel conjectures in \cite{ravenel1984localization}, most of which were proven by Devinatz--Hopkins--Smith in \cite{DHS} and \cite{NilpII}. The main conjecture from \cite{ravenel1984localization} that remained unresolved was the telescope conjecture, which suggests that two natural competing definitions for monochromatic spectra agree. 

More precisely, let $K(n+1)$\footnote{The reader might question why we use $n+1$ rather than $n$.
Indeed, when stating the telescope conjecture, this may seem somewhat unorthodox. However, as suggested by the paper's title, we employ algebraic $K$-theory to develop counterexamples to the telescope conjecture.
Taking into account the redshift phenomena of \cite{rognes2014chromatic}, a significant portion of our analysis pertains to objects with a height that is one less than that for which we are disproving the telescope conjecture.
Thus, the inclusion of the "$+1$."} 
denote height $n+1$ Morava $K$-theory, and let $U$ be a finite $p$-local spectrum of type $n+1$. 
By \cite{NilpII}, there is a self map $v\colon \Sigma^dU\to U$ for $d>0$ inducing an isomorphism after tensoring  with $K(n+1)$. We define $T(n+1)\coloneqq U[v^{-1}]$, and there is a fundamental inclusion of localized categories
\[\Sp_{K(n+1)} \subseteq \Sp_{T(n+1)}.\]

The category $\Sp_{T(n+1)}$ is of interest because it detects the $v_{n+1}$-periodic part of the stable homotopy groups of spheres. On the other hand, computations in $\Sp_{K(n+1)}$ are a priori much more tractable, being closely connected to the algebraic cohomology of the moduli of formal groups.

The telescope conjecture postulates that the inclusion between these two categories is in fact an equality, and as such was originally favored by Occam's razor. For $n+1=0$, the conjecture is trivial. For $n+1=1$, the telescope conjecture was proved by Mahowald at $p=2$ \cite{Bok}, using $bo$-resolutions, and by Miller \cite{Miller1} for $p>2$, using a localized Adams spectral sequence. Both proofs proceed by explicit computation of the homotopy groups on the telescopic side. 

Parts of the analogous computations for heights $n+1\ge2$ led Ravenel to believe that the telescope conjecture is in fact \emph{false} in these cases \cite{Ravenel2}. Subsequent work, based on a certain family of Thom spectra $y(n+1)$, outlined a general strategy for a disproof, suggesting a concrete description for the gap between the telescopic and $K(n+1)$-local sides \cite{Ravenel3}. Variations on this strategy have been considered by
experts in the field \cite{Mark, MIT}, but have not resulted in a disproof.
The main result of this paper is a disproof of the telescope conjecture, at all primes $p$ and at all heights $n+1 \ge 2$.

To do this, we construct for each height $n$ and prime $p$ a family of associative ring spectra $\BP\langle n \rangle^{hp^k\ZZ}$, indexed by integers $k\geq0$, obtained by taking fixed points of Adams operations on the truncated Brown--Peterson spectrum $\BP\langle n \rangle$. 
We then show:


\begin{theorem}\label{thm:main}
Let $p$ be any prime and $n+1\geq 2$. Then, for all $k \ge 0$,
\[L_{T(n+1)}\mathrm{K}\left(\BP\langle n \rangle^{hp^k\Z}\right)\] 
is not $K(n+1)$-local. 
In particular,
\[\Sp_{K(n+1)} \neq \Sp_{T(n+1)}.\]
\end{theorem}

If \Cref{thm:main} is proved for a particular $k=k_1$, it automatically holds for all $0 \le k \le k_1$.  Thus, it will suffice to prove \Cref{thm:main} for $k \gg 0$, and we will often assume $k \gg 0$ below.


\subsection{Cyclotomic hyperdescent}
Before explaining how we prove \Cref{thm:main}, we indicate what our proof says about \textit{how} the telescope conjecture fails.

A key tool in studying $\Sp_{K(n+1)}$ is the fact, due to Devinatz--Hopkins \cite{DevinatzHopkins}, that there is an equivalence  $\SP_{K(n+1)} \cong E_{n+1}({\overline{\FF}_p})^{h\mathbb{G}_{n+1}}$. Here $E_{n+1}(\overline{\FF}_p)$ is a Lubin--Tate spectrum attached to $\Fpbar$, and $\mathbb{G}_{n+1}$ is the extended Morava stabilizer group \cite{GoerssHopkins}. Since $E_{n+1}(\overline{\FF}_p)$ is well understood at the level of homotopy rings, this gives an approach to studying the $K(n+1)$-local sphere via a homotopy fixed point spectral sequence.

Using the Galois theory of Rognes \cite{rognesgalois}, $E_{n+1}(\overline{\FF}_p)$ can be interpreted as the algebraic closure (see \cite{BakerRichter}) of $L_{K(n+1)}\SP$, with Galois group $\mathbb{G}_{n+1}$, so that the Devinatz--Hopkins result allows the $K(n+1)$-local category to be studied via Galois descent. From this perspective, two reasons we have lacked computational control over the $T(n+1)$-local category are:
\begin{enumerate}
	\item We know few explicit Galois extensions of $\SP_{T(n+1)}$.
	\item We do not know that descent is valid for the infinite Galois extensions that exist.
\end{enumerate}

In \cite{carmeli2021chromatic}, Carmeli, Yanovski and the fourth author show that all of the \textit{abelian}  Galois extensions of $\SP_{K(n+1)}$ may be lifted to $\SP_{T(n+1)}$, by constructing them as ``chromatic cyclotomic extensions.'' The most interesting of these extensions are the $p$-cyclotomic extensions \[\SP_{T(n+1)} \to \SP_{T(n+1)}[\omega_{p^{k}}^{(n+1)}],\] which have Galois group $(\ZZ/p^k)^\times$ and are concretely realized as summands of $T(n+1)$-localized suspended Eilenberg--MacLane spaces $L_{T(n+1)} \Sigma^{\infty}_+ \mathrm{K}(\mathbb{Z}/p^k,n+1).$  These telescopic $p$-cyclotomic extensions generalize the classical $p$-cyclotomic extensions $\QQ \to \QQ[\zeta_{p^k}]$ from the case $n+1=0$.  The filtered colimit of the $p$-cyclotomic extensions, denoted $\mathbb{S}_{T(n+1)}[\omega^{(n+1)}_{p^{\infty}}]$, is a $\mathbb{Z}_p^{\times}$-pro-Galois extension, and, in the case $n+1=1$, is the extension $L_{T(1)}\SP \to \KU_p$.


Though the maps \[\SP_{T(n+1)} \to \SP_{T(n+1)}[\omega_{p^{k}}^{(n+1)}]^{h(\ZZ/p^k)^\times}\] 
are equivalences for each $k\geq0$, this \emph{does not} guarantee that the map \[\SP_{T(n+1)} \to \SP_{T(n+1)}[\omega_{p^\infty}^{(n+1)}]^{h\ZZ_p^{\times}}\]
is an equivalence.\footnote{Asking that this map be an equivalence is equivalent to asking that $\SP_{T(n+1)}[\omega_{p^\infty}^{(n+1)}]$ be a faithful module over $\SP_{T(n+1)}$. It is also equivalent to asking that the sheaf of finite $\ZZ_p^{\times}$-sets defined by the finite Galois sub-extensions be a hypersheaf \cite[6.2]{cycloshift}.}
Put differently, localization with respect to $\SP_{T(n+1)}[\omega^{(n+1)}_{p^{\infty}}]$ yields a third category of \emph{cyclotomically complete $T(n+1)$-local spectra}:
\[\Sp_{K(n+1)} \subseteq (\Sp_{T(n+1)})\cycl\subseteq \Sp_{T(n+1)},\]
where both inclusions are potentially strict (see \cite[Question 7.36]{fourier}).

Our proof of \Cref{thm:main} actually shows that the second inclusion is strict: i.e., we prove that $L_{T(n+1)}K(\BP\langle n \rangle^{hp^k\Z})$ is not cyclotomically complete for $n\geq1$ (see \Cref{thm:maincyc}).

\subsection{The proof}

The first ingredient in our proof is the cyclotomic redshift result \cite{cycloshift} of Ben-Moshe, Carmeli, Yanovski and the fourth author. Cyclotomic redshift states that chromatic cyclotomic extensions are compatible with algebraic $K$-theory in the sense that, for any $T(n)$-local $\EE_1$-ring $R$, there are natural $(\ZZ/p^i)^{\times}$-equivariant equivalences:
\[L_{T(n+1)}K(R[\omega_{p^{i}}^{(n)}]) \cong L_{T(n+1)}K(R)[\omega_{p^{i}}^{(n+1)}].\]
The $n=0$ case of this theorem is due to Bhatt--Clausen--Mathew \cite{BhattClausenMathew}.

Using cyclotomic redshift, and the fact that the $(p^k\ZZ_p)$-pro-Galois extension 
\[L_{T(n)}\BP\langle n \rangle^{hp^k\ZZ} \to L_{T(n)}\BP\langle n \rangle\] is closely related to a cyclotomic extension, we deduce that there is an equivalence


\[ L_{T(n+1)}K(\BP\langle n \rangle)^{hp^k\ZZ}\cong L_{T(n+1)}K(\BP\langle n \rangle^{hp^k\ZZ})\cycl\quad \]
for each $k\geq0$. Thus, in order to prove \Cref{thm:main}, it suffices to show that the map 
\[L_{T(n+1)}K(\BP\langle n \rangle^{hp^k\ZZ}) \to L_{T(n+1)}K(\BP\langle n \rangle)^{hp^k\ZZ}\]
is not an equivalence for $k\gg0$. We refer to this map as the $K$-theoretic coassembly map, as it measures the failure of $T(n+1)$-local $K$-theory to commute with the limit $(-)^{hp^k\ZZ}$.

Our way of accessing these $T(n+1)$-local $K$-theory spectra is via trace methods. It follows from the landmark works \cite{CMNN,LMMT,DGM,MitchellVanishing} that, for any $\EE_1$-ring $R$ and $n+1\geq2$, there is a natural equivalence
\[L_{T(n+1)}K(R) \cong L_{T(n+1)}\TC(\tau_{\geq0}R).\] 
In principle, we may therefore replace the $K$-theoretic coassembly map with the map
\[L_{T(n+1)}\TC(\tau_{\geq0}(\BP\langle n \rangle^{hp^k\ZZ})) \to L_{T(n+1)}\TC(\BP\langle n \rangle)^{hp^k\ZZ}.\]

However, this is not the replacement for the $K$-theoretic coassembly map that we use, because the rings $\tau_{\geq0}(\BP\langle n \rangle^{hp^k\ZZ})$ lack regularity properties that make $\TC$ easy to access.\footnote{Using language introduced later in this introduction, they do not satisfy the height $n$ Lichtenbaum--Quillen property. 
	See \cite{lee2023topological} for a discussion and proof of this in the case $n=1$.} Instead we use a variant of the Dundas--Goodwillie--McCarthy theorem, due to the third author \cite{levy2022algebraic}, that applies to $(-1)$-connective rings.  This allows us to replace the $K$-theoretic coassembly map with the $\TC$ coassembly map
\[L_{T(n+1)}\TC(\BP\langle n \rangle^{hp^k\ZZ}) \to L_{T(n+1)}\TC(\BP\langle n \rangle)^{hp^k\ZZ},\]
reducing \Cref{thm:main} to the claim that this $\TC$ coassembly map is not an isomorphism when $k\gg0$.

The key to analyzing the $\TC$ coassembly map is the following result, which allows us to replace the $\ZZ$-action by Adams operations on $\BP\langle n \rangle$ with the \textit{trivial} action:

\begin{theorem}[Asymptotic constancy for $\BP\langle n \rangle$]\label{thm:assconbpn}
	Fix a telescope $T(n+1)$ of a type $n+1$ $p$-local finite spectrum. Then for all $k \gg 0$  there is a commuting square
	
\[\begin{tikzcd}
	{T(n+1)_*\TC(\BP\langle n\rangle^{hp^k\ZZ})} & {T(n+1)_*\TC(\BP\langle n\rangle)^{hp^k\ZZ}} \\
	{T(n+1)_*\TC(\BP\langle n\rangle^{B\ZZ})} & {T(n+1)_*\TC(\BP\langle n\rangle)^{B\ZZ}},
	\arrow["\cong", from=1-1, to=2-1]
	\arrow["\cong", from=1-2, to=2-2]
	\arrow[from=1-1, to=1-2]
	\arrow[from=2-1, to=2-2]
\end{tikzcd}\]
	where the horizontal maps are  $\TC$ coassembly maps.
\end{theorem}


Because of \Cref{thm:assconbpn}, it suffices to show that the $\TC$ coassembly map is not an isomorphism for the trivial $\ZZ$-action on $\BP\langle n\rangle$ (whose homotopy fixed points we write as $\BP\langle n \rangle^{B\Z}$ to emphasize the triviality of the action). To do this we use the following general fact:

\begin{prop}\label{prop:thicksubcat}
	For any $p$-complete $\EE_1$-ring $R$, the $p$-completion of $\TC(R)$ is in the thick subcategory generated by the $p$-completion of the fiber of the coassembly map $\TC(R^{B\ZZ}) \to \TC(R)^{B\ZZ}$.
\end{prop}

This proposition is proven by analyzing the universal case $R = \SP$, and showing after $p$-completion that $\THH(\SP)$ is in the thick subcategory generated by the fiber of $\THH(\SP^{B\Z}) \to \THH(\SP)^{B\Z}$ in the category $\CycSp$ of cyclotomic spectra.\footnote{We remind the reader that the use of the word 'cyclotomic' differs when referring to cyclotomic extensions and cyclotomic spectra.} The general case is obtained by tensoring with $\THH(R)$ in $\CycSp$ and applying the $p$-completed $\TC$ functor. The phenomenon in \Cref{prop:thicksubcat} is closely related to the failure of hyperdescent, nil-invariance, and $\A^1$-invariance in $K(n+1)$-local $K$-theory of $T(n)$-local rings for $n\geq 1$ (see \Cref{thm:failurekn}).

Applying \Cref{prop:thicksubcat}, and tensoring with $T(n+1)$, we learn that the $\TC$ coassembly map is not an equivalence so long as $T(n+1)\otimes \TC(\BP\langle n \rangle) \neq 0$. This follows from the redshift result of the second author and Wilson \cite{hahn2020redshift}, allowing us to finish the proof of \Cref{thm:main}.

\begin{rmk}
Let $U$ denote a type $n+1$ finite complex. To build intuition for \Cref{thm:assconbpn}, the reader might first contemplate the simpler statement that, when $k \gg 0$, the $p^k\ZZ$-action on $U \otimes \BPn$ is trivial. This follows from the fact that $U \otimes \BPn$ is $\pi$-finite and has a unipotent $\Z$-action on its homotopy groups.

In \cite{lee2023topological}, Lee and the third author computed that, for all $k \ge 0$, the tensor product $\mathbb{S}/(p,v_1) \otimes \THH(\BP\langle 1 \rangle^{hp^k\Z})$ 
has the same homotopy ring as the tensor product $\mathbb{S}/(p,v_1) \otimes \THH(\BP\langle 1 \rangle^{\mathrm{B}\Z})$. This is another simpler analog of \Cref{thm:assconbpn}, and at its core rests on the triviality of the $\Z$-action on $\BP\langle 1 \rangle/(p,v_1)$.
\end{rmk}

\subsection{The Lichtenbaum--Quillen property}
The key finiteness property of $\BP\langle n \rangle$ used to prove \Cref{thm:assconbpn} is the \emph{height $n$ Lichtenbaum--Quillen property}:

\begin{dfn*}
We say that an $\mathbb{E}_1$-ring spectrum $R$ satisfies the height $n$ LQ property if $\THH(R)$ is bounded below, and, for any $p$-local finite type $n+2$ complex $V$, $V \otimes \TR(R)$ is bounded.
\end{dfn*}
Equivalently, $R$ satisfies the height $n$ LQ property if $V \otimes \THH(R)$ is bounded in the $t$-structure on cyclotomic spectra of Antieau--Nikolaus \cite{antieau-nikolaus}.

The height $0$ LQ property has substantial history, being closely related to the Lichtenbaum--Quillen conjectures describing $K$-theory of discrete rings in terms of \'etale cohomology \cite{HMII,HMIII,mathew2021k}.  Ausoni and Rognes proved the height $1$ LQ property for $\mathrm{BP}\langle 1 \rangle$ at primes $p \ge 5$  \cite{ausoni2002algebraic}, and had the further vision to highlight height $n$ LQ properties as central to their redshift philosophy.

Work of the second author and Wilson \cite{hahn2020redshift} gave additional techniques for proving LQ properties, which with Raksit were connected to the Bhatt--Morrow--Scholze interpretation of prismatic cohomology \cite{hahn2022motivic,BMS,2023pstragowski}. Using these techniques, the height $n$ LQ property for $\BPn$ was proved for all $n$ and $p$ \cite{hahn2020redshift}.

Here, we establish a tool for descending LQ properties through fixed points by unipotent $\ZZ$-actions. Our proof of \Cref{thm:assconbpn} comes from setting $R=\BP\langle n \rangle$ in the following theorem:



\begin{theorem}[Cyclotomic asymptotic constancy]\label{thm:cohconstgen}
	Suppose that $R$ is a connective $p$-complete $\EE_1\otimes \A_2$\footnote{An $\EE_1\otimes\A_2$-ring is a unital algebra in the category of $\E_1$-rings.}-ring of fp-type $n$
	,\footnote{A connective $p$-complete spectrum $X$ is fp-type $n$ if, for any type $n+1$ finite complex $U$, $U\otimes X$ is $\pi$-finite \cite{mahowald1999brown}.} equipped with a locally unipotent $\ZZ$-action,\footnote{This is the same as asking that the $\ZZ$-action on $\pi_*R/p$ be locally unipotent.} and let $V$ be a finite $p$-local spectrum of type $n+2$.
	
	If $R$ satisfies the height $n$ LQ property, then 
    $R^{hp^k\ZZ}$ satisfies the height $n$ LQ property for all $k\gg0$. 
    Furthermore, for $k \gg 0$ there is a commutative diagram of cyclotomic spectra
	
	\[\begin{tikzcd}
		{V\otimes \THH(R^{hp^k\ZZ})} & {V\otimes \THH(R)^{hp^k\ZZ}} \\
		{V\otimes \THH(R^{B\ZZ})} & {V\otimes \THH(R)^{B\ZZ}},
		\arrow["\cong", from=1-1, to=2-1]
		\arrow["\cong", from=1-2, to=2-2]
		\arrow[from=1-1, to=1-2]
		\arrow[from=2-1, to=2-2]
	\end{tikzcd}\]
	where the horizontal maps are the coassembly maps.
\end{theorem}

Since the above is a diagram of cyclotomic spectra, we also obtain a corresponding square after replacing $\THH$ with $\TC$. 

The height $n$ LQ property for an $\mathbb{E}_1$-ring $R$ implies, for any type $n+1$ spectrum $U$ with $v_{n+1}$-self map $v$, that the map
\[U\otimes \TC(R) \to U[v^{-1}]\otimes \TC(R) = T(n+1)\otimes \TC(R)\]
has bounded above fiber. 
In short, knowledge of $U_*\TC(R)$ through a finite range of degrees is enough to completely determine the periodic homotopy groups $T(n+1)_*\TC(R).$
This property is exactly the opposite of what makes the homotopy groups of $T(n+1)$ so inaccessible:  
for $n+1\geq 2$, there is no bounded range of degrees in which all classes in $\pi_*T(n+1)$ lift to $\pi_*U$.
\footnote{This follows from (2) in the forthcoming work section below, along with Serre's finiteness theorem} 

To exemplify the approachability of telescopic homotopy in the presence of LQ properties, Ausoni and Rognes  were able to completely calculate  $T(2)_*\TC(\BP \langle 1 \rangle)$ for primes $p\ge 5$ and $T(2)=v_2^{-1}\mathbb{S}/(p,v_1)$ \cite{ausoni2002algebraic}. However, it was only recently in \cite{hahn2022motivic} that $T(2) \otimes \TC(\BP \langle 1 \rangle)$ was proved to be $K(2)$-local.


\subsection{Height 2 and the K(1)-local sphere}
Though we need not calculate much about $\TC(\BP\langle n \rangle^{hp^k\ZZ})$ to prove \Cref{thm:main}, it is sometimes possible to make complete calculations. We do so for $n=1,p\geq7$, and $k\gg0$ in \Cref{sec:ht2}.  By restricting to primes $p \ge 7$, we are able to fix as a preferred type $2$ spectrum a homotopy commutative and associative Smith--Toda complex $V(1)=\mathbb{S}/(p,v_1)$, with corresponding telescope $T(2)=v_2^{-1} V(1)$.

We study the connective Adams summand $\ell$ of $\mathrm{ku}_{(p)}$, which is a form of $\BP \langle 1 \rangle$. There is a $\Z$ action on $\ell$ by classical Adams operations, such that the $p$-completion of $\ell^{h\ZZ}$ is the $(-1)$-connective cover of the $K(1)$-local sphere.

Trace theorems of the third author \cite[Theorem B]{levy2022algebraic}, along with devissage results of the first and third author \cite{burklund2023k},
were used in \cite{levy2022algebraic} to show that $\TC(\ell^{h\Z})$ is closely related to the algebraic $K$-theory of the $K(1)$-local sphere.\footnote{This is true for $p>2$, and when $p=2$ there is an analogous statement with $\mathrm{ko}_{(2)}$ replacing $\ell$.} 
In particular, trace methods were used to completely compute $\pi_*K(L_{K(1)}\mathbb{S})[\frac 1 p]$, 
and to make computations of the integral homotopy groups $\pi_*K(L_{K(1)} \Ss)$ in low degrees \cite{levy2022algebraic}.  Furthermore, $V(1)_*\THH(\ell^{hp^k\ZZ})$ was studied extensively by Lee and the third author, for all $k \ge 0$  \cite{lee2023topological}.


 As \Cref{thm:height2comptc} here, we present for $k \gg 0$ a complete computation of $V(1)_*\TC(\ell^{hp^k\ZZ})$, as well as of the $\TC$ coassembly map 
\[V(1)_*\TC(\ell^{hp^k\ZZ}) \to V(1)_*\TC(\ell)^{hp^k\ZZ}.\]
In particular, we deduce the following corollary:

\begin{theorem} \label{thm:introheight2}
	Let $p\geq 7$ be a prime, and let $\ZZ$ act on the Adams summand $L$ of $\mathrm{KU}_{(p)}$ via the $\mathbb{E}_{\infty}$ Adams operation $\Psi^{1+p}$.  Then, for all $k \gg 0$, 
	the $\mathbb{F}_p[v_2^{\pm 1}]$-module map 
 \[T(2)_*\mathrm{K}(L^{hp^k\ZZ}) \to T(2)_*\left(L_{K(2)}\mathrm{K}(L^{hp^k\ZZ})\right)\]
	    may be identified with the direct sum of the maps enumerated below. The degrees of classes are determined from their names via the facts that $|t| = -2, |\lambda_i| = 2p^i-1, |\partial|=-1, |\zeta| = -1$, and the degree of any continuous function (see Notation (\ref{item:firsttop})) is $0$.
	\begin{enumerate}
		\item The projection $\FF_p\{1,\partial\} \oplus \overline{\LCF{\Z_p^{\times}}}\{\partial\zeta\} \to \FF_p\{1,\partial\}$ onto the first factor, tensored over $\FF_p$ with the inclusion $\mathbb{F}_p[v_2^{\pm 1}]\langle\lambda_1,\lambda_2 \rangle \to \mathbb{F}_p[v_2^{\pm 1}] \langle \lambda_1,\lambda_2,\zeta\rangle$.
		\item The map $\FF_p[v_2^{\pm 1}]\langle \zeta \rangle \otimes \LCF{p\ZZ_p} \to \FF_p[v_2^{\pm 1}]\langle \zeta \rangle$ evaluating a continuous function at $0$, tensored over $\FF_p$ with the graded $\FF_p$-vector space on basis elements enumerated below:
		\begin{enumerate}
			\item $t^{d} \lambda_1$, for each $0<d<p$, in degree $2p-1-2d$.
			\item $t^{pd} \lambda_2$, for each  each $0<d<p$, in degree $2p^2-1-2pd$
			\item $t^{d}\lambda_1\lambda_2$, for each $0<d<p$, in degree $2p^2+2p-2-2d$.
			\item $t^{pd} \lambda_1 \lambda_2$, for each $0<d<p$, in degree $2p^2+2p-2-2pd$.
		\end{enumerate}
	\end{enumerate}
	
\end{theorem}

At height $1$, \Cref{thm:main} is closely related to the failure of part of Ausoni--Rognes' original vision of chromatic redshift. Namely, they conjectured \cite[pg 4]{ausoni2002algebraic} that the map
\[V(1) \otimes K(L_{K(1)}\SP) \to V(1) \otimes K(\KU)^{h\ZZ_p^\times}\]
should have bounded above fiber, and so in particular be a $T(2)$-local equivalence. Our results imply that this map is not an equivalence $T(2)$-locally, but rather is the cyclotomic completion map. 

Nevertheless, using \cite{lee2023topological} we prove for $p\geq 5$ and \emph{all} $k \ge 0$ that $\ell^{hp^k\Z}$ satisfies the height $1$ LQ property. When combined with \cite{levy2022algebraic} this implies the following result, which can be considered a replacement for the Ausoni--Rognes conjecture:

\begin{theorem}\label{thm:kk1intro}
	For $p\geq 5$, the map $$V(1) \otimes K(L_{K(1)}\SP) \to V(1)[v_2^{-1}]\otimes K(L_{K(1)}\SP)$$
	has bounded above fiber.
\end{theorem}

\begin{rmk}
Many of our results about $V(1)_*\TC(\ell^{hp^k\ZZ})$ suggest approachable lines of further investigation. For example, one might study $\TC(\ell^{hp^k\Z})$ modulo larger powers of $p$ and $v_1$, for small $k$, or at small primes.
\end{rmk}

\subsection{Forthcoming work}

Although this paper disproves the telescope conjecture at all heights at least $2$ and all primes $p$, many parts of the argument simplify in the case $n=1,p\geq7$. For example, the full strength of \Cref{thm:cohconstgen} is not needed, and can be replaced by a more elementary $\pi_*$-level statement for a Smith--Toda complex $V(2)$. Moreover, the application of cyclotomic redshift is more direct in this case, and the construction of Adams operations is classical. We intend to write a shorter paper expositing a more efficient approach to this simplest case of our theorem.

In forthcoming work of Shachar Carmeli, Lior Yanovski, and the four authors, we aim to explore  additional consequences that our disproof of the telescope conjecture has for stable homotopy theory. In particular, we intend to prove the following three statements for $n \ge 1$:

	\begin{enumerate}
		\item The kernel of $\Pic(\Sp_{T(n+1)}) \to \Pic(\Sp_{K(n+1)})$ is infinite.
		\item For any nonzero telescope $T(n+1)$, there exists some integer $k$ such that $\pi_kT(n+1)$ is not a finitely generated abelian group.
		\item For every non-zero finite $p$-local spectrum $X$ (e.g., 
$X  = \SP_{(p)}$), there exists some $C >0$ such that 
 \[\lim_{M\to \infty}\sum\limits_{i=M}^{M+C} \dim_{\Fp}(\pi_i(X)\otimes \Fp) = \infty.\]
		\end{enumerate}

\subsection{Contents}

We now briefly describe the contents of the paper. In \Cref{sec:bounded}, we begin by reviewing the modern approach to the theory of cyclotomic and polygonic spectra, as developed in the foundational papers \cite{NS}, \cite{antieau-nikolaus}, and \cite{krause2023polygonic}. We then study finiteness and boundedness properties of cyclotomic spectra relevant to understanding the Lichtenbaum--Quillen property. In \Cref{sec:cochaincircle}, we prove \Cref{prop:thicksubcat} and study the cyclotomic spectrum $\THH(\SP^{B\ZZ})$, a key universal ring over which the cyclotomic spectra $\THH(\BP\langle n\rangle^{hp^k\ZZ})$ are modules. We also prove \Cref{prop:thicksubcat} and discuss how it is related to the failure of $\A^1$-invariance, nilinvariance, and hyperdescent in $K(n+1)$-local $K$-theory for $n+1\geq2$. In \Cref{sec:tame}, we prove \Cref{thm:E1A2algtame}, which in particular implies \Cref{thm:cohconstgen}. We do this by first proving a version of \Cref{thm:cohconstgen} at the level of spectra with Frobenius (rather than cyclotomic spectra). We then use the results of the previous sections to bootstrap this to the level of cyclotomic spectra. In \Cref{sec:adamsop}, we construct Adams operations on $\BP\langle n \rangle$ as $\EE_1\otimes \A_2$-algebra automorphisms. The key ingredients here are the $\EE_3$-$\MU$-algebra structure on $\BP\langle n \rangle$ of \cite{hahn2020redshift} and the stable Adams conjecture \cite{Friedlander,clausen2012padic,bhattacharya2022stable}.  In Section \ref{sec:mainthm}, we combine the main theorems of the previous sections, cyclotomic redshift, and other results to disprove the telescope conjecture. 
In \Cref{sec:ht2}, we completely compute the $\TC$ coassembly map for the $p^k\ZZ$-action on $\ell$, mod $(p,v_1)$ and for $p\geq7$, $k\gg0$, in particular proving \Cref{thm:introheight2}. We also prove the height $1$ LQ property for $\ell^{hp^k\ZZ}$ when $p\geq 5$, for all $k\geq0$, and prove \Cref{thm:kk1intro}. Finally, in \Cref{sec:app}, we include material much of which is well known but not treated to the extent needed in the literature.

\subsection*{Acknowledgements}

The authors would like to thank \"Ozg\"ur Bayindir, Shachar Carmeli, Dustin Clausen, Sanath Devalapurkar, Mike Hopkins, Branko Juran, Achim Krause, David Jongwon Lee, Akhil Mathew, Arpon Raksit, Maxime Ramzi, Doug Ravenel, Noah Riggenbach, John Rognes, Andrew Senger, Dylan Wilson, Tristan Yang, Lior Yanovski, and Allen Yuan for discussions related to this paper. We would also like to thank Shachar Carmeli, Shai Keidar, Shay Ben-Moshe, Shaul Ragimov, Doug Ravenel, and Lior Yanovski for comments on an earlier draft of the paper. The second author was supported by the Sloan Foundation and by a grant from the Institute for Advanced Study School of Mathematics. The third author was supported by the NSF Graduate Research Fellowship under Grant No. 1745302. The fourth author was supported by ISF1588/18, BSF 2018389 and the ERC under the European Union's Horizon 2020 research and innovation program (grant agreement No. 101125896).

The first author was supported by NSF grant DMS-2202992. The second author was supported by the Sloan Foundation and by a grant from the Institute for Advanced Study School of Mathematics. The third author was supported by the NSF Graduate Research Fellowship under Grant No. 1745302. The fourth author was supported by ISF1588/18, BSF 2018389 and the ERC under the European Union's Horizon 2020 research and innovation program (grant agreement No. 101125896).

\subsection*{Notations and Conventions}\label{subsec:conv}

\begin{enumerate}
\item We fix a prime $p$, and use $n$ to refer to a chromatic height.
\item We use $\Map_\CC$ to denote mapping spaces in a category $\CC$, and $\map_\CC$ to denote mapping spectra. We drop the $\CC$ when it is clear from context.
\item We use $\SP_{T(n)}$ to refer to the $T(n)$ local sphere.
\item We define constants $m_p^{\A_2}$, $m_p^{hc}$ and $m_p^{\E_1}$ by
\[ m_p^{\A_2} \coloneqq \begin{cases} 2 & p=2 \\ 1 & p \geq 3 \end{cases} \quad m_p^{hc} \coloneqq \begin{cases} 3 & p=2 \\ 2 & p=3 \\ 1 & p \geq 5 \end{cases} \quad m_p^{\E_1} \coloneqq \begin{cases} 3 & p=2 \\ 2 & p \geq 3 \end{cases}. \]
The significance of $m_p^{\A_2}$ (resp. $m_p^{hc}$, $m_p^{\E_1}$) is that for $k \geq m_p^{\A_2}$ ($m_p^{hc}$, $m_p^{\E_1}$) the Moore spectrum $\Ss/p^k$ admits an $\A_2$-algebra (hcring, $\E_1$-algebra) structure \cite{OkaMult} \cite{burklund2022multiplicative}.
\item If $R$ is a commutative ring, we write $R[x]$ for the polynomial algebra over $R$ with generator $x$.
\item If $R$ is a commutative ring, we write $R\langle \epsilon \rangle$ for the exterior algebra over $R$ with generator $\epsilon$.
\item Unless stated otherwise, spectra we consider are implicitly $p$-completed. We denote the category of $p$-complete spectra by $\Sp$.
We denote the $p$-complete sphere by $\SP\in \Sp$. 
\item We use $\Sp^{\omega}$ and $\Sp^{\diamondsuit}$ to denote the categories of compact and dualizable objects in $\Sp$, respectively.
\item If $G$ is an $\E_1$-monoid in spaces, then we use the term $G$-spectra for the category $\Fun(BG, \Sp)$ of functors from the classifying category of $G$ to the category of $p$-complete spectra.
\item We write $\T$ for the compact Lie group $U(1)$.
\item We write $(w) : \T \to \C^\times$ for the weight $w$ character of $\T$,
  $\Ss^{(w)}$ for the associated representation sphere, and
  $a_{(w)} : \Ss^0 \to \Ss^{(w)}$ for the Euler class.
\item We denote by 
\[\CycSp \coloneqq \mathrm{LEq}\left(\Sp^{B\T}\rightrightarrows \Sp^{B\T}\right)\]
the $p$-completion of the category of $p$-typical cyclotomic spectra of \cite{NS}.
\item Similarly the functors
  \[\THH\colon \Alg(\Sp) \to \CycSp\]\[\TC,\TCm,\TP,\mathrm{TR}\colon \CycSp \to \Sp\]
  are all implicitly $p$-completed, as is the functor $K:\Alg(\Sp) \to \Sp$.
\item For $\CC$ a presentably symmetric monoidal category, we use $\THH_\CC:\Alg(\CC) \to \CC$ to denote the $\THH$ relative to $\CC$ functor. For $R \in \CAlg(\CC)$ and $A \in \Alg(\Mod_{\CC}(R)))$, we use $\THH(A/R)$ to mean $\THH_{\Mod(\CC;R)}(A)$.
\item We write $\Perf$ for the category of perfect commutative $\F_p$-algebras.
\item Following \cite[Sec. 5.2]{ECII} and \cite[Prop. 2.2, Cons. 2.33]{chromaticNSTZ},
  there is an adjunction
  \[ \W(-) \colon \Perf \rightleftarrows \CAlg(\Sp) \noloc {\pflat}(-) \]
  where the right adjoint $\pflat$ can be computed as the inverse limit along Frobenius on the commutative ring $\pi_0(-)/p$.
\item \label{item:firsttop} Given a topological space $X$, we write $\LCF{X}$ for the ring of continuous (that is, locally constant) functions from $X$ to $\FF_p$.
\item We write $\overline{\LCF{X}}$ for the quotient $\LCF{X}/\F_p$, where we view $\F_p$ as a subset of $\LCF{X}$ via the inclusion of constant functions. 
\item Given a continuous map $f:Y \to X$ of topological spaces, restriction of functions gives a corresponding ring homomorphism $\LCF{X} \to \LCF{Y}$. We denote this by $\res_f$,  or just $\res$ if $f$ is clear, and call it the restriction map. 
\item \label{item:lasttop} In the case where $f : Y \to X$ is the inclusion of a subspace we write $(-)_{|Y}$ for base change along the map $\W(\res_f)$.
\item For $n \ge 1$, $\BPn$ refers to an $\E_3$-$\MU_{(p)}$-algebra form of the truncated Brown--Peterson spectrum as constructed in \cite[\S 2]{hahn2020redshift}.\label{item:BPn}
\item $E_n$ refers to the height $n$ Lubin--Tate theory constructed by Goerss--Hopkins--Miller \cite[Theorem 5.0.2]{ECII}, associated to the (unique up to isomorphism) height $n$ formal group over $\Fpbar$. \label{item:En} 
\item $\G_n$ refers to the height $n$ extended Morava stabilizer group, which acts on $E_n$ and fits into a short exact sequence
$$1 \to \cO_D^{\times} \to \G_n \to \Gal(\FF_p) \to 1.$$ Here, $\cO_D^{\times}$ is the units in the ring of integers of the division algebra over $\QQ_p$ of Hasse invariant $\frac 1 n$, and $\Gal(\FF_p) \cong \hat{\ZZ}$ is the absolute Galois group of $\FF_p$.\label{item:Gn}
\end{enumerate}

%% file: bounded-short.tex

Throughout this paper we use cyclotomic spectra heavily.
We follow the modern approach to cyclotomic spectra developed in \cite{NS}, and, in \Cref{subsec:dehn}, we also use the theory of $p$-polygonic spectra of  \cite{krause2023polygonic}.
In this section, we begin by recalling the basics of these theories.
The bulk of the section is then dedicated to relating boundedness in the cyclotomic $t$-structure of \cite{antieau-nikolaus} to other ``boundedness properties'' for cyclotomic spectra. In particular, we introduce quantitative versions of the properties studied by the second author and Wilson in \cite{hahn2020redshift} in order to prove the Lichtenbaum--Quillen property for $\BPn$. The main result of \Cref{subsec:almost perfect} is that being almost compact (with respect to the cyclotomic $t$-structure) is a relatively mild condition on a cyclotomic spectrum: for $p$-nilpotent cyclotomic spectra, it is equivalent to the underlying spectrum being almost compact.



\subsection{Preliminaries}

\subsubsection{Equivariant spectra}

Expanding on the conventions section above, we set out notation for working in the category $\Sp^{B\T}$ of $p$-complete spectra with $\T$-action that we will use throughout the paper.

For $w$ a non-negative integer, we denote by \mdef{$\T(w)$} the circle with action given by the quotient $\T/C_w$, where $C_w \to T$ is the cyclic subgroup of order $w$. 
The projection $\T(w) \to *$ 
gives rise to a cofiber sequence
$ \SP[\T/C_w]=\Sigma^{\infty}_+\T(w) \to \SP \to \mdef{\SP^{(w)}} $.
The underlying spectrum of $\SP^{(w)}$
is $\SP^2$, and $\SP^{(w)} \in \Sp^{B\T}$ is therefore tensor invertible. 
For $k\in \ZZ$, we denote by \mdef{$\Sigma^{k(w)}(-)$} the functor of tensoring with $\mdef{\SP^{k(w)}} \coloneqq (\SP^{(w)})^{\otimes k}$.  
Applying $\Sigma^{-(w)}$ to the final map in the above cofiber sequence yields an Euler class
\[\mdef{a_{(w)}}\colon\SP^{-(w)} \to \SP.\]

We next consider the following filtered $\T$-spectrum:
\[ 
\cdots \rightarrow \SP^{-m(1)} \xrightarrow{a_{(1)}} \cdots \xrightarrow{a_{(1)}} \SP^{-2(1)} \xrightarrow{a_{(1)}} 
\SP^{-(1)} \xrightarrow{a_{(1)}} 
\SP. 
\]
For any $X \in \Sp^{\mathrm{B}\T}$, the tensor product of $X$ with the above filtered spectrum gives the \tdef{$a_{(1)}$-Bockstein filtration} on $X$.  Taking
$\T$-fixed points then gives rise to the homotopy fixed point spectral sequence computing $\pi_*X^{h\T}$.\footnote{To compare this construction of the $\T$-homotopy fixed point spectral sequence with another the reader prefers it may help to note that the cofiber of $a_{(1)}^m : \Ss^{-m(1)} \to \Ss$ is isomorphic to $\SP^{S^{2m-1}}$, where the action on $S^{2m-1}$ is as the unit sphere in $\mathbb{C}^{m}$.} Similarly, the usual $\T$-Tate spectral sequence may be identified with the $a_{(1)}$-inverted $a_{(1)}$-Bockstein spectral sequence. 

 Additionally, for $k \geq 0$ we will often use the category $\Sp^{BC_{p^k}}$ of $p$-complete $C_{p^k}$-spectra. We view $C_{p^k} \subset \T$ as the cyclic subgroup of order $p^k$ and will often study $X^{hC_{p^k}},$ $X_{hC_{p^k}}$, and $X^{tC_{p^k}}$ when $X$ is a $\T$-equivariant spectrum. When we use these notations with $k=\infty$, we mean $\Sp^{B\T}$, $X^{h\T},$ $X_{h\T}$, and $X^{t\T}$, respectively. When working in $\Sp^{BC_{p^k}}$ for $k<\infty$, we write $\SP^{(w)},$ $\Sigma^{k(w)}$, and $a_{(w)}$ for the images of the corresponding objects under the restriction map $\Sp^{B\T} \to \Sp^{BC_{p^k}}$.





Given $1 \leq k\leq \infty,$ the $C_p$-Tate construction is a functor 
\[(-)^{tC_p}\colon \Sp^{BC_{p^k}} \to \Sp^{B(C_{p^k}/C_p)} =\Sp^{BC_{p^{k-1}}}  .\]
In the case $k=\infty$, we identify $\T/C_p$ with $\T$ by rescaling, and so consider $(\--)^{tC_p}$ as an endofunctor of $\Sp^{B\T}$.

We record an elementary lemma:

\begin{lem}\label{lem:a1tcp}
    For any $X \in \Sp^{B\T}$, $1 \leq j \leq \infty$, we have $(\Ss^{\T}\otimes X)^{tC_{p^j}} \cong \Sigma^{\infty}_+\T^{tC_{p^j}}\otimes X \cong 0$, so the functors $(-)^{tC_{p^j}}$ take $a_{(1)}\colon X \to \Sigma^{(1)}X$ to an isomorphism.
\end{lem}

\begin{proof}
  This follows because $\T$ has a finite equivariant cell decomposition by free $C_{p^j}$ cells.
\end{proof}

\begin{dfn}\label{dfn:sigmaoperator}
  We recall that any $\T$-spectrum $X$ has a \tdef{Connes operator},
  which is a degree $1$ self-map of non-equivariant spectra 
  \[\sigma \colon \Sigma X \to X.\]
  Viewing $X$ as a module over the group ring $\Ss[\T]$ the Connes operator is given by
  multiplication by the class $\sigma \in \pi_1(\Ss[\T])$ coming from the (pointed)
  identity map $S^1 \to \T$.  
  We also use $\sigma$ to refer to induced operator
  $\sigma\colon\pi_iX \to \pi_{i+1}X$.
\end{dfn}







\subsubsection{Cyclotomic spectra}
We work with the $p$-complete $p$-typical variants of cyclotomic spectra, defined in \cite{NS}. 

\begin{dfn} The category of cyclotomic spectra we use is defined as the lax equalizer
\[\mdef{\CycSp} \coloneqq \mathrm{LEq}\left( \xymatrix{\Sp^{B\T} \ar[r]<2pt> \ar[r]<-2pt> & \Sp^{B\T}} \right)\] 
of the identity functor and the functor $(-)^{tC_p}$.
We use $\mdef{\CycSp_{+}}$ to denote the full subcategory of bounded below objects. In other words, to give a cyclotomic spectrum is to give an $X \in \Sp^{B\T}$ with a map $\varphi_X:X\to X^{tC_p}$ in $\Sp^{B\T}$.
\end{dfn}

\begin{dfn}
	Given $X\in \CycSp$, we let
	$\mdef{\TC^{-}(X)} \coloneqq X^{h\T}$,
	$\mdef{\TP(X)} \coloneqq X^{t\T}$ and
	\[\mdef{\TC(X)}\coloneqq \map_{\CycSp}(\SP,X).\]
	
	We moreover have (see \cite[Example 3.4, Construction 3.18]{antieau-nikolaus}):
	\[\mdef{\TR^{k+1}(X)} := X^{hC_{p^k}} \times_{(X^{tC_p})^{hC_{p^{k-1}}}} \dots \times_{(X^{tC_p})^{hC_{p^2}}} X^{hC_{p^2}} \times_{(X^{tC_p})^{hC_p}} X^{hC_{p}}\times_{X^{tC_p}}X.\]
	and further define $\mdef{\TR(X)}\coloneqq \lim_k\TR^{k}(X)$.
	
	For $X\in \CycSp_+$, we recall that the natural map $ X^{tC_{p^{k+1}}}\to (X^{tC_p})^{hC_{p^k}}$ is an equivalence \cite[Lemma II.4.1]{NS}, simplifying the formula above.
\end{dfn}

$\TC(X)$ is computed for $X \in \CycSp$ by the equalizer
\[ \begin{tikzcd}
	{\TC(X)} & {\TC^{-}(X)} & {(X^{tC_p})^{h\T}}
	\arrow["{\mathrm{can}}"', shift right, from=1-2, to=1-3]
	\arrow["{\mathrm{\varphi}}", shift left, from=1-2, to=1-3]
	\arrow[from=1-1, to=1-2]
\end{tikzcd} \]

where the maps are obtained by applying $(-)^{h\T}$ to the maps $\varphi:X \to X^{tC_p}$ and the map $\mathrm{can}:X^{hC_p} \to X^{tC_p}$. When $X\in \CycSp_+$, the map $\TP(X) \to (X^{tC_p})^{h\T}$ is an equivalence \cite[Lemma II.4.2]{NS}, and $\mathrm{can}$ can be identified with the natural map $\TC^{-} \to \TP$.

There is a natural Frobenius endomorphism \cite[Construction 3.18]{antieau-nikolaus} $F:\TR \to \TR$ such that for $X \in \CycSp$, $\TC(X)$ is also computed as the equalizer

\[ \begin{tikzcd}
	{\TC(X)} & {\TR(X)} & \TR(X)
	\arrow["{1}"', shift right, from=1-2, to=1-3]
	\arrow["{F}", shift left, from=1-2, to=1-3]
	\arrow[from=1-1, to=1-2]
\end{tikzcd} \]

\begin{exm}\label{lem:describe_crtr}
  We define a cyclotomic spectrum $\crTR$\footnote{Our notation is meant to suggest it comes from applying the functor $L_{\langle p^{\infty}\rangle}$ between $p^{\infty}$-polygonic and cyclotomic spectra of \cite{krause2023polygonic} to $\SP$. However we do not use this fact.}. Its underlying spectrum is 
  \[\bigoplus_{j \geq 0} \Ss[\T/C_{p^j}],\] and the Frobenius is the isomorphism (by the Segal conjecture) given by the sum of the composites\footnote{Note that the coassembly map pulling the sum inside the $(-)^{tC_p}$ is an isomorphism by \cite[Corollary 6.7]{yuan2023integral}.}
  \[ \Ss[\T/C_{p^j}] \rightarrow (\Ss[\T/C_{p^{j+1}}])^{hC_p} \to (\Ss[\T/C_{p^{j+1}}])^{tC_p}.\qedhere\]
\end{exm}

%
%

The following can be compared to \cite[Theorem 6.5]{blumberg2016homotopy}:
\begin{lem}\label{lem:trcorep}
  The cyclotomic spectrum $\crTR$ co-represents $\TR(-)$ in $\CycSp_+$
\end{lem}

\begin{proof}
	The formula for mapping out of $\crTR$ is given as the equalizer
	
\[\begin{tikzcd}
	{\map_{\Sp^{B\T}}(\crTR,X)} & {\map_{\Sp^{B\T}}(\crTR,X^{tC_p})}
	\arrow["{\varphi\circ (-)}"', shift right, from=1-1, to=1-2]
	\arrow["{(-)^{tC_p}\circ\varphi}", shift left, from=1-1, to=1-2]
\end{tikzcd}\]
	
	Since for any $Y \in \Sp^{B\T}$ there is a natural isomorphism  $$\map_{\Sp^{B\T}}(\SP[\T/C_{p^k}],Y) \cong Y^{hC_{p^k}}$$, we can identify the above equalizer with the equalizer
	
\[\begin{tikzcd}
	\prod_0^{\infty}X^{hC_{p^i}} & \prod_0^{\infty}(X^{tC_p})^{hC_{p^i}}
	\arrow["{\varphi\circ (-)}"', shift right, from=1-1, to=1-2]
	\arrow["{(-)^{tC_p}\circ\varphi}", shift left, from=1-1, to=1-2]
\end{tikzcd}\]

The map $(-)^{tC_p}\circ \varphi$ can be identified with the canonical maps because of the description of the Frobenius map on $\crTR$. Then this equalizer agrees with $\TR(X)$.
\end{proof}

We will also use the following recognition criterion for $\crTR$:

\begin{lem} \label{lem:identify-crTR}
	Let $R$ be bounded below commutative algebra for which the canonical map
	$R \to R^{tC_p}$ is an isomorphism.
	Suppose we are given an $R$-module in cyclotomic spectra $X$ together with
	\begin{enumerate}
		\item a splitting of $X$ as $X \cong \oplus_{k \geq 0} X_k$ in $\Mod(R)^{B\T}$,
		\item a collection of isomorphisms
		$q_k : X_k \cong R \otimes \Ss[\T/C_{p^k}]$ in $\Mod(R)^{B\T}$ and
		\item a collection of isomorphisms
		$\varphi_k : X_k \to X_{k+1}^{tC_p}$ in $\Mod(R)^{B\T}$
	\end{enumerate}
	such that the cyclotomic Frobenius on $X$ is given by the sum of the map $\varphi_k$
	as a map of $\T$-equivariant $R$-modules.
	Then, there is an isomorphism of $R$-modules in cyclotomic spectra
	\[ X \cong R \otimes \crTR. \]
\end{lem}

\begin{proof}
	Let $\Sp_{sc}^{B\T} \subseteq \Sp^{B\T}$ be the full subcategory
	on those $Y$ for which the natural transformation
	$c : Y \to (p^*Y)^{tC_p}$ is an isomorphism at $Y$.
	Note that this property depends only on the underlying spectrum
	and by exactness of $(-)^{tC_p}$ all of the objects we will consider are in this subcategory.
	
	Let $f_0 \coloneqq q_0^{-1}$.
	By induction on $k \geq 1$ we will construct $\T$-equivariant isomorphisms of $R$-modules
	$f_k: R \otimes \Ss[T/C_{p^k}] \to X_k $ and homotopies making
	the following diagram of $\T$-equivariant $R$-modules commute  
	\[ \begin{tikzcd}
		R \otimes \Ss[\T/C_{p^{k-1}}] 
		\arrow[d, "c"']&
		X_{k-1} \arrow[d, "\varphi"] \\
		R \otimes \SP[\T/C_{p^k}])^{tC_p} &
		X_k^{tC_p}.     
		\arrow["f_{k-1}", from=1-1, to=1-2]
		\arrow["f_k^{tC_p}"', from=2-1, to=2-2]
	\end{tikzcd} \]
	Let $g_k \coloneqq q_k^{tC_p} \circ \varphi_{k-1} \circ f_{k-1} $ and let
	$f_k \coloneqq q_k^{-1} \circ p^*( c^{-1} \circ g_k) $. Then we have 
	\begin{align*}
		f_k^{tC_p} \circ c
		&= \left( q_k^{-1} \circ p^*( c^{-1} \circ g_k) \right)^{tC_p} \circ c
		= (q_k^{-1})^{tC_p} \circ (p^*( c^{-1} \circ g_k))^{tC_p} \circ c \\
		&= (q_k^{-1})^{tC_p} \circ c \circ c^{-1} \circ g_k
		= (q_k^{-1})^{tC_p} \circ (q_k^{tC_p} \circ \varphi_{k-1} \circ f_{k-1}) \\
		&= \varphi_{k-1} \circ f_{k-1}
	\end{align*}
	
	The sum of the maps $f_k$ is the desired isomorphism in $\Mod(R)^{B\T}$
	and sum of the commuting squares lifts it to an isomorphism of
	$R$-modules in cyclotomic spectra.
\end{proof}


\subsubsection{$p$-polygonic spectra}
\label{subsubsec:polygonic}

In \Cref{subsec:dehn}, we make use of the $p$-polygonic spectra of \cite{krause2023polygonic}, which we review here:

\begin{dfn}
	The category of $p$-polygonic spectra is defined as the oplax limit of the functor $(-)^{tC_p}:\Sp^{BC_p} \to \Sp$, i.e
	
	\[\mdef{\pgnsp{\langle p\rangle}} \coloneqq \Sp \vec{\times}_{(-)^{tC_p}}\Sp^{BC_p}\qedhere\]
\end{dfn}
In particular, we see that $\pgnsp{\lpr}$ can be identified with the category of $p$-complete genuine $C_p$-equivariant spectra.
Given $X = (X_0,X_1, X_0\to (X_1)^{t{C_p}})\in \pgnsp{\lpr}$, we sometimes use the following notations from genuine equivariant homotopy theory:
\[X^{\mdef{\Phi C_p}} \coloneqq X_0, X^{\mdef{\Phi e}} \coloneqq X_1, X^{\mdef{hC_p}}\coloneqq X_1^{hC_p}, X^{\mdef{tC_p}}\coloneqq X_1^{tC_p},\] 
and $X^{\mdef{C_p}}$ for the pullback 
\[ \begin{tikzcd}
  {X^{C_p}} \pullback & {X^{\Phi C_p}} \\
  {X^{hC_p}} & {X^{tC_p}}.
  \arrow["{\mathrm{ev}_1}", from=1-2, to=2-2]
  \arrow["{{(-)}^{tC_p}}"', from=2-1, to=2-2]
  \arrow[from=1-1, to=1-2]
  \arrow[from=1-1, to=2-1]
\end{tikzcd} \]

\begin{dfn}
  We denote by $\mdef{\mathrm{res}_{\varphi}} \colon \pgnsp{\lpr} \to \Sp^{\Delta^1}$ the lax symmetric monoidal functor sending $(X, Y, X \to Y^{tC_p})$ to $X \to Y^{tC_p}$.
We also define $\mdef{\mathrm{res}_{\hex}} \colon \CycSp \to \pgnsp{\lpr}$ to be the functor sending a cyclotomic spectrum $X$ to the triple $(X,X,X\xrightarrow{\varphi}X^{tC_p})$.
We abuse notation by writing
$ \mdef{\mathrm{res}_{\varphi}}  \coloneqq{\mathrm{res}_{\varphi}}\circ \mathrm{res}_{\hex} \colon   \CycSp \to \Sp^{\Delta^1} $.
\end{dfn}

\subsubsection{$t$-structures }

The category of spectra admits a $t$-structure such that $\Sp^{\heart}$ is the category of (derived) $p$-complete abelian groups. For every $1\leq k\leq \infty$, this induces a pointwise $t$-structure on $\Sp^{BC_{p^k}}$, where connectivity and coconnectivity are detected on the underlying non-equivariant spectrum.
We shall denote truncation functors with respect to this $t$-structure by $\mdef{\tau_{>b}}$ and $\mdef{\tau_{\leq b}}$.

In \cite{antieau-nikolaus}, Antieau and Nikolaus define the cyclotomic $t$-structure on $\CycSp$. In this $t$-structure an $X\in \CycSp$ is connective if and only if it is connective as a spectrum\footnote{We note that since our categories are $p$-completed, our definitions don't exactly agree.}. We shall denote truncation functors with respect the cyclotomic $t$-structure by $\mdef{\tau^{\mathrm{cyc}}_{>b}}$ and $\mdef{\tau^{\mathrm{cyc}}_{\leq b}}$. 

\begin{ntn}
  Given any of the categories 
  $\cC =\Sp^{BC_{p^k}},\CycSp$,
  we denote by  $\mdef{\cC_{\leq b}, \cC_{>b}, \cC_{[c,b]}}$
  the collection of appropriately bounded objects
  with respect to the aforementioned $t$-structures.
\end{ntn}

\begin{lem}[{\cite[Theorem 9]{antieau-nikolaus}}]
  The functor $\TR\colon\CycSp_+ \to \Sp$ is conservative and reflects connectivity and coconnectivity.
\end{lem}

\begin{proof}
	To reconcile the statement with \cite[Theorem 9]{antieau-nikolaus}, which isn't for the $p$-complete category, we observe that the forgetful functors from the $p$-completed categories to the uncompleted categories are conservative and reflect connectivity and coconnectivity, reducing the statement to the uncompleted case.
\end{proof}

\begin{lem} \label{lem:fil-colim-bounded}
  A filtered colimit of cyclotomic spectra bounded in the range $[a,b]$
  is itself bounded in the range $[a,b+3]$.
\end{lem}

\begin{proof}
  Recall from \cite[Thm. 9]{antieau-nikolaus} that there is a $t$-exact isomorphism between the category of bounded below cyclotomic spectra and the category of $V$-complete bounded below topological Cartier modules.
  The $t$-structure on topological Cartier modules is compatible with filtered colimits, so to prove the lemma it will suffice to show that the $p$-complete $V$-completion functor on bounded
  topological Cartier modules has $t$-amplitude $[0,3]$.

  Given an $M \in \mathrm{TCart}_p^\heartsuit$ we have a formula for the $V$-completion as
  $\lim_n M/V^n$ \cite[Prop. 3.22]{antieau-nikolaus}. Doing this in the $p$-complete category gives $\lim_{n,k} (M/p^k)/V^n$. 
  From this formula we can easily read off that the functor has $t$-amplitude $[0,\infty]$.
  A closer analysis of this formula in fact yields the desired conclusion, that $V$-completion has $t$-amplitude $[0,3]$ (see \cite[Prop. 6.5]{mathew2021k} and use the fact that $M/p^k$ is $p$-complete).
  \qedhere

\end{proof}

\subsubsection{Topological Hochschild homology}
\label{subsubsec:thh}

Many objects in $\CycSp$ and $\pgnsp{\langle p \rangle}$ are constructed via topological Hochschild homology.

\begin{dfn}
  We denote by 
  \[\mdef{\THH}\colon \Alg(\Sp) \to \CycSp\]
  the $p$-completed version of topological Hochschild homology, as defined in \cite{NS}. This functor is symmetric monoidal \cite[IV.2]{NS}.
\end{dfn}

We will often similarly consider $\TC^{-}, \TP,\TC,$ and $\TR^k$ as functors from $\Alg(\Sp)$ by precomposing with $\THH$.

\begin{dfn}
  We let
  \[ \mdef{\THH_{\hex}(-;-)} \colon \mathrm{BiMod} \to \pgnsp{\langle p \rangle} \]
  be the $p$-completed version of the functor 
  of \cite[Theorem 6.31]{krause2023polygonic}.

    This is a lax symmetric monoidal functor because every step of the construction of \cite[Theorem 6.31]{krause2023polygonic} is lax symmetric monoidal as a functor from $\mathrm{BiMod}$.
\end{dfn}


The following is a restatement of \cite[Remark III.1.5.]{NS}, where $\mathrm{Nm}_e^{C_p}$ refers to the Hill--Hopkins--Ravenel norm:

\begin{lem}\label{lem:norm}
  For $V \in \Sp$ there is a natural isomorphism
  $ \THH_{\hex}(\Ss; V) \cong \mathrm{Nm}_e^{C_p}(V) $
  in $\pgnsp{\lpr}$.
\end{lem}

We will denote by $\mdef{\Alg(\Sp)^{\mathrm{EQ}}}$ the category of functors to $\Alg(\Sp)$ from the category 
\[\mathrm{EQ}\coloneqq \cdot \rightrightarrows \cdot\]
in $\Alg(\Sp)$.  Such a diagram $R \rightrightarrows S$ gives $S$ the structure of an $R$-bimodule, and so there is a functor
\[\THH_{\hex}(\--;\--)\colon\Alg(\Sp)^{\EQ} \to \pgnsp.\]


There is the following basic compatibility:

\begin{lem} \label{lem:polygonic-compatibility}
  There is a commuting diagram of lax symmetric monoidal functors
  \[ \begin{tikzcd}
    \Alg(\Sp) \ar[r, "\THH"] \ar[d] & 
    \CycSp \ar[r, "\mathrm{res}_{\varphi}"] \ar[d, "\mathrm{res}_{\hex}"] &
    \Sp^{\Delta^1} \\
    \Alg(\Sp)^{\EQ} \ar[r, "\THH_{\hex}"] &   
    \pgnsp{\lpr} \ar[ur, swap,"\mathrm{res}_{\varphi}"] &     
  \end{tikzcd} \]
  where the left vertical map sends an $\E_1$-algebra $R$ to the constant diagram $R\rightrightarrows R$.
\end{lem}

\begin{proof}\todo{check this}
  The commutation of the square is immediate from the definitions of the cyclotomic \cite[III.2, Corollary III.3.8]{NS} and polygonic \cite[Proposition 6.30]{krause2023polygonic}
  Frobenius maps, upon noting that the inclusion $\Delta^{op} \to \Lambda_{\infty}^{op}$ is cofinal \cite[Proposition B.5]{NS}, and that the composite $$\Delta^{op} \to \Lambda_{\infty}^{op} \to \Lambda_p^{op} \to \mathrm{Free}(C_p)\times_{\mathrm{Fin}}(\mathrm{\Sp}^{\otimes})_{\mathrm{act}}$$ where the third map is as in \cite[III.2]{NS}, agrees with the map of \cite[Construction 6.30]{krause2023polygonic}.
 
  The triangle commutes by the definition of $\res_{\hex}$.
\end{proof}

\begin{lem} \label{lem:tenspr_with_finite}
  Suppose we are given $(A,B) \in \Alg(\Sp)^{\EQ}$ and
  $V \in \Alg(\Sp)$ such that the underlying spectrum of $V$
  is a dualizable  $\Ss$-module.
  The natural map
  \[ \mathrm{res}_{\varphi}\THH_{\hex}(\Ss; V) \otimes \mathrm{res}_{\varphi}\THH_{\hex}(A; B) \to \mathrm{res}_{\varphi}\THH_{\hex}(A; V \otimes B),\]
  coming from the lax symmetric monoidal structure on $\mathrm{res}_{\varphi}\THH_{\hex}(-;-)$,
  is an isomorphism.
\end{lem}

\begin{proof}
  Using the fact that $\THH(-;-)$ is symmetric monoidal we can reduce to showing that the natural map
  \[ \THH_{\hex}(\Ss; V)^{tC_p} \otimes \THH_{\hex}(A; B)^{tC_p} \to (\THH_{\hex}(\Ss; V) \otimes \THH_{\hex}(A; B))^{tC_p} \]
  is an isomorphism.
  Applying the identification $\THH_{\hex}(\Ss; V) \cong V^{\otimes p} \in \Sp^{BC_p}$ (\Cref{lem:norm}) and denoting $X\coloneqq \THH_{\hex}(A; B) \in \Sp^{BC_p}$
   we can a rewrite the above map as 

  \[(V^{\otimes p})^{tC_p}\otimes X^{tC_p} \to (V^{\otimes p}\otimes X)^{tC_p}.\]

  Fixing $X$, we can consider the map as a natural transformation of functors
  $\Sp^{\dualz}\to \Sp$.
  We claim that both the source and target of this natural transformation are exact. Indeed for the source it follows from \cite[Proposition III.1.1]{NS}, and for the target we may use the proof of \cite[Proposition III.1.1]{NS}.

  Since $\Sp^{\dualz}$ is the thick subcategory generated by $\SP$, this reduces the statement to the case $V = \SP$, where
  the claim follows from the Segal conjecture.   
\end{proof}

\input{QBoundedness.tex}

\input{Bokstedt.tex}

\subsection{Almost compact cyclotomic spectra}
\label{subsec:almost perfect}

Although the category of cyclotomic spectra has very few compact objects,
when equipped with the cyclotomic $t$-structure it has a rich collection of
almost compact objects (in the sense of \Cref{subsubsec:almost-compact}).
The main result of this subsection, \Cref{prop:key-finiteness},
shows that a cyclotomic spectrum is almost compact whenever its underlying spectrum is almost compact.
Before we can prove this proposition we will need some preparatory lemmas.

\begin{lem} \label{lem:easy-finiteness-boost}
  Give $\Sp^{B\T}$ the pointwise $t$-structure.
  If the underlying spectrum of $X \in \Sp^{B\T}$ is almost compact,
  then $X$ is almost compact as an object of $\Sp^{B\T}$.
\end{lem}

\begin{proof}
  Let $F : K \to \Sp_{[c,b]}^{B\T}$ be a filtered diagram.
  Then we have isomorphisms
  \[ \begin{tikzcd}
    \colim_{k \in K} \map_{\T}(X, F(k)) \ar[r] \ar[d, "\cong"] &
    \map_{\T}(X, \colim_{k \in K} F(k)) \ar[d, "\cong"] \\
    \colim_{k \in K} \map(X, F(k))^{h\T} \ar[r] \ar[d, "\cong"] &
    \map(X, \colim_{k \in K} F(k))^{h\T} \ar[d, "\cong"] \\    
    \left( \colim_{k \in K} \map(X, F(k)) \right)^{h\T} \ar[r, "\cong"] &
    \left( \colim_{k \in K} \map(X,  F(k)) \right)^{h\T}.
  \end{tikzcd} \]  
  The key observation here is that the objects $\map(X, F(k))$
  are uniformly bounded above and 
  $(-)^{h\T}$ commutes with filtered colimits of spectra that are uniformly bounded above.
\end{proof}

\begin{cnstr} \label{cnstr:retract-r}
  Let $c \leq b$, $m \geq m_p^{\E_1}$ and
  $e \geq b+2(b-c+1)m$ be integers.  
  Let $j$ be the truncation natural transformation
  \[ j\colon  (-) \Rightarrow \tau_{\leq e}(-) \colon  \Mod(\CycSp_{[c,b]}; \Ss/p^m) \to \Sp^{B\T}. \]
  whiskering along 
  \[ (-)^{tC_p}\colon \Sp^{B\T} \to \Sp^{B\T}\]
  We get a natural transformation 
  \[j^{tC_p}\colon  (-)^{tC_p} \Rightarrow (\tau_{\leq e}(-))^{tc_p.}\]
  We will construct a natural $\T$-equivariant retraction $r$ of $j^{tC_p}$.\todo{restate this as a proposition that the fiber is null as a functor.}
\end{cnstr}

\begin{proof}[Details.]
  Let $i$ be the corresponding covering natural transformation
  $ i \colon  \tau_{>e}(-) \Rightarrow (-) $ obtained as the fiber of $j$.
  We will construct $r$ by providing a nullhomotopy of $i^{tC_p}$.

  Let $R \coloneqq \tau_{\leq b-c}^{\mathrm{cyc}}(\Ss/p^m)$
  and let $d \coloneqq (b-c+1)m$.
  Using Lemmas \ref{lem:pn-cb-to-fn}, \ref{lem:fn-to-tn1} and \Cref{rmk:exp-as-a1}
  we pick a nullhomotopy $\epsilon$ of the
  $\T$-equivariant $R^{tC_p}$-linear map
  $a_{(1)}^d \colon \Sigma^{-d(1)}R^{tC_p} \to R^{tC_p}$.
  As any object of $\Mod(\CycSp_{[c,b]}; \Ss/p^m)$ is naturally
  an $R$-module in cyclotomic spectra
  we may construct the following diagram of $\T$-spectra
  which is natural in $X \in \Mod(\CycSp_{[c,b]}; \Ss/p^m)$.
  \[ \begin{tikzcd}[column sep=large, row sep=huge]
    \tau_{>e}X
    \ar[r, "i"] &
    X
    \ar[r, "\varphi"]
    \ar[d, "a_{(1)}^d"] &
    X^{tC_p}
    \ar[r] 
    \ar[d, "a_{(1)}^d"] &
    R^{tC_p} \otimes X^{tC_p}
    \ar[d, "a_{(1)}^d \otimes X^{tC_p}"', ""{name=L, above}]
    \ar[d, bend left=90, "0", ""{name=R, below}] \\
    & \Sigma^{d(1)} X
    \ar[r, "\varphi"] & 
    \Sigma^{d(1)} X^{tC_p} &
    \Sigma^{d(1)} R^{tC_p} \otimes X^{tC_p}
    \ar[l, "\mu"]
    \ar[from=L, to=R, "\epsilon"', Leftrightarrow]
  \end{tikzcd} \]
  From this diagram we obtain a nullhomotopy of the
  natural transformation $\varphi \circ a_{(1)}^d \circ i$.
  
  On the other hand, since every $X \in \Mod(\CycSp_{[c,b]}; \Ss/p^m)$
  satisfies $\Seg{b}$ (by \Cref{prop:AN_bound_to_segal})
  and the source of $i$ is pointwise $(e+1)$-connective
  it follows that the nullhomotopy of $\varphi \circ a_{(1)}^d \circ i$ can be uniquely lifted
  to a nullhomotopy of $a_{(1)}^d \circ i$. 
  Apply $(-)^{tC_p}$ to this nullhomotopy,
  and using the fact that $a_{(1)}^{tC_p}$ is an isomorphism,
  we obtain the desired nullhomotopy of $i^{tC_p}$.
\end{proof}

\begin{lem} \label{lem:tcp-commute}
  The functor
  \[ (-)^{tC_p} : \CycSp \to \Sp^{B\T} \]
  commutes with colimits of uniformly bounded filtered diagrams of cyclotomic spectra.
\end{lem}

\begin{proof}
  Pick an $m \geq m_p^{\E_1}$,
  As forgetting to $\Sp$ and tensoring with $\Ss/p^m$
  is conservative and commutes with colimits it will suffice to instead show that
  \[ X \mapsto (\Ss/p^m \otimes X)^{tC_p} \]
  from $\CycSp$ to $\Sp$ commutes with uniformly bounded filtered colimits.
  
  Let $F\colon K \to \CycSp_{[c,b]}$ be a filtered diagram with colimit\footnote{Note that the colimit is taken in $\CycSp$.} $Y$.
  As the $F(k)$ are all cyclotomically bounded in the
  range $[c,b]$,
  using \Cref{lem:fil-colim-bounded} we see that $Y$ is bounded in the range $[c,b+3]$.
  Using the retraction from \Cref{cnstr:retract-r} with $e \gg 0$
  we learn that the coassembly map
  \[ \colim_{k \in K} (\Ss/p^m \otimes F(k))^{tC_p} \to (\Ss/p^m \otimes Y)^{tC_p} \]
in $\Sp$ is a retract of
  \[ \colim_{k \in K} (\tau_{\leq e}(\Ss/p^m \otimes F(k)))^{tC_p} \to (\tau_{\leq e}(\Ss/p^m \otimes Y))^{tC_p}. \]
  The latter map is an isomorphism since $(-)^{tC_p}$ commutes with uniformly bounded filtered colimits of spectra (for $(-)^{hC_p}$ this is clear and $(-)_{hC_p}$ is a colimit).
  As a retract of an isomorphism is an isomorphism we may conclude.
\end{proof}


\begin{cnstr} \label{cnstr:JY}
  Let $c \leq b$, $m \geq m_p^{\E_1}$ and
  $e \geq b+2(b-c+1)m$ be integers.      
  Using the retraction $r$ from \Cref{cnstr:retract-r}
  we construct the following diagram
  of functors $\CycSp^{\op} \times\ \Mod(\CycSp_{[c,b]}; \Ss/p^m) \to \Sp $\todo{change this to spectra rather than spaces. make sure that doesn't create issues.}
  \[ \begin{tikzcd}[sep=huge] 
    \map_{\T}(-, -)
    \ar[r, "g \mapsto g^{tC_p} \circ \varphi"]
    \ar[d, "g \mapsto j \circ g"] &
    \map_{\T}(-, (-)^{tC_p})
    \ar[dr, equal]
    \ar[d, "g \mapsto j^{tC_p} \circ g"] & \\
    \map_{\T}(-, \tau_{\leq e} (-))
    \ar[r, "g \mapsto g^{tC_p} \circ \varphi"]  &
    \map_{\T}(-, (\tau_{\leq e} (-))^{tC_p})
    \ar[r, "g \mapsto r \circ g"]  &
    \map_{\T}(-, (-)^{tC_p}). 
  \end{tikzcd} \] 

  Let $J$ denote the composite map along the bottom row
  and let $\overline{J}$ denote the further composite
  \[ \map_{\T}(-,\tau_{\leq e} (-)) \xrightarrow{J} \map_{\T}(-, (-)^{tC_p}) \to \map_{\T}(-, \tau_{\leq e} ((-)^{tC_p})).\qedhere \]
\end{cnstr}


\begin{lem} \label{lem:cycsp-map-formula}
  Let $c \leq b$, $m \geq m_p^{\E_1}$ and
  $e \geq b+2(b-c+1)m$ be integers.
  There is an isomorphism of functors
  \[ \CycSp^{\op} \times\ \Mod(\CycSp_{[c,b]}; \Ss/p^m) \to \Sp \]\todo{this can be phrased in terms of right adjoint}
  between the functor sending $(X,Y)$ to $\map_{\CycSp}(X, Y)$ and the functor sending $(X,Y)$ to the equalizer
  \[ \begin{tikzcd}[sep=huge]
    \mathrm{Eq}\Bigg( \map_{\T}(X,\tau_{\leq e} Y) \arrow[r,yshift=-1ex,"\overline{J}"'] \arrow[r,yshift=1ex, "g \mapsto (\tau_{\leq e}\varphi) \circ g"] & \map_{\T}(X, \tau_{\leq e}( Y^{tC_p} ) ) \Bigg)
  \end{tikzcd} \]
  where $\overline{J}$ is the natural transformation from \Cref{cnstr:JY}.  
\end{lem}

\begin{proof}
    
  Let $Z(-,-)$ denote the functor that we wish to compare with
  $\map_{\CycSp}(-,-)$.
  We have a commuting square 
  \[ \begin{tikzcd}
    \map_{\T}(-, -) \ar[d, "\varphi \circ -"] \ar[r] &
    \map_{\T}(-, \tau_{\leq e} (-)) \ar[d, "(\tau_{\leq e}\varphi) \circ -"] \\
    \map_{\T}(-, (-)^{tC_p}) \ar[r] &
    \map_{\T}(-, \tau_{\leq e} ((-)^{tC_p}))
  \end{tikzcd} \]
  and 
  \Cref{cnstr:JY} provides us with a commuting square
  \[ \begin{tikzcd}
    \map_{\T}(-, -) \ar[d, "g \mapsto g^{tC_p} \circ \varphi"] \ar[r] &
    \map_{\T}(-, \tau_{\leq e} (-)) \ar[d, "\overline{J}"] \\
    \map_{\T}(-, (-)^{tC_p}) \ar[r] &
    \map_{\T}(-, \tau_{\leq e} ((-)^{tC_p})).
  \end{tikzcd} \]
  
  Note that the horizontal arrows in the two squares agree.
  Taking differences of the vertical maps we construct a 3x3 diagram of fiber sequences
  which includes the desired comparison map
  \[ \begin{tikzcd}
    D(-,-) \ar[d] \ar[r] &
    \map_{\CycSp}(-,-) \ar[r] \ar[d] &
    Z(-,-) \ar[d] \\
    \map_{\T}(-, \tau_{> e}(-)) \ar[r] \ar[d, "H"] &
    \map_{\T}(-, -)
    \ar[d, "g \mapsto (\varphi \circ g) - (g^{tC_p} \circ \varphi)"] \ar[r] &
    \map_{\T}(-, \tau_{\leq e} (-))
    \ar[d, "((\tau_{\leq e}\varphi) \circ -) - \overline{J}"] \\
    \map_{\T}(-, \tau_{> e}((-)^{tC_p})) \ar[r]  &
    \map_{\T}(-, (-)^{tC_p}) \ar[r] &
    \map_{\T}(-, \tau_{\leq e} ((-)^{tC_p}))
  \end{tikzcd} \]
  We show the comparison map is an isomorphism by showing that
  the map $H$ is an isomorphism.

  Taking fibers in the two squares above we can read off that
  $H$ is the difference of two maps $H_1$ and $H_2$ coming from the respective squares.
  We analyze $H_1$ and $H_2$ separately.
  First we show $H_1$ is an isomorphism.
  Examining the first square we see that $H_1 \cong ((\tau_{ > e}\varphi) \circ -)$.
  From \Cref{prop:AN_bound_to_segal} we know that the target satisfies $\Seg{b}$
  and therefore $(\tau_{ > e}\varphi \circ -)$ is an isomorphism since $e>b$. 
  
  Next we show $H_2$ is null.
  From the diagram constructing $J$ in \Cref{cnstr:JY}
  and the definition of $\overline{J}$ we obtain a refinement of
  the second square to a diagram
  \[ \begin{tikzcd}
    \map_{\T}(-, -) \ar[d, "g \mapsto g^{tC_p} \circ \varphi"'] \ar[r] &
    \map_{\T}(-, \tau_{\leq e} (-)) \ar[d, "\overline{J}"] \ar[dl, "J"] \\
    \map_{\T}(-, (-)^{tC_p}) \ar[r] &
    \map_{\T}(-, \tau_{\leq e} ((-)^{tC_p})).
  \end{tikzcd} \]
  It follows that the induced map on horizontal fibers is nullhomotopic.
  The difference $H = H_1 - H_2$ is thus an isomorphism and we may conclude that the term $D(-,-)$ is zero.
\end{proof}

\begin{prop} \label{prop:key-finiteness}
  Let $X$ be a bounded below, $p$-nilpotent cyclotomic spectrum.
  If the underlying spectrum of $X$ is almost compact,
  then $X$ is almost compact as a cyclotomic spectrum.
\end{prop}

\begin{proof}
  Without loss of generality we may assume that $X$ is connective.
  In order to prove the proposition we must show that
  for every filtered diagram $F : K \to \CycSp_{[c,b]}$ with colimit\footnote{Note that the colimit is taken in $\CycSp$, not $\CycSp_{[c,b]}$.} $Y$ the natural comparison map
  \[ \colim_{k \in K} \Map_{\CycSp}(X, F(k)) \to \Map_{\CycSp}(X, Y) \]
  is an isomorphism.
  As the $F(k)$ are all cyclotomically bounded in the range $[c,b]$,
  using \Cref{lem:fil-colim-bounded} we see that $Y$ is bounded in the range $[c,b+3]$.
    
  Pick an $m \geq m_p^{\E_1}$ such that $p^m$ acts by zero on $X$.
  Then we have $(\Ss/p^m)^{\vee} \otimes X \cong X \oplus \Sigma X$
  and as a consequence it suffices to prove the proposition for $(\Ss/p^m)^{\vee} \otimes X$.
  Dualizing the copy of $(\Ss/p^m)^{\vee}$ over to the other side we return to considering $X$,
  but are now free to assume that $F^{\rhd}$ takes values in $\Mod(\CycSp_{[c,b+4]}; \Ss/p^m)$.

  Using \Cref{lem:easy-finiteness-boost}
  and the hypothesis that the underlying spectrum of $X$ is almost compact
  we learn that the underlying $\T$-spectrum of $X$ is almost compact
  (with respect to the pointwise $t$-structure on $\T$-spectra).


  We analyze the comparison map using the formulas for maps of cyclotomic spectra from
  \Cref{lem:cycsp-map-formula}.
  Taking $e = b+4 + 2(b-c+5)m$ we obtain the following diagram
  where each column is a fiber sequence:
  \[ \begin{tikzcd}
    \colim_{k \in K} \Map_{\CycSp}(X, F(k)) \ar[r] \ar[d] &
    \Map_{\CycSp}(X, Y) \ar[d] \\
    \colim_{k \in K} \Map_{\T}(X, \tau_{\leq e} F(k)) \ar[r] \ar[d] &
    \Map_{\T}(X, \tau_{\leq e} Y) \ar[d] \\
    \colim_{k \in K} \Map_{\T}(X,  \tau_{\leq e}( F(k)^{tC_p} )) \ar[r] &
    \Map_{\T}(X, \tau_{\leq e}( Y^{tC_p} )).
  \end{tikzcd} \] 
  We will prove the proposition by showing that
  the lower two horizontal maps are isomorphisms.
  
  Using the facts that
  (i) $X$ is connective,
  (ii) $X$ is almost compact as a $\T$-spectrum and
  (iii) the truncation functors on $\Sp^{B\T}$ commute with filtered colimits
  of $\Ss/p^m$-modules,\footnote{The restriction to torsion $\T$-spectra may seem odd, but recall our convention that $\Sp^{B\T}$ consists of $p$-complete spectra with $\T$-action.} we have isomorphisms
  \begin{align*}
    \colim_{k \in K} &\Map_{\T}(X, \tau_{\leq e} F(k)) 
    \cong \colim_{k \in K} \Map_{\T}(X, \tau_{[0,e]} F(k))
    \cong \Map_{\T}(X, \colim_{k \in K} \tau_{[0,e]} F(k)) \\
    &\cong \Map_{\T}(X,  \tau_{[0,e]} \colim_{k \in K} F(k)) \cong  \Map_{\T}(X,  \tau_{[0,e]}Y)
    \cong \Map_{\T}(X, \tau_{\leq e} Y).
  \end{align*}
  Similarly, we have isomorphisms
  \begin{align*}
    \colim_{k \in K} &\Map_{\T}(X,  \tau_{\leq e}( F(k)^{tC_p} ))
    \cong \colim_{k \in K} \Map_{\T}(X,  \tau_{[0,e]}( F(k)^{tC_p} )) \\
    &\cong  \Map_{\T}(X, \colim_{k \in K} \tau_{[0,e]}( F(k)^{tC_p} )) 
    \cong  \Map_{\T}(X,  \tau_{[0,e]}( \colim_{k \in K} F(k)^{tC_p} )) \\
    &\cong  \Map_{\T}(X,  \tau_{[0,e]}( (\colim_{k \in K} F(k) )^{tC_p} ))
    \cong  \Map_{\T}(X,  \tau_{\leq e}( Y^{tC_p} ))
  \end{align*}
  where the key step is using \Cref{lem:tcp-commute} to move the filtered colimit past the $(-)^{tC_p}$.
\end{proof}

The converse to \Cref{prop:key-finiteness} is more straightforward to prove.

\begin{lem} \label{lem:coh-implies-coh}
  Let $X$ be a cyclotomic spectrum.
  If $X$ is almost compact as a cyclotomic spectrum,
  then its underlying spectrum is almost compact as well.
\end{lem}
  
\begin{proof}
  Let $\Ss^{\T} \otimes -$ denote the functor
  sending a spectrum $Y\in \Sp$ to the cyclotomic spectrum
  \[ \varphi: \Ss^\T \otimes Y \to 0  \]
  using that $(\Ss^{\T}\otimes Y)^{tC_p} = 0$ (\Cref{lem:a1tcp}). 
  Computing maps out to such objects we find
  \[ \Map_{\CycSp}(X, \Ss^{\T} \otimes Y) \cong \Map_\T(X, \Ss^\T \otimes Y) \cong \Map(X,Y) \]
  that $\Ss^{\T} \otimes -$ is right adjoint to the forgetful functor
  $ \CycSp \to \Sp $.
  Using the formula for mapping out to $\Ss^\T \otimes Y$ and the fact that
  $\Ss^{\T}$ is $(-1)$-connective we can read off that
  the functor $\Ss^\T \otimes -$ has $t$-amplitude $[-1,0]$.  

  
  If the right adjoint to a functor $F$ commutes with colimits
  and has bounded $t$-amplitude, then $F$ sends almost compact objects to almost compact objects.
  Applying this to the forgetful functor $\CycSp \to \Sp$ we may now conclude.
\end{proof}

We end the section by using the theory we have developed to prove that $\crTR$, the object corepresenting $\TR$, has bounded cyclotomic $t$-amplitude.

\begin{lem} \label{lem:t-amplitude-repTR}
  For any set $S$ the functor
  \[ (\oplus_{S}\crTR) \otimes - : \CycSp_{+} \to \CycSp_{+} \]
  has $t$-amplitude bounded in the range $[0,4]$.
\end{lem}
\todo{comment: tamp [0,3] if you restrict to p-nil.}

\begin{proof}
  Let $Y = \oplus_{S}\crTR$.
  The lower bound in the lemma follows from the fact that
  the underlying spectrum of $Y$ is $\geq 0$.
  We prove the upper bound in several stages
  beginning with the following subclaim:
  \begin{itemize}
  \item[(A)] Given an $X \in (\Sp^{B\T})_{\leq b}$ on which $p^m$ acts by zero,
    $Y \otimes X \in (\Sp^{B\T})_{\leq b+1} $.
  \end{itemize}
  As the forgetful functor $\Sp^{B\T}\to \Sp$ is $t$-exact, conservative, and symmetric monoidal, it suffices to prove this $t$-amplitude bound in $\Sp$. As an object of $\Sp$, we have (see \Cref{lem:describe_crtr})
  \[ Y = \bigoplus_{S}\Big((\bigoplus_{i=0}^{\infty}\SP) \bigoplus \Sigma (\bigoplus_{i=0}^{\infty}\SP)\Big).\]
  It will thus suffice to show that $(\oplus_I \Ss) \otimes X$ is $b$-truncated.
  This follows from the fact that a sum of $p^m$-torsion objects is
  $p^m$-torsion (and therefore $p$-complete).
  \begin{itemize}
  \item[(B)] If $X$ is cyclotomically bounded, then
    $\fib(\varphi_{Y \otimes X}) \cong Y \otimes \fib(\varphi_X)$.
  \end{itemize}
  The cyclotomic Frobenius of $Y \otimes X$
  is given by the composite of maps in $\Sp^{B\T}$
  \[ Y \otimes X \xrightarrow{Y\otimes \varphi_X} Y\otimes X^{tC_p} \xrightarrow{\varphi_Y\otimes X^{tC_p}} Y^{tC_p} \otimes X^{tC_p} \to (Y \otimes X)^{tC_p}. \]
  From \Cref{lem:describe_crtr}, we know that the second map is an equivalence and the underlying $C_p$-spectrum of $Y$ is isomorphic to $\bigoplus_S\big(\Ss\otimes {\T} \oplus \bigoplus_{i=1}^\infty (\Ss^0 \oplus \Ss^{1})\big)$. Using \Cref{lem:a1tcp} and \cite[Corollary 6.7]{yuan2023integral}, we may rewrite the third map above as 
  \[ \left(\bigoplus_{S \times \mathbb{N}} (\Ss^0 \oplus \Ss^{1})  \right) \otimes X^{tC_p} \to \left( \left(\bigoplus_{S \times \mathbb{N}} (\Ss^0 \oplus \Ss^{1}) \right)\otimes X\right)^{tC_p}. \]
  Since $X$ is cyclotomically bounded, after writing the infinite sum as a filtered colimit over the finite sub-sums it follows from \Cref{lem:tcp-commute} that this map is an isomorphism. We now obtain the desired isomorphism.
  \begin{itemize}
  \item[(C)] Suppose we are given an $X \in {\CycSp}^\heartsuit$ on which $p$ acts by zero,
    then $Y \otimes X$ is cyclotomically $3$-truncated.
  \end{itemize}
  From \Cref{prop:omnibus}(3,4) we learn that $X$ satisfies $\ScanVn{2}$ and $\Seg{0}$.
  In particular there exist some $d$ such that the composite 
  \[\tau_{>2}X \to X\xrightarrow{a_{(1)}^{d}} \Sigma^{d(1)}X\]
  is null. 
  Using (A) we learn that
  $Y \otimes \tau_{\leq 2}X \in (\Sp^{B\T})_{\leq 3}$.  
  The diagram below now witnesses that $Y \otimes X$ satisfies $\ScanVn{3}$.
  \[ \begin{tikzcd}
    & {Y\otimes \tau_{>2}X} \\
    {\tau_{>3}(Y\otimes X)} & {Y\otimes X} & {\Sigma^{d(1)}Y\otimes X} \\
    & {Y \otimes \tau_{\leq 2}X}
    \arrow[from=2-2, to=3-2]
    \arrow[from=1-2, to=2-2]
    \arrow[from=2-1, to=2-2]
    \arrow["0"', from=2-1, to=3-2]
    \arrow[dashed, from=2-1, to=1-2]
    \arrow[from=2-2, to=2-3]
    \arrow["0"', from=1-2, to=2-3]
  \end{tikzcd}\]

  Using (A) again we find that
  $Y \otimes \fib(\varphi_X) \in (\Sp^{B\T})_{\leq 1}$.
  Rewriting this with the isomorphism from (B) we conclude that $Y \otimes X$
  satisfies satisfies $\Seg{1}$.
  Applying \Cref{prop:omnibus}(1,2) we learn that $Y \otimes X$
  is cyclotomically $3$-truncated as desired.
  
  We are now ready to complete the proof.
  Suppose we are given a cyclotomic spectrum $X \in [c,b]$,
  we would like to show that $Y \otimes X$ is cyclotomically $(b + 4)$-truncated.
  Without loss of generality we may reduce to the case where $X$ is in the heart.
  As we work in $p$-complete cyclotomic spectra we have an isomorphism
  $ Y \otimes X \cong \lim (Y \otimes X/p^k) $.
  Using this, and the fact that $X/p^k$ is a $k$-fold extension of copies of $X/p$,
  we may reduce to showing that $Y \otimes X/p$ is cyclotomically $4$-truncated.  
  Finally, splitting $X/p$ up using the $t$-structure we reduce to claim (C).  
\end{proof}

%% file: QBoundedness.tex
\subsection{Quantitative forms of boundedness}
\label{subsec:implications}

In this subsection we explore a collection of ``boundedness'' conditions that a cyclotomic spectrum $X$ might satisfy. The results proved here are primarily used in \Cref{sec:tame} to prove that certain cyclotomic spectra are bounded in the cyclotomic $t$-structure.
This subsection can be viewed as a quantitative version of \cite[Section 3.5]{hahn2020redshift}.

\subsubsection{Canonical vanishing}

\begin{dfn}
  Suppose we are given an $X \in \Sp^{B\T}$ which is bouned below.
  We say that: 
  \begin{enumerate}
  \item  $X$ satisfies \mdef{weak canonical vanishing} with parameter $b$
    if for each $1\leq j\leq \infty$ and $m > b$ the canonical maps
    \[ \pi_mX^{hC_{p^j}} \to \pi_mX^{tC_{p^j}} \]
    are zero. We abbreviate this by saying that $X$ satisfies $\WcanV{b}$.
  \item $X$ satisfies \mdef{strong canonical vanishing} with parameter $b$
    if there exists some $d\geq 0$ for which the composition
    \[ \tau_{> b}X \to X \xrightarrow{a_{(1)}^d} \Sigma^{d(1)} X \]
    is $\T$-equivariantly null.
    We abbreviate this by saying that $X$ satisfies $\ScanVn{b}$.\qedhere
  \end{enumerate}
\end{dfn}
Note that $\WcanV{b}\implies \WcanV{b+1}$ and similarly $\ScanVn{b}\implies\ScanVn{b+1}$.

\begin{lem}\label{lem:srw_canv}
  Let $X \in \Sp^{B\T}$ be bounded below.
  For all $b \in \ZZ$,
  if $X$ satisfies ${\ScanVn{b}}$ then $X$ satisfies ${\WcanV{b}}$.
\end{lem}

\begin{proof}
  As $X$ satisfies $\ScanVn{b}$ we may find a $d$ such that the composite
  \[ \tau_{> b}X \to X \xrightarrow{a_{(1)}^d} \Sigma^{d(1)} X \]
  is $\T$-equivariantly null.
  We learn that $X$ satisfies $\WcanV{b}$ from the following diagram:  
  \[ \begin{tikzcd}
    \tau_{> b}( X^{hC_{p^j}} ) \ar[r]  &
    (\tau_{> b} X)^{hC_{p^j}} \ar[r] \ar[d] &
    (\tau_{> b} X)^{tC_{p^j}} \ar[d] \ar[dr, "0"] & \\
    & X^{hC_{p^j}} \ar[r, "\mathrm{can}"] &
    X^{tC_{p^j}} \ar[r, "\cong", "(a_{(1)}^d)^{tC_{p^j}}"'] &
    (\Sigma^{d(1)}X)^{tC_{p^j}}, 
  \end{tikzcd} \]
  where the bottom left horizontal  map is an isomorphism by \Cref{lem:a1tcp}.
\end{proof}

\subsubsection{Exponent of Nilpotence}

We will also need control over exponents of nilpotence in the sense of \cite{AkhilDescent}.

\begin{rec}\label{rec:exponents1}
  In \cite{AkhilDescent}, in the process of analyzing descent,
  Mathew sets out the basic properties of the
  \emph{exponent of nilpotence} of an object with respect to a ring.
  We briefly recall some of these properties.
  
  Let $\CC$ be a stable, exactly\footnote{i.e the tensor product is exact in each variable} symmetric monoidal category. Let $B$ be an $\E_1$-algebra in $\CC$, and
  let $I_B \to \one \to B$ be the associated fiber sequence.  
  Given an $X \in \CC$, the exponent of nilpotence of $X$ with respect to $B$, denoted $\mdef{\exp_B(X)}$, is the smallest $m$ such that either of following (equivalent) conditions hold:
  \begin{enumerate}
  \item $X$ is a retract of an object with an $m$-step resolution by $B$-modules, or 
  \item The map $I_B^{\otimes m} \to \one$ becomes null upon tensoring with $X$.
  \end{enumerate}
  By \cite[Proposition 2.3]{AkhilDescent} the function $\exp_B(-)$ is sub-additive in cofiber sequences, e.g. given a cofiber sequence $X \to Y \to Z$ we have
  \[ \exp_B(Y) \leq \exp_B(X) + \exp_B(Z). \qedhere\]
\end{rec}

\begin{rmk} \label{rmk:exp-as-a1}
  In the category of $\T$-spectra we have a cofiber sequence
  \[ \Ss^{-(1)} \xrightarrow{a_{(1)}} \Ss \to \Ss^{\T},\]
  and so we find for any $\T$-spectrum $X$
  that $\exp_{\Ss^{\T}}(X) \leq t$ if and only if $a_{(1)}^t \otimes X$ is null.
\end{rmk}

\begin{lem} \label{lem:quant-fp-to-tate1}\hfill
  \begin{enumerate}
  \item Let $\CC$ be a stable, exactly symmetric monoidal category, and let
 $R$ and $B$ be $\E_1$-algebras in $\CC$.
    Then, for every $R$-module $M$, 
    \[ \exp_B(M) \leq \exp_B(R). \]    
  \item Let $F : \CC \to \DD$ be an exact lax symmetric monoidal functor
    between stable categories.
    Let $A \in \Alg(\CC)$ and $B \in \Alg(\DD)$ be $\E_1$-algebras.
    Then, for each $M \in \CC$, 
    \[ \exp_{B}(F(M)) \leq \exp_A(M) \cdot \exp_B(F(A)). \]
  \item Let $\CC$ be a stable presentably symmetric monoidal category with compatible $t$-structure, and $ \iota \colon \Sp \to \cC$ the unique functor in $\CAlg(\mathrm{Pr}^L)$.
    If $X \in \CC_{[c,b]}$  and $p^m$ acts by zero on $X$, then
    \[ \exp_{\iota \F_p}(X) \leq (b-c+1)m. \]
  \end{enumerate}
\end{lem}

\begin{proof}
  Part (1).
  Using the $R$-module structure we write $M$ as a retract of $M \otimes R$.
  Now we observe that, if $R \otimes \left( I_B^{\otimes m} \to \one \right)$ is null,
  then $M \otimes R \otimes \left( I_B^{\otimes m} \to \one \right)$ is null as well.
  
  Part (2).
  Resolving $M$ by $A$-modules, and
  using exactness of $F$ and subadditivity of $\exp_{B}(-)$,
  we reduce to the case where $M$ is an $A$-module.
  Using the lax symmetric monoidal structure on $F$
  we can give $F(M)$ an $F(A)$-module structure.
  Using the conclusion of (1), it now suffices to prove (2) in the case where $M=A$,
  but in this case the lemma reduces to the equality
  $\exp_B(F(A)) = \exp_B(F(A))$.

  Part (3).
  Resolving $X$ by its $t$-structure homotopy groups
  we reduce to the case $X$ is in the heart. 
  Using the fact that $\CC^\heartsuit$ is a $1$-category,
  we can give $X$ the structure of a $\iota(\Z/p^m)$-module.
  By (1) we are then reduced to the case $X=\iota(\Z/p^m)$.
  By (2) we learn that 
  \[\exp_{\iota \Fp}(\iota(\Z/p^m))\leq \exp_{\Fp}(\Z/p^m)\cdot \exp_{\iota \Fp}(\iota \Fp) \leq m \cdot 1. \qedhere\]
\end{proof}

\begin{dfn}\label{dfn:conditions-3}
 Suppose we are given an $X\in \Sp^{B\T}$ which is bounded below.
  \begin{itemize}
  \item $X$ satisfies \mdef{Tate nilpotence} of exponent $d$ if
    $\exp_{\Ss^{\T}}(X^{tC_p}) \leq d$,
    where $X^{tC_p}$ is given its residual $\T/C_p$-action.
  \item $X$ satisfies \mdef{$\Fp$ nilpotence} of exponent $d$ if
    $\exp_{\F_p}(X) \leq d$. \qedhere
  \end{itemize}  
\end{dfn}

\begin{lem}\label{lem:fn-to-tn1}
  Given a bounded below $X \in \Sp^{B\T}$ we have
  \[ \exp_{\Ss^{\T}}(X^{tC_p}) \leq \exp_{\F_p}(X). \]
\end{lem}

\begin{proof}
  Applying \Cref{lem:quant-fp-to-tate1}(2)
  to the functor \[(-)^{tC_p}\colon \Sp_p^{B\T} \to \Sp_p^{B\T},\] it will suffice to show that
  \[ \exp_{\Ss^\T}(\F_p^{tC_p}) = 1. \]  
  Per \Cref{rmk:exp-as-a1} (and using that $\F_p^{tC_p}$ is an algebra)
  it will suffice for us to show that
  the image of $a_{(1)}$ in $\pi^{\T}_{-(1)}(\F_p^{tC_p})$ is zero. But the map $a_{(1)}:\FF_p^{tC_p} \to \Sigma^{(1)}\FF_p^{tC_p}$ is obtained by applying $(-)^{tC_p}$ to the map $a_{(p)}:\FF_p \to \Sigma^{(p)}\FF_p$, so it will suffice 
  to argue that $a_{(p)} \otimes \F_p$ is zero.
  The map $a_{(p)}$ can be written as the composite
  \[ \Ss \xrightarrow{a_{(1)}} \Ss^{(1)} \xrightarrow{\underline{p}} \Ss^{(p)},\]
  where $\underline{p}$ denotes the pointed suspension of
  the $p^{\mathrm{th}}$ power map on $\C^\times$ (with the weight $1$ $\T$-action on the source and the weight $p$ $\T$-action on the target).
  There are orientations
  $\F_p \otimes \Ss^2 \cong \F_p \otimes \Ss^{(1)}$ and $\F_p \otimes \Ss^2 \cong \F_p \otimes \Ss^{(p)}$, under which the map $\underline{p}$ is identified with $p$ (i.e zero).
\end{proof}

\subsubsection{Boundedness of cyclotomic spectra}

We now consider the above properties in the context of cyclotomic spectra, and use the notations 
${\WcanV{b}}$ and $\ScanVn{b}$, and refer to Tate and $\F_p$ nilpotence exponent 
to indicate properties of the underlying $\T$-spectrum.

\begin{dfn}\label{dfn:sc}
  For $X \in \CycSp$, and $b\in \ZZ$.
  We say that $X$  satisfies the \mdef{Segal condition}\footnote{This is because this condition is related to the Segal conjecture.} with parameter $b$ 
  if the fiber of the map $X \xrightarrow{\varphi} X^{tC_p}$ is $b$-truncated.
  We abbreviate this by saying that $X$ satisfies $\mdef{\Seg{b}}$.
\end{dfn}
  
\begin{lem} \label{lem:sc-tn-to-scv1}
  Let $X$ be a bounded below cyclotomic spectrum.
  If $X$ satisfies $\Seg{b}$ and $\exp_{\Ss^{\T}}(X^{tC_p}) \leq d$
  then $X$ satisfies $\ScanVn{b+2d}$.
\end{lem}

\begin{proof}
  Consider the following diagram of $\T$-spectra
  \[ \begin{tikzcd}[sep=large]
    & & \Sigma^{d(1)} \fib(\varphi) \ar[d] \\
    \tau_{> b+2d}X \ar[r] \ar[rru, dashed] & 
    X \ar[r, "a_{(1)}^d"] \ar[d, "\varphi"] &
    \Sigma^{d(1)} X \ar[d, "\varphi"] \\
    & X^{tC_p} \ar[r, "a_{(1)}^d"] &
    \Sigma^{d(1)} X^{tC_p}.
  \end{tikzcd} \]
  Using \Cref{rmk:exp-as-a1} and the hypothesis that $ \exp_{\Ss^{\T}}(X^{tC_p})\leq d$
  we know the bottom horizontal arrow is null
  and therefore that the dashed lift exists.
  On the other hand $\Sigma^{d(1)} \fib(\varphi)$ is $(b+2d)$-truncated 
  since $X$ satisfies $\Seg{b}$.
  The dashed lift is therefore null.
\end{proof}

\begin{lem} \label{lem:wcv-sc-to-cb}
  Let $X\in \CycSp_+$.
  If $X$ satisfies $\WcanV{b}$ and $\Seg{b}$,
  then $X \in \CycSp_{\leq b}$.
\end{lem}

\begin{proof}
  This lemma is a quantitative version of \cite[Thm. 3.3.2(f)]{hahn2020redshift}
  and we follow the argument there closely.
  A similar argument can be found in \cite{mathew2021k}.

  We have $\TR(X) \cong \lim_{k}\TR^k(X)$,
  where for every $k$ there is a pullback square
  \[ \begin{tikzcd}[sep=huge]
    \TR^{k+1}(X) \pullback \ar[rr] \ar[d] & & X^{hC_{p^k}}  \ar[d,"\can"] \\      
    \TR^{k}(X) \ar[r] & X^{hC_{p^{k-1}}} \ar[r,"\varphi^{hC_{p^{k-1}}}"] & X^{tC_{p^k}}. 
  \end{tikzcd} \]

  Thus, by induction on $k$ we learn that the fiber of the map $\TR^{k+1}(X) \to X^{hC_{p^k}}$ has a finite filtration
  with associated graded given by
  \[ \left\{ \fib \left( \varphi^{hC_{p^j}} : X^{hC_{p^j}} \to X^{tC_{p^{j+1}}} \right) \right\}_{j<k}.\]
  Since $X$ satisfies $\Seg{b}$,
  these are all $b$-truncated and we learn that
  the maps $\TR^{k+1}(X) \to X^{hC_{p^k}}$ induce an injection on $\pi_{m}(-)$
  for $m > b$.\footnote{See also \cite[Cor. II.4.9]{NS} for a weaker bound.}
  Additionally, as $X$ satisfies $\Seg{b}$
  the Frobenius map $X^{hC_{p^{k-1}}} \to (X^{tC_p})^{hC_{p^{k-1}}} \cong X^{tC_{p^k}}$
  induces an injection on $\pi_{m}(-)$  for $m > b$.
  
  Thus both horizontal maps in the square are injective on $\pi_{m}(-)$  for $m > b$. Now, since $X$ satisfies $\WcanV{b}$, the right vertical map in zero on $\pi_{m}(-)$  for $m > b$.
  We deduce that the map $\TR^{k+1}(X) \to \TR^{k}(X)$
  induces the zero map on $\pi_m(-)$ for $m > b$.
  Thus, the limit is $b$-truncated, proving the result.
\end{proof}

\begin{lem} \label{lem:pn-cb-to-fn}
  Suppose we are given an $X \in \CycSp_{[c,b]}$ on which $p^m$ acts by zero.
  Then $\exp_{\F_p}(X) \leq (b-c+1)m$.
\end{lem}

\begin{proof}  
  Applying \Cref{lem:quant-fp-to-tate1}(3) in the case $\CC={\CycSp}$ and $X=X$
  we learn that $X$ is $\F_p$-nilpotent of exponent $(b-c+1)m$
  as a cyclotomic spectrum.
  Using that the forgetful functor down to $\T$-spectra is symmetric monoidal we can apply \Cref{lem:quant-fp-to-tate1}(2) to conclude that
  the underlying $\T$-spectrum of $X$ is $\F_p$-nilpotent with the same bound.
\end{proof}

We also recall the following result:

\begin{prop}[{\cite[Prop 2.25]{antieau-nikolaus}}]\label{prop:AN_bound_to_segal}
  Any $X \in \CycSp_{\leq b}$ satisfies $\Seg{b}$.
\end{prop}

We collect the previous results to conveniently refer to later:

\begin{prop} \label{prop:omnibus}
  Let $X \in \CycSp$ be a bounded below cyclotomic spectrum.
  \begin{enumerate}
  \item If $X$ satisfies $\ScanVn{b}$, then $X$ satisfies $\WcanV{b}$.
  \item If $X$ satisfies $\WcanV{b}$ and $\Seg{b}$, then $X \in  \CycSp_{\leq b}$.
  \item If $X \in \CycSp_{\leq b}$, then $X$ satisfies $\Seg{b}$.
  \item If $X \in \CycSp_{[c,b]}$ and $p^m$ acts by zero on $X$,
    then $X$ satisfies $\ScanVn{b+2(b-c+1)m}$.
  \end{enumerate}
\end{prop}
\begin{proof}
  (1) is a restatement of \Cref{lem:srw_canv}.
  (2) is a restatement of \Cref{lem:wcv-sc-to-cb}.
  (3) is a restatement of \Cref{prop:AN_bound_to_segal}.
  (4) is obtained by combining
  Lemmas \ref{lem:pn-cb-to-fn}, \ref{lem:fn-to-tn1}, \ref{prop:AN_bound_to_segal} and \ref{lem:sc-tn-to-scv1}.
\end{proof}

%% file: Bokstedt.tex
\subsection{The B\"okstedt class}
\label{subsec:bokstedt1}

In this subsection we explore another boundedness condition, specific to rings in cyclotomic spectra: having a B\"okstedt class. As it turns out, a bounded below $R$ has a B\"okstedt class exactly when it is cyclotomically bounded and we use this as a hook for proving boundedness in later sections.\footnote{As a warning, the notion of B\"okstedt class we use in this paper is quite strong: for example, $\THH(\Z_p)/p$ does not have a B\"okstedt class because $\mathrm{TR}(\Z_p)/p$ is not bounded.}

\begin{dfn} \label{dfn:weak-bokstedt-class}
  Let $R$ be an $h\A_2$-ring (\Cref{dfn:ha2}) in cyclotomic spectra satisfying $\Seg{b}$.
  We say that an element $\mu \in \pi_{2p^k}R$ for $k\geq 0$ is a \mdef{B\"okstedt class} if
  \begin{enumerate}
  \item[(a)] both $\mu$ and $\varphi(\mu)$ are hcentral (\Cref{dfn:hcentral}),
  \item[(b)] $\mu$ is in the image of the  $\T$-transfer map $\Sigma R_{h\T} \to R$
  \item[(c)] $\varphi(u) \in \pi_{2p^k}(R^{tC_p})$ is a unit.\qedhere
  \end{enumerate}
\end{dfn}

\todo{added segal condition here. Check this doesn't break anything.}
\begin{lem}\label{lem:move-bok11}
  Let $i \colon R \to S$ be an hcentral map of $h\A_2$-rings in cyclotomic spectra satisfying $\Seg{b}$. If $\mu \in \pi_{2p^k}R$ is a B\"okstedt class, then the class $i(\mu) \in \pi_{2p^k}R$ is a B\"okstedt class.
\end{lem}

\begin{proof}
  First, as $i$ is hcentral by hypothesis, by \Cref{rmk:restrict-hcentral}, both $i(\mu)$ and $\varphi(i(\mu))$ are hcentral. Since $\varphi(\mu)$ is a unit, $\varphi(i(\mu)) = i^{tC_p}(\varphi(\mu))$ is a unit as well. Condition (b), that $\mu$ be in the image of the transfer, follows from naturality of the transfer map.
\end{proof}

\begin{lem} \label{lem:sc-bo-to-bok1}
  Let $R$ be an $h\A_2$-ring in cyclotomic spectra satisfying $\Seg{b}$
  and let $\mu \in \pi_{2p^k}R$ be a B\"okstedt class. 
  \begin{enumerate}
  \item[(1)] $R/\mu$ is $(b+1+2p^k)$-truncated.
  \item[(2)] The induced map $\varphi : R[\mu^{-1}] \to R^{tC_p}$ is an isomorphism 
  \end{enumerate}
\end{lem}

\begin{proof}
  Using that $\varphi(\mu)$ is a unit we have an isomorphism
  $ \fib(\varphi_R)/\mu \cong R/\mu $
  and we can then use our hypothesis that $R$ satisfies $\Seg{b}$ to conclude that
  $R/\mu \cong  \fib(\varphi_R)/\mu$ is $(b+1+2p^m)$-truncated.
  This verifies (1).
  As $\fib(\varphi_R)$ is $b$-truncated
  the action of $\mu$ on this fiber
  is locally nilpotent.
  It follows that $\fib(\varphi_R)[\mu^{-1}] = 0$ and
  we obtain (2).  
\end{proof}




\todo{Why was there an assumption that b geq 1?}
\begin{lem} \label{lem:cb-to-bok-uni1}
  Let $b \geq 0$ and let $m \geq 2m_p^{\A_2}$\todo{could be mphc. rwb: It cannot.}.
  The cyclotomic spectrum $R=\SP/p^{m} \otimes \tau_{\leq b}^{\mathrm{cyc}}\SP$
  admits an hcring structure in $\CycSp$, satisfies $\Seg{b+1}$
  and has a B\"okstedt class $\mu \in \pi_{2p^{k}}R$ for $k= (\lfloor b/2 \rfloor + 1)m$.
\end{lem}

\begin{proof}
  Let $R \coloneqq \SP/p^m \otimes \tau_{\leq b}^{\mathrm{cyc}}\SP$.
  The truncation of the unit, $\tau_{\leq b}^{\mathrm{cyc}}\SP$, is a commutative algebra
  and, since $m \geq m_p^{hc}$, $\SP/p^m$ admits an hcring structure,
  therefore $R$ admits an hcring structure.
  Using that $R \in \CycSp_{[0,b+1]}$ and $p^m$ acts by zero on $R$ we learn from
  \Cref{prop:omnibus} that $R$ satisfies $\Seg{b+1}$ and $\WcanV{b+1+2(b+2)m}$.
  Moreover, the spectrum $\TC(R)$ is bounded in the range $[-1, b+1]$.
  
  Give $\TC(R)$ the trivial $\T$-action and consider the $\T$-tate sseq
  which has signature
  \[ \pi_*\TC(R)[t^{\pm 1}] \cong E_2^{s,t} \Longrightarrow \pi_s\TC(R)^{t\T} \]
  where $|t| = (-2,2)$. 
  Note that this spectral sequence degenerates at the $E_{b+4}$-page for degree reasons
  and the powers of $t$ can only (possibly) support a
  $d_r$-differential for $r \leq b+2$ and even.
  Using that $\TC(R)$ is an hcring,
  the Leibniz rule for differentials implies (inductively) that for every $i \in \Z$
  \[ d_{2r}( (t^{i \cdot p^{(r-1)m}})^{p^m} ) = p^m \cdot d_{2r}(t^{i \cdot p^{(r-1)m}}) = 0 \]
  Put together, we obtain a unit
  $u \in \pi_{2p^{b'm}}\TC(R)^{t\T}$
  detected by the permanent cycle $t^{-p^{b'm}}$
  where $b' \coloneqq \lfloor b/2 \rfloor+1$. 
    
  
  
  The natural $\T$-equivariant hcring map $\TC(R) \to R$
  gives us an hcring map
  $ \TC(R)^{t\T} \to R^{t\T} $
  and we let $v$ denote the image of $u$ under this map.
  From the fact that $R$ satisfies $\Seg{b+1}$,
  we learn that $\varphi^{h\T} \colon R^{h\T} \to R^{t\T}$
  is an isomorphism on homotopy groups starting in degree $b+3$.
  Similarly, examining the cofiber sequence 
  \[ \Sigma R_{h\T} \xrightarrow{\mathrm{Nm}} R^{h\T} \xrightarrow{\mathrm{can}} R^{t\T} \]
  the fact that $R$ satisfies $\WcanV{b+1+2(b+2)m}$ implies that
  the map $\mathrm{Nm}$ is surjective on homotopy groups in degrees larger than $b+1+2(b+2)m$.
  Since $ 2p^k = 2p^{b'm} > b+1+2(b+2)m \geq b+3 $
  we may lift $v$ along $\varphi^{h\T}$ and $\mathrm{Nm}$ to a class
  $\hat{\mu} \in \pi_{2p^{k}}\Sigma R_{h\T}$.  

  Let $\mu \coloneqq \mathrm{tr}(\hat{\mu}) \in \pi_{2p^{k}}R$.
  We will show that $\mu$ is a B\"okstedt class.
  Condition (a) is automatic since $R$ is an hcring.
  Condition (b) follows from the fact that $\mu= \mathrm{tr}(\hat{\mu})$.
  Finally we wish to prove that $\varphi(\mu)$ is a unit.
  From the commutative diagram 
  \[ \begin{tikzcd}
    {} & {\Sigma R_{h\T}} & {R^{h\T}} & {R^{t\T}} \\
    && R & {R^{tC_p}}
    \arrow["{\mathrm{Nm}}", from=1-2, to=1-3]
    \arrow["{\varphi^{h\T}}", from=1-3, to=1-4]
    \arrow["i", from=1-3, to=2-3]
    \arrow["{i'}", from=1-4, to=2-4]
    \arrow["\varphi", from=2-3, to=2-4]
    \arrow["{\mathrm{tr}}", from=1-2, to=2-3]
  \end{tikzcd} \]
  we can read off that $ \varphi(\mu) = i'(v) $
  and, since $v$ is a unit and $i'$ an hcring map,
  that $\phi(\mu)$ is a unit as well.
\end{proof}

  

\todo{rwb: Theres a stupid zero ring thing here. If zero is not between c and b then R is zero and therefore has a bokstedt class via the unit 0. This comment is just so we have a record of why this isn't an issue.}

\begin{cor} \label{lem:cb-to-bok-hc}
  Let $c \leq b$ and $m \geq m_p^{\A_2}$.
  Let $R$ be an $h\A_2$ ring in cyclotomic spectra with $p^m=0$ such that
  $ R \in \CycSp_{[c,b]}$.
  Then $R$ admits a B\"okstedt class $\mu \in \pi_{2p^{k}}R$ for $k = (b-2c+2)m$.
\end{cor}
\todo{The old version of this contained an error in the explicit number. This means any number relying on it needs to be changed.}

\begin{proof}
    Because $R$ is cyclotomically $\leq b$, it satisfies $\Seg{b}$ by \Cref{prop:AN_bound_to_segal}.

  Let $e \coloneqq b-2c$.
  We claim that the boundedness hypothesis on $R$ implies it can naturally be refined to an $h\A_2$-ring in
  $\Mod(\CycSp; \tau_{\leq e}^{\mathrm{cyc}}\Ss)$. 
  Indeed, the underlying cyclotomic spectrum uniquely lifts to $\Mod(\CycSp; \tau_{\leq e}^{\mathrm{cyc}}\Ss)$ because it is bounded in the range $[c,b]$, 
  and the multiplication map $R\otimes R \to R$ uniquely factors through $\tau_{\leq b}(R\otimes R)$, 
  which is bounded in the range $[2c, b]$ and agrees with $\tau_{\leq b}(R\otimes_{\tau_{\leq e}^{\mathrm{cyc}}\SP} R)$.

  Using \Cref{lem:moore-hcentral} we now obtain
  an hcentral ring map
  \[ i : \Ss/p^{2m} \otimes \tau_{\leq e}^{\mathrm{cyc}}\Ss \to R. \]
  Applying \Cref{lem:cb-to-bok-uni1} to
  $\Ss/p^{2m} \otimes \tau_{\leq e}^{\mathrm{cyc}}\Ss$
  we learn that it has a B\"okstedt class $\mu$ in degree $2p^{k'}$
  where $k' = (\lfloor e/2 \rfloor + 1)2m$.
  Replacing $\mu$ by its $p^{\mathrm{th}}$ power (if necessary)
  we may replace $k'$ by $k = (e+2)m$.
  Applying \Cref{lem:move-bok11} to $i$
  we learn that $i(\mu)\in \pi_{2p^{k}}R$ is a B\"okstedt class. 
\end{proof}

\begin{lem} \label{lem:wbock-to-scv}
  Let $R$ be $p$-nilpotent $h\A_2$-algebra in cyclotomic spectra and let $\mu \in \pi_{k}R$
  be a class such that
  \begin{enumerate}
      \item $R/\mu$ is $b$-truncated and
      \item $\mu$ is in the image of the transfer map 
      $\Sigma R_{h\T} \to R$. 
  \end{enumerate}
  Then $R$ satisfies $\ScanVn{b}$.
\end{lem}

\begin{proof}
  Recall that the norm and transfer fit into a diagram
  \[ \begin{tikzcd}
    \colim_d (\Sigma^{-1+d(1)} \Ss/a_{(1)}^{d} \otimes R)^{h\T} \ar[d, equal] \ar[r] &
    R^{h\T} \ar[d, equal] \ar[r] &
    (\Ss^{\T} \otimes R)^{h\T} \ar[d, equal] \\
    \Sigma R_{h\T} \ar[r, "\mathrm{Nm}"] &
    R^{h\T} \ar[r] &
    R.
  \end{tikzcd} \]
  
  Let $\hat{\mu}$ be a lift of $\mu$ along the transfer.
  Since $R$ is $p$-nilpotent the colimit in the top left commutes with $\pi_*(-)$\footnote{Recall we are working with $p$-complete spectra, so homotopy groups in general do not compute with filtered colimits.}
  and we learn that there exists some $d \gg 0$ such that
  $a_{(1)}^{d} \cdot \hat{\mu} = 0$.
  Using this we construct a diagram of $\T$-spectra
  witnessing the desired nullhomotopy.  
  \[ \begin{tikzcd}[sep=large]
    & \tau_{> b}R \ar[d] \ar[dl, dashed] \ar[dr, "0"] & \\
    \Sigma^{2p^k} R \ar[r, "\Nm(\hat{\mu}) \cdot -"] \ar[dr, "0"] &
    R \ar[d, "a_{(1)}^{d}"] \ar[r] &
    R/\Nm(\hat{\mu}) \\    
    & \Sigma^{d(1)} R &
  \end{tikzcd}\]  
\end{proof}

%% file: zeta3.tex
Given a $p$-complete ring spectrum with a $\Z$-action  
$R \in \Alg(\Sp)^{B\Z}$, 
the first half of this paper is devoted to studying the system of cyclotomic spectra 
\[ \begin{tikzcd}
  {\THH(R^{h\ZZ})} \ar[r] &
  {\THH(R^{hp\ZZ})} \ar[r] &
  {\THH(R^{hp^{2}\ZZ})} \ar[r] &
  {\cdots} \ar[r] &
  {\THH(R)}.
\end{tikzcd} \]
In this section specifically, we analyze the initial example, where $R=\Ss$ with the trivial $\ZZ$-action. In other words, we look at the system of commutative algebras in cyclotomic spectra $\THH(\SP^{Bp^{k}\Z})$. The key idea governing our analysis is that, since $\Ss^{B\Z}$ is the $\Ss$-cochains on $B\Z_p$, $\THH(\Ss^{B\Z})$ is controlled by the geometry the free loop space of $B\ZZ_p$.
We highlight that the underlying cyclotomic spectrum of $\THH(\Ss^{B\Z})$ was studied in \cite{malkiewich2017topological}, and the underlying commutative algebra was studied in \cite[Lemma 4.6]{lee2023topological}.

In the final subsection we prove \Cref{prop:thicksubcat} from the introduction,
by analyzing the fiber of the coassembly map $\THH(\Ss^{B\Z}) \to \THH(\Ss)^{B\Z}$.
This proposition is the key statement we use to show that various
$\TC$ coassembly maps are not an isomorphisms.

We make significant use of the spherical Witt vectors functor, which is a way of producing canonical lifts of perfect $\F_p$-algebras to formally \'etale $p$-complete commutative algebras. We refer the reader to Conventions (\ref{item:firsttop})-(\ref{item:lasttop}) for relevant notation and recall the main properties of this construction:

\begin{prop}[{\cite[Sec. 5.2]{ECII}, \cite[Prop. 2.2, Cons. 2.33]{chromaticNSTZ}}] \label{prop:witt-and-tilt}
  There is an adjunction
  \[ \W(-) \colon \Perf \rightleftarrows \CAlg(\Sp) \noloc {\pflat}(-) \]
  where the right adjoint $\pflat$ can be computed as
  the inverse limit along Frobenius on the commutative ring $\pi_0(-)/p$.

  This adjunction witnesses $\Perf$ as a colocalization of $\CAlg(\Sp)$.
  The essential image of the (fully faithful) functor $\W$
  consists of those connective $R \in \CAlg(\Sp)$ such that
  $\F_p \otimes R $ is a discrete perfect $\F_p$-algebra.
  In this situation, we have $R \cong \W(R\otimes \FF_p)$.  
\end{prop}

\subsection{As a commutative algebra}

We begin by analyzing the system of commutative algebras $\THH(\SP^{Bp^k\Z})$, saving discussion of $\T$-equivariance and cyclotomic structure for later subsections.
In fact, it will be simpler to analyze $\THH(\Ss^{BM})$ for any (discrete) finite projective $\Z_p$-module $M$, together with its functoriality in $M$.
Our primary example of interest is then obtained by specializing to the system
$0 \to \cdots \to p^2\Z_p \to p\Z_p \to \Z_p$.

\begin{rmk}
  We remind the reader that, for any free finite rank $\ZZ$-module $M$, the map $\SP^{BM_p} \to \SP^{BM}$ is an equivalence. Indeed, the map $BM \to BM_p$ induces an equivalence on $p$-complete suspension spectra since it is an equivalence on $\F_p$-homology, so this follows from taking duals.
\end{rmk}

\def\latt{\mathrm{Latt}_{\Z_p}}
\begin{dfn}
  Let $\latt$ be the category of discrete, finite, projective $\Z_p$-modules.
\end{dfn}

\begin{cnstr} \label{cnstr:shear}
  Let $\DD$ be a category with all finite limits.
  Given a $G \in \mathrm{Mon}(\DD)^{\mathrm{gp}}$,
  we construct the following diagram in $\DD$ 
  \[ \begin{tikzcd}[column sep=huge]
    G \ar[r, "\triangle"] \ar[d, "\id"] &
    G \times G \ar[d, "{(a,b) \mapsto (a,b \cdot a^{-1})}"] &
    G \ar[l, "\triangle"'] \ar[d, "\id"] \\
    G \ar[r, "{a \mapsto (a, 1)}"] &
    G \times G &
    G, \ar[l, "{a \mapsto (a, 1)}"']
  \end{tikzcd} \] 
  natural in $G$.
  Taking the pullback of these spans 
  and remembering the maps out to the left column,
  we obtain a diagram
  \[ \begin{tikzcd}
    \cL G \ar[rr, "\cong"', "\mathrm{sh}"] \ar[dr] & &
    G \times \Omega_e G \ar[dl] \\
    & G. & 
  \end{tikzcd} \]
  Herel
  the left vertical map is restriction along the inclusion of the basepoint $* \to \T$,
  and the right vertical map is the projection to the first factor.
\end{cnstr}

\begin{lem}\label{lem:assembly}
  The counit of the $\W \vdash \pflat$ adjunction,
  applied to the natural assembly map,
  yields an identification of commutative algebras
  \[ \begin{tikzcd}
    \W\LCF{A} \ar[r, "i"] \ar[d] & \W\LCF{A^\delta} \ar[d] \\
    \Ss \otimes_{\Ss^{BA}} \Ss \ar[r] & \Ss^{\Omega_eBA} 
  \end{tikzcd} \]
  natural in $A \in \latt$.
  Here, $A^\delta$ denotes $A$ with the discrete topology,
  and the top horizontal map is the inclusion of
  continuous functions into all functions.
\end{lem}

\begin{proof}
  First, we claim that $\SP\otimes_{\SP^{BA}}\SP$ and $\Ss^{\Omega_eBA}$ are bounded below.
  For the former this follows by writing $A$ as $\Z_p^j$ for some $j$ and
 expanding $\Ss \otimes_{\SP^{BA}}\Ss$ as $(\SP\otimes_{\SP^{B\Z_p}}\SP)^{\otimes j}$,
  which is connective by \cite[Lemma 3.1]{levy2022algebraic}.
  For the latter it is clear.
    
  By \Cref{prop:witt-and-tilt} it will now suffice to prove that the indicated arrows in the diagram below are isomorphisms.
  \[ \begin{tikzcd}
    \colim_k \Fp \otimes_{\Fp^{BA/p^k}} \Fp \ar[r, "\cong", "(1)"'] \ar[d, "\cong"', "(3)"] &
    \Fp \otimes_{\Fp^{BA}} \Fp \ar[d] &
    (\SP \otimes_{\SP^{BA}}\SP) \otimes \Fp \ar[l, "\cong"', "(2)"] \ar[d] \\
    \colim_k \Fp^{\Omega_eBA/p^k} \ar[r] \ar[d, "\cong"', "(5)"] &
    \Fp^{\Omega_eBA} \ar[d, "\cong"', "(6)"] &
    \Ss^{\Omega_eBA} \otimes \F_p \ar[l, "\cong"', "(4)"] \\
    \LCF{A} \ar[r, "i"] &
    \LCF{A^\delta} &
  \end{tikzcd} \]
  
  The top left square is constructed from
  the assembly map for the colimit over $k$ and
  the assembly map for the suspension of augmented commutative algebras.
  The map labeled (3) is an isomorphism by convergence of the
  Eilenberg--Moore spectral sequence (see \cite[Corollary 1.1.10]{DAGXIII}).
  The map labeled (1) is an isomorphism since $H_c^*(A; \F_p) \cong H^*(A; \F_p)$.
  \todo{Too oblique? cite serre good}
  
  The top right square is constructed using
  the unit map $\Ss \to \F_p$, $\F_p$-linearization
  and the assembly map for the suspension of augmented commutative algebras.
  The maps (2) and (4) are isomorphisms since
  $\F_p \otimes \Ss^{BA} \cong \Fp^{BA}$ and
  $\F_p \otimes \Ss^{\Omega_eBA} \cong \Fp^{\Omega_eBA}$.
  This follow from the fact that $\F_p \otimes -$ commutes with
  finite limits and arbitrary products that are uniformly bounded below.
  
  The bottom square is contructed from the natural identifications
  and the fact that continuous functions on an $A \in \latt$ with values in $\F_p$
  are locally constant.
\end{proof}

\begin{lem}\label{lem:thh-s1}
  The counit of the $\W \vdash \pflat$ adjunction,
  the natural map $\Ss^{BA} \to \THH(\Ss^{B\A})$,
  and the assembly map for the $\T$-shaped colimit $\THH$
  along the cochains functor $\Ss^{(-)}$
  together induce the following identification of commutative algebras  
  \[ \begin{tikzcd}
    \Ss^{BA} \otimes \W\LCF{A} \ar[r] \ar[d, "\cong"] &
    \Ss^{BA} \otimes \W\LCF{A^{\delta}} \ar[d, "\cong"] \\
    \THH(\Ss^{BA}) \ar[r] &
    \Ss^{\mathcal{L}BA}
  \end{tikzcd} \]
  natural in $A \in \latt$.
\end{lem}

\begin{proof}
  The functor $\Ss^{(-)} \colon \Spc \to \CAlg(\Sp)^{\op}$
  preserves products, and therefore takes grouplike commutative monoids to grouplike commutative monoids.
  Using naturality along the functor $\Ss^{(-)} : \Spc \to \CAlg(\Sp)^{\op}$,
  \Cref{cnstr:shear} provides us with identifications of $\Ss^{BA}$-algebras
  \[ \begin{tikzcd}
    \Ss^{BA} \otimes \left( \Ss \otimes_{\Ss^{BA}} \Ss \right) \ar[r] \ar[d, "\cong"] &
    \Ss^{BA \times \Omega_eBA} \ar[d, "\cong"] \\
    \THH(\Ss^{BA}) \ar[r] &
    \Ss^{\cL BA},
  \end{tikzcd} \]
  natural in $A \in \latt$.
  The lemma now follows from \Cref{lem:assembly} and 
  the fact that $ \pflat( \Ss^{BA} \otimes R) \cong \pflat(R)$
  for any $R \in \CAlg(\Sp)$ ($\pflat$ is a right adjoint with values in a $1$-category).
\end{proof}

\begin{lem} \label{lem:coassembly-po}
  The coassembly map for the limit over $BA$
  and the counit of the $\W \vdash \pflat$ adjunction
  fit into a pushout square of commutative algebras
  \[ \begin{tikzcd}
    \W(\LCF{A}) \ar[d] \ar[r, "{(-)_{|0}}"] &
    \W(\F_p) \ar[d] \\
    \THH(\Ss^{BA}) \ar[r] &
    \Ss^{BA} \pushout
  \end{tikzcd} \]
  natural in $A \in \latt$. Moreover, the lower horizontal map coincides with the map $\colim_{\T}\SP^{B\A} \to \colim_*\SP^{B\A}$ induced by the map $\T \to *$, where the colimit is taken in $\CAlg(\Sp)$.
\end{lem}

\begin{proof}
  As $\pflat $ is a right adjoint with values in a $1$-category,
  we have $\pflat \Ss^{BA} \cong \F_p$.
  It now follows from \Cref{lem:thh-s1} and the fact that the composite
  \[ \Ss^{BA} \to \THH(\Ss^{BA}) \to \Ss^{BA} \]
  is the identity that the square is a pushout. 
  
    To see the claim about the map induced by $\T \to *$, it suffices to consider the square
  
  \begin{center}
  	\begin{tikzcd}
  		{\THH(\SP^{BA})} \ar[r]\ar[d] &\THH(\SP)^{B\A} \ar[d,"\cong"]\\
  		\SP^{B\A}\ar[r,"\Id"] & \SP^{B\A},
  	\end{tikzcd}
  \end{center}
  whose horizontal maps are the $BA$ coassembly maps for the identity and $\THH$ functors, and whose vertical maps are induced by the natural transformation $\THH(\--) \to \--$ arising from the projection $\T \to *$. 
\end{proof}

The following notation will be useful in the sequel:

\begin{cnstr} \label{cnstr:zeta}
  Let $\epsilon \in \pi_{-1}\Ss^{B\Z}$ be
  the class corresponding to the element $1$ in $H^1(B\Z; \Z)$,
  and let $\zeta \in \pi_{-1}\THH(\Ss^{B\Z})$ be
  the image of $\epsilon$ under the natural map
  $\Ss^{B\Z} \to \THH(\Ss^{B\Z})$.
\end{cnstr}

\subsection{As a \texorpdfstring{$\T$}{T}-equivariant commutative algebra}

Having understood the commutative algebra structure of $\THH(\Ss^{B\Z})$,
we turn to describing the circle action.
Under the assembly map to $\Ss^{\cL B\Z_p}$, the circle action is compatible with
the circle action on the free loop space $\cL B\ZZ_p$,
and it is through this that we gain control over the situation.

\begin{dfn}
  For $w \in \Z_p$, we let $B\Z_p(w)$ denote $B\Z_p$ with the action of  $\T = B\Z$ 
  coming from left multiplication via the homomorphism $w \colon B\Z \to B\Z_p$.

  Let $R \in \CAlg(\Sp)$. Taking $R$-valued cochains on $B\Z_p(w)$
  we obtain a $\T$-action on the commutative $R$-algebra $R^{B\Z_p}$,
  which we denote $R^{B\Z_p(w)} \in \CAlg(\Sp)_{R/-}^{B\T}$.
  Note that when $w \in \Z \subseteq \Z_p$, we have $R^{B\Z_p(w)} \cong R^{\T/C_{w}}$.
\end{dfn}

\begin{exm} \label{exm:loop-action}
  Consider the $\T$-action on $\cL B\ZZ_p$, the free loop space of the $p$-adic circle.
  On the connected component of the degree $w$ map $w \colon B\Z \to B\Z_p$,
  the rotation action on the source circle becomes the `$w$-speed' rotation action on the target. Altogether, we obtain a $\T$-equivariant isomorphism
  \[ \cL B\Z_p \cong \coprod_{w \in \Z_p} B\Z_p(w).\hfill\qedhere \]
\end{exm}

\begin{lem} \label{lem:sigma-zeta}
  In $\pi_0\THH(\Ss^{B\Z})$ we have $\sigma(\zeta) = (1 + \eta\zeta) \cdot \Id_{\Z_p}$.\footnote{Note that $\pi_0\LCF{\mathbb{Z}_p}$ is the ring of cont. functions from 
  $
  \ZZ_p$ to itself. By $\Id_{\ZZ_p} \in \pi_0\W\LCF{
  \ZZ_p
  }$ we mean the identity function.}
\end{lem}

\begin{proof}
  From \Cref{lem:thh-s1} we can read off that the assembly map
  $ \THH(\Ss^{B\Z_p}) \to \Ss^{\cL B\Z_p}$ is injective on homotopy groups, and
  it will therefore suffice to compute $\sigma(\zeta)$ in the target.
  Breaking things up accross the coproduct in \Cref{exm:loop-action}
  and using the fact that $\Z \subset \Z_p$ is dense 
  we reduce to computing $\sigma(\epsilon)$ in $\pi_0(\Ss^{B\Z(w)})$ where $w \in \Z$.
  Writing $\Ss^{B\Z(w)}$ as $w^*\Ss^{\T}$ and using that the degree $w$ map $\T \to \T$ sends $\sigma$ to $w \sigma$ we find that it will suffice to compute that
  $\sigma(\epsilon) = 1 + \eta\epsilon$ in $\Ss^{\T}$.
  After rationalization this is straightforward,
  the general case follows from the fact that $\sigma \circ \sigma = \eta \sigma$ (see \cite[Sec. 3.5]{antieau-nikolaus}).
\end{proof}

\begin{cnstr}\label{cnstr:psi-p}
  Let $R \in \CAlg(\Sp)$.
  Tensoring the $p$-fold covering map $p \colon \T \to \T/C_p$
  with the commutative algebra $R$ we obtain a 
  map of $\T$-equivariant commutative algebras
  \[ \psi_p \colon \THH(R) \to p^*\THH(R) \]
  natural in $R$
  where $p^*$ denotes the restriction map $\Sp^{B\T/C_p} \to \Sp^{B\T}$.

  As the map $p\colon \T \to \T/C_p$ preserves the base-point
  $\psi_p$ also carries the structure of a commutative $R$-algebra map
  (but not compatibly with the circle action).
\end{cnstr}

\begin{lem} \label{lem:rescale}
  The map $\psi_p$ refines to a $\T$-equivariant isomorphism
  \[ \psi_p \colon \THH(\Ss^{B\Z})_{|p\Z_p} \xrightarrow{\cong} p^*\THH(\Ss^{B\Z}) \]
  which induces the map $\res_{p} \colon C^0(p\Z_p) \to C^0(\Z_p)$ on $\pflat$.
  In particular, the $C_p$ action on $\THH(\Ss^{B\Z})_{|p\Z_p}$ is trivializable.
\end{lem}

\begin{proof}
  As the assembly maps to $\Ss^{\cL B\Z_p}$ are injective on homotopy groups by \Cref{lem:thh-s1} and the construction of $\psi_p$ is natural it will suffice to prove the corresponding claim for $\Ss^{\cL B\Z_p}$. Here we observe that precomposition with the degree $p$ map of $S^1$ produces the map $\cL B\Z_p \to \cL B\Z_p$ which sends the circle at component $a \in \Z_p$ to the circle at component $pa$ isomorphically.
\end{proof}

\begin{lem} \label{lem:shells}
  For each $k \geq 0$
  there is an isomorphism of $\W\LCF{p^k\ZZ_p^{\times}}$-modules in $\Sp^{B\T}$
  \[ \THH(\SP^{B\Z_p})_{|p^k\ZZ_p^{\times}} \cong \W\LCF{p^k\ZZ_p^{\times}} \otimes \Sigma^{-1}\Ss[\T/C_{p^k}]. \]
  In particular,
  if $M$ is a $\THH(\SP^{B\Z_p})_{|\ZZ_p^{\times}}$-module,
  then $M^{tC_{p^j}} = 0$ for all $1 \leq j \leq \infty$.
\end{lem}


\begin{proof}
  Using \Cref{lem:rescale} it will suffice to handle the case where $k=0$.

  Using the restriction of $\zeta$ to $\pi_{-1}\THH(\SP^{B\Z_p})_{|\ZZ_p^{\times}}$
  we construct an induced $\T$-equivariant map of $\W\LCF{\Z_p^\times}$-modules
  \[ z : \W\LCF{\ZZ_p^{\times}} \otimes \Sigma^{-1}\Ss[\T] \to \THH(\SP^{B\Z_p})_{|\ZZ_p^{\times}}. \]
  On homotopy groups this gives a map of
  $\pi_*\W\LCF{\ZZ_p^{\times}}$-modules
  \[ z : (\pi_*\W\LCF{\Z_p^\times})\{ [*], [\T] \} \to \LCF{\Z_p^\times}\{ 1, \zeta \} \]
  with $z([*]) = \zeta$.
  The Connes operator and \Cref{lem:sigma-zeta}
  now let us compute that
  \[ z([\T]) = z(\sigma([*])) = \sigma(z([*])) = \sigma(\zeta) = \Id_{\Z_p}(1+\eta\zeta)\]
  The first claim follows from that fact that,
  when restricted to $\Z_p^{\times}$, $\Id_{\Z_p}$ is a unit.

  The second claim follows from the first since
  we now know $a_{(1)} = 0$ as a $\T$-equivariant self-map of $\THH(\SP^{B\Z_p})_{|\ZZ_p^{\times}}$
  and therefore that $a_{(1)}$ acts by zero on any $\T$-equivariant
  $\THH(\SP^{B\Z_p})_{|\ZZ_p^{\times}}$-module $M$.
\end{proof}

\begin{lem} \label{lem:pflat-tcp}
  Let $R \in \CAlg(\Sp)^{BC_p}$ be bounded below.
  In the associated spain 
  \[ \pflat R \leftarrow \pflat R^{hC_p} \to \pflat R^{tC_p} \]
  the left arrow is an isomorphism if the action on $\pi_0R$ is trivial
  (e.g. when the $C_p$ actions extends to a $\T$-action)
  and the right arrow is an isomorphism if the $C_p$-action is trivial.
\end{lem}

\begin{proof}
  The first claim follows from 
  $\pflat(-)$ being a right adjoint 
  that factors through the functor $\pi_0(-)$
  and has target a $1$-category.

  We now prove the second claim.
  Using the fact that
  (i) the Postnikov tower refines the map
  $\Ss \to \Z_p$
  to an $\omega$-indexed tower of square-zero extensions,
  (ii) the fact that $(-)^{hC_p}$ and $(-)^{tC_p}$
  are exact and commute with uniformly bounded below limits and
  (iii) nil-invariance of $\pflat$ \cite[Lem. 2.16]{chromaticNSTZ}
  we reduce to proving the lemma for $\Z_p \otimes R$.
  In this case we may identify the map
  $ \pi_*(\Z_p \otimes R)^{hC_p} \to \pi_*(\Z_p \otimes R)^{tC_p}$
  with the map 
  $ \pi_*(\F_p \otimes R)[\![t]\!] \to \pi_*(\F_p \otimes R)(\!(t)\!) $
  where $|t| = -2$.
  As $R$ is bounded below and $t$ is in a negative degree
  this is a nil-extension in degree $0$
  and we may conclude. 
\end{proof}

\begin{rmk} \label{rmk:things-agree}
  In the setting of \Cref{lem:pflat-tcp},
  a trivialization of the $C_p$-action on $R$
  gives a section $R \to R^{hC_p}$ of the map $R^{hC_p} \to R$
  and through this we see that the map
  \[ \pflat R \to \pflat R^{tC_p} \]
  associated to a trivialization of the $C_p$-action on $R$
  will agree with the isomorphism from \Cref{lem:pflat-tcp}.
\end{rmk}

\begin{cnstr} \label{cor:tateTHHSZ}
  Using Lemmas \ref{lem:shells}, \ref{lem:rescale} and \ref{lem:pflat-tcp} 
  we fix an identifcation
  \[ \begin{tikzcd}
    \pflat \THH(\SPp^{B\ZZ}) \ar[d, "\cong"] &
    \pflat \THH(\SPp^{B\ZZ})^{hC_p} \ar[l] \ar[r] \ar[d, "\cong"] &
    \pflat \THH(\SPp^{B\ZZ})^{tC_p} \ar[d, "\cong"] \\
    \LCF{\Z_p} &
    \LCF{\ZZ_p} \ar[l, "\Id"] \ar[r, "{(-)_{|p\Z_p}}"] &
    \LCF{p\ZZ_p}
  \end{tikzcd} \]
  which on the left column agrees with the one from \Cref{lem:thh-s1}.
\end{cnstr}

\subsection{As a cyclotomic spectrum}

The cyclotomic Frobenius on $\THH(\SP^{B\Z})$ is an isomorphism and it turns out that this cyclotomic spectrum is closely related to the one corepresenting $\TR$.

\begin{prop}\label{prop:cycfrobuniv}
  The cyclotomic Frobenius map
  $\varphi \colon \THH(\Ss^{B\Z_p}) \to \THH(\Ss^{B\Z_p})^{tC_p}$
  is an isomorphism
  and under the identifications from \Cref{cor:tateTHHSZ}
  $\pflat(\varphi)$ is given by $\res_{1/p}$
  where $1/p$ is the map $1/p \colon p\Z_p \to \Z_p$.  
\end{prop}

\begin{proof}
  We will construct the following diagram
  \[ \begin{tikzcd}
    \Ss^{B\ZZ_p} \ar[r,"\Delta_p"] \ar[d] &
    ((\Ss^{B\Z_p})^{\otimes p})^{tC_p} \ar[d,"(\mu_p^*)^{tC_p}"] \ar[r, "m"] &
    (\Ss^{B\ZZ_p})^{tC_p} \ar[d] &
    \Ss^{B\Z_p} \ar[l, "{\mathrm{can}}"', "\cong"] \ar[d] \\
    \THH(\SPp^{B\ZZ_p}) \ar[r,"\varphi"] &
    \THH(\SPp^{B\ZZ_p})^{tC_p} \ar[r, "\psi_p^{tC_p}"] &
    (p^*\THH(\SPp^{B\ZZ_p}))^{tC_p} &
    \THH(\SPp^{B\ZZ_p}) \ar[l, "\cong"] 
  \end{tikzcd} \]
  of shape $\Delta^3 \times ( * \to B\T )$.
  The left square is from \cite[Cor. IV.2.3]{NS}.
  The middle square is obtained by tensoring the commutative algebra
  $\Ss^{B\Z_p}$ with the square associated to the short exact sequence
  $C_p \to \T \to \T/C_p$ and applying $(-)^{tC_p}$.
  The right square is constructed using a 
  trivialization of the $C_p$ action
  on $* \to \T/C_p$.
  The indicated isomorphisms are from \cite[corollary 6.7]{yuan2023integral}.

  The map $\psi_p^{tC_p}$ is an isomorphism by
  Lemmas \ref{lem:rescale} and \ref{lem:shells}.
  Using the universal property of $\THH$ as a commutative algebra
  we will show that the composite along the bottom row is the identity
  by showing that the composite along the top row is the identity.
  The composite $m \circ \Delta_p$ is the Tate-valued Frobenius.
  As the Tate-valued Frobenius composed with $\mathrm{can}^{-1}$
  is the identity for $\Ss$,
  naturality in the limit over $B\Z_p$ implies it is the identity for
  $\Ss^{B\Z_p}$ as well.\todo{is this last step valid.}

  The identification of $\pflat \varphi$ now follows from
  \Cref{lem:rescale}, \Cref{lem:pflat-tcp} and \Cref{cor:tateTHHSZ}.  
\end{proof}

\begin{prop}\label{prop:coassembly}
  Let $\crTR \in \CycSp$ denote the object corepresenting $\TR(-)$.
  There is a fiber sequence of cyclotomic spectra
  \[ \W (\LCF{\Z_p^\times}) \otimes \Sigma^{-1}\crTR \to \THH(\Ss^{B\Z}) \to \THH(\Ss)^{B\Z} \]
  where the second map is the natural coassembly map. 
\end{prop}

\begin{proof}	
  Let $F$ denote the fiber of the coassembly map
  $ \THH(\Ss^{B\Z}) \to \THH(\Ss)^{B\Z} $.
  As it is the fiber of a map of commutative algebras,
  $F$ is naturally a non-unital commutative algebra in $\CycSp$.

  Using \Cref{lem:coassembly-po} we can read off that
  as a $\T$-equivariant nonunital commutative algebra
  $F$ is isomorphic to the direct sum
  $ \oplus_{k \geq 0} F_{|p^k\Z_p^\times} $.
  By \Cref{lem:shells}
  we have isomorphisms of $\W\LCF{\Z_p^\times}$-modules in $\Sp^{B\T}$
  \[ F_{|p^k\Z_p^\times} \cong \W\LCF{\Z_p^\times} \otimes \Sigma^{-1}\Ss[\T/C_{p^k}] \]
  and by \Cref{prop:cycfrobuniv} the cyclotomic Frobenius on $F$
  breaks up as a sum\footnote{By this we mean the sum of the maps composed with the assembly map for $tC_p$ for the infinite sum.} of isomorphisms
  $ F_{|p^k\Z_p^\times} \cong F_{|p^{k+1}\Z_p^\times}^{tC_p} $
  of $\W\LCF{\Z_p^\times}$-modules in $\Sp^{B\T}$.\footnote{Note that here we have freely rescaled $p^k\Z_p^\times$ to make the same algebra act on all objects.}
  The proposition now follows from \Cref{lem:identify-crTR}.
\end{proof}

\begin{cor} \label{cor:t-amplitude-W}
  The cyclotomic spectrum $\THH(\Ss^{B\Z})$ has $t$-amplitude in the range $[-1,3]$ as an object of $\CycSp_{+}$.
\end{cor}
\begin{proof}
  Using the cofiber sequence from \Cref{prop:coassembly} and
  the fact that $\THH(\Ss)$ is the unit of cyclotomic spectra
  we can read off that it will suffice to show that
  $ \W (\LCF{\Z_p^\times}) \otimes \Sigma^{-1}\repTR$
  has $t$-amplitude in the range $[-1,3]$.
  Writing $\W (\LCF{\Z_p^\times})$ as a sum of spheres,
  we conclude by applying \Cref{lem:t-amplitude-repTR}.
\end{proof}

\subsection{The failure of hyperdescent and \texorpdfstring{$\A^1$}{A1}-invariance}
\label{sec:knlocal}

Having studied $\THH(\SP^{B\Z})$ as a cyclotomic spectrum,
we briefly pause to extract some consequences. The following implies \Cref{prop:thicksubcat} of the introduction.

\begin{cor}\label{cor:nilthicksub}
  Let $R \in \Alg(\Sp)$ be connective.
  Then $\W\LCF{\Z_p^\times} \otimes \TC(R)$ is in
  the thick subcategory generated by the fiber of the coassembly map
  \[ \TC(R^{B\Z}) \to \TC(R)^{B\Z}. \]
\end{cor}

\begin{proof}
  From the fiber sequence
  $\TC \to \TR \xrightarrow{1-F} \TR$,
  we deduce that there is a corresponding cofiber sequence of corepresenting objects
  $ \crTR \to \crTR \to \Ss$
  in $\CycSp$.
  Combining this with \Cref{prop:coassembly} we learn that
  $\W\LCF{\Z_p^\times}$ is in the thick subcategory of $\CycSp$ generated by
  the fiber of the coassembly map $\THH(\SP^{B\ZZ}) \to \THH(\SP)^{B\ZZ}$.
  
  Using the fact that $\THH$ is symmetric monoidal,
  the fact that $\W(\LCF{\ZZ_p^\times})$ is a $p$-adic sum of spheres and
  the fact that $\TC:\CycSp_{\geq0} \to \Sp$ preserves colimits
  \cite[Theorem 2.7]{clausen2021}, we conclude.
\end{proof}

%% file: knlocal.tex
The remainder of this section is dedicated to a discussion of the consequences of \Cref{cor:nilthicksub} for our understanding of $K(n+1)$-local $K$-theory of $T(n)$-local rings. 
Note that the material in in the remainder of the section is not cited in later sections, though the argument of \Cref{thm:failurekn} is used in the proof of \Cref{thm:maincyc}.


Recall that $K(1)$-local $K$-theory of connective commutative $\Q$-algebras is an incredibly well behaved invariant.

\begin{enumerate}
\item It satisfies \'etale hyperdescent under mild finiteness conditions \cite{thomason1985algebraic} (see \cite{clausen2021hyperdescent}). 
\item It satisfies nil-invariance \cite[Corollary 4.23]{LMMT}.
\item It satisfies $\mathbb{A}^1$-invariance: $L_{K(1)}K(R) \cong L_{K(1)}K(R[t])$ \cite[Corollary 4.24]{LMMT}.
\end{enumerate}

Unfortunately, as a consequence of \Cref{cor:nilthicksub},
$K(n+1)$-local $K$-theory of $T(n)$-local algebras does not
share any of these properties when $n \geq 1$. This follows from the following result, setting $X = L_{K(n+1)}\SP$.

\begin{thm}\label{thm:failurekn}
  Let $R$ be a $T(n)$-local $\E_{1}$-ring for $n\geq 1$ and let $X$ be a spectrum.
  If $F(R)\coloneqq L_{T(n+1)}(X\otimes K(R)) \neq 0$, then none of the maps
  \begin{align*}
    F(R) &\to  F(R[t]) \\
    F(R) &\to F(R\langle \epsilon_{-1} \rangle) \\
    F(R^{B\ZZ}) &\to F(R)^{B\ZZ}
    \end{align*}
  are isomorphisms.
\end{thm}

\begin{rmk}
  Recall that $R\langle \epsilon_{-1} \rangle$ denotes
  a trivial square-zero extension by a class in degree $-1$.
  Note also that $L_{T(n+1)}K(R[t]) \cong L_{T(n+1)}K(L_{T(n)}(R[t]))$
  by the purity results of \cite{LMMT}.
\end{rmk}
  
\begin{proof}
  We first observe that the first two claims are equivalent to each other. Indeed, it follows from \cite[Theorem 4.1]{land2023k}\footnote{see also the discussion in \cite[Example 4.9]{burklund2023k}} that $K(R[\epsilon_{-1}]) \cong K(R)\oplus \Sigma^{-1}NK(R)$ and $K(R[t])\cong K(R)\oplus NK(R)$, so that both claims are equivalent to the nonvanishing of $L_{T(n+1)}(X\otimes NK(R))$. By \Cref{thm:purity}, this is equivalent to asking that $L_{T(n+1)}(X\otimes N\TC(\tau_{\geq0}R)) \neq 0$.
 
 By \cite{mccandless2021curves}, the cofiber of $\SP_p \to \THH(\Ss[t])_p$ corepresents $\TR_p$ in $(\CycSp)_p$, where $(\CycSp)_p$ is the $p$-completion of the category of integral cyclotomic spectra as opposed to the category of $p$-typical $p$-complete cyclotomic spectra. Thus $\Ss_p$ is in the thick subcategory it generates in $(\CycSp)_p$. Tensoring this with $\THH(R)$ and applying $L_{T(n+1)}(X\otimes \TC)$, we learn that $L_{T(n+1)}(X\otimes \TC(R)) \neq 0$ is in the thick subcategory generated by $L_{T(n+1)}(X\otimes N\TC(\tau_{\geq0}R) )$, allowing us to conclude that the latter is nonzero.
	
 The last statement about $$F(R^{B\ZZ}) \to F(R)^{B\ZZ}$$ not being an equivalence follows similarly, one only need apply \Cref{cor:nilthicksub} and \Cref{thm:TCequalK} instead of \cite{mccandless2021curves} and \Cref{thm:purity}.
\end{proof}

\begin{rmk}
  The first map in \Cref{thm:failurekn} tests $\mathbb{A}^1$-invariance.
  The second map tests nil-invariance.\footnote{Note that in the $T(n)$-local setting many examples are periodic and there can be no particular distinction between square-zero extension on classes in positive degree and square-zero extensions on classes in negative degree.}
  The third map tests $\ZZ_p$-Galois hyperdescent along the $p$-adic $\ZZ_p$-Galois extension $R^{B\Z} \to R$.
\end{rmk}

The results of \cite{chromaticNSTZ} imply that \Cref{thm:failurekn} applies to any $T(n)$-local commutative algebra for $X = \SP$ or $X= L_{K(n+1)}\SP$.

\begin{rmk}
  The failure of $\Z_p$-Galois hyperdescent for $T(n+1)$-local $K$-theory for $n \geq 1$
  in the case of a trivial action leads to an expectation that hyperdescent should also fail for $\Z_p$-Galois extensions that sufficiently ``close'' to the trivial action. The next section is devoted to studying one instance where this expectation is correct.
  
  Indeed, we disprove the telescope conjecture
  by showing that hyperdescent fails $T(n+1)$-locally for the
  $\ZZ_p$-Galois extension
  $L_{T(n)}\BP\langle n \rangle^{h\ZZ} \to L_{T(n)}\BP\langle n \rangle$
  (it behaves enough like a trivial action)
  then constrasting this with the results of \cite{cycloshift}
  which imply that this extension \textit{does}
  satisfy hyperdescent $K(n+1)$-locally.
\end{rmk}


%% file: tame-action-short.tex




In this section we prove our cyclotomic asymptotic constancy theorem (\Cref{thm:cohconstgen} of the introduction). This theorem says that under strong finiteness hypotheses a ring spectrum $R$ with a $\Z$-action will behave as though the $p^k\Z$-action obtained by restriction is trivial when computing $\THH$ and $\TC$. We begin by introducing the precise finiteness hypotheses we will need.

\begin{dfn}
  Let $n \geq -1$. Following \cite{mahowald1999brown} we say that
  a $p$-complete bounded below spectrum $X$ is of \mdef{fp-type $n$}
  if, for each finite spectrum $U$ of type $n+1$, $U \otimes X$ is $\pi$-finite.\footnote{Note that by the thick subcategory theorem, it suffice to check this property for a single finite, type $n+1$ spectrum $U$ with $K(n+1) \otimes U \neq 0$.}
\end{dfn}

We will also need a cyclotomic version of being of fp-type $n$.

\begin{dfn} \label{dfn:LQ}
  Let $n\geq -1$. We say that an $\E_1$-ring $R$ satisfies the \mdef{height $n$ Lichtenbaum--Quillen property} if $\THH(R)$ is bounded below and, for each finite spectrum $V$ of type $n+2$, the cyclotomic spectrum $V \otimes \THH(R)$ is bounded.\footnote{As above.}
\end{dfn}

\begin{rmk}
  If $R$ satisfies the height $n$ LQ property,
  then for any finite spectrum $U$ of type $n+1$ the localization map
  \[ U \otimes \TC(R) \to L_{T(n)}(U \otimes \TC(R)) \]
  induces an isomorphism on $\pi_*(-)$ for $*\gg 0$.
  which can be viewed as a $\TC$ version of the Lichtenbaum--Quillen conjecture for $R$.
\end{rmk}

The main theorem of the section is then the following:

\begin{thm} \label{thm:E1A2algtamebegin}
  Suppose we are given an $R \in \Alg_{\E_1\otimes \A_2}(\Sp^{B\Z,u})$ that is
  connective, of fp-type $n \geq -1$, and satisfies the height $n$ LQ property.
  Then, for all $k \gg 0$ sufficiently large,
  $R^{hp^k\Z}$ satisfies the height $n$ LQ property as well.

  Furthermore, if $V$ is a finite spectrum of type $n+2$,
  then for $k \gg 0$ sufficiently large there is an isomorphism
  \[ V \otimes \THH(R^{hp^k\Z}) \cong V \otimes \THH(R^{Bp^k\Z}) \]
  of $\THH(\SP^{Bp^k\Z})$-modules in cyclotomic spectra.
\end{thm}

\subsection{\texorpdfstring{$\THH$}{THH} of algebras with locally unipotent \texorpdfstring{$\Z$}{Z}-actions}
\input{loc-unipotent.tex}

\subsection{The Dehn twist trivialization}\label{subsec:dehn}
\input{dehn.tex}

\subsection{Bootstrapping trivializations to \texorpdfstring{$\CycSp$}{CycSp}}\label{subsec:bootstrap}

The previous parts of this section show give conditions under which $V \otimes \THH(R^{h\ZZ})$ as a spectrum with Frobenius map is constant in that it is isomorphic to $V \otimes \W(\LCF{\Z_p}) \otimes \THH(R)^{h\ZZ}$. The goal of this subsection is to provide tools to upgrade this spectrum level statement to the level of cyclotomic spectra in order to prove \Cref{thm:E1A2algtame}. We begin with a tool that helps us prove boundedness properties for $W$-modules:

\begin{lem} \label{lem:move-inward-can-van}
  Let $M$ be a $p$-nilpotent $W$-module in cyclotomic spectra.
  If $M_{|p^k\Z_p}$ satisfies $\WcanV{b}$ and $M$ satisfies $\Seg{b}$, then
  $M$ is cyclotomically $b$-truncated.
\end{lem}

\begin{proof}  
  Let $U_{i,j}$ be the subset of $\ZZ_p$ consisting of elements $a \in \ZZ_p$ such that $i \leq v_{p}(a)\leq j$.  
  Let 
  \[ X_{i,j} \coloneqq \mathrm{Eq} \Bigg( \begin{tikzcd}
    \prod_{r \geq 0} (M_{|U_{i,j}})^{hC_{p^r}} \ar[r, yshift=-1ex, "\varphi"'] \ar[r, yshift=1ex, "\can"] & 
    \prod_{r \geq 1} (M_{|U_{i+1,j+1}})^{tC_{p^r}} \end{tikzcd} \Bigg)
  \]
  for $i < j$ where we have used the splitting of $M$ along idempotents in $\pi_0W$ to pick out components of $\can$ and $\varphi$.
  
  Using the fact that $\varphi_{|p^k\Z_p}$ factors through $M_{|p^{k+1}}\Z_p$ and the fact that $M^{tC_p} = M^{tC_p}|_{p\Z_p}$ (see \Cref{prop:cycfrobuniv}),  
  we construct a natural finite filtration
  \[ X_{0,1} \to X_{0,2} \to \cdots \to X_{0,k} \to  \TR(X) \]
  with associated graded given by
  $X_{0,1}, X_{1,2}, \dots, X_{k-1,k}, X_{k,\infty}$.
  We will show that $\TR(X)$ is $b$-truncated
  by proving that each of these terms $b$-truncated.

  First we show that $X_{i-1,i}$ is $b$-truncated.
  The localized canonical map used in $X_{i-1,i}$ is zero since the target is $0$ when restricted to $U_{i-1,i}$.
  Again using the fact that $\varphi_{|p^k\Z_p}$ factors through $M_{|p^{k+1}\Z_p}$,
  this means that $X_{i-1,i}$ is isomorphic to
  $\fib \left( \prod_{j \geq 0} (\varphi_{|U_{i-1,i}})^{hC_{p^j}} \right)$.
  The latter object is $b$-truncated since it is a product of homotopy fixed points of retracts of the fiber of $\varphi$, which is $b$-truncated by hypothesis.
  
  Now we show that $X_{k,\infty}$ is $b$-truncated as well.
  This time we use the weak canonical vanishing hypothesis which tells us that
  the localized canonical maps used in constructing $X_{k,\infty}$ are zero
  on homotopy groups for $* \geq b+1$.
  As above, it follows that $X_{k, \infty}$ is $b$-truncated since the Frobenius maps
  induce as isomorphism on $\pi_{s}(-)$ for $s > b+1$ and an injection for $s=b+1$.
\end{proof}

\begin{prop}\label{prop:Rhz-cyc-bounded}
  Let $m \geq m_p^{\A_2}$
  and let $R$ be a $h\A_2$-ring in $W$-modules in cyclotomic spectra with $p^m=0$.
  Suppose that
  \begin{enumerate}
  \item $R_{|0}$ is cyclotomically bounded in the range $[c, b]$.
  \item There is an isomorphism of graded $h\A_2$-rings 
    \[ \pi_*R \cong \pi_0(\W\LCF{\Z_p}) \otimes\pi_* R_{|0}. \]
    such that the map $R \to R_{|0}$ is given by restriction to $0$ at the level of $\pi_*$.
  \item $R$ satisfies $\Seg{b'}$.
  \end{enumerate}
  Then $R$ is cyclotomically bounded in the range $[c, e]$ where $e = b'+1+2p^{(b-2c+2)m}$.
\end{prop}
\todo{rwb: the pi star input needed to be strengthened. ishan: i don't think so. you just have to interpret the restrictions as happening at the level of W-modules instead of $\pi_*R$-modules. The reason to write this proposition this way is that it makes section 7.4 a bit easier. The proof would work with the $\pi_0\W\LCF{\Z_p}$ replaced by any augmented ring. rwb: after a closer look I think you're right.}

\begin{proof}
  Let $k \coloneqq (b-2c+2)m$ and let $e\coloneqq(b'+1+2p^k)$.
  From assumptions (1) and (2) it follows that
  the underlying spectrum of $R$ is $c$-connective.
  What remains is to show that $R$ is $e$-truncated as a cyclotomic spectrum.  
  
  Using the boundedness of $R_{|0}$ and \Cref{lem:cb-to-bok-hc}
  we pick a B\"okstedt class $\mu \in \pi_{2p^k}R_{|0}$.
  Assumption (2) now allows us to now consider the associated class
  $1 \otimes \mu \in \pi_{2p^{k}}(R)$
  whose restriction to $0$ is $\mu$.

  As $\mu$ is a B\"okstedt class,
  we may in particular fix a constant $d$ and a class
  $\alpha \in \pi_{2p^k}(\Sigma^{-1+d(1)}\Ss/a_{(1)}^d \otimes R_{|0})^{h\T}$  
  such that the image of $\alpha$ under the partial transfer map
  \[ (\Sigma^{-1+d(1)}\Ss/a_{(1)}^d \otimes R_{|0})^{h\T} \to R_{|0}^{h\T} \]
  is $\mu$.
  Writing $R_{|0}$ as $\colim_j R_{|p^j\Z_p}$ and
  using the fact that $\pi_{2p^k}(\Sigma^{-1+d(1)}\Ss/a_{(1)}^d \otimes -)^{h\T}$
  commutes with filtered colimits of uniformly $p$-nilpotent objects,
  we can, for $j \gg 0$, we pick a lift $\widetilde{\alpha}$ of $\alpha$ to
  $\pi_{2p^k}^{\T}(\Sigma^{-1+d(1)}\Ss/a_{(1)}^d \otimes R_{|p^j\Z_p})$.
  Again using the filtered colimit we can read off that,
  since both $1 \otimes \mu$ and the partial transfer of $\widetilde{\alpha}$ 
  restrict to $\mu$ at the zero fiber, these classes must become equal at some finite stage.
  In other words, after restriction to $p^j\Z_p$ for $j \gg 0$
  the class $1 \otimes \mu$ is in the image of the $d$-partial transfer
  (and in particular in the image of the transfer).
  
  Since $\mu$ is a B\"okstedt class, by \Cref{lem:sc-bo-to-bok1},
  $(R_{|0})/\mu$ is $e$-truncated.
  From this we learn that the map
  $\pi_*(\Sigma^{2p^k}R_{|0}) \xrightarrow{\mu \cdot -} \pi_*(R_{|0})$
  is an isomorphism for $*>e$ and injective for $*=e$.
  Assumption (2) now lets us conclude that multiplication by $1 \otimes \mu$
  on $R$ satisfies the same property and therefore that
  $\cof((1\otimes \mu) \cdot -)$
  is bounded in the range $[c,e]$.
  The same is true after restriction to $p^j\Z_p$
  as this is given by inverting an idempotent on the $W$-module 
  $\cof((1\otimes \mu) \cdot -)$.
	
  Applying \Cref{lem:wbock-to-scv} and \Cref{prop:omnibus}(1)
  to $R_{|p^j\Z_p}$ (together with the class $1 \otimes \mu$) we learn that
  $R_{|p^j\Z_p}$ satisfies $\WcanV{e}$ for $j \gg 0$.
  Finally, we apply \Cref{lem:move-inward-can-van} using assumption (3)
  in order to conclude that $R$ is cyclotomically $e$-truncated.
\end{proof}

\Cref{prop:Rhz-cyc-bounded} pairs nicely with the next proposition.

\begin{prop} \label{prop:bdd-to-triv}  
  Let $X$ be a $p$-nilpotent $W$-module in cyclotomic spectra
  and let $X_k \coloneqq W_k \otimes_W X $.
  If we suppose that 
  \begin{enumerate}
  \item The $X_k$ are uniformly cyclotomically bounded.
  \item There is an isomorphism of $W$-modules in spectra
    \[ X \cong W \otimes X_\infty. \]
  \item $X_\infty$ is almost compact as a cyclotomic spectrum.
  \end{enumerate}
  Then, for all $k \gg 0$ sufficiently large,
  there is an isomorphism of $W_k$-modules in cyclotomic spectra
  \[ X_k \cong W_k \otimes X_\infty. \] 
\end{prop}

\begin{proof}
  The first step is producing a comparison map.
  Since the $X_k$ are uniformly cyclotomically bounded,
  using the almost compactness of $X_\infty$ we learn that there exists a $k_1 \gg 0$
  and a lift $\psi$ of the indentity map $X_{\infty} \to X_\infty$
  through $X_{k_1}$ as display below.
  \[ \begin{tikzcd}
    & X_{k_1} \ar[dr] & \\
    X_\infty \ar[ur, dashed, "\psi"] \ar[rr, "p"] & & X_\infty
  \end{tikzcd} \]
  Using the $W_{k}$-module structure on $X_{k}$ for $k \geq k_1$ we can linearize $\psi$
  to produce a compatible system of maps of $W_{k}$-modules 
  \[ \psi_{k} : W_k \otimes X_\infty \to X_k \]
  for $k \geq k_1$ which recovers the identity map $X_\infty \to X_\infty$
  in the colimit as $k \to \infty$.

  We now show that $\psi_k$ is an isomorphism for $k \gg 0$.
  Both the $X_k$ and the $W_k \otimes X_\infty$ are uniformly bounded
  (for the latter see \Cref{cor:t-amplitude-W})
  so it will suffice to instead show that for every $s \in \Z$ there exists a $k \gg 0$
  such that the map $\pi_s(\psi_k)$ is an isomorphism.

  Using assumption (1) we can read off that
  \[ \pi_*X_k \cong \pi_*(W_k \otimes X_\infty) \cong \pi_0(\W\LCF{p^k\Z_p})\langle \zeta \rangle \otimes \pi_*X_\infty. \]
  Let $N_{k,s} \coloneqq \pi_0(\W\LCF{p^k\Z_p})\langle \zeta \rangle \otimes \pi_sX_\infty$
  and let $\psi_{k,s} : N_{k,s} \to N_{k,s}$ be the endomorphism
  corresponding to $\pi_s(\psi_k)$.
  Recall from above that $\psi_{\infty,s}$ is an isomorphism.
  Let $\overline{N}_{k,s}$ and $\overline{\psi}_{k,s}$ be
  the reduction mod $\zeta$ of $N_{k,s}$ and $\psi_{k,s}$ respectively.
  From the expression above we can read off that
  \begin{enumerate}
  \item[(a)] $\mathrm{Tor}^1_{\Z_p\langle \zeta \rangle}( \Z_p, N_{k,*} ) = 0$
    and therefore that $\psi_{k,s}$ is an isomorphism
    if both $\overline{\psi}_{k,s}$ and $\overline{\psi}_{k,s+1}$ are isomorphisms and
  \item[(b)] $\overline{N}_{k,s} \cong \pi_0(\W\LCF{p^k\Z_p}) \otimes \pi_sX_\infty$
    with the variance in $k$ given by restriction.
  \end{enumerate}
  
  Using assumption (3) we know that $\pi_sX_\infty$ is finitely generated, $p$-nilpotent
  and in view of (b) and the fact that $\overline{\psi}_{\infty,s}$ is an isomorphism
  we may apply \Cref{lem:map-formula2}(1)
  (with $\CC=\Sp$, the $R_\alpha$ the system $\W\LCF{p^k\Z_p}$ and $X=Y=\overline{N}_{k,s}$)
  and conclude that $\overline{\psi}_{k,s}$ is an isomorphism for $k \gg 0$.
  Using (a) this in turn implies that $\psi_{k,s}$ is an isomorphism for $k \gg 0$.
  This means $\pi_s\psi_k$ is an isomorphism and the proposition follows.
\end{proof}

\subsection{Finishing the argument}

In this subsection, we piece together the ingredients from earlier subsections to prove this section's main theorem.

We begin with a lemma that allows us to replace our finite spectra with connective rings:

\begin{lem}\label{lem:replacefinspec}
  Let $X$ be finite spectrum of type $w$ and let $q \geq 0$.
  There exists a connective $\E_k$-algebra $V$
  whose underlying spectrum is finite and of type $n$
  such that $X$ is a retract of $V \otimes X$.
\end{lem}

\begin{proof}
  From \cite[Thm. 1.3]{burklund2022multiplicative} we know that there exists a connective,
  finite spectrum $V'$ of type $w$ with an $\mathbb{E}_1$-algebra structure.
  For $i \geq 1$, let $V^{(i)}$ denote the cofiber of the map $I^{\otimes i} \to \SP$ where $I \to \SP \to V'$ is a cofiber sequence.
  $V^{(i)}$ is of type $w$ and when $i \geq q+1$ it admits an $\E_q$-algebra structure by
  \cite[Theorem 1.5]{burklund2022multiplicative}.
  By \Cref{rec:exponents1} $X$ is a retract of $X \otimes V^{(i)}$, for $i \gg 0$.
  Taking $V = V^{(i)}$ for $i$ sufficiently large satisfies the conditions of the lemma.
\end{proof}

We now prove the main theorem of this section:
%

\begin{thm}[Cyclotomic asymptotic constancy] \label{thm:E1A2algtame}
  Let $R \in \Alg_{\E_1\otimes \A_2}(\Sp^{B\Z,u})$ be connective of fp-type $n\geq -1$, and let $X$ be a finite spectrum of type $n+2$. Suppose that $X\otimes \THH(R)$ is bounded in the range $[c,b]$.

  There exists a function $k(R,X,n)$ such that for all $k \geq k(R,X,n)$,
  $X \otimes \THH(R^{hp^k\Z})$ is bounded in the range $[c-1,b+3]$
  and there is an isomorphism of $W_k$-modules in cyclotomic spectra        	
  \[ X\otimes \THH(R^{hp^k\Z})\cong X\otimes W_k\otimes  \THH(R). \]
\end{thm}

	
	
	

\begin{proof}
  The proof of this theorem mostly consists of
  combining the piecemeal results proved throughout this section.
  For ease of hypothesis-tracking we subdivide the proof into
  a sequence of labelled sub-claims.
  
  We apply \Cref{lem:replacefinspec} with $q=2$ to obtain a
  connective, finite spectrum $V$ of type $n+2$ equipped with an 
  $(\E_1 \otimes \A_2)$-algebra structure
  such that $X$ is a retract of $X\otimes V$. 
  For most of the proof we will work with $V$, returning to $X$ at the end.

  \begin{itemize}
  \item[(0)] $V \otimes \THH(R)$ and $X \otimes \THH(R)$ are almost compact in $\CycSp$.
  \end{itemize}
  Since $R$ is fp-type $n$, $R/p$ is bounded below and has level-wise finite homotopy groups.
  By \Cref{lem:thhac} the underlying spectra of $V\otimes \THH(R)$ and $X\otimes \THH(R)$ are almost compact. \Cref{prop:key-finiteness} now implies (0).


  \vspace{5pt}  
  \textbf{Part 1:} Locally unipotent actions
  
  \begin{itemize}
  \item[(1a)] For each $k \geq 0$, the $p^k\Z$-action on $\THH(R)$ is unipotent.
  \item[(1b)] For each $k \geq 0$, there is an isomorphism of $W_k$-modules
    \[ \THH(R^{hp^k\Z}) \cong W_k \otimes_W \THH(R^{h\Z}). \]
  \item[(1c)] For each $k \geq 0$, there is an isomorphism
    \[ \THH(R^{hp^k\Z})_{|0} \cong \THH(R)^{hp^k\Z}. \]
  \end{itemize}
  (1a), (1b) and (1c) follow form 
  Lemmas \ref{lem:unipotent_THH}, \ref{lem:thhbcres} and \ref{lem:coassemblygenTHH}
  respectively.

  \begin{itemize}
  \item[(1d)] The $p^k\ZZ$-action on $V \otimes R \in \Alg_{\E_1 \otimes \A_2}(\Sp^{B\Z,u})$ is trivial for $k\gg0$.
  \end{itemize}
  The $(\E_1\otimes \A_2)$-algebra $V \otimes R$ is
  connective, $\pi$-finite (implies $p$-nilpotent, bounded and almost compact)
  and has a locally unipotent $\Z$-action,
  therefore we may apply \Cref{lem:trivializeacalg} to obtain (1d).

  \begin{itemize}
  \item[(1e)] The $p^k\ZZ$-action on $V\otimes \THH(R)$ in $\Alg_{\A_2}(\CycSp^{B\Z,u})$ is trivial for $k\gg0$.
  \end{itemize}    
  The $\A_2$-algebra in cyclotomic spectra $V \otimes \THH(R)$ is
  connective, bounded, $p$-nilpotent, almost compact (by (0))
  and has a locally unipotent $\Z$-action (by (1a)),
  therefore we may apply \Cref{lem:trivializeacalg} to obtain (1e).
  
  \vspace{5pt}
  \textbf{Part 2:} The Dehn twist trivialization

  \begin{itemize}
  \item[(2a)] There are isomorphisms of $\A_2$-algebras in $W_k$-modules in spectra,
    \[ V \otimes \THH(R^{hp^k\Z}) \cong V \otimes \W\LCF{p^k\Z_p} \otimes \THH(R)^{hp^k\Z} \cong V \otimes W_k \otimes \THH(R) \]
    for $k\gg0$ sufficiently large.
  \item[(2b)] The cyclotomic spectrum $V \otimes \THH(R^{hp^k\Z})$ satisfies $\Seg{b}$
    for $k \gg 0$ sufficiently large.
  \end{itemize}  
  The second isomorphism in (2a) follow from (1e).
  Using (1d) we lift the pair $(R,V)$ to an $\A_2$-algebra object of $\UAlg$.
  (2a) now follows form \Cref{thm:assconst_main1}(1).

  The cyclotomic spectrum $V \otimes \THH(R)$ satisfies $\Seg{b}$ by \Cref{prop:AN_bound_to_segal}. It follows that we may prove (2b) using \Cref{thm:assconst_main1}(2).
  The additional hypothesis needed for \Cref{thm:assconst_main1}(2)
  follows from \Cref{lem:tcp-commute} using the fact that $V \otimes \THH(R)$ is cyclotomically bounded. 


  \vspace{5pt}
  \textbf{Part 3:} Bootstrapping the trivialization to $\CycSp$
  
  \begin{itemize}
  \item[(3a)] There is a $k' \gg 0$ sufficiently large so that
    the cyclotomic spectra $V \otimes \THH(R^{hp^k\Z})$ are uniformly bounded for $k \geq k'$.
  \end{itemize}
  For each $k$, $V\otimes \THH(R^{hp^k\ZZ})$ is an h$\A_2$-ring in $W_k$-modules. We wish to prove that $V\otimes \THH(R)^{hp^k\ZZ})$ is uniformly bounded in the cyclotomic $t$-structure as $k$ varies. To do this we will verify the hypotheses of \Cref{prop:Rhz-cyc-bounded} for the h$\A_2$-ring in $W=W_k$-modules in cyclotomic spectra $\THH(R^{hp^k\Z})$ for $k\gg0$. We will do this where the constants $c,b,b'$ are independent of $k$ so that the bound is uniform in $k$.

  Hypothesis (1) is satisfied because $(V \otimes \THH(R^{hp^k\ZZ}))_{|0} \cong V \otimes \THH(R)^{hp^k\ZZ}$ by (1c), which is bounded in the range $[c-1,b]$ if $V\otimes \THH(R)$ is bounded in the range $[c,b]$.
  Hypothesis (2) follows from (1a) and (2a).
  Hypothesis (3) is (2b).

  \begin{itemize}
  \item[(3b)] There is an isomorphism of $W_k$-modules in cyclotomic spectra
    \[ V \otimes \THH(R^{hp^k\Z}) \cong V \otimes W_k \otimes \THH(R) \]
     for all $k \gg 0$ sufficiently large.
  \end{itemize}
  We prove this using \Cref{prop:bdd-to-triv} with $X = V \otimes \THH(R)^{hp^k\Z}$
  where $k$ is large enough for both (2a) and (3a) to hold.
  Hypothesis (1) follows from (3a) and (1b).
  Hypothesis (2) follows from (2a) and (1b).
  Hypothesis (3) follows form (0) and (1b).

  \begin{itemize}
  \item[(3c)] There is an isomorphism of $W_k$-modules in cyclotomic spectra
    \[ X \otimes \THH(R^{hp^k\Z}) \cong X\otimes W_k\otimes \THH(R) \]
    for all $k\gg0$ sufficiently large.
  \end{itemize}
  Recall that $V$ was constructed so that $X$ is a retract of $X \otimes V$.
  This means $X \otimes \THH(R^{hp^k\Z})$ is a retract of $X \otimes V \otimes \THH(R^{hp^k\Z})$ as a $W_k$-module in cyclotomic spectra.
  Applying \Cref{lem:map-formula2}(2) along the system of the $W_k$
  using \Cref{cor:t-amplitude-W}, (0), (1b), (3b) we conclude.
	
  To finish the proof, we note that \Cref{cor:t-amplitude-W} along with (3c) implies the claimed bound in the $t$-structure for $k \gg 0$. 
\end{proof}

\begin{cor}\label{cor:E1A2algtamecoassembly}
  Let $R \in \Alg_{\E_1\otimes \A_2}(\Sp^{B\Z,u})$ be connective of fp-type $n\geq -1$, and let $X$ be a finite spectrum of type $n+2$. Suppose that $R$ satisfies the height $n$ LQ property.

  Then, there exists a $k_{R,X,n}$ such that for all $k\geq k_{R,X,n}$, there is a commutative square in cyclotomic spectra as below, where the horizontal maps are the coassembly maps:
  \[ \begin{tikzcd}
    X \otimes \THH(R^{hp^k\Z}) \ar[r,""] \ar[d,"\cong"] &
    X \otimes \THH(R)^{hp^k\Z}\ar[d,"\cong"]\\
    X \otimes \THH(R^{B\Z})\ar[r,""] &
    X \otimes \THH(R)^{B\Z}
  \end{tikzcd} \]
\end{cor}

\begin{proof}
  We apply \Cref{thm:E1A2algtame} to learn that for all $k\gg0$ is an isomorphism of $W_k$-modules in cyclotomic spectra
	
  $$X\otimes \THH(R^{hp^k\Z})\cong X\otimes W_k\otimes  \THH(R)$$
  Since the coassembly map is given by base change along the map $W_k \to \THH(\SP)^{Bp^k\Z}\cong \SP^{B\Z}$ by \Cref{lem:coassemblygenTHH}, \Cref{lem:unipotent_THH}, we are done.
\end{proof}

We now extract another corollary, which in particular allows us to understand coassembly maps telescopically in terms of the trivial action. Applying this to $R = \BPn$ and inverting $v_{n+1}$ (which is done in \Cref{thm:maincyc}), we recover \Cref{thm:assconbpn} from the introduction. The key lemma lets us access $T(n+1)$-homology by working modulo a power of $v_{n+1}$:

\begin{lem} \label{lem:unmodout}
  Let $X$ be a spectrum and let $v : \Sigma^k X \to X$ be a self map 
  such that $X/v$ is bounded in the range $[c,b]$.
  There are isomorphisms
  \begin{enumerate}
  \item $\pi_s(-) \cong \pi_s(-/v^a)$ for $s < c+ak$ and
  \item $\pi_s(-) \cong \pi_{s-k}(-)$ for $s > b$.
  \end{enumerate}
  Moreover, these isomorphisms are natural in
  maps compatible with the self maps.
\end{lem}

We note in particular that for $a \gg 0$ we have $b<c+ak$
and therefore the homotopy groups of an
$X$ as in \Cref{lem:unmodout} are determined by $\pi_*(X/v^a)$ for $a\gg0$.

\begin{proof}
  The first claim follows from examining the cofiber sequence
  $ \Sigma^{ak}X \xrightarrow{v^a} X \to X/v^a $
  and noting that $\Sigma^{ak}X$ is $(c+ak)$-connective.
  The second claim follows from examining the cofiber sequence
  $ \Sigma^{k}X \xrightarrow{v} X \to X/v $
  and noting that $X/v$ is $b$-truncated.
\end{proof}  
	
	

\begin{cor}[Telescopic asymptotic constancy]\label{cor:telasscon}
  Let $R \in \Alg_{\E_1\otimes \A_2}(\Sp^{B\Z,u})$ be connective of fp-type $n\geq 0$ and satisfy the height $n$ LQ property.
  Fix a $V \in \Sp^{\omega}$ of type $n+1$ with $v_{n+1}$-self map $v$.
  
  Then, there exists a function $k(R,V,v,n)$ such that for all $k\geq k(R,V,v,n)$,
  there is a commutative diagram of $\Z[v]$-modules as below,
  where the horizontal maps are the coassembly maps
  \[ \begin{tikzcd}
    V_*\TC(R^{hp^k\Z})\ar[r,""]\ar[d,"\cong"] & V_*\TC(R)^{hp^k\Z}\ar[d,"\cong"]\\
    V_*\TC(R^{B\Z})\ar[r,""] & V_*\TC(R)^{B\Z}
  \end{tikzcd} \]
\end{cor}

\begin{proof}
    
  Since $R$ satisfies the height $n$ LQ property, by applying \Cref{cor:t-amplitude-W}, we learn that $V/v\otimes \THH(R^{B\Z})\cong V/v\otimes W\otimes \THH(R)$ is bounded in the cyclotomic $t$-structure. This means that $V/v\otimes\THH(R^{B\Z})$ and $V/v\otimes\THH(R)^{B\Z}$ are bounded in the range $[b+1,c]$ for some $b,c$.
  Applying \Cref{cor:E1A2algtamecoassembly} we learn that for $k \gg 0$,
  $V/v\otimes \THH(R^{hp^k\Z})$ is bounded in the same range.
    
  Choose an $a \gg 0$, we now apply \Cref{cor:E1A2algtamecoassembly} again,
  this time to obtain a square 
  \[ \begin{tikzcd}
    V/v^a \otimes \THH(R^{hp^k\Z}) \ar[r,""] \ar[d,"\cong"] &
    V/v^a \otimes \THH(R)^{hp^k\Z}\ar[d,"\cong"]\\
    V/v^a \otimes \THH(R^{B\Z})\ar[r,""] &
    V/v^a \otimes \THH(R)^{B\Z}    
  \end{tikzcd} \]
  for $k \gg 0$.
  Expanding this to a cube using the self map $v$ and applying
  \Cref{lem:unmodout} we obtain the desired conclusion.
\end{proof}  
    

\begin{wrn}
  Note that the isomorphism in \Cref{cor:telasscon} is of \textbf{homotopy groups only} and generally cannot be upgraded to a spectra level isomorphism. Indeed in the key example of this paper the assembly map will become an equivalence after $K(n+1)$ localization for the Adams action  but not for the trivial action. 
\end{wrn}

%% file: loc-unipotent.tex
In this subsection we begin our investigation of the $\THH$ of algebras with unipotent $\ZZ$-actions. Here we focus on some elementary observations which follow from the work in \Cref{sec:cochaincircle} in a straightforward way.

\begin{cnv}
  Throughout this section we let
  $\mdef{W_k} \coloneqq \THH(\Ss^{Bp^k\Z})$ and $\mdef{W} \coloneqq W_0$.
  The system of natural maps
  \[ \Ss^{B\Z} \to \Ss^{Bp\Z} \to \Ss^{Bp^2\Z} \to \cdots \to \Ss \]
  induces a corresponding sequence of maps of commutative algebras in cyclotomic spectra
  \[ W \to W_1 \to W_2 \to \cdots W_\infty \cong \Ss. \]

  Given an $R \in \Alg(\Sp^{B\Z,u})$ we will use the natural
  $W_k$-modules structure on the cyclotomic spectrum $\THH(R^{hp^k\ZZ})$
  throughout the section.

  For a $W_k$-module $X$, we sometimes use $X_{|0}$ to refer to the cyclotomic spectrum $X\otimes_{W_k}\SP^{Bp^k\Z}$, where the map $W_k \to \SP^{Bp^k\Z}$ is the coassembly map, which on underlying is given by restriction to $0$ by \Cref{lem:coassembly-po}.
\end{cnv}

\begin{rmk}
  Note that given $j<k$ in fact
  $\Ss^{Bp^j\Z} \cong \Ss^{Bp^{k}\Z}$ and therefore $W_j \cong W_{k}$.
  However, the natural map $W_j \to W_{k}$ is not an isomorphism.
\end{rmk}

\begin{lem}\label{lem:unipotent_THH}
  Let $R \in \Alg(\Sp^{{B\ZZ,u}})$.
  The induced $\Z$ action on $\THH(R) \in \CycSp$ is locally unipotent
  (as is the action on the underlying spectrum of $\THH(R)$).
  
  If $R$ is connective, then the action on $\TC(R)$ is locally unipotent as well.
\end{lem}

\begin{proof}  
  The forgetful functor $\CycSp \to \Sp$ is a conservative left adjoint,
  therefore by \Cref{lem:check-unip} it suffices to show
  that the action on the spectrum $\THH(R)$ is locally unipotent.
  For this we observe that by \Cref{lem:tensor-unip} $\Sp^{{B\ZZ,u}}$ is closed under tensor products and colimits and use the tensor product formula for $\THH$.

  The final claim follows by applying \Cref{lem:check-unip} to
  $\TC\colon\Sp \otimes \CycSp_{\geq 0} \to \Sp $
  using \cite[Theorem 2.7]{clausen2021}.
  \todo{do not change, it breaks it.}
\end{proof}  

\begin{lem}\label{lem:thhbcres}  
  Let $R \in \Alg(\Sp^{B\ZZ,u})$.
  The natural map 
  \[ W_k \otimes_W \THH(R^{h\ZZ}) \xrightarrow{\cong} \THH(R^{hp^k\ZZ}) \]
  is an isomorphism of $W_k$-modules in $\CycSp$.
\end{lem}

\begin{proof}
  Since $\THH$ is a symmetric monoidal functor $\Alg(\Sp) \to \CycSp$, it will suffice to show that the map $R^{h\ZZ} \otimes_{\SP^{B\ZZ}}\SP^{Bp^k\ZZ} \to R^{hp^k\ZZ}$ is an isomorphism. This follows from \Cref{lem:basechangepZ} and \Cref{lem:unipotent_THH}.
\end{proof}

\begin{lem}\label{lem:coassemblygenTHH}
  The coassembly map fits into a commuting diagram of lax symmetric monoidal functors from
  $ \Alg(\Sp^{B\Z,u}) $ to $ \CycSp $,
  \[ \begin{tikzcd}
    & \THH((-)^{h\Z}) \ar[dl] \ar[dr] & \\
    \THH((-)^{h\Z})_{|0} \ar[rr, "\cong"] & & \THH(-)^{h\Z}.
  \end{tikzcd} \]
\end{lem}

\begin{proof}
  We begin with the lax symmetric monoidal natural transformation  
  \[ \THH(R^{h\Z}) \to \THH(R)^{h\Z} \]
  given by the coassembly map.
  We know from \Cref{lem:coassembly-po} that when $R=\Ss$ the coassembly map
  agrees with the restriction map $\THH(\Ss^{B\Z}) \to \THH(\Ss^{B\Z})_{|0}$ so by base-change
  we obtain a lax symmetric monoidal natural transformation under $\THH(R^{h\Z})$:
  \[ \epsilon \colon \THH(R^{h\Z})_{|0} \to \THH(R)^{h\Z}. \]
  
  In order to prove $\epsilon$ is an isomorphism
  we use \Cref{rmk:unip-affine} to reduce to checking it is an isomorphism after applying
  $\Ss \otimes_{\Ss^{B\Z}} - $.
  Now we have isomorphisms
  \begin{align*}
    \Ss \otimes_{\Ss^{B\Z}} \THH(R^{h\Z})_{|0}
    &\cong \Ss \otimes_{\Ss^{B\Z}} \THH(R^{h\ZZ}) \otimes_{\W\LCF{\Z_p}} \W(\F_p) 
    \cong W_{\infty} \otimes_{W} \THH(R^{h\ZZ}) \\
    &\cong \THH(R) 
    \cong \Ss \otimes_{\Ss^{B\Z}} \THH(R)^{h\Z}
  \end{align*}
  where we have used
  \Cref{lem:thh-s1} in the second step,
  \Cref{lem:thhbcres} in the third step and
  \Cref{lem:basechangepZ} and \Cref{lem:unipotent_THH} in the final step.
\end{proof}

  %
  

%% file: dehn.tex
Our next goal is to prove a weak version of \Cref{thm:E1A2algtamebegin} at the level of the underlying spectrum of $\THH$ together with the Frobenius map (rather than at the level of cyclotomic spectra). 

We will work with rings $R\in \Alg(\Sp^{B\Z,u})$ equipped with a finite spectrum $V\in \Sp^{\diamondsuit}$ and a trivialization of the $\Z$-action on $V \otimes R$. It will be convenient to also assume that $V$ has an associative algebra structure and encode this as a part of the data, that is we consider the following category:

\begin{dfn}
  We let $\UAlg$ be the presentably symmetric monoidal category
  defined by the pullback below
  \[ \begin{tikzcd}
    \UAlg \ar[r, "{u,v}"] \ar[d] \pullback &  
    \Alg(\Sp)^{B\Z,u} \times \Alg(\Sp^{\diamondsuit})
    \ar[d, "{(R,V)\mapsto(R\otimes V)}"] \\
    \Alg(\Sp) \arrow[r, "{\mathrm{triv}}"'] &
    \Alg(\Sp)^{B\Z,u}.
  \end{tikzcd} \]

  We will typically denote objects of \mdef{$\UAlg$} by pairs $(R,V)$ where
  $R$ is the algebra with locally unipotent $\Z$-action and
  $V$ is the finite algebra
  (leaving all other data implicit).
\end{dfn}

Our goal is now to study the functor
\[ (R,V) \mapsto V \otimes \THH(R^{h\ZZ}). \]
More precisely, we will prove the following theorem which serves as a replacement for $\Z_p$-hyperdescent for $\THH$.

\begin{thm} \label{thm:assconst_main1}
  There is a natural transformation of lax symmetric monoidal functors
  \[ \eta : \W(\LCF{\Z_p}) \otimes V \otimes \mathrm{res}_{\varphi}\left( \THH(R)^{h\Z} \right)
  \Rightarrow V \otimes \mathrm{res}_{\varphi} \THH(R^{h\Z}) : \UAlg \to \Sp^{\Delta^1} \]
  such that
  \begin{enumerate}
  \item $\eta$ becomes an isomorphism after composing with
    pullback along $i_0 : * \to \Delta^1$.
  \item $\eta$ becomes an isomorphism upon 
    restricting to the subcategory of those $(R,V)$ for which the assembly map
  \[ ((V \otimes \THH(R))^{tC_p})^{\oplus \infty} \to
  ((V \otimes \THH(R))^{\oplus \infty})^{tC_p} \]
  is an isomorphism.
  \end{enumerate}
\end{thm}

\subsubsection{The setup}
\label{subsubsec:notation}

Before proceeding we will need to set up a certain amount of notation for the functors we will use for manipulating objects of $\UAlg$.

\begin{itemize}
\item Let \mdef{$\EQ$} be the category $0 \rightrightarrows 1$.
\item Let $\mdef{\tau} : \EQ \to \Delta^1$ be the functor identifying the two arrows.
\item Recall that $B\Z^{\rhd}$ sits in a pushout square
  \[ \begin{tikzcd}
    B\Z \ar[d] \ar[r, "i_1"] \ar[d, "\pi"] &
    B\Z \times \Delta^1 \ar[d, "\pi'"] \\
    * \ar[r, "i_1"] &
    B\Z^{\rhd}. \pushout
  \end{tikzcd} \]
\item Let $\mdef{\rho}$ be the essentially surjective functor $\EQ \to B\Z^{\rhd}$.
\item Let $\mdef{c}$ be the functor $B\Z^{\rhd} \to *$.
\item Let $\mdef{e}$ be the functor $\EQ \to *$.
\item By abuse of notation we will also use \\
  $\mdef{\tau}$ for the map $\EQ \times B\Z \to \Delta^1 \times B\Z$,\\
  $\mdef{\pi}$ for the map $\EQ \times B\Z \to \EQ$, \\
  $\mdef{\pi}$ for the map $\Delta^1 \times B\Z \to \Delta^1$ and \\
  $\mdef{e}$ for the map $\EQ \times B\Z \to B\Z$.
\end{itemize}

\begin{cnstr}
  We may construct a commuting diagram of symmetric monoidal functors
  \[ \begin{tikzcd}[column sep=huge]
    \Alg(\Sp)^{B\Z,u} \times \Alg(\Sp^{\diamondsuit})
    \ar[r, "{(R,V)\mapsto(R\otimes V)}"']
    \ar[d, "{(R,V) \mapsto (R \to V \otimes R)}"] &
    \Alg(\Sp)^{B\Z,u} \ar[d] &
    \Alg(\Sp) \ar[l, "{\mathrm{triv}}"] \ar[d, equal] \\
    (\Alg(\Sp)^{B\Z})^{\Delta^1} \ar[r, "{\mathrm{ev}_1}"'] &
    \Alg(\Sp)^{B\Z} &
    \Alg(\Sp) \arrow[l, "{\mathrm{triv}}"] 
  \end{tikzcd} \]
  and upon taking pullbacks of the horizontal cospans this gives us a symmetric monoidal functor
  \[ \mdef{\mathrm{cn}} : \UAlg \to \Alg(\Sp)^{B\Z^{\rhd}} \]
  and a symmetric monoidal natural isomorphism
  \[ \rho^*\mathrm{cn}(R,V) \cong \pi^*\left( \Ss \rightrightarrows V \right) \otimes e^*R.\hfill\qedhere \]
\end{cnstr}

\begin{cnstr}
  Given a diagram $R \rightrightarrows S  \in \Alg(\cC)^{\EQ}$, $S$ is naturally an $S$-$S$-bimodule, so by using the two different maps $R \to S$, we can forget down to the structure of an $R$-$R$-bimodule. We thus obtain a symmetric monoidal functor
  \[ \mdef{\bm} : \Alg(\Sp)^{\EQ} \to \Bimod. \]
  and a symmetric monoidal natural isomorphism between $\bm \circ e^*$ and the functor sending
  $R \in \Alg(\Sp)$ to $(R,R) \in \Bimod $ where $R$ is viewed as an $R$-$R$-bimodule in the standard way.
\end{cnstr}

\begin{itemize}
\item Let $\mdef{\THH^{\EQ}} \coloneqq \THH \circ \bm $ be the symmetric monoidal functor
  sending a diagram $R \rightrightarrows S$ to
  the THH of $R$ with coefficients in the $R$-bimodule $S$.
\item Let $\mdef{\THH^{\EQ}_{\hex}} \coloneqq \THH_{\hex} \circ \bm$ be the symmetric monoidal functor
  obtained by composing $\bm$ with the $p$-polygonic $\THH$ functor discussed in \Cref{subsubsec:thh}.
  Note that there is a natural isomorphism of symmetric monoidal functors
  $ \THH^{\EQ}_{\hex}(-)^{\Phi C_p} \cong \THH^{\EQ}(-)$ (see \Cref{subsubsec:polygonic} and \cite[Theorem D]{krause2023polygonic}).
\end{itemize}

\subsubsection{The tilt of the unit}

Our next task will be to determine the tilt of
$\THH_{\hex}^{\EQ}(\pi_*\rho^*\Ss)$.

\begin{cnstr} \label{cnstr:flat-hex}
  If we write the unit $\Ss \in \Alg(\Sp)^{B\Z^{\rhd}}$ as
  $c^*\Ss$ we may use the identification
  $ \pi_*\rho^*c^*\Ss \cong e^*\Ss^{B\Z} $
  and \Cref{lem:polygonic-compatibility}
  to obtain isomorphisms
  \[ \THH_{\hex}^{\EQ}(\pi_*\rho^*\Ss) \cong \THH_{\hex}^{\EQ}(e^*\Ss^{B\Z}) \cong \mathrm{res}_{\hex}\THH(\Ss^{B\Z}). \]
  Expanding this out into an isotropy seperation square using \Cref{lem:polygonic-compatibility}
  we obtain an identification of the diagrams of commutative algebras below.
  \[ \begin{tikzcd}
    \THH_{\hex}^{\EQ}(\pi_*\rho^*\Ss)^{C_p} \ar[r] \ar[d] \pullback &
    \THH_{\hex}^{\EQ}(\pi_*\rho^*\Ss)^{\Phi C_p} \ar[d] &
    \THH_{\hex}^{\EQ}(\pi_*\rho^*\Ss)^{C_p} \ar[r] \ar[d] \pullback &
    \THH(\Ss^{B\Z}) \ar[d, "\varphi"] \\
    \THH_{\hex}^{\EQ}(\pi_*\rho^*\Ss)^{hC_p} \ar[r] \ar[d] &
    \THH_{\hex}^{\EQ}(\pi_*\rho^*\Ss)^{tC_p} &
    \THH(\Ss^{B\Z})^{hC_p} \ar[r, "\mathrm{can}"] \ar[d] &
    \THH(\Ss^{B\Z})^{tC_p}  \\
    \THH_{\hex}^{\EQ}(\pi_*\rho^*\Ss)^{\Phi e} & &
    \THH(\Ss^{B\Z}) &    
  \end{tikzcd} \]
  Applying $\pi_0^{\flat}(-)$ to the diagrams above and using \Cref{cor:tateTHHSZ} and \Cref{prop:cycfrobuniv} to identify
  the objects and morphisms we obtain the following diagram:
  \[ \begin{tikzcd}
    \LCF{p^{-1}\Z_p} \ar[r, "(-)_{|\Z_p}"] \ar[d, "{\res_{1/p}}"] &
    \LCF{\Z_p} \ar[d, "{\res_{1/p}}"] \\
    \LCF{\Z_p} \ar[r, "(-)_{|p\Z_p}"] \ar[d, "\cong"] &
    \LCF{p\Z_p} \\
    \LCF{\Z_p}. &
  \end{tikzcd} \]

  Through this we obtain a distinguished map of commutative algebras
  \[ \W \LCF{p^{-1}\Z_p} \to \THH_{\hex}^{\EQ}(\pi_*\rho^*\Ss)^{C_p}. \]
  Since $\THH_{\hex}^{\EQ}(\pi_*\rho^*(-))$ is lax symmetric monoidal, it canonically refines to a lax symmetric monoidal functor
  \begin{align*}
    \Alg(\Sp)^{B\Z^{\rhd}}
    &\xrightarrow{\THH_{\hex}^{\EQ}(\pi_*\rho^*(-))}
    \Mod(\pgnsp{\lpr}; \THH_{\hex}^{\EQ}(\pi_*\rho^*\Ss)) \\
    &\to \Mod(\pgnsp{\lpr}; \W \LCF{p^{-1}\Z_p}).
  \end{align*}
  which we, by abuse of notation, also denote $\THH_{\hex}^{\EQ}(\pi_*\rho^*(-))$.
\end{cnstr}

The key point in the construction above is that it gives us access to restriction operations
indexed by $p^{-1}\Z_p$ on the output of $\THH_{\hex}^{\EQ}(\pi_*\rho^*(-))$.
The main way we will use this is in the following lemma:

\begin{lem} \label{lem:outer-vanish}
  For any $A \in \Alg(\Sp)^{B\Z^{\rhd}}$ we have
  \[ \mathrm{res}_{\varphi}\left( \THH_{\hex}^{\EQ}(\pi_* \rho^* A )_{|p^{-1}\Z_p^\times} \right) = 0. \]
\end{lem}

\begin{proof}
  $\mathrm{res}_{\varphi}( \THH_{\hex}^{\EQ}(\pi_* \rho^* A )_{|p^{-1}\Z_p^\times} )$ is naturally a
  $\mathrm{res}_{\varphi}( \THH_{\hex}^{\EQ}(\pi_* \rho^* \Ss )_{|p^{-1}\Z_p^\times} )$-module
  and it will therefore suffice to prove that the latter object is zero.
  Examining the third diagram in \Cref{cnstr:flat-hex} we can read off that
  \[ \pi_0^{\flat}\left( \mathrm{res}_{\varphi}( \THH_{\hex}^{\EQ}(\pi_* \rho^* \Ss )_{|p^{-1}\Z_p^\times} ) \right) \cong \left( \LCF{\Z_p}_{|p^{-1}\Z_p^{\times}} \xrightarrow{\res_{1/p}} \LCF{p\Z_p}_{|\Z_p^\times} \right) \cong \left( 0 \to 0 \right).\hfill\qedhere \]
\end{proof}

\subsubsection{Constructing the Dehn twist}

We now upgrade the functor $\THH_{\hex}^{\EQ}(\pi_*\rho^*(-))$ to a functor
taking values in $\pgnsp{\lpr}^{*\mm\Z,u}$.
We refer to this additional $\Z$-action as the Dehn twist
and its existence is instrumental in the proof of \Cref{thm:assconst_main1}.


\begin{cnstr} \label{cnstr:dehn}
We start by considering the functor  
\[\mdef{\sigma_0} \colon \EQ \to B\ZZ\]
sending the top map to $0$ and the bottom map to $1$.
Using $\sigma_0$ we construct an automorphism of $\EQ \times B\Z$
via the formula $\sigma(a,b) \coloneqq (a, \sigma_0(a) + b)$.
The automorphism $\sigma$ naturally fits into a commutative diagram
\[\begin{tikzcd}
	{\EQ\times B\ZZ } & {\EQ\times B\ZZ } \\
	\EQ & \EQ.
	\arrow["\pi", from=1-1, to=2-1]
	\arrow["\pi", from=1-2, to=2-2]
	\arrow["\mathrm{Id}", from=2-1, to=2-2]
	\arrow["\sigma", from=1-1, to=1-2]
\end{tikzcd}\]
Viewing this automorphism as a $\Z$-action and taking homotopy orbits we obtain a functor\footnote{Note that $B\Z \cong *\mm\Z$. We maintain this notational distinction purely to help the reader to keep in mind the role played by the different circles we see.}
\[\mdef{\bar{\pi}} \colon  (\EQ\times B\ZZ )\mm\ZZ \to\EQ\times *\mm\ZZ. \]
Similarly, the map $\rho : \EQ \times B\Z \to B\Z^{\rhd}$ is naturally equivariant for this $\Z$-action and so we obtain a corresponding functor
\[ \mdef{\bar{\rho}} \colon (\EQ\times B\ZZ )\mm\ZZ \to B\Z^{\rhd} \times *\mm\Z. \]
Finally, we let $\bar{\varrho}$ be the composite of $\bar{\rho}$ with the
projection $B\Z^{\rhd} \times *\mm\Z \to B\Z^{\rhd}$.
\end{cnstr}

\begin{rmk}\label{rmk:dehnaction}
  Associated to the pullback squares
  \[ \begin{tikzcd}
    B\Z^{\rhd} \ar[d, "t"] &
    \EQ \times B\Z \ar[r, "\pi"] \ar[l, "\rho"] \ar[d, "t"]&
    \EQ  \ar[d, "t"] \\
    B\Z^{\rhd} \times *\mm\Z &
    (\EQ \times B\Z) \mm\Z \ar[r, "\bar{\pi}"] \ar[l, "\bar{\rho}"] &
    \EQ \times (*\mm\Z)             
  \end{tikzcd} \]
  we have an isomorphism
  \[ t^*\bar{\pi}_*\bar{\varrho}^*A \cong \pi_*t^*\bar{\varrho}^*A \cong \pi_*\rho^*A \]
  natural in $A \in \CC^{B\Z^{\rhd}}$.
  In other words, $\bar{\pi}_*\bar{\varrho}^*A$ equips $\pi_*\rho^*A$ with a $\Z$-action.
\end{rmk}

\begin{dfn}
  Given an $A \in \Alg(\Sp)^{B\Z^{\rhd}}$ 
  we will refer to the $\Z$-action on $\THH_{\hex}^{\EQ}(\bar{\pi}_*\bar{\varrho}^*A)$
  as the \mdef{Dehn twist}.
\end{dfn}

\begin{rmk}
  The reason we refer to this action as a Dehn twist is that
  the classifying space of $\EQ \times B\Z$ is a torus and the map $\sigma$ is a Dehn twist on this torus.
  The map $\pi$ is then the map witnessing that if we view a torus as a trivial $S^1$-bundle on a circle, then the Dehn twist can be made into a bundle automorphism.
  The map $\bar{\varrho}$ can be described as witnessing that the Dehn twist acts trivially on the associated nodal curve.
\end{rmk}

\subsubsection{The Dehn twist on the unit}
\input{dehn2.tex}

\subsubsection{The trivialization}

\begin{lem} \label{lem:yea-nilp}
  Let $A \in \Alg(\Sp)^{B\Z^{\rhd}}$.
  The Dehn twist action on $\THH_{\hex}(\bar{\pi}_*\bar{\varrho}^*A)$
  is locally unipotent.
\end{lem}

\begin{proof}
  Per \Cref{lem:check-loc-nilp} it will suffice to argue that the $\Z$-actions on
  $\THH_{\hex}^{\EQ}(\bar{\pi}_*\bar{\varrho}^*A)^{\Phi e}$ and
  $\THH_{\hex}^{\EQ}(\bar{\pi}_*\bar{\varrho}^*A)^{\Phi C_p}$ are each locally unipotent.
  From the formulas
  \begin{align*}
    \THH_{\hex}(A;M)^{\Phi C_p} &\cong \THH(A;M) \\
    \THH_{\hex}(A;M)^{\Phi e} &\cong \THH(A;M^{\otimes_A p}) \\
    \THH(A;M) &\cong A \otimes_{A \otimes A} M \cong \colim A \otimes (A \otimes A)^{\otimes \bullet} \otimes M \\
    M^{\otimes_A p} &\cong \colim_{(\Delta^\op)^{\times p-1}} \left( M \otimes A^{\otimes \bullet} \otimes M \otimes \cdots \otimes A^{\otimes \bullet} \otimes M \right)
  \end{align*}
  and the fact that $\Sp^{B\Z,u}$ is closed under colimits and $\otimes$
  we learn that it will suffice for us to argue that
  the Dehn twist action on
  $\bar{\iota}^* \bar{\pi}_*\bar{\varrho}^*A$ is trivial
  where $\bar{\iota} : (\bullet \coprod \bullet) \times *\mm\Z \to \EQ \times *\mm\Z$
  is the natural embedding.
  
  Consider the following diagram
  \[\begin{tikzcd}
	& {B\Z^{\rhd}} \\
	{B\Z\coprod B\Z} & {(B\Z\coprod B\Z)\times *//\Z} & {(\EQ\times B\Z)//\Z} \\
	{\bullet \coprod \bullet} & {(\bullet \coprod \bullet)\times *//\Z} & {\EQ\times *//\Z}
	\arrow["{\bar{i}}", from=2-2, to=2-3]
	\arrow["{\bar{\varrho}}"', from=2-3, to=1-2]
	\arrow["{\bar{\varpi}}", from=2-2, to=2-1]
	\arrow["j", from=2-1, to=1-2]
	\arrow["\bar{\pi}"', from=2-3, to=3-3]
	\arrow["{\bar{\iota}}", from=3-2, to=3-3]
	\arrow["q"', from=2-1, to=3-1]
	\arrow["\varpi", from=3-2, to=3-1]
	\arrow["{\bar{q}}", from=2-2, to=3-2]
  \end{tikzcd}\]
  Where $q,\bar{q},\varpi,\bar{\varpi}$ are projections and 
  $j$ embed the first copy of $B\Z$ into $\BZgg$ and send the second copy to the cone point. Further, $q$ and $\bar{\pi}$ are cartesian fibrations and and both squares are pullback squares.\footnote{All these assertions are easily verifiable by hand as all the categories that appear in the diagram are $1$-categories with only two objects.}  

  
  We now have isomorphisms
  \[ \bar{\iota}^* \bar{\pi}_*\bar{\varrho}^*A \cong \bar{q}_* \bar{i}^* \bar{\varrho}^*A \cong \bar{q}_* \bar{\varpi}^* j^*A \cong \varpi^* q_* j^*A \]
  from which we can read off that the Dehn twist action is trivial as desired.  
\end{proof}


\begin{prop} \label{prop:dehn-twist-trivialization}
  There is a natural isomorphism of lax symmetric monoidal functors
  $ \Alg(\Sp)^{B\Z^{\rhd}} \to \pgnsp{\lpr}$
  between the functor sending $A \in \Alg(\Sp)^{B\Z^{\rhd}}$ to
  $\THH_{\hex}( \pi_* \rho^* A )$
  and the functor sending $A \in \Alg(\Sp)^{B\Z^{\rhd}}$ to    
  \[ \left( \W \LCF{\Z_p} \otimes \THH_{\hex}(\pi_* \rho^* A )_{|0} \right) \oplus \THH_{\hex}(\pi_* \rho^* A )_{|p^{-1}\Z_p^\times}. \] 
\end{prop}

\begin{proof}
  By lax symmetric monoidality, $\THH_{\hex}^{\EQ}( \bar{\pi}_* \bar{\varrho}^*(-))$ naturally lands
  in $\Z$-equivariant modules over
  $\THH_{\hex}^{\EQ}( \bar{\pi}_* \bar{\varrho}^*\Ss)$.
  Restricting along the counit of the spherical Witt vector, tilt adjunction after applying $(-)^{C_p}$ we obtain a refinement of $\THH_{\hex}^{\EQ}( \bar{\pi}_* \bar{\varrho}^*(-))$
  landing in
  \[ \Mod \left(\pgnsp{\lpr}^{*\mm\Z} ;\  \W \pi_0^\flat(\THH_{\hex}^{\EQ}(\bar{\pi}_*\bar{\varrho}^*\Ss)^{C_p}) \right). \]
  The $\Z$-equivariant identification from \Cref{cor:dehn-refined} now gives us a natural $\Z$-equivariant splitting
  \[ \THH_{\hex}^{\EQ}(\bar{\pi}_* \bar{\varrho}^* A ) \cong \THH_{\hex}^{\EQ}(\bar{\pi}_* \bar{\varrho}^* A )_{|\Z_p} \oplus \THH_{\hex}^{\EQ}(\bar{\pi}_* \bar{\varrho}^* A)_{|p^{-1}\Z_p^\times} \]
  where the first term is a $\Z$-equivariant $\W \LCF{\overrightarrow{\ZZ_p}}$-module.
  Applying \Cref{prop:unipotent_SW} (the necessary unipotence hypotheses having been checked in \Cref{lem:yea-nilp})
  we now obtain the desired lax symmetric monoidal identification
  \[ \THH_{\hex}^{\EQ}(\bar{\pi}_*\bar{\varrho}^* A )_{|\Z_p} \cong \W\LCF{\overrightarrow{\ZZ_p}} \otimes \left( \THH_{\hex}^{\EQ}(\pi_*\rho^* A )_{|\Z_p} \right)_{|0}.\hfill\qedhere \]
\end{proof}

\begin{lem} \label{lem:unroll}
  There is a symmetric monoidal natural isomorphism of functors from
  $\UAlg$ to $\W\LCF{p^{-1}\Z_p}$-modules in $\pgnsp{\lpr}$
  \[ \Nm(V) \otimes \mathrm{res}_{\hex} \THH(R^{h\Z})
  \cong \THH_{\hex}^{\EQ}(\pi_*\rho^*\cn(R,V)). \]
\end{lem}

\begin{proof}
  Using
  \Cref{lem:norm},
  \Cref{lem:tenspr_with_finite} and
  \Cref{lem:polygonic-compatibility},
  we obtain isomorphisms
  \begin{align*}
    \THH_{\hex}^{\EQ} &(\pi_*\rho^*\cn(R,V))
    \cong \THH_{\hex}^{\EQ}\left(\pi_*\left( \pi^*(\Ss \rightrightarrows V) \otimes e^* R \right) \right) 
    \cong \THH_{\hex}^{\EQ}\left( (\Ss \rightrightarrows V) \otimes \pi_* e^* R \right) \\
    &\cong \THH_{\hex}^{\EQ}\left( (\Ss \rightrightarrows V) \otimes e^* R^{h\Z} \right) 
    \cong \Nm(V) \otimes \THH_{\hex}^{\EQ}\left( e^* R^{h\Z} \right) \\
    &\cong \Nm(V) \otimes \mathrm{res}_{\hex} \THH\left( R^{h\Z} \right).\hfill\qedhere
  \end{align*}   
\end{proof}

\begin{proof}[{Proof (of \Cref{thm:assconst_main1})}]
  Using \Cref{lem:unroll}, 
  \Cref{prop:dehn-twist-trivialization} and
  \Cref{lem:outer-vanish} we obtain symmetric monoidal natural isomorphisms
  \begin{align*}
    V &\otimes \mathrm{res}_{\varphi} \THH(R^{h\Z})
    \cong \res_{\varphi} \THH_{\hex}^{\EQ}(\pi_*\rho^*\cn(R,V)) \\
    &\cong \res_{\varphi}\left( \left( \W(\LCF{\Z_p}) \otimes \THH_{\hex}^{\EQ}(\pi_*\rho^*\cn(R,V))_{|0} \right) \oplus  \THH_{\hex}^{\EQ}(\pi_*\rho^*\cn(R,V))_{|p^{-1}\Z_p} \right) \\
    &\cong \res_{\varphi}\left( \W(\LCF{\Z_p}) \otimes \THH_{\hex}^{\EQ}(\pi_*\rho^*\cn(R,V))_{|0} \right).
  \end{align*}

  Using \Cref{lem:unroll} a second time together with
  the identification of zero fibers from \Cref{lem:coassemblygenTHH}
  we can simplify the inner part of this last term further
  \begin{align*}
    \THH_{\hex}^{\EQ} &(\pi_*\rho^*\cn(R,V))_{|0}
    \cong \left( \Nm(V) \otimes \mathrm{res}_{\hex} \THH(R^{h\Z}) \right)_{|0} \\
    &\cong \Nm(V) \otimes \mathrm{res}_{\hex} \left( \THH(R^{h\Z}) \right)_{|0} 
    \cong \Nm(V) \otimes \mathrm{res}_{\hex} \THH(R)^{h\Z}.
  \end{align*}
  Plugging this into the above and using the fact that $\mathrm{res}_{\varphi}(\Nm(V)) \cong V$
  and $\mathrm{res}_{\varphi}(-)$ is symmetric monoidal when one of the inputs is
  a compact object of $\Sp^{C_p}$ we arrive at a natural isomorphism
  \[ V \otimes \mathrm{res}_{\varphi}\THH(R^{h\Z})
  \cong V \otimes \mathrm{res}_{\varphi}\left( \W(\LCF{\Z_p}) \otimes \mathrm{res}_{\hex}( \THH(R)^{h\Z} ) \right). \]

  Using the assembly map 
  \[ \W(\LCF{\Z_p}) \otimes \mathrm{res}_{\varphi}\left(  \mathrm{res}_{\hex}( \THH(R)^{h\Z} ) \right) \to \mathrm{res}_{\varphi}\left( \W(\LCF{\Z_p}) \otimes \mathrm{res}_{\hex}( \THH(R)^{h\Z} ) \right)   \]
  together with \Cref{lem:polygonic-compatibility} we now obtain the desired lax symmetric monoidal natural transformation
  \[ \W(\LCF{\Z_p}) \otimes V \otimes \mathrm{res}_{\varphi}( \THH(R)^{h\Z} ) \to 
  V \otimes \mathrm{res}_{\varphi}(\THH(R^{h\Z})). \]

  For the final claims we must analyze when the assembly map above is an isomorphism.
  As the pullback along $0 \coprod 1 \to \Delta^1$ is conservative it suffices to
  analyze the assembly maps for $(-)^{\Phi C_p}$ and $(-)^{tC_p}$ separately.
  The former is colimit preserving (proving (1)),
  so we are reduced to the case of $(-)^{tC_p}$ where this follows from the given hypothesis
  (here we use that the condition that the assembly map be an isomorphism is
  stable under cofiber sequences and that
  the underlying spectrum of $\W(\LCF{\Z_p})$ is a sum of spheres).
\end{proof}

%% file: dehn2.tex

\begin{prop} \label{prop:dehn-action}
  Under the isomorphism 
  \[ \pflat(\THH^{\EQ}_{\hex}(\pi_*\rho^*\Ss)^{\Phi C_p}) \cong \LCF{\Z_p} \]
  from \Cref{cnstr:flat-hex}
  the Dehn twist action is identified with restriction along the map $+1 : \Z_p \to \Z_p$.
\end{prop}

The idea behind this proposition is that $\THH_{\hex}(\pi_*\rho^*\Ss)^{\Phi C_p} \cong \THH(\SP^{B\Z})$ is a continuous version of the cochains of the free loop space on the $p$-adic circle, which can be thought of as the space of sections of the projection $B\ZZ_p\times B\ZZ \to B\ZZ$. The Dehn twist acts as an automorphism of $B\ZZ_p\times B\ZZ$ which by construction induces the automorphism on the space of sections that on $\pi_0$ is the map $+1:\ZZ_p \to \ZZ_p$.

Before proving the proposition we will need a lemma simplifying the computation of $\THH^{\EQ}$ for commutative algebras.

\begin{lem}\label{lem:thhcolim}
  There is a commutative diagram of symmetric monoidal functors	
  \[ \begin{tikzcd}
    {\CAlg(\CC)^{\EQ}} & {\Alg(\CC)^{\EQ}} \\
    {\CAlg(\CC)} & \CC
    \arrow["{\THH^{\EQ}}", from=1-2, to=2-2]
    \arrow[from=1-1, to=1-2]
    \arrow["\colim"', from=1-1, to=2-1]
    \arrow[from=2-1, to=2-2]
  \end{tikzcd} \]
  where $\THH^{\EQ}$ denotes the $\THH$ relative to $\CC$.
\end{lem}

\begin{proof}
  For a diagram $A \rightrightarrows B$ in $\CAlg(C)^{\EQ}$, the colimit can be computed as $A\otimes_{A\otimes A}B$, which is also a formula for $\THH^{\EQ}$.
\end{proof}

\begin{proof}[Proof of \Cref{prop:dehn-action}]
  As explained in \Cref{subsubsec:notation} we have an $\Z$-equivariant isomorphism of commutative algebras
  \[ \THH^{\EQ}_{\hex}(\bar{\pi}_*\bar{\varrho}^*\Ss)^{\Phi C_p} \cong \THH^{\EQ}(\bar{\pi}_*\bar{\varrho}^*\Ss) \]
  so it will suffice to focus on the latter.
  Next we note that the operations $\bar{\pi}_*$ and $\bar{\varrho}_*$ commute with the forgetful functor from commutative algebras to algebras, so we are free to work at the level of commutative algebras
  By \Cref{lem:thhcolim}, we can replace $\THH^{\EQ}$ with the colimit in $\CAlg(\Sp)$.
  This lets us rewrite our object of interest as
  $ \colim_{\EQ}\bar{\pi}_*\bar{\varrho}^*\SP $.
  

  The operations $\bar{\pi}_*$ and $\bar{\varrho}_*$ are also
  compatible with the symmetric monoidal, limit preserving, functor
  $\SP^{(-)}:\Spc^{\op} \to \CAlg(\Sp)$.
  This gives us, for each space $X \in (\Spc^{\op})^{B\Z^{\rhd}}$,
  a Dehn twist-equivariant assembly map
  \[ \colim_{\EQ} \bar{\pi}_*\bar{\varrho}^* \Ss^{X} \to \Ss^{\lim_{\EQ}\bar{\pi}_*\bar{\varrho}^*X}. \]
  Taking $X$ to be the constant diagram on a point,
  $\lim_{\EQ}\pi_*\rho^*X$
  is the free loop space of $B\Z$,
  and it follows from \Cref{lem:thh-s1}
  that this assembly map is (on $\pflat$) the map
  $\LCF{\ZZ_p} \to \FF_p^{\ZZ}$
  taking a continuous function $\ZZ_p \to \FF_p$ to its restriction to $\ZZ \subset \ZZ_p$.
  The map $\LCF{\ZZ_p} \to \FF_p^{\ZZ}$ is injective since $\ZZ$ is dense in $\ZZ_p$, so to verify the claim of the proposition, it will suffice to show that the Dehn twist action on
  $\FF_p^{\ZZ} \cong \pflat \Ss^{\lim_{\EQ}\bar{\pi}_*\bar{\varrho}^*\bullet} $
  is given by the automorphism coming from pre-composition with the map $+1: \ZZ \to \ZZ$.
  This in turn reduces to determining the Dehn twist action on $\pi_0$ of
  $\lim_{\EQ}\bar{\pi}_*\bar{\varrho}^*\bullet$.
  

  Unrolling definitions and using the fact that $\pi : \EQ \times B\Z \to \EQ$ is a cartesian fibration we identify $\lim_{\EQ} \pi_*\rho^* \bullet$ with the space of sections of $\pi$
  and identify the Dehn twist action with the action on the space of sections generated by the map $\sigma$ from \Cref{cnstr:dehn}.
  The $\pi_0$ of the space of sections of $\pi$ can be identified with
  pairs of integers $(i,j)$
  (recording the values of the top and bottom arrow respectively)
  modulo the equivalence relation generated by $(i,j) = (i+1,j+1)$
  (coming from the automorphism of the source (or target) object).
  The autormorphism $\sigma$ acts by sending a section $(i,j)$ to $(i,j+1)$.
  The proposition follows.\qedhere
  \todo{I believe the sign here depends how THH and THHEQ are actually identified (i.e. this should depend on an orientation of EQ) I dont think we are precise enough for this choice to have appeared previously, so I would say we're just correct up to retroactively figuring out which orientation we needed.}

\end{proof}

\begin{cor} \label{cor:dehn-refined}
  There is an isomorphism of $\Z$-equivariant algebras
  \[ \pflat(\THH_{\hex}(\bar{\pi}_*\bar{\varrho}^*\Ss)^{C_p})
  \cong \LCF{\overrightarrow{\ZZ_p}} \times (\LCF{p^{-1}\Z_p^\times}, \psi) \]
  lying over the isomorphism from \Cref{cnstr:flat-hex}
  for some continuous automorphism $\psi$ of $p^{-1}\Z_p^\times$.
\end{cor}

\begin{proof}  
  In \Cref{cnstr:flat-hex} we identified the natural map
  \[ \pflat(\THH_{\hex}(\pi_*\rho^*\Ss)^{C_p}) \to \pflat(\THH_{\hex}(\pi_*\rho^*\Ss)^{\Phi C_p}) \]
  with the map
  \[ \LCF{p^{-1}\Z_p} \xrightarrow{(-)_{|\Z_p}} \LCF{\Z_p}. \]
  Let $\psi$ denote the Dehn twist actions on the source.
  From Stone duality we know that $\psi$ acts as
  precomposition by some continuous automorphism $p^{-1}\Z_p \to p^{-1}\Z_p$.
  The fact that the map above is compatible with the Dehn twist action
  implies that for $c \in \Z_p$ we have $\psi(c) = c+1$.
  It follows that $\psi$ preserves $p^{-1}\Z_p^\times$ as well
  and we obtain the desired splitting.
\end{proof}

%% file: adamsop.tex
\label{sec:adamsop}
This section, which use no inputs from the other sections of the paper, is devoted to the construction of well-behaved \emph{Adams operations} on $\BPn$.  We will need such Adams operations to disprove the telescope conjecture at height $n+1$ in \Cref{sec:mainthm}. We note that another disproof of the of the telescope conjecture at height $2$ and primes at least $7$ is given in \Cref{sec:ht2} which is independent of this section.


\begin{cnv} \label{cnv:ell}\label{cnv:BPnEn}
  In this section, we do \textit{not} implicitly $p$-complete objects as in much of the rest of the paper.
  
  For the remainder of this section,
  we let $\mdef{\ell} \coloneqq m_p^{\E_1}$ 
  
  As they are important for this section we remind the reader of conventions (\ref{item:BPn})-(\ref{item:Gn}):
  \begin{itemize}
  \item \mdef{$\BPn$} refers to an $\E_3$-$\MU_{(p)}$-algebra form of the truncated Brown--Peterson spectrum as constructed in \cite[\S 2]{hahn2020redshift}.
  \item \mdef{$E_n$} refers to the height $n$ Lubin--Tate theory constructed by Goerss--Hopkins--Miller \cite[Theorem 5.0.2]{ECII}, associated to the (unique up to isomorphism) height $n$ formal group over $\Fpbar$.\qedhere
  \item \mdef{$\G_n$} refers to the height $n$ extended Morava stabilizer group, which acts on $E_n$ and fits into a short exact sequence
$$1 \to \cO_D^{\times} \to \G_n \to \Gal(\FF_p) \to 1$$ where \mdef{$\cO_D^{\times}$} is the units in the maximal order of the division algebra over $\QQ_p$ of Hasse invariant $\frac 1 n$, and $\Gal(\FF_p) \cong \hat{\ZZ}$ is the absolute Galois group of $\FF_p$.
  \end{itemize}
\end{cnv}

Our first goal will be to construct an Adams operation $\Psi^{\ell}$ on the Lubin--Tate theory $E_n$, and the $p$-localized complex bordism spectrum $\MU_{(p)}$.

\begin{cnstr} \label{cnstr:psi-En}
  Let $\mdef{\Psi^{\ell}} \in \G_n$ be the commutative algebra automorphism
  \[ \Psi^{\ell} {\colon}E_n \to E_n, \]
  arising from the action of $\Z_{p}^\times$ on the chosen formal group of height $n$ over $\Fpbar$, that acts on $\pi_{2k} E_n$ by multiplication by $\ell^k$.

  We write $E_n^\Psi \in \CAlg(\Sp_{K(n)}^{B\Z})$ for
  $E_n$ equipped with the $\ZZ$-action whose generating automorphism is $\Psi^{\ell}$.
\end{cnstr}

Next we construct an Adams operation on $\MU$.

\begin{cnstr} \label{cnstr:adams-conj}
  The stable Adams conjecture, as proved in \cite{Friedlander},\footnote{See also  \cite[Lemma 2.2]{clausen2012padic}, \cite[Section 16.2]{norms}, \cite[Appendix A]{bhattacharya2022stable}. The exact statement we use here is \cite[Theorem 1.8]{bhattacharya2022stable}.} provides us with a commuting diagram of infinite loop maps
  \[ \begin{tikzcd}
    \mathrm{BU}_{(p)} \ar[rr, "\Psi^{\ell}"] \ar[dr, "J"] & &
    \mathrm{BU}_{(p)} \ar[dl, "J"] \\
    & \mathrm{BSL}_1(\mathbb{S}_{(p)}). & 
  \end{tikzcd} \]
  Thomifying this diagram we obtain a commutative algebra automorphism 
  \[ \Psi^{\ell} {\colon}\MU_{(p)} \to \MU_{(p)} \]
  that we call the \mdef{Adams operation on $\MU_{(p)}$}.
  Note that $\Psi^\ell$ depends on the choice of homotopy in the diagram of infinite loop spaces above. We fix a choice of this homotopy for the remainder of the paper. 

  We write $\MU_{(p)}^\Psi \in \CAlg(\Sp^{B\Z})$ for
  $\MU_{(p)}$ equipped with the $\ZZ$-action whose generating automorphism is $\Psi^{\ell}$.
\end{cnstr}

We now turn to the main goal of the section: constructing an Adams operation $\Psi^\ell$ on 
each of the truncated Brown--Peterson spectrum $\BPn$.

\begin{thm} \label{thm:Adams-ops-exist}
  The $(\E_1 \otimes \A_2)$-$\MU_{(p)}$-algebra underlying
  the $\E_3$-$\MU_{(p)}$-algebra $\BPn$ admits
  a lift to an object
  \[ \BPn^\Psi \in \Alg_{\E_1 \otimes \A_2}\left(\Mod(\Sp^{B\Z}; \MU_{(p)}^{\Psi}) \right) \]
  such that
  \begin{enumerate}
  \item there is a map $ \iota {\colon}\BPn^\Psi \to E_n^\Psi $ in
    $\Alg_{\E_1}\left( \Mod(\Sp^{B\Z}; \MU_{(p)}^\Psi) \right)$,
  \item an identification
    \[ \begin{tikzcd}
      L_{T(n)} \BPn^\Psi \ar[r, "\iota"] \ar[d, "\cong"] & E_n^\Psi \ar[d, "\cong"] \\
      (E_n^\Psi)^{h \mu_{p^n-1} \rtimes \hat{\Z}} \ar[r] & E_n^\Psi
    \end{tikzcd} \]
    in $\Alg_{\E_1}(\Sp^{B\Z})$ where $\mu_{p^n-1} \rtimes \hat{\Z} \subseteq \G_n$
    fits into a map of exact sequences
    \[ \begin{tikzcd}
      \mu_{p^n-1} \ar[r, hook] \ar[d, hook] &
      \mu_{p^n-1} \rtimes \hat{\Z} \ar[r, two heads] \ar[d, hook] &
      \hat{\Z} \ar[d, "\cong"] \\
      \cO_D^\times \ar[r, hook] & 
      \G_n \ar[r, two heads] &
      \Gal(\F_p), 
    \end{tikzcd} \]
  \item and the underlying $\Z$-action on $\BPn$ is locally unipotent in $p$-complete spectra after $p$-completion.
  \end{enumerate}  
\end{thm}

This theorem is the only result from this section used in later sections of the paper.

\begin{rmk}
  It would be interesting to know whether the Adams operation of \Cref{thm:Adams-ops-exist} underlies an $\mathbb{E}_2$-algebra or $\mathbb{E}_3$-algebra automorphism. If so, one could equip $\THH(\BPn^{h\ZZ})$ with an $\mathbb{E}_1$-algebra or $\mathbb{E}_2$-algebra structure. As it is, \Cref{thm:Adams-ops-exist} equips $\THH(\BPn^{h\ZZ})$ with a unital multiplication map, but provides no guarantee of either homotopy associativity or commutativity.
\end{rmk}



\begin{rmk}
  Complex conjugation on $\MU_{(p)}$ can be viewed as the action of the Adams operation $\Psi^{-1}$.  This is not the same as the action of $\Psi^{\ell}$ we study here, even at $p=2$, but it may be noteworthy that the interaction between complex conjugation, Morava $E$-theory, and $\BPn$ has received a great deal of prior attention.  This interaction is at the heart of the Hill--Hopkins--Ravenel approach to Kervaire invariant $1$ and the computation of $C_{2^{k}}$ fixed points of Morava $E$-theory \cite{HuKriz,HHR,hahn2020real,beaudry2021models}.  From the point of view of Real homotopy theory, $\BPn$ with its action by $\Psi^{-1}$ is expected to be an $\mathbb{E}_{2\sigma+1}$-$\MU_{\mathbb{R}}$-algebra, and in particular an $\mathbb{E}_{2\sigma+1}$-algebra \cite[Remark 1.0.14]{hahn2020redshift}. Since $2\sigma+1$ contains only one copy of the trivial representation, this would not imply a $C_2$-action on $\BPn$ by $\mathbb{E}_2$-algebra maps. However, $\Z$-equivariantly $S^{2\sigma} \cong S^2$, so $\Psi^{-1}$ itself would be an $\mathbb{E}_3$-algebra map.
\end{rmk}

\subsection{An \texorpdfstring{$\E_4$}{E4} complex orientation of Lubin--Tate theory}

Here, we explain how the $\E_3$-$\MU_{(p)}$-algebra structure on $\BP\langle n\rangle$ can be used to consrtruct a map of $\E_3$-$\MU_{(p)}$-algebras $\BP\langle n \rangle \to E_n$. We use this along with the self centrality of $E_n$ to construct an $\E_4$-algebra map $\MU_{(p)} \to E_n$. 

\begin{cnstr}\label{cnstr:root}
  In \cite[Proposition 2.6.2]{hahn2020redshift}, it is shown that $\BPn$ admits
  the structure of an $\mathbb{E}_{3}$-$\MU_{(p)}[y]$-algebra, where $y$ in degree $2p^{n}-1$ acts by $v_n$.
  Here, \mdef{$\MU_{(p)}[y]$} is the Thom spectrum of a composite of commutative monoid maps
  \[ \begin{tikzcd}
    \mathbb{N} \arrow[r,"p^n-1"] & \mathbb{Z} \arrow[r] & \mathrm{Pic}(\MU_{(p)}). 
  \end{tikzcd} \]
  Let $\MU_{(p)}[y^{1/(p^{n}-1)}]$ denote the Thom spectrum of the composite
  \[ \begin{tikzcd}
    \mathbb{N} \arrow[r,"1"] & \mathbb{Z} \arrow[r] & \mathrm{Pic}(\MU_{(p)}),
  \end{tikzcd} \]
  the natural commutative $\MU_{(p)}$-algebra map
  $\MU_{(p)}[y] \to \MU_{(p)}[y^{1/p^{n}}]$
  allows us to construct an $\E_3$-$\MU_{(p)}$-algebra
  \[ \BPn \left[ v_n^{1/(p^n-1)} \right] \coloneqq \MU_{(p)}[y^{1/(p^{n}-1)}] \otimes_{\MU_{(p)}[y]} \BP\langle n \rangle.\hfill\qedhere\]
\end{cnstr}

\begin{lem} \label{lem:root-galois}
  The commutative algebra map 
  \[ \W(\Fp) \otimes \MU_{(p)}[y^{\pm1}]
  \to \colim_k \W(\F_{p^k}) \otimes \MU_{(p)}[y^{\pm 1/p^{n}}] \]
  obtained from the map from \Cref{cnstr:root}
  by inverting $y$ and tensoring up to $\W(\Fpbar)$ on the target
  is a $\mu_{p^n-1} \rtimes \hat{\Z}$ pro-Galois extension.
\end{lem}

\begin{proof}
  Let $A \coloneqq \W(\F_p) \otimes \MU_{(p)}[y^{\pm1}]$,
  and let $B \coloneqq \colim_k \W(\F_{p^k}) \otimes \MU_{(p)}[y^{\pm 1/p^{n}}]$.
  Using the fact that  $\pi_0\W(\mathbb{F}_{p^n})$ has
  a primitive $(p^n-1)^{\mathrm{st}}$ root of unity
  we can read off that the map $\pi_*A \to \pi_*B$ is
  a graded $\mu_{p^n-1} \rtimes \hat{\Z}$-pro-Galois extensions of graded commutative rings.
  Using \cite[Corollary 10.1.5]{rognesgalois} (as $k$ varies) we lift this
  to the structure of a $\mu_{p^n-1} \rtimes \hat{\Z}$-pro-Galois extension
  on the map $A \to B$.
\end{proof}

%
%

The following proposition was first observed as \cite[Example 8.7]{rootadjunction}:


\begin{prop} \label{prop:detectEn}
  There is an identification of underlying $\E_3$-algebras
  \[ \begin{tikzcd}
    L_{T(n)}\BPn \ar[r] \ar[d, "\cong"] &
    L_{T(n)}\left( \W(\Fpbar) \otimes \BPn \left[ v_n^{1/(p^n-1)} \right] \right)
    \ar[d, "\cong"] \\
    E_n^{h \mu_{p^n-1} \rtimes \hat{\Z}} \ar[r] & E_n,
  \end{tikzcd} \]
  where $\mu_{p^n-1} \rtimes \hat{\Z} \subseteq \G_n$
  fits into a map of exact sequences 
  \[ \begin{tikzcd}
    \mu_{p^n-1} \ar[r, hook] \ar[d, hook] &
    \mu_{p^n-1} \rtimes \hat{\Z} \ar[r, two heads] \ar[d, hook] &
    \hat{\Z} \ar[d, "\cong"] \\
    \cO_D^\times \ar[r, hook] & 
    \G_n \ar[r, two heads] &
    \Gal(\F_p). 
  \end{tikzcd} \]
\end{prop}

\begin{proof}
  Let $A \coloneqq L_{T(n)}\BPn$
  and let $B \coloneqq L_{T(n)}\left( \W(\Fpbar) \otimes \BPn \left[ v_n^{1/(p^n-1)} \right] \right) $.  
  The homotopy ring of $A$ is the completion of the graded ring
  $\Z_{(p)}[v_1,\dots, v_{n-1},v_n^{\pm1}]$ at the Landweber ideal $(p,\dots,v_{n-1})$.
  It follows that the homotopy ring of $B$ is obtained from this by
  adjoining a $(p^n-1)^{\mathrm{st}}$-root of $v_{n}$ that lives in degree $2$,
  base changing to $W(\Fpbar)$, and re-completing at $(p,\dots,v_{n-1})$.
	
  Since $B$ is an $\MU_{(p)}$-algebra with homotopy ring satisfying Landweber's criterion, the natural map
  \[ \pi_*(B) \otimes_{\pi_*(\MU_{(p)})} \MU_{(p),*}(X) \to B_*(X) \]
  is an isomorphism, so that $B$ is a Landweber exact homotopy associative and commutative ring. Because its homotopy ring agrees with that of a Lubin--Tate theory, it follows from \cite[Corollary 4.53]{ramzi2023separability} that it agrees as an $\E_3$-algebra with a Lubin--Tate theory, which is $E_n$ since there is one formal group over $\Fpbar$ of height $n$ up to isomorphism.

  Let $G = \mu_{p^n-1} \rtimes \hat{\Z}$.
  From \Cref{lem:root-galois} we obtain an $\E_3$-algebra
  $\mu_{p^n-1} \rtimes \hat{\Z}$-action on
  $B$ over $A$ such that $A \cong B^{hG}$.
  Examining the action on $\pi_*B$ we read off that the map
  $G \to \pi_0\Aut_{\E_3}(B)$ is injective.
  Using \cite[Corollary 4.53]{ramzi2023separability} again we know that 
  the $\E_3$-automorphisms of $E_n$ are the discrete set $\G_n$.
  Viewing $G$ as a subgroup of $\G_n$ through the identification $B \cong E_n$
  and examining the action of $\hat{\Z}$ on $\W(\Fpbar) \subset \pi_0E_n$
  gives the map of exact sequences.
\end{proof}

\begin{cnstr}\label{cnst:BPntoEn}
  As a consequence of \Cref{prop:detectEn}, we obtain a refinement of the 
  underlying $\E_3$-algebra structure on $E_n$ to an $\E_3$-$\MU_{(p)}$-algebra structure
  and a map of $\EE_3$-$\MU_{(p)}$-algebras
  \[ \iota {\colon}\BP\langle n \rangle \to E_n.\hfill \]
  We note that this map is injective on $\pi_*$ by the proof of \Cref{prop:detectEn}.
\end{cnstr}

Recall that for an $\E_m$-algebra $R$, its center $\mathcal{Z}_{\mathbb{E}_m}(R)$ is the terminal $\E_{m+1}$-algebra $A$ equipped with a lift of $R$ to an $\E_m$-$A$-algebra \cite[Section 5.3]{HA} \cite{francis2013tangent}.
In particular, the $\E_3$-$\MU_{(p)}$-algebra structure on $E_n$ corresponds to an $\E_4$-algebra map $\MU_{(p)} \to \mathcal{Z}_{\EE_3}E$.



\begin{prop}\label{prop:selfcentral}
  Let $m\geq2$. The natural $\E_m$-algebra map $\mathcal{Z}_{\E_m}(E_n) \to E_n$ is an isomorphism,
  and the $\E_{m+1}$-algebra structure on $\mathcal{Z}_{\E_m}(E_n)$ agrees with
  the restriction of the commutative algebra structure on $E_n$.
\end{prop}

\begin{proof}
  It follows from \cite[Proposition 3.16]{francis2013tangent}
  that $\mathcal{Z}_{\E_m}(E_n)$ can be computed as
  \[\End_{\int_{\mathbb{R}^m-0}E_n}E_n,\]
  Using the fact that $E_n$ is $K(n)$-local we obtain an isomorphism
  \[ \map_{\int_{\mathbb{R}^m-0}E_n}(E_n, E_n)
  \cong L_{K(n)} \map_{L_{K(n)}\int_{\mathbb{R}^m-0}E_n}(E_n, E_n), \]
  so for the first claim it suffices to show that the map
  $\int_{\mathbb{R}^m-0}E_n\to E_n$ is a $K(n)$-local isomorphism.
  Using the isomorphism 
  \[ \int_{\mathbb{R}^m-0}E_n \cong \colim_{S^{m-1}}E_n, \]
  from \cite[Corollary 3.27]{francis2013tangent},
  where the colimit is taken in commutative algebras,
  it will suffice to argue that $E_n$ is codiscrete as a $K(n)$-local commutative algebra.
  In order to check that $E_n$ is codiscrete it suffices to show that the map
  $\colim_{S^1} E_n \cong \THH(E_n) \to E_n$ is $K(n)$-locally an isomorphism.
  This follows from the fact that $E_n$ is a $K(n)$-local pro-Galois extension of
  $L_{K(n)}\SP$ \cite[5.4.6]{rognesgalois}.
	
  It remains to check that the $\E_{m+1}$-structure on
  $\mathcal{Z}_{\mathbb{E}_m}(E_n) \cong E_n$ is the usual one.
  The commutative algebra structure on $E_n$ restricts to
  an $\E_m$-$E_n$-algebra structure on $E_n$.
  This gives us an $\E_{m+1}$-algebra map $E_n \to \mathcal{Z}_{\mathbb{E}_m}(E_n)$
  whose composite with the $\mathbb{E}_m$-algebra map
  $\mathcal{Z}_{\mathbb{E}_m}(E_n) \to E_n$
  is the identity $\mathbb{E}_m$-algebra map.
  This means that the $\E_{m+1}$-algebra map $E_n \to \mathcal{Z}_{\mathbb{E}_m}(E_n)$
  is an isomorphism.
\end{proof}

The following corollary is immediate from \Cref{prop:selfcentral} and \Cref{prop:detectEn}.

\begin{cor}\label{cor:E4upgrade}
  The $\E_3$-$\MU_{(p)}$-algebra structure on $E_n$ constructed via \Cref{prop:detectEn}
  arises from an $\E_4$-algebra map $\MU_{(p)} \to E_n$. 
\end{cor}

\subsection{A \texorpdfstring{$\Z$}{Z}-equivariant \texorpdfstring{$\E_3$}{E3} complex orientation of Lubin--Tate theory}

In this subsection, we show that the underlying $\E_3$-algebra map of the
$\E_4$-algebra map $\MU_{(p)} \to E_n$ constructed in the previous subsection
is compatible with the action of the Adams operation $\Psi^{\ell}$.

\begin{prop} \label{prop:E-thy-Adams}\todo{check}
  The underlying $\E_3$-algebra map of any $\E_4$-algebra map
  $\MU_{(p)} \to E_n$
  can be refined to a map of $\Z$-equivariant $\E_3$-algebras
  $\MU_{(p)}^\Psi \to E_n^\Psi$.
\end{prop}

Before embarking on the proof of \Cref{prop:E-thy-Adams}, we set some notation.

\begin{dfn}
  Suppose we are given an integer $a > 0$ and a $k \in \Z_{(p)}^{\times}$.
  \begin{itemize}
  \item We let $ \mdef{S^{(a,[k])}} \in \Spc_*^{B\Z} $
    be the $\ZZ$-equivariant pointed space
    consisting of $S_{(p)}^a$ together with a degree $k$ self-map.
  \item We let \mdef{$\Omega^{a,[k]}$} be the right adjoint to $S^{(a,[k])} \wedge -$.
    This is given by a $\Z$-equivariant mapping space out of $S^{a,[k]}$.
  \item We let $\mdef{\Ss^{(a,[k])}} \coloneqq \Sigma^\infty S^{(a,[k])} \in \Sp_{(p)}^{B\Z}$.
  \item Given a $p$-local $\Z$-equivariant infinite loop space $M=\Omega^{\infty}m$, 
    we let \mdef{$B^{a,[k]}M$} be the $\Z$-equivariant infinite loop space
    $\Omega^{\infty}( \Ss^{(a,[k])} \otimes m)$.
    In particular, $\Omega^{a,[k]} B^{a,[k]} M \cong M$.\qedhere
  \end{itemize}
\end{dfn}

Note that smash products of $\Z$-equivariant spheres are given by the formula 
\[S^{(a_1,[k_1])} \wedge S^{(a_2,[k_2])} \cong S^{(a_1+a_2,[k_1 k_2])}.\]

\begin{proof}[Proof of \Cref{prop:E-thy-Adams}]
  The space of $\E_4$-algebra maps $\MU_{(p)} \to E_n$ is isomorphic to
  the space of nullhomotopies of the composite
  \[\mathrm{BU}_{(p)} \stackrel{J}{\to} \mathrm{BSL}_1(\mathbb{S}_{(p)}) \to \mathrm{BSL}_1(E_n)\]
  in the category of four-fold loop maps \cite[Theorem 3.5]{BAC}.\footnote{Recall that $\mathrm{SL}_1$ is the $1$-connective cover of $\mathrm{GL}_1$.}
  Equivalently, this is the space of nullhomotopies in pointed spaces of the composite
  \[ \mathrm{B}^4\mathrm{BU}_{(p)} \to \mathrm{B}^5\mathrm{SL}_1(\mathbb{S}_{(p)}) \to \mathrm{B}^5\mathrm{SL}_1(E_n) \]
  
  We will prove the theorem by showing that every nullhomotopy of this composite
  can be refined to a $\Z$-equivariant nullhomotopy of the composite
  \begin{align}
    \mathrm{B}^{4,[\ell^2]}\mathrm{BU}_{(p)} \to \mathrm B^{5,[\ell^2]}\mathrm{SL}_1(\mathbb{S}_{(p)}) \to \mathrm{B}^{5,[\ell^2]}\mathrm{SL}_1(E_n^\Psi). \label{eqn:thom}
  \end{align}
  Here, both $\mathrm{BU}_{(p)}$ and $\mathrm{SL}_1(E_n^\Psi)$ are equipped with $\Z$ actions via $\Psi^{\ell}$, and $\mathrm{SL}_1(\mathbb{S}_{(p)})$ is equipped with trivial $\Z$-action. The $\Z$-equivariant infinite loop map $\mathrm{BU}_{(p)} \to \mathrm{B}\mathrm{SL}_1(\mathbb{S})$ is given by the solution to the stable Adams conjecture fixed in \Cref{cnstr:adams-conj}.
  
  To see that this is sufficient, note that the functor
  $\Omega^{4,[\ell^2]}(-) {\colon}\Spc_*^{B\Z} \to \Spc_*^{B\Z}$
  takes values in triple loop spaces.
  This is because $S^{4,[\ell^2]} \cong S^{3,[0]} \wedge S^{1,[\ell^2]}$,
  where first smash factors has trivial $\Z$-action.

  First, we show that the composite in (\ref{eqn:thom}) is $\Z$-equivariantly nullhomotopic.
  Let $X, Y \in \Spc_*^{B\Z}$ be the equivariant mapping spaces
  \[ \Map_{\Spaces_*}(\mathrm{B}^{4,[\ell^2]}\mathrm{BU}_{(p)},\mathrm{B}^{5,[\ell^2]}\mathrm{SL}_1(E_n^\Psi)) \text{ and } \Map_{\Spaces_*}\left(\mathrm{B}^{4,[\ell^2]}\mathrm{BU}_{(p)},\mathrm{B}^{5,[\ell^2]}\mathrm{SL}_1(E_n^\Psi)_{\mathbb{Q}}\right) \]
  respectively (where the $\ZZ$-action is by conjugation)
  so that we have an isomorphism
  \[ \Map_{\Spaces_*^{B\ZZ}}(\mathrm{B}^{4,[\ell^2]}\mathrm{BU}_{(p)},\mathrm{B}^{5,[\ell^2]}\mathrm{SL}_1(E_n^\Psi)) \cong X^{h\Z},\]
  between the space of $\Z$-equivariant maps and the fixed points of $X$.
  The obstruction to (\ref{eqn:thom}) being null is now a class in $\pi_0(X^{h\Z})$.
  By \Cref{lem:trivial-action} below, the homotopy groups of $X$ 
  are torsion free, concentrated in odd degrees, the $\ZZ$-action on $\pi_1$ is trivial
  and the natural map $X \to Y$ is a $\pi_*$-injection.
  Altogether this implies that the natural map
  $ \pi_0(X^{h\Z}) \to \pi_0(Y^{h\Z}) $
  is injective. To prove that (\ref{eqn:thom}) is null
  it now suffices to prove that the further composite
  \begin{align*}
    \mathrm{B}^{4,[\ell^2]}\mathrm{BU}_{(p)} \to \mathrm B^{5,[\ell^2]}\mathrm{SL}_1(\mathbb{S}_{(p)}) \to \mathrm{B}^{5,[\ell^2]}\mathrm{SL}_1(E_n^\Psi) \to \mathrm{B}^{5,[\ell^2]}\mathrm{SL}_1(E_n^\Psi)_{\mathbb{Q}}
  \end{align*}
  is $\ZZ$-equivariantly nullhomotopic.
  This composite factors through
  $\mathrm B^{5,[\ell^2]} \mathrm{SL}_1(\mathbb{S}_{(p)})_{\Q}$
  which is contractible.

  It remains to show that one can find a $\ZZ$-equivariant nullhomotopy of (\ref{eqn:thom}) lifting any given non-equivariant nullhomotopy. Homotopy classes of equivariant and nonequivariant nulhomotopies are torsors over $\pi_1(X^{h\Z})$ and $\pi_1(X)$ respectively.
  Using \Cref{lem:trivial-action} again we see that these groups agree since
  $\pi_2(X)$ is trivial, and $\pi_1(X)$ has a trivial action.
  Thus any given non-equivariant nullhomotopy can be refined to an equivariant nullhomotopy.
\end{proof}

\begin{lem} \label{lem:trivial-action}\todo{check}
  The homotopy groups of the pointed mapping space
  \[ \Map_{\Spaces_*}(\mathrm{B}^{4,[\ell^2]}\mathrm{BU}_{(p)},\mathrm{B}^{5,[\ell^2]}\mathrm{SL}_1(E_n^{\Psi})) \]
  are torsion free,
  concentrated in odd degrees,
  have $\Psi^\ell$ act by $\ell^k$ on $\pi_{2k+1}$,
  and embed into the homotopy groups of the mapping space to the rationalization
  \[ \Map_{\Spaces_*}(\rB^{4,[\ell^2]}\mathrm{BU}_{(p)},\rB^{5,[\ell^2]}\mathrm{SL}_1(E_n^{\Psi})_{\QQ}). \]
\end{lem}

\begin{proof}
  Let $X, Y \in \Spc_*^{B\Z}$ be the equivariant mapping spaces
  \[ \Map_{\Spaces_*}\left(\rB^{4,[\ell^2]}\mathrm{BU}_{(p)},\rB^{5,[\ell^2]}\mathrm{SL}_1(E_n^\Psi) \right) \text{ and } \Map_{\Spaces_*}\left(\rB^{4,[\ell^2]}\mathrm{BU}_{(p)},\rB^{5,[\ell^2]}\mathrm{SL}_1(E_n^\Psi)_{\mathbb{Q}}\right) \]
  respectively (where the $\ZZ$-action is by conjugation).    
  We begin by analyzing the underlying (non-equivariant) spaces of $X$ and $Y$.
  Filtering the target spaces
  $\rB^{5,[\ell^2]}\mathrm{SL}_1(E_n^{\Psi})$
  and $\rB^{5,[\ell^2]}\mathrm{SL}_1(E_n^{\Psi})_{\Q}$
  by their Postnikov towers we obtain spectral sequences
  of signature
  \[ \begin{tikzcd}
    H^s( \BU\langle 6 \rangle_{(p)} ;\ \pi_t(\rB^{5}\mathrm{SL}_1(E_n)) )
    \ar[r, Rightarrow] \ar[d] &
    \pi_{t-s}(X) \ar[d] \\
    H^s( \BU\langle 6 \rangle_{(p)} ;\ \pi_t(\rB^{5}\mathrm{SL}_1(E_n)_{\Q}) )
    \ar[r, Rightarrow] &
    \pi_{t-s}(Y)
  \end{tikzcd} \]
  and a map between them (recall that the
  underlying space of $\rB^{4,[\ell^2]}\BU_{(p)}$ is $\BU\langle 6 \rangle_{(p)}$).

  $E_n$ has torsion free homotopy concentrated in even degrees
  and the integral cohomology of $\BU\langle 6 \rangle$
  is a finite sum of copies of $\Z$ concentrated only in even degrees
  \cite{Singer} (cf. \cite[Corollary 4.7]{cube}).
  Therefore, on the $E_1$-page these spectral sequences are
  concentrated entirely in degrees with $t-s$ odd and the map between them is injective.  
  In particular, both spectral sequence degenerate at the $E_1$-page
  and converge strongly.
  From this we learn that $\pi_*X$ and $\pi_*Y$ are concentrated in odd degrees
  and the map $\pi_*X \to \pi_*Y$ is injective.

  In order to complete to proof it will now suffice
  (using the injectivity proved above)
  to analyze the $\Z$-action on $\pi_*(Y)$.
  The Postnikov tower of $\Sigma^{5, [\ell^2]}\mathrm{sl}_1(E_n^\Psi)) \otimes \Q$
  splits $\ZZ$-equivariantly
  (here we use that the group ring $\Q[\Z]$ has projective dimension $1$).
  As a consequence we obtain a $\Z$-equivariant isomorphism,
  \[ \pi_{2k+1}(Y) \cong \prod_{t-s=2k+1} H^{s} \left( \rB^{4,[\ell^2]}\BU_{(p)}; \Q \otimes \pi_{t} ( \rB^{5,[\ell^2]}\mathrm{SL}_1(E_n^\Psi) ) \right). \]
  We now analyze the $\Z$-actions on each term in this product. Using the fact that the cohomology of $\BU\langle 6 \rangle$ is finitely generated in each degree we obtain a $\Z$-equivariant isomorphism
  \[
  H^{s} \left( \rB^{4,[\ell^2]}\BU_{(p)};
  \Q \otimes \pi_{t} (\rB^{5,[\ell^2]}\mathrm{SL}_1(E_n^\Psi)) \right) \cong
  H^{s}(\rB^{4,[\ell^2]}\BU_{(p)}; \Q) \otimes \pi_{t}( \rB^{5,[\ell^2]}\mathrm{SL}_1(E_n^\Psi)).
  \]

  Returning to \Cref{cnstr:psi-En} we now determine that
  \begin{itemize}
  \item[(a)] $\Psi^\ell$ acts on $\pi_{2k}E_n^\Psi$ by $\ell^{k}$,
  \item[(b)] $\Psi^\ell$ acts on
    $\pi_{2k+1}\left( \rB^{5,[\ell^2]}\mathrm{SL}_1(E_n^\Psi) \right)$
    by $\ell^{k}$,
  \item[(c)] $\Psi^\ell$ acts on $\pi_{2k}\bu_{(p)}$ by $\ell^{k}$,
  \item[(d)] $\Psi^\ell$ acts on $\pi_{2k}\rB^{4,[\ell^2]}\BU_{(p)}$ by $\ell^k$,
  \end{itemize}

  Using that the rational homology of $\BU\langle 6 \rangle$ 
  is a polynomial algebra generated off of the image of the homotopy groups
  we can now work out the action on the rational homology (and cohomology)
  of $\rB^{4,[\ell^2]}\BU_{(p)}$,
  \begin{itemize}      
  \item[(e)] $\Psi^\ell$ acts on $H_{2k}(\rB^{4,[\ell^2]}\BU_{(p)}; \Q)$ by $\ell^k$ and
  \item[(f)] $\Psi^\ell$ acts on $H^{2k}(\rB^{4,[\ell^2]}\BU_{(p)}; \Q)$ by $\ell^{-k}$.
  \end{itemize}
  Using (b), (f) and the isomorphism above we determine that
  $\Psi^\ell$ acts on $\pi_{2k+1}(Y)$ by $\ell^k$.
\end{proof}

\begin{cor} \label{cor:E2EnAdams}
  The $\E_3$-$\MU_{(p)}$-algebra structure on $E_n$ constructed via \Cref{prop:detectEn}
  refines to a lift of $E_n^\Psi$
  to a $\Z$-equivariant $\E_2$-$\MU_{(p)}^\Psi$-algebra $E_n^\Psi$.
\end{cor}

\begin{proof}
  This follows by applying
  \Cref{cor:E4upgrade}, \Cref{prop:E-thy-Adams} and \Cref{prop:selfcentral}.
\end{proof}

\subsection{Digression: \texorpdfstring{$\E_1$}{E1}-cells}

The goal of this digression is to set up a theory of $\E_1$-cells that works well when applied to both the category of $\MU_{(p)}$-modules and the category of $\Z$-equivariant $\MU_{(p)}$-modules. In \Cref{subsec:e1a2}, we use this to first upgrade $\BP\langle n \rangle$ to a $\Z$-equivariant $\E_1$-$\MU_{(p)}$-algebra, and then to a $\Z$-equivariant $(\E_1\otimes \A_2)$-$\MU_{(p)}$-algebra. Our treatment here takes inspiration from \cite{beardsley2021skeleta}.

\begin{cnv}
  Throughout this subsection $\CC$ will denote a fixed
  stable, presentably symmetric monoidal category
  equipped with a $t$-structure $(\CC_{\geq 0}, \CC_{\leq 0})$
  compatible with the monoidal structure.
\end{cnv}

\begin{rec}	
  Recall from \cite[Corollary 1.4.4.5]{HA} that the universal colimit preserving functor
  from a presentable category $\DD$ to a stable presentable category is the stabilization,
  \[ \Sigma_{\DD}^\infty {\colon}\DD \to \Sp(\DD). \]	
  By \cite[Chapter 4, Theorem 4.2]{francis2008derived},
  for $A$ an $\E_1$-algebra in $\CC$,
  the stabilization of $\Alg(\CC)_{/A}$ 
  can be identified with the functor sending $B \to A$ to the $A$-bimodule $\fib( A\otimes A \to A\otimes_BA )$.
  The right adjoint is given by the trivial square zero extension functor.	
  Note that in the case $A = 1_\CC$, the category of bimodules is just $\CC$.
\end{rec}



%
  
%

\begin{dfn}
  We say that an algebra $A$ in $\CC$ is \mdef{connected} if
  it is connective and the unit map induces an isomorphism $\pi_0\one \cong \pi_0A$.\footnote{Compare with ``connected''  (graded) Hopf algebras.}
  Let $\Alg(\CC)^{\geq1} \subseteq \Alg(\CC)$ denote
  the full subcategory of connected algebras.
\end{dfn}

Note that $\Alg(\CC)^{\geq1} \subseteq \Alg(\CC)$ is closed under colimits.
Given a connected algebra $A$, its base change to $\pi_0A$
acquires a canonical augmentation $A \otimes \pi_0\one \to \pi_0\one$
via the $0$-truncation map.

\begin{cnstr} \label{cnstr:H}
  We define $\mdef{\H}$ to be the composite 	
  \[ \Alg(\CC)^{\geq1} \xrightarrow{ \otimes \pi_0\one }
  \Alg(\Mod(\CC; \pi_0\one))_{-/\pi_0\one} \xrightarrow{\Sigma^\infty} \Mod(\CC; \pi_0\one) \]
  where the second functor is stabilization.
  Concretely, $\H$ can be computed by the formula
  \[ A \mapsto \fib( \pi_0\one \to \pi_0\one\otimes_{A \otimes \pi_0\one} \pi_0\one). \]	
  By construction, $\H$ preserves colimits.
  The stabilization adjunction gives a natural map
  \[ A \to A\otimes \pi_0\one \to \pi_0\one \oplus \H(A).\hfill\qedhere \]
\end{cnstr}

\begin{exm}\label{exm:hfree}
	Let us suppose that $\on = \tau_{\leq0}\on$. Then for a free algebra $\on\{X\}$, we have $\H(\on\{X\}) = X$, and the natural map $\bigoplus_0^{\infty}X^{\otimes i} \cong \on\{X\} \to \H(\on\{X\})\cong  \on \oplus X$ can be identified with the projection onto the arity $1$ summand.
\end{exm}

\begin{exm} \label{exm:base-bpn}
  Take $\CC=\Mod(\MU_{(p)})$, then 
  $\pi_*\H(\BPn)$ is torsion free, finitely generated
  and concentrated in positive odd degrees.    
\end{exm}

\begin{proof}[Details.]
  Expanding out the formula for $\H$ we obtain
  \[ \Z_{(p)} \oplus \Sigma \H(\BPn) \cong \Z_{(p)} \otimes_{\BPn \otimes_{\MU_{(p)}} \Z_{(p)}}\Z_{(p)}. \]
  Applying the Tor spectral sequence and using the fact that the map
  $\pi_*\MU_{(p)} \to \pi_*\BPn$
  is the quotient by a regular sequence of even classes, we learn that
  $\pi_*(\BPn \otimes_{\MU_{(p)}} \Z_{(p)})$ is an exterior algebra over $\Z_{(p)}$
  on the suspensions of the classes in the regular sequence.
  There are no multiplicative extensions between these generators for filtration reasons (see, e.g., \cite[Proposition 3.6]{angeltveit2008topological}).\todo{Note: you don't actually have to resolve these extensions...}
  Applying the Tor spectral sequence again, we obtain a spectral sequence
  computing $\pi_*\H(\BPn)$ whose $E_2$-page is 
  a divided power algebra over $\Z_{(p)}$ on classes in positive even degrees.
  The spectral sequence degenerates at $E_2$ for bidegree reasons,
  and we obtain the desired conclusion.
\end{proof}

A key tool of this subsection is the following analog of the Hurewicz theorem.

\begin{lem} \label{lem:Hurewicz}
  Let $f {\colon}A \to B$ be a map in $\Alg(\CC)^{\geq 1}$ and $k \geq 1$.
  \begin{enumerate}
  \item $\fib(f)$ is $k$-connective if and only if
    $\fib(\H(f))$ is $k$-connective.
  \item If $\fib(f)$ is $k$-connective,
    then the natural map $ \fib(f) \to \fib(\H(f)) $
    is $(k+1)$-connective.
  \end{enumerate}
\end{lem}

\begin{proof}
  First, note that for $M \in \CC$ which is bounded below
  $M$ is $s$-connective iff $M \otimes \pi_0\one$ is $s$-connective.
  Thus, without loss of generality we may
  replace $f$ with the map $f\otimes \pi_0\one$ of augmented $\pi_0\one$-algebras and
  replace $\CC$ with $\Mod(\CC; \pi_0\one)$.

    
  Part (2).
  Using the simplicial resolution of $f$ coming from the free augmented algebra, augmentation ideal monad we reduce to the case where $f$ is of the form $\one\{X\} \to \one\{Y\}$, the map induced by a $k$-connective map $g {\colon}X \to Y$, where $X,Y$ are $1$-connective. In this situation we may identify (see \Cref{exm:hfree}) the fiber sequence
  $F \to \fib(f) \to \fib(\H(f))$ with the fiber sequence
  \[ \oplus_{j \geq 2} \fib(g^{\otimes j}) \to \oplus_{j \geq 1} \fib(g^{\otimes j}) \to \fib(g). \]
  Using that $\fib(g)$ is $k$-connective and $X,Y$ are $1$-connective
  it follows that $\fib(g^{\otimes j})$ is $(k-1+j)$-connective.
  This suffices to conclude.

  Part (1).
  Since both $A$ and $B$ are connected algebras $\fib(f)$ is $0$-connective.
  Applying (2) inductively we find that
  $\fib(f)$ is $k$-connective if and only if $\fib(\H(f))$ is $k$-connective.
\end{proof}

\begin{lem}\label{lem:cells}	
  Let $f {\colon}A \to B$ be a map in $\Alg(\CC)^{\geq1}$ and let
  \[ \H(A) = M_0 \xrightarrow{m_1} M_1 \xrightarrow{m_2} M_2 \to \cdots \to M_\infty = \H(B) \]
  be an increasing filtration such that
  \begin{enumerate}
  \item[(a)] $m_i$ is $k_i$-connective for
    a non-decreasing sequence $k_i$ with $k_1 \geq 0$ and
  \item[(b)] for each $i\geq 1$ there are objects $X_i \in \CC_{\geq k_i}$ such that
    $X_i \otimes \pi_0\one \cong \fib(m_i)$
    and $[X_i, Y]=0$ for any $(k_{i}+2)$-connective $Y$. 
  \end{enumerate}
  There exists a filtration in $\Alg(\CC)^{\geq 1}$
  \[ A:= {R}_0 \xrightarrow{{r}_1} {R}_1 \to \cdots \to R_\infty \xrightarrow{} B \]
  such that
  \begin{enumerate}
  \item the map $R_i \to B$ is $k_{i+1}$-connective,
  \item the functor $\H(-)$ sends the filtration $\{ R_i \}$ to the filtration $\{ M_i \}$
  \item each map $r_i$ fits into a pushout square in $\Alg(\CC)$
    \[ \begin{tikzcd}
      \one \{ X_i \} \ar[r, "\mathrm{aug}"] \ar[d] & \one \ar[d] \\
      R_{i-1} \ar[r, "r_i"] & R_i,
    \end{tikzcd} \]
  \item the map $R_\infty \to B$ is $\infty$-connective.
  \end{enumerate}  
\end{lem}

\begin{proof}
  We construct a filtration between $A$ and $B$ satisfying (1), (2) and (3) by induction on $i$. For the base case $i=0$, we take $A = R_0$, which clearly satisfies (2), and we consider condition (3) to be vacuous in this case. Condition (1) in this case and in the inductive case follows from condition (2) by applying \Cref{lem:Hurewicz}
  and using the hypothesis that $k_j$ are non-decreasing. Thus we will only focus on (2) and (3) for the inductive step.
  
  For the inductive step, suppose we are given $A \to R_{i-1} \to B$ and an identification of $\H(A) \to \H(R_{i-1}) \to \H(B)$ with $M_0 \to M_{i-1} \to M_\infty$. We will construct an $R_i$ satisfying the same condition using pushout of the form in (3).
  

  Consider the following diagram in $\CC$
  \[ \begin{tikzcd}
    X_{i} \ar[r,dashed,"g_i"] \ar[d] &
    \ar[d] \fib(R_{i-1} \to B) \\
    \fib(m_i) \ar[r, "h_i"] &
    \fib(\H({R}_{i-1}) \to \H(B))    
  \end{tikzcd} \]
  where the right vertical map is $(k_i + 1)$-connective by condition (1) and
  \Cref{lem:Hurewicz}.
  As $X_{i}$ has no nonzero maps to a $(k_i+2)$-connective objects,
  the obstructions to producing the dashed lift $g_i$ vanishes
  and we fix a choice of lift.
  Using $g_i$ we can define ${R}_{i}$ to be the pushout of $\E_1$-algebras
  that fits into the following diagram  
  \[ \begin{tikzcd}
    \one \{ X_i \} \ar[r, "\mathrm{aug}"] \ar[d,"\overline g_i"] & \one \ar[d] \ar[r] & \one \ar[d] \\
    R_{i-1} \ar[r,"r_i"] & R_i \pushout \ar[r] & B. 
  \end{tikzcd} \]
   $\on\{X_i\}$ is connective since $X_i$ is, so after applying $\tau_{\leq0}$ (which preserves colimits), the pushout square defining $R_i$ becomes the pushout square in $\Alg(\CC^{\heart})$:
  
  \begin{center}
  	\begin{tikzcd}
  		{\tau_{\leq0}\on\{X_i\}} \ar[r]\ar[d] &\on \ar[d]\\
  		\ar[r] \on& \tau_{\leq0}R_i
  	\end{tikzcd}
  \end{center}
  
  which shows that $R_i$ is connected, since this pushout is the free algebra on the suspension of $X_i$, which is $0$ since $\CC^{\heart}$ is discrete.
  
  Applying $\fib(\H(-))$ to rows of the square above we obtain a
  $\pi_0\one$-linear map $u_i$
  fitting into the following diagram
  \[ \begin{tikzcd}
    X_i \ar[r] \ar[dr, "g_i"'] &
    \fib(\one\{X_i\} \to \one) \ar[r] \ar[d] &
    X_i \otimes \pi_0\one \ar[d, "u_i"] \\
    & \fib(R_{i-1} \to B) \ar[r] &
    \fib(\H(R_{i-1}) \to \H(B)).
  \end{tikzcd} \]
  Comparing this diagram with the one above
  we observe they are both examples of $\pi_0\one$-linearizations
  and therefore obtain an isomorphism $X_i \otimes \pi_0\one \cong \fib(m_i)$
  that identifies $u_i$ with $h_i$. Condition (2) follows.

  In order to complete the proof of the lemma we must check that the filtration we have constructed satisfies condition (4). Using the compatibility of $\H(-)$ together with condition (2) we see that the map $\H(R_\infty) \to \H(B)$ is an isomorphism. It now follows from \Cref{lem:Hurewicz} that the map $R_\infty \to B$ is $\infty$-connective.
\end{proof}

We will also need to understand how $\H$ interacts with tensor products.

\begin{lem}\label{lem:tenscells}
  Let $A,B \in \Alg(\CC)^{\geq1}$. There is a natural identification
  \[ \fib\left( \H(A\coprod B) \to \H(A\otimes B) \right)
  \cong \H(A) \otimes_{\pi_0\one} \H(B). \]
\end{lem}

\begin{proof}
  Without loss of generality, we may replace $\CC$ with $\Mod(\CC; \pi_0\one)$,
  $A$ with $A\otimes\pi_0\one$ and $B$ with $B \otimes \pi_0\one$. 	
  For any $A \in \Alg(\CC)_{-/\one}$, there is a natural identification
  $\one\otimes_A\one \cong \one \oplus \Sigma \H(A)$.
  Using this we then compute that
  \begin{align*}
    \one\otimes_{A\otimes B}\one
    &\cong (\one\otimes_{A}\one) \otimes (\one\otimes_{B}\one)
    \cong (\one \oplus \Sigma \H(A)) \otimes (\one \oplus \Sigma \H(B)) \\
    &\cong \one \oplus \Sigma \H(A) \oplus \Sigma \H(B) \oplus \Sigma^2 \H(A)\otimes \H(B)\hfill\qedhere
  \end{align*}

  By keeping track of these isomorphisms via the map from $A\coprod B \to A\otimes B$, we conclude. 
\end{proof}

\begin{exm}\label{exm:keyexmzeq}\hfill
  \begin{enumerate}
  \item If $\CC$ is the category of $p$-local spectra,
    then $\pi_0\one \cong \ZZ_{(p)}$, whose modules are $\ZZ_{(p)}$-modules.
  \item If $\CC$ is the category of $\MU_{(p)}$-modules,
    then $\pi_0\one \cong \ZZ_{(p)}$, whose modules are $\ZZ_{(p)}$-modules.
  \item If $\CC$ is the category of $\ZZ$-equivariant $\MU_{(p)}^\Psi$-modules,
    then $\pi_0\one \cong \ZZ_{(p)}$ (with trivial action),
    whose modules are $\ZZ$-equivariant $\ZZ_{(p)}$-modules.\qedhere
  \end{enumerate}
\end{exm}

\begin{dfn}
  In the category $\Sp_{(p)}^{B\Z}$ of $\Z$-equivariant $\Ss_{(p)}$-modules
  we let $\mdef{\Ss_{(p)}(j)}$\todo{slight overlap with earlier notation. Doen't matter tho} be the invertible object
  given by $\Ss_{(p)}$ with $\Psi^\ell$ acting by $\ell^j \in (\pi_0\Ss_{(p)})^\times$.

  Given a $\Z$-equivariant commutative $\Ss_{(p)}$-algebra $R$
  we let $R(j) \in \Mod(\Sp^{B\Z}; R)$ be the corresponding invertible object
  obtained by base change.
  Note that $\Z_{(p)}(j) \in \Mod(\ZZ_{(p)})^{B\ZZ}$ lies in the heart.
\end{dfn}

Next we prove a lemma useful for checking the vanishing condition in
\Cref{lem:cells}(b). This starts with the observation that in the category of spectra
$[\Ss, Y] = 0$ for any $1$-connective $Y$.
Similarly, if $R$ is a connective commutative algebra, then in $R$-modules
$[R, Y] = 0$ for any $1$-connective $R$-module $Y$.

\begin{lem}\label{lem:lifttomuandsplitpost}
  Let a $\Z$-equivariant commutative $\Ss_{(p)}$-algebra $R$.
  If $Y \in \Mod(\Sp^{B\Z}; R)$ is $2$-connective, then
  for all $j$ we have
  $ [ R(j), Y ] = 0 $.
\end{lem}
  
	
\begin{proof}
  As $R(j)$ is obtained by base change from $\Ss_{(p)}(j)$ and restriction preserves connectivity it will suffice to prove the lemma in the case $R=\Ss_{(p)}$.
  If we let $\Psi_Y {\colon}Y \to Y$ be the action on $Y$ then we have
  \[ \map_{\Sp^{B\Z}}( \Ss_{(p)}(j), Y ) \cong \fib( Y \xrightarrow{\ell^j - \Psi_Y} Y).\hfill\qedhere \]
\end{proof}

\subsection{A \texorpdfstring{$\Z$}{Z}-equivariant \texorpdfstring{$\EE_1\otimes\A_2$}{E1 times A2} form of \texorpdfstring{$\BPn$}{BPn}}\label{subsec:e1a2}


We now prove an $\E_1$-version of \Cref{thm:Adams-ops-exist} which we then upgrade to the full statement.

\begin{prop} \label{prop:e1adamsopexist}
  The underlying $\E_1$-$\MU_{(p)}$-algebra of $\BPn$ lifts to a
  \[ \BPn^\Psi \in \Alg(\Mod(\Sp^{B\Z}; \MU_{(p)}^\Psi) \]
  in such a way that there is a $\Z$-equivariant $\E_1$-$\MU_{(p)}^\Psi$-algebra map
  $ \iota {\colon}\BPn^{\Psi} \to E_n^\Psi $
  to the underlying $\Z$-equivariant $\E_1$-$\MU_{(p)}^\Psi$-algebra of the $E_n^\Psi$
  from \Cref{cor:E2EnAdams}.
\end{prop}

\begin{proof}
  The $\ZZ$-action on $\MU_{(p)}^\Psi$ gives rise to a $\ZZ$-action on
  $\Mod(\MU_{(p)})$ and there is a symmetric monoidal isomorphism
  \[ \Mod(\Sp^{B\Z}; \MU_{(p)}^\Psi) \cong \Mod(\MU_{(p)})^{h\Z}. \]  
  As consequence of this identification, to prove the proposition
  it will suffice for us to construct a square of $\E_1$-$\MU_{(p)}$-algebras
  \[ \begin{tikzcd}
    \BPn \ar[r] \ar[d, "\iota"] &
    (\Psi^\ell)_*(\BPn) \ar[d, "(\Psi^\ell)_*(\iota)"] \\
    E_n \ar[r, "\Psi^\ell"] &
    (\Psi^\ell)_*(E_n).
  \end{tikzcd} \]
  where the map on the bottom row comes from \Cref{cor:E2EnAdams}
  and the vertical maps are obtained via \Cref{prop:detectEn}.

  We will construct this square by obstruction theory using a presentation of $\BPn$ by $\E_1$-cells coming from \Cref{lem:cells}. Using a splitting of the Postnikov tower of $\H(\BPn)$ and the concentration in odd degrees from \Cref{exm:base-bpn} we construct a filtration
  \begin{align*}
    \H(\MU_{(p)}) = 0
    &\to \Sigma^1 \pi_1(\H(\BPn))
    \to \Sigma^1 \pi_1(\H(\BPn)) \oplus \Sigma^3 \pi_3(\H(\BPn)) \\
    &\to \cdots \to \oplus_{i>0} \Sigma^{2i-1}\pi_{2i-1}(\H(\BPn)) \cong \H(\BPn).
  \end{align*}
  As each $\pi_{2i-1}(\H(\BPn))$ is a finite sum of copies of $\Z_{(p)}$ by \Cref{exm:base-bpn} we may take $X_i$ to be a sum of copies of $\Sigma^{2i-2}\MU_{(p)}$ and apply \Cref{lem:cells}. From this we obtain a filtration
  \[ \MU_{(p)} = R_0 \xrightarrow{r_1} R_1 \xrightarrow{r_2} R_2 \to \cdots \to R_{\infty} \xrightarrow{\cong} \BPn \]
  of $\E_1$-$\MU_{(p)}$-algebras and pushouts
  \[ \begin{tikzcd}
    \MU_{(p)}\{ \oplus \Sigma^{2i-2}\MU_{(p)} \} \ar[r, "\mathrm{aug}"] \ar[d] &
    \MU_{(p)} \ar[d] \\
    R_{i-1} \ar[r, "r_i"] & R_i. \pushout 
  \end{tikzcd} \]
    


  
  We now proceed by induction constructing a commuting diagram of $\E_1$-$\MU_{(p)}$-algebras
  \[ \begin{tikzcd}
    \MU_{(p)}\{ \oplus \Sigma^{2i-2}\MU_{(p)} \} \ar[d, "\mathrm{aug}"] \ar[r] &
    R_{i-1} \ar[d, "r_i"] \ar[ddr, bend left, "f_{i-1}"] & \\
    \MU_{(p)} \ar[r] & R_{i} \pushout \ar[d] \ar[dr,dashed,swap,"f_{i}"] \\
    & \BPn \ar[d] &
    (\Psi^\ell)_*(\BPn) \ar[d] \\
    & E_n\ar[r,"\Psi^{\ell}"] & (\Psi^\ell)_*(E_n).
  \end{tikzcd} \]  

  Let $x$ denote a map $ \oplus \Sigma^{2i-2}\MU_{(p)} \to R_{i-1}$ along which the cell producing $R_{i}$ is attached. We will first check that $ f_{i-1} \circ x $ is nullhomotopic.
  Since the source of $x$ is a free $\MU_{(p)}$-module and
  $\pi_*(\Psi^\ell)_*\BPn \to \pi_*(\Psi^\ell)_*E_n$ is injective (\Cref{cnst:BPntoEn}),
  it will suffice to check that the further composite
  $(\Psi^\ell)_*(\iota) \circ f_{i-1} \circ x$ is nullhomotopic.
  This is implied by the fact that the diagram of solid arrows commutes
  and $x$ maps to zero in along the augmentation.

  Any nullhomotopy of $ f_{i-1} \circ x $ produces a map $f_{i}$
  together with a homotopy filling the triangle within the desired diagram.
  It remains to pick a homotopy making the trapezoid below the triangle commute.
  This is equivalent to checking that two nullhomotopies of the composite
  \[ \oplus \Sigma^{2i-2}\MU_{(p)} \xrightarrow{x} R_{i-1} \to (\Psi^\ell)_*(E_n) \]
  are compatible. Since the source is a free $\MU_{(p)}$-module
  and $\pi_{2i-1}E_n \cong 0$, the space of such nullhomotopies is connected.
\end{proof}

\begin{lem}\label{lem:MU-formula}
  $\Psi^\ell$ acts on $\pi_{2k}\MU_{(p)}^\Psi$ via multiplication by $\ell^k$.
\end{lem}

\begin{proof}
  As the homotopy groups of $\MU_{(p)}$ are torsion free it suffices to prove the lemma after rationalization. After rationalizing $\mathrm{BSL}_1(\mathbb{S}_{(p)})$ becomes trivial and so we have a canonical $\Z$-equivariant identification
  \[ \Q \otimes \MU_{(p)}^\Psi \cong \Q \otimes \Sigma_+^\infty \BU. \]
  The lemma now follows from the fact that the rational homology of $\BU$ is a polynomial algebra on classes $b_i$ in the Hurewicz image and the formula for the action of the Adams operation $\Psi^\ell$ on the homotopy of $\BU$.
\end{proof}

\begin{cor} \label{cor:BPn-action}
  $\Psi^\ell$ acts on $\pi_{2k}\BPn^\Psi$ via multiplication by $\ell^k$.
\end{cor}

\begin{proof}
  The unit map $\MU_{(p)}^\Psi \to \BPn^\Psi$ is surjective on homotopy groups.
\end{proof}

\begin{lem}\label{lem:BPnfixcells}
  $\pi_*^{\heart}\H(\BPn^\Psi)$ is concentrated in odd degrees,
  and each homotopy group $\pi_{2i-1}^\heart\H(\BPn^\Psi)$
  has a finite filtration whose associated graded consists of copies of 
  $\Z_{(p)}(j)$ for $j<i$.
\end{lem}

\begin{proof}
  The proof of this lemma proceeds by keeping track of actions in \Cref{exm:base-bpn}.
  As there we have an isomorphism of $\Z$-equivariant $\Z_{(p)}$-modules
  \[ \Z_{(p)} \oplus \Sigma \H(\BPn^\Psi)
  \cong \Z_{(p)} \otimes_{\BPn^\Psi \otimes_{\MU_{(p)}^\Psi} \Z_{(p)}} \Z_{(p)}. \]  

  Examining the Tor spectral sequence we find that
  $\pi_*^\heart(\BPn^\Psi \otimes_{\MU_{(p)}^\Psi}\ZZ_{(p)})$ is
  an exterior algebra over $\ZZ_{(p)}$ on a collection of classes 
  $\Sigma^{2i+1}\ZZ_{(p)}(i)\{\sigma b_i\}$
  where $i \in \{ i > 0\ |\ i \neq p-1, \dots, p^n -1 \}$.
  The classes $\sigma b_i$ are suspensions of
  polynomial algebra generators of $\MU_{(p)}^\Psi$
  and the action comes from \Cref{lem:MU-formula}.

  Applying the Tor spectral sequence again we obtain a
  spectral sequence computing $\pi_*^\heart(\Z_{(p)} \oplus \Sigma \H(\BPn^\Psi))$
  whose $E_2$-page is a divided power algebra over $\Z_{(p)}$
  on classes $\Sigma^{2i+2}\ZZ_{(p)}(i)\{b_i\}$.
  This spectral sequence degenerates at $E_2$
  and so we obtain a filtration on
  $\pi_*^\heart(\H(\BPn^\Psi))$ with the desired properties.
\end{proof}



Next we upgrade the $\E_1$-algebra from \Cref{prop:e1adamsopexist}
to an $(\E_1 \otimes \A_2)$-algebra.

\begin{lem} \label{lem:+A2}
  The underlying $(\E_1\otimes \A_2)$-$\MU_{(p)}$-algebra of $\BPn$ lifts to a
  \[ \BPn^\Psi \in \Alg_{\E_1 \otimes \A_2}(\Mod(\Sp^{B\Z}; \MU_{(p)}^\Psi)). \]
\end{lem}

\begin{proof}
  Consider the $\Z$-equivariant $\E_1$-$\MU_{(p)}^\Psi$-algebra $\BPn^\Psi$ from \Cref{prop:e1adamsopexist}. We will lift $\BPn^\Psi$ to a $\ZZ$-equivariant $(\E_1 \otimes \A_2)$-$\MU_{(p)}^\Psi$-algebra by equipping it with a unital multiplication in the category of $\ZZ$-equivariant $\EE_1$-$\MU_{(p)}^\Psi$-algebras. To do this we produce an extension as indicated in the diagram below.
  \[ \begin{tikzcd}
    \BPn^\Psi \coprod \BPn^\Psi \ar[rr,"\nabla"] \ar[dr,"s"] & & \BPn^\Psi \\
    & \ar[ur,dashed] \BPn^\Psi \otimes_{\MU_{(p)}^\Psi} \BPn^\Psi & 
  \end{tikzcd} \]

  We will do this using a cellular obstruction theory argument. To do this we apply \Cref{lem:cells} to place a filtration on the map $s$. Using \Cref{lem:tenscells} we know that $\fib(\H(s))$ is isomorphic to $\H(\BPn^\Psi) \otimes_{\ZZ_{(p)}} \H(\BPn^\Psi)$. From \Cref{lem:BPnfixcells} we can now conclude that the homotopy groups of $\fib(\H(s))$ are concentrated in positive even degrees and that $\pi_{2i}^\heart\fib(\H(s))$ has a finite filtration with associated graded given by copies of $\Z_{(p)}(j)$ with $j<i$. Even-ness and \Cref{lem:lifttomuandsplitpost} together imply that the Postnikov tower of $\fib(\H(s))$ splits (all $k$-invariants vanish). Splicing a splitting of the Postnikov tower together with the filtrations on the homotopy groups and picking the $X_i$ to be of the form $\Sigma^{2k_i} \MU_{(p)}^\Psi(j)$ (see \Cref{lem:lifttomuandsplitpost}) we may apply \Cref{lem:cells} to obtain a filtration
  \[ \BPn^\Psi \coprod \BPn^\Psi = R_0 \xrightarrow{r_1} R_1 \to \cdots \to R_{\infty} \xrightarrow{\cong} \BPn^\Psi \otimes_{\MU_{(p)}^\Psi} \BPn^\Psi \]  
  where $R_i$ is obtained from $R_{i-1}$ by a pushout along the augmentation of a free algebra on $\Sigma^{2k+1}\MU_{(p)}^\Psi(j)$ with $j<k$.

  Thus, the obstructions to making the extension are a sequence of $\Z$-equivariant maps
  \[ \Sigma^{2k}\MU_{(p)}^\Psi(j) \to \BPn^\Psi \]
  where $j < k$. As the homotopy of $\BPn$ is even, these maps are determined by their value on $\pi_{2k}^\heart(-)$. On the other hand, any map $ \Z_{(p)}(j) \to \pi_{2k}^\heart\BPn^\Psi $ is null since $\pi_{2k}^\heart\BPn^\Psi$ is a sum of copies of $\Z_{(p)}(k)$ by \Cref{cor:BPn-action} and $j<k$.
    
	
  Given a nullhomotopy, we note that forgetting down to a non-equivariant nullhomotopy must agree with the one coming from the underlying $(\E_1\otimes \A_2)$-structure of $\BPn$, since there is a unique such nulhomotopy in that case (as $\BPn$ has even homotopy groups). Thus the $\Z$-equivariant $(\E_1\otimes \A_2)$-$\MU_{(p)}$-algebra structure refines the nonequivariant one.
\end{proof}

We are now ready to put all the pieces together and conclude.

\begin{proof}[Proof (of \Cref{thm:Adams-ops-exist}).]
  The $\Z$-equivariant lift $\BPn^\Psi$ of the $(\E_1 \otimes \A_2)$-$\MU_{(p)}$-algebra structure on $\BPn$ was constructed in \Cref{lem:+A2}. The $\Z$-equivariant $\E_1$-$\MU_{(p)}^\Psi$-algebra map $\iota {\colon}\BPn^\Psi \to E_n^\Psi$ was constructed in \Cref{prop:e1adamsopexist}. The identification of the underlying $\E_1$-algebras comes from \Cref{prop:detectEn} and this upgrades to an identification of $\Z$-equivariant $\E_1$-algebras because $\Psi^\ell \in \G_n$ is central and therefore commutes with $\F_p^\times \rtimes \hat{\Z}$. The $\Z$-action on the $p$-completion of $\BPn$ is locally unipotent in $p$-complete spectra by \Cref{cor:BPn-action}, the fact that $\pi_*\BPn$ is concentrated in degrees divisible by $2p-2$ and \Cref{cor:check-unip-pi}.
\end{proof}

%% file: disproof.tex
In this section, we study the algebraic $K$-theory of the fixed points of the $\ZZ$-action on $\BPn$ constructed in \Cref{sec:adamsop}.  For $n \ge 1$ and $k \ge 0$, we prove that the $T(n+1)$-localized coassembly map
\[ L_{T(n+1)}K(\BP\langle n\rangle^{hp^k\ZZ}) \to L_{T(n+1)}K(\BP\langle n\rangle)^{hp^k\ZZ} \]
is not an equivalence, but becomes an equivalence after $K(n+1)$-localization. Thus, the telescope conjecture fails.

We do this by looking at the coassembly map from two highly divergent perspectives, which are connected via trace theorems:
\begin{enumerate}
\item From the perspective of locally unipotent $\Z$-actions on ring spectra,
  the results of \Cref{sec:tame} tell us that the coassembly map cannot be an isomorphism.
\item From the perspective of \emph{cyclotomic redshift} of \cite{cycloshift},
  the map \[L_{T(n)}\BPn^{hp^k\Z} \to L_{T(n)}\BPn\] splits after base change to the maximal abelian extension of the $K(n)$-local sphere, 
and therefore the coassembly map is a $K(n+1)$-local isomorphism.
\end{enumerate}

We begin in \Cref{subsec:KTC} with a discussion of the connection between $\TC$, which has been the subject of the paper so far, and algebraic $K$-theory.
In \Cref{subsec:univdescent}, we set up abstract machinery relating coassembly maps and descent. In \Cref{subsec:cyclo-descent}, we apply this machinery to show that cyclotomic redshift implies that $L_{T(n+1)}K(\BP\langle n\rangle^{hp^k\ZZ}) \to L_{T(n+1)}K(\BP\langle n\rangle)^{hp^k\ZZ}$ becomes an equivalence after cyclotomic completion.
Finally, with all requisite inputs assembled, we exhibit a counterexample to the telescope conjecture in \Cref{subsec:failure}.

\subsection{\texorpdfstring{$K$}{K}-theory and \texorpdfstring{$\TC$}{TC}}
\label{subsec:KTC}
\input{K-to-TC.tex}

\subsection{Descent and coassembly}
\label{subsec:univdescent}
\input{univdescent}


\subsection{Cyclotomic hyperdescent}
\label{subsec:cyclo-descent}
\input{cyclo-descent.tex}

\subsection{The failure of the telescope conjecture}
\label{subsec:failure}

We are now ready to prove the following refinement of \Cref{thm:main} from the introduction.

\begin{thm}\label{thm:maincyc}
    Let $\BP\langle n \rangle$ be as in \Cref{thm:Adams-ops-exist}. For every prime $p$, height $n \geq 1$ and $k \geq 0$, there is a diagram
      \[ \begin{tikzcd}
    L_{T(n+1)} K(L_{T(n)}\BP\langle n \rangle^{hp^k\ZZ}) \ar[d] &
    L_{T(n+1)} K(\BP\langle n \rangle^{hp^k\ZZ}) \ar[r, "\cong"] \ar[l, "\cong"'] \ar[d] &
    L_{T(n+1)} \TC(\BP\langle n \rangle^{hp^k\ZZ}) \ar[d] \\    
    L_{T(n+1)} K(L_{T(n)}\BP\langle n \rangle)^{hp^k\ZZ} &
    L_{T(n+1)} K(\BP\langle n\rangle)^{hp^k\ZZ} \ar[l, "\cong"'] \ar[r, "\cong"] &    
    L_{T(n+1)}\TC(\BP\langle n \rangle)^{hp^k\ZZ}
  \end{tikzcd} \]
  where the vertical maps are coassembly maps. 
  The vertical maps are not isomorphisms, but rather
  exhibit the target as the cyclotomic completion of the source.
  In particular, this gives a counterexample to the height $n+1$ telescope conjecture.
\end{thm}

\begin{proof}  
    By applying \Cref{thm:TCequalK} to $R = \BP\langle n \rangle$, we get the diagram as in the theorem statement, where the vertical maps are coassembly maps. By applying \Cref{prop:bpncycmodel} and \Cref{lem:coassembly-descent}, the left vertical map is an equivalence after cyclotomic completion, and the target is cyclotomically complete. Thus the vertical maps are the cyclotomic completion maps.

    It remains to show that the vertical maps are not isomorphisms, or equivalently, the source of the coassembly map is not cyclotomically complete. 
  Since cyclotomic completion is smashing and the maps between $L_{T(n+1)}K(L_{T(n)}\BP\langle n \rangle^{hp^k\ZZ})$ as $k$ varies are $\EE_{\infty}$-maps (\Cref{thm:Adams-ops-exist}), we learn that if the result holds for $k$, i.e $L_{T(n+1)}K(L_{T(n)}\BP\langle n \rangle^{hp^k\Z})$ is not cyclotomically complete, then it also holds for all $k'<k$. 
  

  Thus it suffices to prove that the right coassembly map is not an isomorphism for all $k\gg0$.


  Choose $U$ to be a finite spectrum of type $n+1$ with a $v_{n+1}$-self map $v$, so that $T(n+1)\coloneqq U[v^{-1}]\neq 0$. $\BP\langle n \rangle$ is fp-type $n$, has a locally unipotent $p^k\ZZ$-action (\Cref{thm:Adams-ops-exist}, \Cref{lem:res-u}), is an $\EE_1\otimes 
  \A_2$-algebra,
  and satisfies the height $n$ Lichtenbaum--Quillen property by \cite{hahn2020redshift}, so we map apply \Cref{cor:telasscon} and invert $v$ to obtain for all $k\gg0$ a square

  	\begin{center}\begin{tikzcd}
			T(n+1)_*\TC(\BP\langle n\rangle ^{hp^k\Z})\ar[r,""]\ar[d,"\cong"] &T(n+1)_*\TC(\BP\langle n \rangle)^{hp^k\Z}\ar[d,"\cong"]\\
			T(n+1)_*\TC(\BP\langle n \rangle^{B\Z})\ar[r,""] & T(n+1)_*\TC(\BP\langle n \rangle)^{B\Z}
	\end{tikzcd}\end{center}

Thus we are reduced to showing that the lower coassembly map is not an isomorphism. We know from \cite{hahn2020redshift} that
  $T(n+1)\otimes \TC(\BP\langle n \rangle)$ is non-zero, so since by \Cref{cor:nilthicksub} $T(n+1)\otimes \TC(\BP\langle n \rangle)$ is in the thick subcategory generated by the fiber of the coassembly map 

  $$T(n+1)\otimes \TC(\BP\langle n \rangle^{B\Z}) \to T(n+1)\otimes \TC(\BP\langle n \rangle)^{B\Z}$$
  we can conclude the proof.
\end{proof}




%% file: K-to-TC.tex
In this subsection, we explain why $T(n+1)$-localized $K$-theory and $T(n+1)$-localized $\TC$ coincide for our examples of interest. We begin with connective rings, where we have the following theorem proved in \cite{CMNN} and \cite{LMMT}:


\begin{thm}[{\cite[Purity Theorem, Cor. 4.29]{LMMT}, \cite[Cor. 4.11]{CMNN}}]
  \label{thm:purity}
  Let $R$ be a connective $\E_1$-algebra.
  For $n \geq 1$
  , the $(T(n) \oplus T(n+1))$-localization map
  and the cyclotomic trace induce isomorphisms
  \[ L_{T(n+1)}K(L_{T(n) \oplus T(n+1)}R) \xleftarrow{\cong} L_{T(n+1)}K(R) \xrightarrow{\cong} L_{T(n+1)}\TC(R).  \]
\end{thm}

To disprove the telescope conjecture, we will need to understand the topological cyclic homologies of fixed points of $\ZZ$-actions on connective $\E_1$-algebras.
Such fixed points are $(-1)$-connective, but often not $0$-connective, so the above theorem does not apply.
In order to surmount this obstacle, we use the following theorem, which is phrased in terms of the notion of truncating invariant from \cite[Def. 3.1]{LT}:


\begin{thm}[{\cite[Theorem B]{levy2022algebraic}}]
  \label{dgm-1connective}
  Let $R$ and $S$ be connective $\E_1$-algebras with $\Z$-action.
  Let $f \colon R \to S$ be a $1$-connective, $\Z$-equivariant $\E_1$-algebra map.
  For any truncating invariant $E$, the induced map
  \[ E(R^{h\ZZ}) \to E(S^{h\ZZ}) \]
  is an isomorphism.
\end{thm}

As a corollary of this theorem we prove an analog of \Cref{thm:purity} for $(-1)$-connective rings.


\begin{cor} \label{thm:red_main} \label{thm:TCequalK}
  For $n\geq1$, let $R$ be a $T(n+1)$-acyclic, connective $\E_1$-algebra with $\Z$-action.
  The coassembly map, $T(n)$-localization map, and cyclotomic trace fit into a commuting diagram
  \[ \begin{tikzcd}
    L_{T(n+1)} K(L_{T(n)}R^{h\ZZ}) \ar[d] &
    L_{T(n+1)} K(R^{h\ZZ}) \ar[r, "\cong"] \ar[l, "\cong"'] \ar[d] &
    L_{T(n+1)} \TC(R^{h\ZZ}) \ar[d] \\    
    L_{T(n+1)} K(L_{T(n)}R)^{h\ZZ} &
    L_{T(n+1)} K(R)^{h\ZZ} \ar[l, "\cong"'] \ar[r, "\cong"] &    
    L_{T(n+1)}\TC(R)^{h\ZZ},
  \end{tikzcd} \]
  where each horizontal map is an isomorphism.
\end{cor}

\begin{proof}
  The lower horizontal maps are isomorphisms by \Cref{thm:purity}.
  The upper left map is an isomorphism by the purity theorem from \cite{LMMT}.
  We now focus our attention on the upper right map.

  Let $K_{\mathrm{inv}}(-)$ be the fiber of the cyclotomic trace map, so that our goal is now to prove that $L_{T(n+1)}K_{\mathrm{inv}}(R^{h\Z}) = 0$.
  By the Dundas--Goodwillie--McCarthy theorem \cite{DGM}, $L_{T(n+1)}K_{\mathrm{inv}}(-)$ is a truncating invariant.
  Applying \Cref{dgm-1connective} to the $1$-connective $\Z$-equivariant map
  $ R \to \pi_0R $, we obtain an isomorphism
  \[ L_{T(n+1)}K_{\mathrm{inv}}(R^{h\Z}) \cong L_{T(n+1)}K_{\mathrm{inv}}((\pi_0R)^{h\Z}). \]
  Since $n \ge 1$, Mitchell's vanishing theorem \cite{MitchellVanishing} implies, for any $\Z$-algebra $A$, that $L_{T(n+1)}K_{\mathrm{inv}}(A) = 0$.
  Since the $\Z^{B\Z}$-algebra structure on $(\pi_0R)^{h\Z}$
  restricts to a $\Z$-algebra structure, we learn that
  $L_{T(n+1)}K_{\mathrm{inv}}((\pi_0R)^{h\Z}) = 0$,
  which completes the proof.
\end{proof}

%% file: univdescent.tex
In this subsection, we recall and develop general abstract machinery relating coassembly and descent.  At the end, this machinery is applied to prove \Cref{lem:coassembly-descent}, which is a statement specifically about coassembly maps for $T(n+1)$-localized algebraic $K$-theory.


\begin{cnv}\label{cnv:topcat}
    Throughout this subsection $\CC$ will denote a category with finite limits and colimits, such that finite coproducts are disjoint and universal.  We suppose $F\colon \CC^{op} \to \DD$ is a product preserving functor into some other category $\DD$.
\end{cnv}

\begin{exm}\label{exm:descenttopexamples}
Suppose $\CC_0$ is a presentably symmetric monoidal additive category.  Then $\CC=\CAlg(\CC_0)^{op}$ satisfies the conditions of \Cref{cnv:topcat}.  Indeed, the fact that coproducts are disjoint follows from the fact that products in $\CAlg(\CC_0)$ are computed on underlying, with the projections to each factor given by inverting orthogonal idempotents. The fact that coproducts are universal follows from the fact that base change preserves finite products, since finite products preserve finite coproducts.

In particular, the main example of interest will be $\CC=\CAlg\left(\Sp_{T(n)}\right)^{op}$, with $F$ the functor $L_{T(n+1)} K\colon\CAlg\left(\Sp_{T(n)}\right) \to \Sp_{T(n+1)}$.
\end{exm}      

\subsubsection{F-covers}

\begin{dfn} \label{dfn:Fcover}
    An \mdef{$F$-cover} in $\CC$ is a   morphism  $f\colon A \to B$ in $\CC$ that is a universal $F$-descent morphism  in the sense of \cite[Definition 3.1.1]{liu2012enhanced}. In other words for every map $g$ that is a base change of $f$, $F$ satisfies descent for the \v{C}ech nerve of $g$.

    We recall \cite[Lemma 3.1.2]{liu2012enhanced}
    the collection of $F$-covers is closed under composition and base change, and if $f\circ g$ is an $F$-cover, so is $f$.
\end{dfn}


  We do not need to use \Cref{lem:Fdescenttop} below, but rather include it in order to explain our terminology: $F$-covers are exactly the covers in the universal $F$-descent topology.
    
\begin{prop}\label{lem:Fdescenttop}
    There exists a Grothendieck topology on $\CC$ called the \mdef{universal $F$-descent topology} where a sieve on an object $C \in \CC^{op}$ is a cover iff it contains a finite collection of maps $C_i \to C$ such that the map $\coprod_iC_i\to C$ in $\CC$ is an $F$-cover. In particular, the sieve corresponding to a finite collection of map $C_i \to \CC$ is a cover in this topology iff the map $\coprod C_i \to \CC$ is an $F$-cover.

    Moreover, for any functor $\mathscr{F}\colon\CC^{op} \to \DD$, $\mathscr{F}$ is a sheaf iff it preserves finite products and satisfies descent for the \v{C}ech nerve of any $F$-cover. In particular, $F$ is a universal $F$-descent sheaf.
\end{prop}

\begin{proof}
    To construct such a topology, we wish to apply \cite[Proposition A.3.2.1]{SAG} to $\CC$ equipped with the collection of $F$-covers. Conditions (a) and (b) follow from \cite[Lemma 3.1.2]{liu2012enhanced}, and condition (d) follows by assumption on $\CC$.

    It remains to check that (c) is satisfied, i.e that if $f_i\colon a_i \to b_i$ is a universal $F$-descent morphism for a finite collection of maps, then so is $\coprod_i f_i$. Since coproducts are universal, it is enough to prove that $F$ is an $F$-descent morphism. Because coproducts are disjoint and $F$ is product preserving, the \v{C}ech complex for $\coprod_if_i$ is the product of the \v{C}ech complexes for each $f_i$, so the result follows.


    For the statement in the second paragraph, we apply \cite[Proposition A.3.1.1]{SAG}, where condition (e) is satisfied by hypothesis on $\CC$.
\end{proof}




\begin{lem}\label{lem:checklocally}
	Let $\DD$ be a category with a Grothendieck topology, and suppose that the square below is a pullback square, and $f'$ is a cover.
	\begin{center}
		\begin{tikzcd}
			{c} \ar[r,"g"]\ar[d,"f"] & c'\ar[d,"f'"]\\
		d	\ar[r,"g'"] & d'
		\end{tikzcd}
	\end{center}
	
	Then $g$ is an cover iff $g'$ is an cover.
\end{lem}
\begin{proof}
	If $g'$ is a cover, $g$ is too since covers are closed under pullbacks. Conversely, if $g$ is a cover, then $g\circ f'=g'\circ f$, is too, which implies that $g'$ is.
\end{proof}

\begin{rmk}\label{rmk:checklocally}
	By \Cref{lem:Fdescenttop}, the above lemma applies to any $\DD=\CC$ as in \Cref{cnv:topcat}. We note that in this case, the lemma could have been proven directly using the closure properties of $F$-covers \cite[Lemma 3.1.2]{liu2012enhanced}.
\end{rmk}

The following lemmas lets us transfer the notions of $F$-covers along functors:
\begin{lem}\label{lem:transfersheafcover}
    Suppose we have functors
    $$\CC^{op}\xrightarrow{F}\DD \xrightarrow{B}\DD'$$ 
    where $B$ preserves finite products and totalizations which exist, and $F$ preserves finite products.
    Then if $f \in \CC$ is a map,
    \begin{enumerate}
        \item If $f$ is an $F$-cover, it is a $B\circ F$-cover
        \item If $B$ is conservative and $f$ is a $B\circ F$-cover, it is an $F$-cover.
    \end{enumerate}

\end{lem}

\begin{proof}
    For (1) it is enough to show that any $F$-descent morphism is also a $B\circ F$-descent morphism, but this follows from the fact that $F$ preserves totalizations. (2) follows since conservative totalization preserving functors reflect totalizations.
\end{proof}

\begin{lem}\label{lem:covprecompose}
    Let $\CC,\CC'$ be as in \Cref{cnv:topcat}, and suppose we have functors 
$$\CC'^{op}\xrightarrow{A^{op}}\CC^{op} \xrightarrow{F} \DD$$ where $F$ and $F\circ A^{op}$ preserve finite products and $A$ preserves pullbacks. Then if $A(f)$ is an $F$-cover, $f$ is an $F\circ A^{op}$-cover.
\end{lem}

\begin{proof}
    Since $A$ preserves finite limits, it preserves the \v{C}ech complex and base change, from which the result follows.
    
\end{proof}



The following lemma will be useful in giving examples of $F$-covers:
\begin{lem}\label{lem:descentcovers}
    Let $G = \Omega X$ be a loop space. Suppose that $a$ is a $G$-equivariant object in $\CC$, such that the map $F(-)\to F(a\times-)^{hG}$ is an isomorphism. Then the map $a \to *$ is an $F$-cover.
\end{lem}

\begin{proof}
    
    Since the functor $(-)^{hG}\colon\CC^{BG} \to \CC$ preserves limits and the forgetful functor $\CC^{BG} \to \CC$ is conservative, by \Cref{lem:transfersheafcover} it is enough to show that $f$ is a $F(a\times-)$-cover. Using \Cref{lem:covprecompose}, it is enough to show that $a\times f$ is an $F$-cover. But this is true because the map $a\times a\to a$ admits a section given by the diagonal map.
\end{proof}

We now turn to comparing $L_{T(n+1)}K\cycl$-descent with coassembly maps.



\begin{lem} \label{lem:extract-stone}
  Let $R \in \Alg(\Sp)$, and
  let $S$ be a Stone topological space.
  There is a natural isomorphism
  \[ K(\W(\LCF{S}) \otimes R) \cong \W(\LCF{S}) \otimes K(R). \]
\end{lem}

\begin{proof}
  Write $S$ as a cofiltered limit $\varprojlim S_\alpha$
  of finite discrete spaces $S_\alpha$.
  Using the fact that each of the functors $K(-)$, $\W(-)$ and $- \otimes R$
  commute with filterd colimits and finite products
  (for $\W(-)$ see \cite[Prop. 2.2, Lem. 2.10]{chromaticNSTZ})
  the composite $K(\W(-) \otimes R)$ does as well.
  We now construct a chain of isomorphisms 
  \begin{align*}
    K(\W(\LCF{S}) &\otimes R)
    \cong K\left(  \W\left( \varinjlim_{\alpha} \F_p^{\times S_\alpha} \right) \otimes R \right) 
    \cong \varinjlim_{\alpha}  K\left(  \W\left( \F_p \right) \otimes R \right)^{\times S_\alpha} \\
    &\cong \left( \varinjlim_{\alpha} \Ss^{\times S_\alpha} \right) \otimes K\left( \Ss \otimes R \right) 
    \cong \left( \varinjlim_{\alpha}  \W\left( \F_p \right)^{\times S_\alpha} \right) \otimes K(R) \\
    &\cong \W\left( \varinjlim_{\alpha} \F_p^{\times S_\alpha} \right)\otimes K(R) 
    \cong \W( \LCF{S}) \otimes K(R). 
  \end{align*}
\end{proof}

\begin{lem} \label{lem:extract-stone2}
  For $n\geq1$, Let $R \in \Alg(\Sp_{T(n)})$ and 
  Let $S$ be a Stone topological space.
  There is a natural isomorphism
  \[ L_{T(n+1)}K(L_{T(n)}(\W(\LCF{S}) \otimes R)) \cong L_{T(n+1)}(\W(\LCF{S}) \otimes K(R)). \]
\end{lem}

\begin{proof}
  As $R$ is $T(n)$-local, the map
  $\W(\LCF{S}) \otimes R \to L_{T(n)}(\W(\LCF{S}) \otimes R)$
  is a $(T(n) \oplus T(n+1))$-local isomorphism and
  therefore we have a natural isomorphism  
  \[ L_{T(n+1)}K(L_{T(n)}(\W(\LCF{S}) \otimes R)) \cong L_{T(n+1)}K(\W(\LCF{S}) \otimes R) \]
  by the purity theorem of \cite{LMMT}.
  The lemma now follows from \Cref{lem:extract-stone}.  
\end{proof}

\begin{lem} \label{lem:coassembly-descent}
    For $n\geq1$, let $R \in \CAlg(\Sp_{T(n)}^{B\Z,u})$, and $F = L_{T(n+1)}K$.
    There is a commuting triangle, natural in $R$,
    \[ \begin{tikzcd}
        & F(R^{h\Z}) \ar[dr] \ar[dl] & \\
        F(R)^{h\Z} \ar[rr, no head, "\cong"] & & \lim_{\Delta} F( R^{\otimes_{R^{h\Z}} \bullet +1} )
    \end{tikzcd} \]
    identifying the coassembly map with the \v{C}ech nerve of $R^{h\Z} \to R$. 
    In particular, the coassembly map for $R$ is an isomorphism if and only if $R$ satisfies $F$-descent.
\end{lem}

\begin{proof}
    We begin by constructing the following square with \v{C}ech covers 
    and coassembly maps
    \[ \begin{tikzcd}
        F(R^{h\Z}) \ar[r] \ar[d] &
        \lim_{\Delta} F\left( R^{h\Z} \otimes_{R^{h\Z}} R^{\otimes_{R^{h\Z}} \bullet } \right) \ar[d] \\
        F(R)^{h\Z} \ar[r] &
        \left( \lim_{\Delta} F\left( R \otimes_{R^{h\Z}} R^{\otimes_{R^{h\Z}} \bullet } \right) \right)^{h\Z} .  
    \end{tikzcd} \]
    The bottom horizontal arrow is an isomorphism since the map $R \to R \otimes_{R^{h\Z}} R$ admits a retraction.
    It will now suffice for us to argue that for any commutative $R$-algebra $A$, 
    the coassembly map
    \[ F(R^{h\Z} \otimes_{R^{h\Z}} R \otimes_R A  ) \to F(R \otimes_{R^{h\Z}} R \otimes_R A  )^{h\Z} \]
    is an isomorphism.
    
    Using \Cref{lem:basechangepZ}, \Cref{rmk:unip-affine} and \Cref{lem:thh-s1}
    we construct a cube
    \[ \begin{tikzcd}
        \Ss^{B\Z} \ar[rr] \ar[dr] \ar[dd] & & 
        \Ss \ar[dd] \ar[dr] & \\
        & R^{h\Z} \ar[rr] \ar[dd] & &
        R \ar[dd] \\
        \Ss \ar[rr] \ar[dr] & & 
        \W\LCF{\overrightarrow{\ZZ_p}} \ar[dr] \\
        & R \ar[rr] & & 
        R \otimes_{R^{h\Z}} R        
    \end{tikzcd} \]
    where every face is a pushout of commutative algebras 
    and terms on the bottom face have $\Z$-actions.
    The identification of the $\Z$-action on the back right comes from the proof of part (3) of \Cref{prop:unipotent_SW}.
    From the right face this cube we can read off that there is an isomorphism of $\Z$-equivariant commutative $R$-algebras
    \[ R \otimes_{R^{h\Z}} R \cong \W\LCF{\overrightarrow{\ZZ_p}} \otimes R \]
    where on the left hand side the $\Z$-action is via the action on the left tensor factor and the $R$-algebra structure is via the right tensor factor.
    This lets rewrite the coassembly map above as 
    \[ F( A  ) \to F(\W\LCF{\overrightarrow{\ZZ_p}} \otimes A  )^{h\Z}. \]
    This map is an isomorphism by \Cref{lem:extract-stone2} and \Cref{prop:unipotent_SW} (3).
\end{proof}

%% file: cyclo-descent.tex
In this subsection we introduce the final key idea in giving a counter-example to the telescope conjecture: cyclotomic redshift. This result, proven in \cite{cycloshift}, is about the compatibility of chromatically localized $K$-theory with the cyclotomic extensions. We use cyclotomic redshift to show that the map $\SP_{T(n)} \to \SP_{T(n)}^{ab}$ is a $L_{T(n+1)}K\cycl$-cover, where $\SP_{T(n)}^{ab}$ is the lift of the maximal abelian extension of the $K(n)$-local sphere to the $T(n)$-local sphere constructed in \cite{carmeli2021chromatic}. We show that since the extensions $L_{T(n)}\BP\langle n \rangle^{hp^k\ZZ} \to L_{T(n)}\BP\langle n \rangle$ admits a retraction after base change to $\SP_{T(n)}^{ab}$, they are also $L_{T(n+1)}K\cycl$-covers.

We refer the reader to the introduction for notation regarding the chromatic cyclotomic extensions of \cite{carmeli2021chromatic}.

\begin{rec}\label{rec:smashing}\label{cnstr:iwasawa}
    We recall that the Galois group $\ZZ_p^\times$ of the infinite $p$-cyclotomic extension $\Ss_{T(n)}[\omega_{p^{\infty}}]$ is isomorphic via the $p$-adic logarithm to the product $T_p\times \ZZ_p$, where $T_p$ is the torsion subgroup that is $\{\pm1\}$ for $p=2$ and $\FF_p^\times$ for $p>2$.

    We let $T_p\times \ZZ$ be the subgroup obtained via the inclusion $\ZZ \to \ZZ_p$.


In \cite[Proposition 6.19]{fourier} it is proven that the functor
  $(-)\cycl : \Sp_{T(n)} \to \Sp_{T(n)}$ given by
  \[ X \mapsto (\Ss_{T(n)}[\omega_{p^\infty}^{(n)}]^{h(T_p\times \ZZ)})\otimes X \]
  is a smashing, symmetric monoidal localization, where the $\ZZ$ action comes from restriction of the $\ZZ_p$-action to a generator.
  We call the local objects for this localization
  cyclotomically complete $T(n)$-local spectra and there are natural inclusions
  \[ \Sp_{K(n)} \subset (\Sp_{T(n)})\cycl \subset \Sp_{T(n)}. \]

\end{rec}


    

\begin{dfn}
    We define $\SP_{T(n)}^{\mathrm{ab}}$ to be $\W(\overline{\FF}_p)\otimes\SP_{T(n)}[\omega_{p^\infty}^{(n)}] \in \CAlg(\Sp_{T(n)})$. By \cite{carmeli2021chromatic}, this is a lift of the maximal abelian Galois extension of the $K(n)$-local sphere to a pro-Galois extension of the $T(n)$-local sphere, and we equip it with a $\Z\times (T_p\times \Z)$-action via the $\Z$-action on $\W(\overline{F}_p)$ coming from the Frobenius and the $T_p\times \Z$-action on $\SP_{T(n)}[\omega_{p^\infty}^{(n)}]$.
\end{dfn}

Cyclotomic redshift can be phrased as the following
hyperdescent result for the $K$-theory of the $p$-cyclotomic extensions.

  \todo{Do we really need the iwasawa one?}\todo{ishan: look at lem:6.22 and end of section 7, that is where it is used. }

\begin{thm}[{\cite[Theorem 5.11, Proposition 5.17]{cycloshift}}]
  \label{thm:cyclotomicredshiftinput}
  Let $n\geq0$, and $R \in \CAlg(\Sp_{T(n)})$. The natural
lax symmetric monoidal transformations
\[ L_{T(n+1)}K(R) \to L_{T(n+1)}K(\Ss_{T(n)}[\omega_{p^{\infty}}^{(n)}]^{hT_p} \otimes R)^{h\Z} \]
\[L_{T(n+1)}K(R)\to L_{T(n+1)}K(\Ss_{T(n)}[\omega_{p^{\infty}}^{(n)}] \otimes R)^{h(T_p\times \Z)} \]
  exhibit the target as the cyclotomic completion of the source, where the tensor products are taken in $\CAlg(\Sp_{T(n)})$.
\end{thm}
\begin{cor}\label{cor:cycredshift}
    For $n\geq 0$, the map $\SP_{T(n)} \to \SP_{T(n)}[\omega_{p^{\infty}}^{(n)}]$ is a $L_{T(n+1)}K\cycl$-cover.
\end{cor}

\begin{proof}
    This follows directly from \Cref{lem:descentcovers} and \Cref{thm:cyclotomicredshiftinput}, as we can identify $L_{T(n+1)}K\cycl$ with $(L_{T(n+1)}K(\Ss_{T(n)}[\omega_{p^{\infty}}^{(n)}] \otimes R)\cycl)^{h(T_p\times \Z)}$ as a functor. 
\end{proof}

Next, we extend \Cref{cor:cycredshift} to apply to $\SP_{T(n)}^{\mathrm{ab}}$. To do this, we show $\TC$ satisfies hyperdescent along
the $\pi_0$-\'etale extension $\W(\F_p) \to \W(\overline{\FF}_p)$.

\begin{lem} \label{lem:primetopcyc}
  Let $R \in \Alg(\Sp)$ 
  and give $\W(\Fpbar)$ the $\Z$-action by Frobenius.
  The natural map
  \[ \TC(R) \xrightarrow{\cong} \TC(\W(\Fpbar) \otimes R)^{h\Z} \]
  is an isomorphism. 
\end{lem}

\begin{proof}
  The equivalence follows from applying $\TC$ to the equivalence of cyclotomic spectra:
  \begin{align*}
    \THH(R)
    &\cong \W(\Fpbar)^{h\Z} \otimes \THH(R)
    \cong \THH(\W(\Fpbar))^{h\Z} \otimes \THH(R) \\
    &\cong (\THH(\W(\Fpbar)) \otimes \THH(R))^{h\Z}
    \cong \THH(\W(\Fpbar) \otimes R)^{h\Z}.\qedhere
  \end{align*}  
\end{proof}

\begin{cor}\label{cor:sabcover}
    For $n\geq 1$, the map $\SP_{T(n)} \to \W(\overline{\FF}_p)\otimes \SP_{T(n)}$ is a $L_{T(n+1)}K$-cover, and $\SP_{T(n)}\to \SP_{T(n)}^{\mathrm{ab}}$ is a $L_{T(n+1)}K\cycl$-cover.
\end{cor}
\begin{proof}
    By the definition of $\SP_{T(n)}^{\mathrm{ab}}$, \Cref{cor:cycredshift}, and the fact that $L_{T(n+1)}K$-covers are also $L_{T(n+1)}K\cycl$-covers (\Cref{lem:transfersheafcover}), it is enough to show the first statement. 



      $\tau_{\geq 0}R$ (resp. $\W(\Fpbar) \otimes \tau_{\geq 0}R$) is
  a connective, $T(n+1)$-acyclic $\E_1$-algebra whose $T(n)$-localization is $R$
  (resp. $L_{T(n)}(\W(\Fpbar) \otimes R)$).
  Applying \Cref{thm:red_main} we obtain a commuting diagram
  whose horizontal maps are isomorpshism.
  \[ \begin{tikzcd}
    L_{T(n+1)}K(R) \ar[r, no head, "\cong"] \ar[d] &
    L_{T(n+1)}\TC(\tau_{\geq 0} R) \ar[d] \\
    L_{T(n+1)}K(L_{T(n)}(\W(\Fpbar) \otimes R))^{h\Z} \ar[r, no head, "\cong"] &
    L_{T(n+1)}\TC(\W(\Fpbar) \otimes \tau_{\geq 0} R)^{h\Z} 
  \end{tikzcd} \]
  The right vertical map is an isomorphism by \Cref{lem:primetopcyc} so we conclude by \Cref{lem:descentcovers}.
\end{proof}

The following lemma will be useful for knowing that the target of the coassembly map is cyclotomically complete:

\begin{lem}\label{rmk:target-complete}
    For $n\geq1$, if $R \in \CAlg(\SP_{T(n)})$ and there is a map $\SP[\omega_{p^{\infty}}]^{hT_p} \to L_{T(n)}(\W(\overline{\FF}_p)\otimes R)$ in $\CAlg(\Sp_{T(n)})$, then $L_{T(n+1)}K(R)$ is cyclotomically complete.
\end{lem}

\begin{proof}
    Because $\SP_{T(n)} \to \W(\overline{\FF}_p)\otimes \SP_{T(n)}$ is a $L_{T(n+1)}K$-cover by \Cref{cor:sabcover}, it is enough to show that for $R' := L_{T(n)}(\W(\overline{\FF}_p)\otimes R)$, $L_{T(n+1)}K(R')$ is cyclotomically complete. 
By \Cref{thm:cyclotomicredshiftinput}, it is enough to show that the map $L_{T(n+1)}K(R') \to L_{T(n+1)}K(\Ss_{T(n)}[\omega_{p^{\infty}}^{(n)}]^{hT_p}\otimes R')$ has a retraction, but such a retraction can be produced using the map $\Ss_{T(n)}[\omega_{p^{\infty}}^{(n)}]^{hT_p}\to R'$ from the hypothesis.
\end{proof}

\begin{lem}\label{lem:intermediateextn}
    Given maps $A \to B \to C\to D$ where $A \to B$ is a finite Galois extension and $B \to C$ is a pro-Galois extension, then the map $B\otimes_AD \to C\otimes_AD$ admits a retraction.
\end{lem}
    \begin{proof}
        Without loss of generality, we can assume $C = D$, since if the map has a retraction after base change to $C$ it also does for $D$.
        
        Let $G$ be the Galois group of $A \to C$ and let $H$ be the Galois group of $A \to B$. Then the map $B\otimes_AC \to C\otimes_AC$ is isomorphic to the map
        $C^{G/H} \to C^{G}$ induced by the surjection of profinite sets $G \to G/H$. But this map has a section since $G/H$ is finite.
    \end{proof}





\begin{prop} \label{prop:bpncycmodel}
  For $n\geq1$, consider $R=L_{T(n)}\BP\langle n \rangle$ with the $\EE_{\infty}$-$\Z$-action from
  \Cref{thm:Adams-ops-exist}. Then for each $k\geq0$, the map $f_k:R^{hp^k\ZZ} \to R$ is a $L_{T(n+1)}K\cycl$-cover, and moreover $L_{T(n+1)}K(R)$ is cyclotomically complete.
\end{prop}

\begin{proof}
  We begin with a few recollections and some notation.
  \begin{itemize}
  \item Let $\ell$ be as in \Cref{cnv:ell}.
  \item Let $\Psi^\ell \in Z(\G_n) \subseteq \G_n$ be as in \Cref{cnv:BPnEn}, and let $\ZZ \subset \G_n$ be the subgroup generated by $\Psi^\ell$.
  \item Let $C \subseteq \G_n$ be the cyclic subgroup of order $p^n-1$
  from \Cref{thm:Adams-ops-exist}.
  \item Let $\mathrm{cyc}:\mathbb{G}_n \to \ZZ_p^{\times}$ be
    the $p$-adic cyclotomic character of \cite[Definition 5.6]{carmeli2021chromatic}.
  \item Let $\overline{\mathrm{cyc}}$ be the further quotient
    $ \G_n \xrightarrow{\mathrm{cyc}} \Z_p^\times \twoheadrightarrow \Z_p. $
  \item Recall that $\G_n$ has a subgroup $\cO_D^\times \subseteq \G_n$
    where $\cO_D^\times$ is the units in the ring of integers of the
    division algebra $D$ over $\Q_p$ with Hasse invariant $\frac 1 n$.
  \item Let $j = v_p(\ell^{np^k}-1)$.
  \end{itemize}
  
  \Cref{thm:Adams-ops-exist} implies that the map $f_k$ after base change to $L_{T(n)}\W(\overline{\FF}_p)$ becomes isomorphic to the map 
  $(E_n^{hC})^{hp^k\Z} \to E_n^{hC}$, which we denote $\tilde{f}_k$.    
  We claim there is a pushout square of $T(n)$-local commutative algebras
  \[ \begin{tikzcd}
    \Ss_{T(n)}[\omega_{p^j}^{(n)}]^{hT_p} \ar[r] \ar[d] &
    \Ss_{T(n)}[\omega_{p^\infty}^{(n)}]^{hT_p} \ar[d] \\
    (E_n^{hC})^{hp^k\Z} \ar[r,"\tilde{f}_k"] &
    E_n^{hC}.
  \end{tikzcd} \]
To see that this claim lets us finish the argument, we first note that the right vertical map along with \Cref{rmk:target-complete} lets us conclude $L_{T(n+1)}K(R)$ is cyclotomically complete. 

Next, by applying the pushout square and \Cref{lem:intermediateextn}, we learn that 
$\tilde{f}_k\otimes \SP[\omega_{p^{\infty}}^{(n)}] = f_k\otimes \SP_{T(n)}^{\mathrm{ab}}$ admits a retraction, and hence is a $L_{T(n+1)}K\cycl$-cover. By \Cref{rmk:checklocally} and \Cref{cor:sabcover}, we find that $f_k$ is also a $L_{T(n+1)}K\cycl$-cover. 
  
  We turn to proving the claim.  
  By \cite[Theorem 5.8]{carmeli2021chromatic},
  the $p$-adic cyclotomic character restricts to the determinant map on $\cO_D^\times$.
  This lets us compute that
  \[ {\mathrm{cyc}}(\Psi^\ell) = \ell^n. \]
 We observe that since $C$ is cyclic and $\Z_p$ is torsion-free
  $\overline{\mathrm{cyc}}$ factors through $\G_n/C$.
  Together these facts allows us to construct a pullback square of quotients of $\G_n$
  \[ \begin{tikzcd}
    \G_n/C \ar[r] \ar[d, "\overline{\mathrm{cyc}}"] &
    (\G_n/C)/\Psi^{\ell^{p^k}} \ar[d] \\    
    \Z_p \ar[r] &
    \Z_p/p^j.
  \end{tikzcd} \]  

  Applying $E_n^{h(\ker(-))}$ to this square we obtain the desired
  pushout square of $T(n)$-local commutative algebras
  after $K(n)$-localization. The pushout square holds before $K(n)$-localization too since $K(n)$-localization is smashing in $\Sp_{T(n)}$ and the lower rings in the square are $K(n)$-local.
  \qedhere

\end{proof}
	%
	%
	%
	%



%% file: TCcomputation2.tex
\label{sec:ht2}

In \Cref{sec:mainthm}, we showed that the telescope fails by proving that the coassembly map
\[T(n+1)\otimes\TC(\BP\langle n \rangle^{hp^k\ZZ}) \to T(n+1)\TC(\BP\langle n \rangle)^{hp^k\ZZ}\] 
is not an isomorphism for $n\geq1,k\gg0$. The goal of this section is to completely compute $\pi_*$ of a version of this coassembly map when $n=1,k\gg0$, and $p\geq 7$.

Specifically, we study the  connective Adams summand $\ell$, which is a form of $\BP \langle 1 \rangle$.\footnote{In this section, we follow historical convention and use $\ell$ to denote the Adams summand. In previous sections $\ell$ was used to denote an element of $\ZZ_p^{\times}$, and we trust the reader may distinguish these uses from context.} The $\ZZ$-action on $\ell$ is generated by the classical Adams operation $\Psi^{1+p}$. 
We restrict to primes $p\geq7$ in order to make use of Smith--Toda complexes $V(0)=\mathbb{S}/p$, $V(1)=\mathbb{S}/(p,v_1)$, and $V(2)=\mathbb{S}/(p,v_1,v_2)$, which exist as hcrings for such primes \cite[Proposition 3.3]{yy77}. We take $T(2)$ to be $v_2^{-1}V(1)$.

Our main result is true for all $k$ larger than a fixed positive integer, which depends only on the prime $p$:\footnote{We expect that this integer does not in fact depend on $p$, and may be taken to be as small as $2$ or $3$.}

\begin{thm} \label{thm:height2comptc}
	Let $p\geq 7$ be a prime, and let $\ZZ$ act on the connective Adams summand $\ell$ via the $\mathbb{E}_{\infty}$ Adams operation $\Psi^{1+p}$.  Then, for all $k \gg 0$, 
	the $\mathbb{F}_p[v_2]$-module map 
 \[V(1)_*\TC(\ell^{hp^k\ZZ}) \to V(1)_*\TC(\ell)^{hp^k\ZZ}\]
	    may be identified with the direct sum of the maps enumerated below. The degrees of classes are determined from their names via the facts that $|t| = -2, |\lambda_i| = 2p^i-1, |\partial|=-1, |\zeta| = -1$, and the degree of any locally constant function is $0$.
	\begin{enumerate}
		\item The projection $\FF_p\{1,\partial\} \oplus  \overline{\LCF{\Z_p^{\times}}}\{\partial\zeta\} \to \FF_p\{1,\partial\}$ onto the first factor, tensored over $\FF_p$ with the inclusion $\mathbb{F}_p[v_2]\langle\lambda_1,\lambda_2 \rangle \to \mathbb{F}_p[v_2] \langle \lambda_1,\lambda_2,\zeta\rangle$.
		\item The map $\FF_p[v_2]\langle \zeta \rangle \otimes \LCF{p\ZZ_p} \to \FF_p[v_2]\langle \zeta \rangle$ evaluating a continuous function at $0$, tensored over $\FF_p$ with the graded $\FF_p$-vector space on basis elements enumerated below:
		\begin{enumerate}
			\item $t^{d} \lambda_1$, for each $0<d<p$, in degree $2p-1-2d$.
			\item $t^{pd} \lambda_2$, for each  each $0<d<p$, in degree $2p^2-1-2pd$
			\item $t^{d}\lambda_1\lambda_2$, for each $0<d<p$, in degree $2p^2+2p-2-2d$.
			\item $t^{pd} \lambda_1 \lambda_2$, for each $0<d<p$, in degree $2p^2+2p-2-2pd$.
		\end{enumerate}
	\end{enumerate}
	
\end{thm}

 As we will explain, by inverting $v_2$ we recover the following corollary, which in turn implies \Cref{thm:introheight2} from the introduction:

\begin{cor} \label{cor:height2comp}
Let $p\geq 7$ be a prime, 
and let $\ZZ$ act on the Adams summand $L$ of $\mathrm{KU}_{(p)}$ via the $\mathbb{E}_{\infty}$ Adams operation $\Psi^{1+p}$.  Then, for all $k \geq 0$, 
	the map 
 \[T(2)\otimes \mathrm{K}(L^{hp^k\ZZ}) \to T(2)\otimes \mathrm{K}(L)^{hp^k\ZZ}\]
        is both the cyclotomic completion map and the $K(2)$-localization map. 
        
        At the level of $\pi_*$ as an $\FF_p[v_2^{\pm1}]$-module, this map for $k\gg0$ may be identified with the direct sum of the maps enumerated below. The degrees of classes are determined from their names via the facts that $|t| = -2, |\lambda_i| = 2p^i-1, |\partial|=-1, |\zeta| = -1$, and the degree of any locally constant function is $0$.
	\begin{enumerate}
		\item The projection $\FF_p\{1,\partial\} \oplus  \overline{\LCF{\Z_p^{\times}}}\{\partial\zeta\} \to \FF_p\{1,\partial\}$ onto the first factor, tensored over $\FF_p$ with the inclusion $\mathbb{F}_p[v_2^{\pm 1}]\langle\lambda_1,\lambda_2 \rangle \to \mathbb{F}_p[v_2^{\pm 1}] \langle \lambda_1,\lambda_2,\zeta\rangle$.
		\item The map $\FF_p[v_2^{\pm 1}]\langle \zeta \rangle \otimes \LCF{p\ZZ_p} \to \FF_p[v_2^{\pm 1}]\langle \zeta \rangle$ evaluating a continuous function at $0$, tensored over $\FF_p$ with the graded $\FF_p$-vector space on basis elements enumerated below:
		\begin{enumerate}
			\item $t^{d} \lambda_1$, for each $0<d<p$, in degree $2p-1-2d$.
			\item $t^{pd} \lambda_2$, for each  each $0<d<p$, in degree $2p^2-1-2pd$
			\item $t^{d}\lambda_1\lambda_2$, for each $0<d<p$, in degree $2p^2+2p-2-2d$.
			\item $t^{pd} \lambda_1 \lambda_2$, for each $0<d<p$, in degree $2p^2+2p-2-2pd$.
		\end{enumerate}
	\end{enumerate}
	
\end{cor}

\begin{rmk}
We think of $\ell^{hp^k\ZZ}$ as a bounded below model (analogous to an order in the ring of integers) of the height $1$ local field obtained as 
    a finite subextension of the Iwasawa extension $\SP_{T(1)}[\omega_{p^{\infty}}^{(1)}]^{hT_p}$.
\end{rmk}


We now proceed with the proofs of \Cref{thm:height2comptc} and \Cref{cor:height2comp}.  After these are proved, we discuss in \Cref{sec:LQht1} some qualitative results about $\mathrm{K}(L_{K(1)} \mathbb{S})$ itself, which do not require passage to a finite Galois extension to prove.



\subsection{Recollections regarding \texorpdfstring{$V(2) \otimes \THH(\ell)$}{V2 THH(BP1)}}

In this subsection, we recall a few classical results about the cyclotomic spectrum $\THH(\ell)$. Most importantly we have the following:

\begin{thm}[Ausoni--Rognes]\label{thm:AR}
For $p\geq5$, the cyclotomic spectrum
\[V(2) \otimes \THH(\ell)\]
is bounded in the cyclotomic $t$-structure.  In other words, $\ell$ satisfies the height $1$ LQ property. 
\end{thm}

\begin{proof}
This follows from \cite[Theorem 8.5]{ausoni2002algebraic} (see also 
\cite[Theorem G]{hahn2020redshift}).
\end{proof}

From now on let $p\geq7$ unless stated otherwise.

Next, we recall the explicit identification of $V(1)_*\TC(\ell)$ by Ausoni and Rognes, together with additional algebraic identifications of the objects $V(2)_*\THH(\ell)$, $V(2)_*\THH(\ell)^{tC_p}$, $V(2)_*\TC^{-}(\ell)$, and $V(2)_*\TP(\ell)$. An alternative perspective on some of these calculations, with spectral sequence diagrams, may be found in \cite[Section 6]{hahn2022motivic}.

\begin{rec} \label{rec:lambda1}
Ausoni and Rognes \cite[Definitions 1.3 and 1.8]{ausoni2002algebraic}, following work of B\"okstedt and Madsen, constructed classes
\[\lambda_1,\lambda_2 \in V(0)_*K(\ell).\]
Following these authors, we will also use $\lambda_1$ and $\lambda_2$ to refer to the images of these classes in trace invariants under $V(0)_*K(\ell)$. In particular, under this convention both
\[\varphi^{h\T}:V(2)_*\TC^{-}(\ell) \to V(2)_*\TP(\ell), \text{ and}\]
\[\can:V(2)_*\TC^{-}(\ell) \to V(2)_*\TP(\ell),\]
are $\mathbb{F}_p\langle \lambda_1,\lambda_2\rangle$-module maps.
\end{rec}

There is an isomorphism of $\mathbb{F}_p$-algebras
\[V(1)_*\THH(\ell) \cong \mathbb{F}_p[\mu] \otimes \Lambda(\lambda_1,\lambda_2),\]
due to McClure--Staffeldt \cite{mcclure} \cite[Proposition 2.6]{ausoni2002algebraic}.  Since $v_2$ is trivial in this ring, killing it introduces an exterior generator and we arrive at the following:

\begin{cnv} \label{cnv:THHellnames}
We fix a preferred isomorphism of $\mathbb{F}_p$-algebras
\[
V(2)_*\THH(\ell) \cong \mathbb{F}_p[\mu]\langle \lambda_1,\lambda_2,\epsilon\rangle,
\]
such that $\lambda_1$ and $\lambda_2$ are as in \Cref{rec:lambda1} and $\epsilon$ corresponds to a choice of nullhomotopy of $v_2$ in $V(1)_*\THH(\ell)$.  Here, $|\mu|=2p^2,|\lambda_1|=2p-1,$ and $|\lambda_2|=|\epsilon|=2p^2-1$.
\end{cnv}

\begin{rmk}
We will soon see that $\mu$ is a B\"okstedt class in the sense of \Cref{subsec:bokstedt1}.
\end{rmk}

\begin{prop} \label{prop:TtateTHHell}
The $\T$-Tate spectral sequence
\[\mathrm{E}_2=\mathbb{F}_p[\mu,t^{\pm 1}]\langle \lambda_1,\lambda_2,\epsilon\rangle \implies V(2)_*\TP(\ell)\] has differentials
\[d_2(\epsilon)=t\mu,\]
\[d_{2p}(t)=t^{p+1}\lambda_1,\]
\[d_{2p^2}(t^p)=t^{p^2+p}\lambda_2,\]
with all other differentials determined by multiplicative structure and the facts that $\lambda_1$ and $\lambda_2$ are permanent cycles.  The $\mathrm{E}_{\infty}$-page is isomorphic to $\mathbb{F}_p[t^{\pm p^2}]\langle \lambda_1,\lambda_2\rangle.$
\end{prop}

\begin{proof}
This follows from work of Ausoni--Rognes \cite[Section 6]{ausoni2002algebraic}, who compute the much more difficult $\T$-Tate spectral sequence for $V(1)_*\TP(\ell)$.  More precisely,  the map of $\T$-Tate spectral sequences induced by the $\T$-equivariant map $V(1)_*\THH(\ell) \to V(2)_*\THH(\ell)$ induces the differentials on $t$ and $t^{p}$.  The $d_2$ differential on $\epsilon$ follows from \cite[Proposition 4.8]{ausoni2002algebraic}.
\end{proof}

Note that the $\T$-homotopy fixed point spectral sequence
\[\mathrm{E}_2=\mathbb{F}_p[\mu,t]\langle \lambda_1,\lambda_2,\epsilon\rangle \implies V(2)_*\TC^{-}(\ell)\]
is formally determined by the $\T$-Tate spectral sequence through truncation.

\begin{dfn}
Let $N$ denote the submodule of the $\mathbb{F}_p\langle \lambda_1,\lambda_2\rangle$-module $V(2)_*\TC^{-}(\ell)$ that consists of classes simultaneously in the kernel of $\mathrm{can}$ and of positive Nygaard filtration (i.e., of positive filtration in the $\T$-homotopy fixed point spectral sequence).
\end{dfn}

Examining the canonical map between the $\T$-homotopy fixed point spectral sequence and the $\T$-Tate spectral sequence, we arrive at the following result:

\begin{cor} \label{cor:Ncalculation}
The canonical map
\[\mathrm{can}:V(2)_*\TC^{-}(\ell) \to V(2)_*\TP(\ell)\]
is trivial in degrees $* \ge 2p^2$.  The subspace $N$ of $V(2)_*\TC^{-}(\ell)$ is a $4(p-1)$ dimensional $\mathbb{F}_p$-vector space, with basis elements detected in the $\T$-homotopy fixed point spectral sequence by:
\begin{enumerate}
    \item $t^{d} \lambda_1$, for each $0<d<p$, in degree $2p-1-2d$
    \item $t^{pd} \lambda_2$, for each each $0<d<p$, in degree $2p^2-1-2pd$
    \item $t^{d}\lambda_1\lambda_2$, for each $0<d<p$, in degree $2p^2+2p-2-2d$.
    \item $t^{pd} \lambda_1 \lambda_2$, for each $0<d<p$, in degree $2p^2+2p-2-2pd$.
\end{enumerate}
Classes in $(3)$ are obtained from classes in $(1)$ by multiplication by $\lambda_2$, and classes in $(4)$ are obtained from classes in $(2)$ by multiplication by $\lambda_1$.
\end{cor}

We next turn to the cyclotomic Frobenius map.

\begin{prop} \label{prop:ellFrob}
The cyclotomic Frobenius 
\[\varphi:V(2)_*\THH(\ell) \to V(2)_*\THH(\ell)^{tC_p}\]
is the universal $\mathbb{F}_p$-algebra map that inverts $\mu$.  In other words,
\[
V(2)_*\THH(\ell)^{tC_p} \cong \mathbb{F}_p[\varphi(\mu)^{\pm 1}]\langle \lambda_1,\lambda_2,\varphi(\epsilon) \rangle.
\]
\end{prop}

\begin{proof}
It suffices to identify
\[\varphi:V(1)_*\THH(\ell) \to V(1)_*\THH(\ell)^{tC_p}\]
with the map
\[\mathbb{F}_p[\mu]\langle \lambda_1,\lambda_2 \rangle \to \mathbb{F}_p[\varphi(\mu)^{\pm 1}] \langle \lambda_1,\lambda_2 \rangle.\]
This is \cite[Theorem 5.5]{ausoni2002algebraic}.
\end{proof}

\begin{cor} \label{cor:varphiNygaard}
The map
\[\varphi^{h\T}:V(2)_*\TC^{-}(\ell) \to V(2)_*\TP(\ell)\]
is an isomorphism in degrees $* \gg 0$.  
It is trivial on all classes of positive Nygaard filtration.  
\end{cor}

\begin{proof}
The first claim follows from the fact that the fiber of $\varphi$ is bounded above, and so the fiber of $\varphi^{h\T}$ is as well.  To see the second claim, we must check that the $\mathrm{E}_{\infty}$-page of the $\T/C_p$-homotopy fixed point spectral sequence for
$\left(\THH(\ell)^{tC_p}\right)^{h\T/C_p}$ is concentrated on the $0$-line. This $\T/C_p$-homotopy fixed point spectral sequence begins with \[E_2=\mathbb{F}_p[\varphi(\mu)^{\pm 1},t]\langle \varphi(\epsilon),\lambda_1,\lambda_2\rangle.\]
We may calculate the $d_2$ differential using the $\varphi$ map from the $\T$-homotopy fixed point spectral sequence determined by \Cref{prop:TtateTHHell}.  After running the $d_2$ differential, we find that the $\mathrm{E}_3$-page is concentrated on the $0$-line.
\end{proof}

\begin{cor} \label{cor:V2TCell}
As a graded $\mathbb{F}_p \langle \lambda_1,\lambda_2\rangle$-module, $V(2)_*\TC(\ell)$ is isomorphic to the direct sum of $N$ and $\mathbb{F}_p\langle \lambda_1,\lambda_2,\partial\rangle$, where $|\partial|=-1$. 
\end{cor}

\begin{proof}
This follows immediately from \cite[Theorem 0.3]{ausoni2002algebraic}, but we indicate some details that will generalize to the study of $V(2)_*\TC(\ell^{\mathrm{B}\ZZ})$.
 We must calculate the equalizer and coequalizer of the map
\[\varphi^{h\T}-\can:V(2)_*\TC^{-}(\ell) \to V(2)_*\TP(\ell).\]

Since the $\varphi^{h\T}$ is trivial on classes of positive Nygaard filtration, we may form the diagram
\[
\begin{tikzcd}
\mathrm{ker}_1 \arrow{r} \arrow{d} 
& \mathrm{Nyg}_{\ge 1}\left(V(2)_*\TC^{-}(\ell) \right) \arrow{r}{0-\mathrm{can}} \arrow{d}
& \mathrm{Nyg}_{\ge 1} \left(V(2)_*\TP(\ell)\right) \arrow{r} \arrow{d}
& \mathrm{coker}_1 \arrow{d}\\
\mathrm{ker} \arrow{r} \arrow{d}
& V(2)_*\TC^{-}(\ell) \arrow{r}{\varphi-\can}  \arrow{d}
& \mathbb{F}_p[t^{\pm p^2}]\langle \lambda_1,\lambda_2\rangle \arrow{r} \arrow{d}
&\mathrm{coker} \arrow{d} \\
\mathrm{ker}_2 \arrow{r} 
& \mathrm{Nyg}_{=0}\left(V(2)_*\TC^{-}(\ell^{\mathrm{B}\ZZ})\right) \arrow{r}{\overline{\varphi}-\overline{\mathrm{can}}} 
&  \left( V(2)_*\TP(\ell^{\mathrm{B}\ZZ})\right) / \mathrm{Nyg}_{\ge 1}  \arrow{r} & \mathrm{coker}_2.
\end{tikzcd}
\]
and apply the snake lemma.  From \Cref{prop:TtateTHHell} and \Cref{cor:Ncalculation} we see that $\mathrm{ker}_1$ is exactly $N$, while $\mathrm{coker}_1$ is zero.  Furthermore, since $\can$ is trivial in degrees $* \ge 2p^2$, we identify the map 
\[\overline{\can}:\mathbb{F}_p[\mu]\langle \lambda_1,\lambda_2 \rangle \to \mathbb{F}_p[\varphi(\mu)] \langle \lambda_1,\lambda_2 \rangle\]
as the map that kills $\mu$. 
All of this proves that $V(2)_*\TC(\ell)$ admits a filtration with associated graded the direct sum of:
\begin{itemize}
\item $\mathrm{ker}_1=N$
\item $\mathrm{ker}_2 \cong \mathbb{F}_p\langle \lambda_1,\lambda_2 \rangle$, and
\item $\Sigma^{-1}\mathrm{coker}_1 \cong \mathbb{F}_p\langle \lambda_1,\lambda_2\rangle\{\partial\}$.
\end{itemize}
To conclude, we observe that $\mathbb{F}_p\langle \lambda_1,\lambda_2\rangle$ extension problems in the above filtration are ruled out by the facts that $\mathrm{ker}_2$ and $\mathrm{coker}_1$ are free modules. \qedhere
\end{proof}

Finally, Ausoni and Rognes prove that the $v_2$-Bockstein spectral sequence converging to $V(1)_*\TC(\ell)$ degenerates, with no differentials or extension problems.

\begin{thm}[Theorem 0.3 of \cite{ausoni2002algebraic}] As a graded $\mathbb{F}_p[v_2] \langle \lambda_1,\lambda_2 \rangle$-module, $V(1)_*\TC(\ell)$ is isomorphic to the direct sum of:
\begin{enumerate}
\item $N \otimes_{\mathbb{F}_p} \mathbb{F}_p[v_2]$,
\item $\mathbb{F}_p[v_2]\langle \lambda_1,\lambda_2,\partial \rangle$.
\end{enumerate}
\end{thm}

\subsection{Calculations regarding \texorpdfstring{$V(2) \otimes \THH(\ell^{B\ZZ})$}{V2 THH(BP1 BZ)}}

Using \Cref{thm:AR}, \Cref{thm:E1A2algtame}, and the fact that $\ell$ is a $\ZZ$-equivariant $\EE_{\infty}$-ring of fp-type $1$, we deduce that the coassembly map 
\[V(1)_*\TC(\ell^{hp^k\ZZ}) \to V(1)_*\TC(\ell)^{hp^k\ZZ}\]
is, as an $\mathbb{F}_p[v_2]$-module map for $k \gg 0$, the same as the coassembly map
\[V(1)_*\TC(\ell^{\mathrm{B}\ZZ}) \to V(1)_*\TC(\ell)^{\mathrm{B}\ZZ}\]
associated to the trivial $\ZZ$ action on $\ell$.

In this subsection we study the cyclotomic spectrum $V(2) \otimes \THH(\ell^{B\ZZ})$ at a fixed prime $p \ge 7$, obtaining analogs of many of the results and proofs from the previous subsection.

First we observe that the natural map
\[V(2)_*\THH(\ell) \to V(2)_*\THH(\ell^{\mathrm{B}\Z})\]
defines elements $\mu,\zeta,\epsilon,\lambda_1,$ and $\lambda_2$ in the codomain, according to \Cref{cnv:THHellnames}.  Furthermore, since $\lambda_1$ and $\lambda_2$ lift to $V(2)_*\TC(\ell)$, they also lift to $V(2)_*\TC(\ell^{\mathrm{B}\Z})$. Additionally, we use the natural map
\[\THH(\Ss^{\mathrm{B}\ZZ}) \to \THH(\ell^{\mathrm{B}\ZZ})\]
to equip with $V(2)_*\THH(\ell^{\mathrm{B}\ZZ})$ with the struture of a $\LCF{\Z_p}$-algebra.  In terms of these elements, we identify the coassembly map in $\THH$ as follows:

\begin{prop}
The coassembly map
\[V(2)_*\THH(\ell^{\mathrm{B}\Z}) \to V(2)_*\THH(\ell)^{\mathrm{B}\Z}\]
may be identified with the $\mathbb{F}_p$-algebra map
\[\LCF{\Z_p}[\mu]\langle \zeta,\epsilon,\lambda_1,\lambda_2 \rangle \to \mathbb{F}_p[\mu] \langle \zeta,\epsilon,\lambda_1,\lambda_2\rangle\]
that evaluates a continuous function at $0$.
\end{prop}

\begin{proof}
We apply \Cref{lem:thh-s1} to the isomorphism of $\mathbb{E}_{\infty}$-algebras
\[\THH(\ell^{\mathrm{B}\Z}) \simeq \THH(\ell) \otimes \THH(\Ss^{\mathrm{B}\Z}).\]
which is compatible with the coassembly map.
\end{proof}

We next turn to the $\T$-homotopy fixed point and $\T$-Tate spectral sequences converging to $V(2)_*\TC^{-}(\ell^{\mathrm{B}\ZZ})$ and $V(2)_*\TP(\ell^{\mathrm{B}\ZZ})$, respectively.

\begin{thm} \label{thm:ht1tate}
The $\T$-Tate spectral sequence
\[\mathrm{E}_2=C^0(\mathbb{Z}_p)[\mu,t^{\pm 1}]\langle \lambda_1,\lambda_2,\epsilon,\zeta\rangle \implies V(2)_*\TP(\ell^{B\Z})\] has differentials
\[d_2(\epsilon)=t\mu,\]
\[d_2(\zeta)= f t \zeta,\]
\[d_{2p}(t)=t^{p+1}\lambda_1,\]
\[d_{2p^2}(t^p)=t^{p^2+p}\lambda_2,\]
where $f \in \LCF{\Z_p}$ is a function vanishing exactly on $p\Z_p$.
All other differentials are determined by multiplicative structure and the facts that $\lambda_1$, $\lambda_2$, and elements of $\LCF{p\Z_p}\langle \zeta \rangle$ are permanent cycles. The $\mathrm{E}_{\infty}$-page is isomorphic to $\LCF{p\Z_p}[t^{\pm p^2}]\langle \zeta,\lambda_1,\lambda_2\rangle.$ 
\end{thm}

\begin{proof}
We are studying here the $\T$-Tate spectral sequence associated to the $\T$-equivariant spectrum
\[V(2) \otimes \THH(\ell) \otimes \THH(\Ss^{B\ZZ}).\]

By \cite[Corollary 1.3]{malkiewich2017topological}, this is additively isomorphic to 
\[V(2) \otimes \THH(\ell) \oplus \bigoplus_{n \ge 1} V(2) \otimes \THH(\ell) \otimes \mathbb{S}^{\T/C_n},\]
and as in \Cref{lem:thh-s1} and \Cref{exm:loop-action} the projection map to the summand indexed by $n$ is given by evaluation of locally constant functions at $n$.

We will now calculate the $\T$-Tate spectral sequence associated to 
\[V(2) \otimes \THH(\ell) \otimes \Ss^{\T/C_n}.\] As $n$ ranges over all positive integers, these spectral sequences determine that of \Cref{thm:ht1tate}.

The $\T$-Tate spectral sequence associated to 
\[V(2) \otimes \THH(\ell) \otimes \Ss^{\T/C_{n}}\]
is of signature
\[\mathrm{E}_2=\mathbb{F}_p[\mu,t^{\pm 1}]\langle \lambda_1,\lambda_2,\epsilon,\zeta \rangle \implies V(2)_*\THH(\ell)^{tC_{n}}.\]

Using the restriction map from the $\T$-Tate spectral sequence associated to $V(2) \otimes \THH(\ell)$, as calculated in \Cref{prop:TtateTHHell}, one sees that there is a $d_2$ differential $d_2(\epsilon)=t\mu$, while 
\[d_2(\lambda_1)=d_2(\lambda_2)=d_2(\mu)=d_2(t)=0.\]

If $n$ has $p$-adic valuation $0$, then $V(2) \otimes \THH(\ell) \otimes \Ss^{\T/C_n}$ is induced, and so the spectral sequence becomes trivial at the $\mathrm{E}_3$-page after a non-zero differential on $\zeta$.

If $n$ has positive $p$-adic valuation, it follows from the calculations of Ausoni--Rognes \cite[Section 6]{ausoni2002algebraic} that $V(2)_*\THH(\ell)^{tC_{n}}$ is non-trivial, so the $d_2$ differential on $\zeta$ must be trivial.
 In this case the $\mathrm{E}_3$-page is isomorphic to
\[\mathbb{F}_p[t^{\pm 1}] \langle \lambda_1,\lambda_2, \zeta \rangle.\]

If $n$ is of $p$-adic valuation $1$, then the spectral sequence must converge to $V(2)_*\THH(\ell)^{tC_p}$, which was calculated in \Cref{prop:ellFrob}.  The only differentials consistent with both \Cref{prop:TtateTHHell} and \Cref{prop:ellFrob} are
\[d_{2p}(t)=t^{p+1}\lambda_1, \text{ and}\]
\[d_{2p^2}(t^p)=t^{p^2+p}\lambda_2,\]
as induced from \Cref{prop:TtateTHHell}.  In particular, $\zeta$ must be a permanent cycle in order to obtain an $\mathrm{E}_{\infty}$-page of size at least that of $V(2)_*\THH(\ell)^{tC_p}$, and one is left with an $\mathrm{E}_{\infty}$-page of $\mathbb{F}_p[t^{\pm p^2}] \langle \lambda_1,\lambda_2,\zeta \rangle$.

For $n$ of $p$-adic valuation greater than $1$, the transfer map $\SP^{\T/C_{(n/p)}}\to \SP^{\T/C_{n}}$ sends $\zeta$ to $\zeta$, so ensures that $\zeta$ is a permanent cycle. Now the only differentials compatible with \cite[Proposition 6.3]{ausoni2002algebraic} and the restriction maps $\SP^{\T} \to \SP^{\T/C_{n}}$ are the same pattern of differentials as in the $p$-adic valuation $1$ case.
\end{proof}

The $\T$-homotopy fixed point spectral computing $V(2)_*\TC^{-}(\ell^{B\ZZ})$ is formally determined as a truncation of the above $\T$-Tate spectral sequence.  In particular, we have the following corollary:

\begin{cor} \label{cor:kercan}
When restricted to classes of positive Nygaard filtration, the $\mathrm{can}$ map
\[\mathrm{Nyg}_{\ge 1}V(2)_*\TC^{-}(\ell^{\mathrm{B}\Z}) \to \mathrm{Nyg}_{\ge 1}V(2)_*\TP(\ell^{\mathrm{B}\Z})\]
has trivial cokernel. The kernel of this restricted $\mathrm{can}$ map is isomorphic to
\item \[\LCF{p\Z_p} \langle \zeta \rangle \otimes_{\mathbb{F}_p} N.\]
\end{cor}

\begin{proof}
The subspace $\LCF{p\Z_p} \langle \zeta \rangle \otimes_{\mathbb{F}_p} N$ is clearly in the kernel of the restricted $\mathrm{can}$ map.  At the level of homotopy fixed point spectral sequences, we can see that $\mathrm{can}$ is an isomorphism on the remaining quotient of $V(2)_*\TC^{-}(\ell^{\mathrm{B}\Z})$
\end{proof}

\begin{prop} \label{prop:ht1frob}
The cyclotomic Frobenius 
\[\varphi:V(2)_*\THH(\ell^{\mathrm{B}\Z}) \to V(2)_*\THH(\ell^{\mathrm{B}\Z})^{tC_p}\]
is the universal map that inverts $\mu$. In other words, there is a preferred isomorphism
\[
V(2)_*\THH(\ell^{B\Z})^{tC_p} \cong \LCF{p\Z_p}[\varphi(\mu)^{\pm 1}]\langle \lambda_1,\lambda_2,\varphi(\epsilon),\varphi(\zeta) \rangle.
\]
Under this isomorphism, $\varphi$ carries a locally constant function $x \mapsto f(x)$ to the function $x \mapsto f(x/p)$.
\end{prop}

\begin{proof}
This follows by combining \Cref{prop:ellFrob} and \Cref{prop:cycfrobuniv}.
\end{proof}

\begin{lem}
    The cyclotomic Frobenius map
    \[\varphi^{h\T}:V(2)_*\TC^{-}(\ell^{\mathrm{B}\Z}) \to V(2)_*\TP(\ell^{\mathrm{B}\Z})\]
    annihilates all classes of positive Nygaard filtration.
\end{lem}

\begin{proof}
It will suffice to prove that the $\mathrm{E}_{\infty}$-page of the $(\T/C_p)$-homotopy fixed point spectral sequence 
\[\LCF{p\Z_p}[\varphi(\mu)^{\pm 1},t]\langle\lambda_1,\lambda_2,\varphi(\zeta),\varphi(\epsilon)\rangle \implies V(2)_*(\THH(\ell^{\mathrm{B}\ZZ})^{tC_p})^{h\T/C_p}\]
is trivial above the $0$-line.  As in the proof of \Cref{cor:varphiNygaard}, this is already true at the $\mathrm{E}_3$-page, because of $d_2$ differentials induced from the $\T$-homotopy fixed point spectral sequence for $V(2)_*\TC^{-}(\ell^{\mathrm{B}\ZZ})$. 
\end{proof}

\begin{thm} \label{thm:v2tcell}
The graded $\mathbb{F}_p\langle \lambda_1,\lambda_2 \rangle$-module map
\[V(2)_*\mathrm{TC}(\ell^{\mathrm{B}\Z}) \to V(2)_*\mathrm{TC}(\ell)^{\mathrm{B}\Z}\]
is isomorphic to the direct sum of the following three $\mathbb{F}_p\langle \lambda_1,\lambda_2\rangle$ module maps:
\begin{enumerate}
\item The map \[\LCF{p\Z_p} \langle \zeta \rangle \otimes_{\mathbb{F}_p} N \to \mathbb{F}_p\langle \zeta \rangle \otimes_{\mathbb{F}_p} N\]
evaluating a function at $0$.
\item The inclusion \[\mathbb{F}_p\langle \lambda_1, \lambda_2 \rangle \to \mathbb{F}_p\langle \lambda_1,\lambda_2,\zeta \rangle.\]
\item The projection \[\overline{\LCF{\Z_p^{\times}}}\{\partial \zeta\}
\oplus \mathbb{F}_p\{\partial\} \to \mathbb{F}_p\{\partial\}\]
onto the second factor, tensored over $\mathbb{F}_p$ with the inclusion from $(2)$.
\end{enumerate}
\end{thm}
\begin{proof}
Our first goal will be to compute the equalizer and coequalizer of the maps
\[
\can,\varphi^{h\T}:V(2)_*\TC^{-}(\ell^{\mathrm{B}\ZZ}) \to V(2)_*\TP(\ell^{\mathrm{B}\ZZ}),
\]
thereby computing $V(2)_*\TC(\ell^{\mathrm{B}\ZZ})$.

On classes of positive Nygaard filtration, $\varphi^{h\T}$ is trivial.  It follows that $\varphi^{h\T}-\mathrm{can}$ takes classes of positive Nygaard filtration to classes of positive Nygaard filtration.  As a result, we may consider the following diagram and apply the snake lemma:

\[
\begin{tikzcd}
\mathrm{ker}_1 \arrow{r} \arrow{d} 
& \mathrm{Nyg}_{\ge 1}\left(V(2)_*\TC^{-}(\ell^{\mathrm{B}\ZZ}) \right) \arrow{r}{0-\mathrm{can}} \arrow{d}
& \mathrm{Nyg}_{\ge 1} \left(V(2)_*\TP(\ell^{\mathrm{B}\ZZ})\right) \arrow{r} \arrow{d}
& \mathrm{coker}_1 \arrow{d}\\
\mathrm{ker} \arrow{r} \arrow{d}
& V(2)_*\TC^{-}(\ell^{\mathrm{B\ZZ}}) \arrow{r}{\varphi^{h\T}-\can}  \arrow{d}
& \LCF{p\Z_p}\langle \zeta,\lambda_1,\lambda_2\rangle[t^{\pm p^2}] \arrow{r} \arrow{d}
&\mathrm{coker} \arrow{d} \\
\mathrm{ker}_2 \arrow{r} 
& \mathrm{Nyg}_{=0}\left(V(2)_*\TC^{-}(\ell^{\mathrm{B}\ZZ})\right) \arrow{r}{\overline{\varphi}-\overline{\mathrm{can}}} 
&  \left( V(2)_*\TP(\ell^{\mathrm{B}\ZZ})\right) / \mathrm{Nyg}_{\ge 1}  \arrow{r} & \mathrm{coker}_2.
\end{tikzcd}
\]

By \Cref{cor:kercan}, we identify $\mathrm{ker}_1$ with $\LCF{p\Z_p}\langle \zeta \rangle \otimes_{\mathbb{F}_p} N$, and observe that $\mathrm{coker}_1$ is trivial. We next check that $\mathrm{ker}_2$ and $\mathrm{coker}_2$ are the other two summands. To do, it is helpful to observe that $\overline{\varphi}$ must be an isomorphism in large degrees, because $\varphi^{h\T}$ is and $\mathrm{ker}_1$ is bounded above.


By construction, the domain of $\overline{\varphi}-\overline{\can}$ can be identified with the $0$-line of the $\mathrm{E}_{\infty}$-page of the homotopy fixed point spectral sequence, which is calculated by \Cref{thm:ht1tate} to be
\[\LCF{p\Z_p}[\mu] \langle \zeta,\lambda_1,\lambda_2 \rangle \oplus \LCF{\Z_p^{\times}}[\mu] \langle \lambda_1,\lambda_2,\epsilon \rangle \subset V(2)_*\THH(\ell^{\mathrm{B}\Z}).\]
We next note that $\overline{\varphi}$ carries $\mu$ to a class detected by a unit multiple of $t^{-p^2}$, and indeed otherwise $\varphi$ could not be an isomorphism in large degrees. \Cref{thm:ht1tate} then implies the codomain of $\overline{\varphi}-\overline{can}$ can be identified 
with
\[\LCF{p\Z_p}[(\overline{\varphi}-\overline{\can})(\mu)]\langle \zeta, \lambda_1, \lambda_2 \rangle,\] since there is a map from this ring into $V(2)_*\TP(\ell^{B\Z})$ which is an isomorphism on the associated graded groups after projecting to the quotient by positive Nygaard filtration. 

As in \Cref{prop:cycfrobuniv} and \Cref{prop:ht1frob}, $\overline{\varphi}$ carries a function $f \in C^0(p\Z)$ to the function $x \mapsto f(x/p)$, where $f(x/p)$ is interpreted as zero if $x/p \in \Z_p^{\times}$.  Similarly, \Cref{prop:cycfrobuniv} determines $\overline{\varphi}$ on $\zeta$. Furthermore, $\overline{\varphi}$ is an $\mathbb{F}_p\langle \lambda_1,\lambda_2 \rangle$ module map, leaving only $\overline{\varphi}(\epsilon)$ undetermined.  

Now, the fact that $\overline{\varphi}$ is an isomorphism in large degrees is enough to prove that $\overline{\varphi}(\epsilon)$ is sent to $\zeta$ times a function non-trivial on elements of $p$-adic valuation $1$, and so $\mathrm{ker}_2$ and $\mathrm{coker}_2$ must be $\mathbb{F}_p\langle \lambda_1,\lambda_2\rangle$ and $\overline{\LCF{\Z_p^{\times}}}\{\zeta\} \oplus \mathbb{F}_p$, respectively.

All of this describes a filtration on $V(2)_*\TC(\ell^{\mathrm{B\Z}})$, with associated graded the direct sum of
\begin{itemize}
\item $\mathrm{ker}_1 \cong \LCF{p\Z_p}\langle \zeta \rangle \otimes_{\mathbb{F}_p} N$,
\item $\mathrm{ker}_2 \cong \mathbb{F}_p\langle \lambda_1,\lambda_2 \rangle$, and
\item $\Sigma^{-1}\mathrm{coker}_2 \cong \Sigma\overline{\LCF{\Z_p^{\times}}}\{ \partial \zeta\} \oplus \mathbb{F}_p \{ \partial\}$
\end{itemize}
This filtration is by construction compatible with that of the proof of \Cref{cor:V2TCell} above, along the coassembly map. We therefore see that the coassembly map is as indicated, up to possible filtration jumps. Such filtration jumps are precluded by $\mathbb{F}_p\langle \lambda_1,\lambda_2 \rangle$-module structure. 
\end{proof}

\subsection{Coassembly for \texorpdfstring{$V(1)_*\TC(\ell^{\mathrm{B}\Z})$}{V1 TC(BP1 BZ)}}

In the previous subsection, we computed the coassembly map
\begin{equation} \label{eqn:ellV2coassembly}
V(2)_*\TC(\ell^{\mathrm{B}\Z}) \to V(2)_*\TC(\ell)^{\mathrm{B}\Z}.
\end{equation}
To finish the proof of \Cref{thm:height2comptc}, it remains only to prove that the $\mathbb{F}_p[v_2]$-module map
\[
V(1)_*\TC(\ell^{\mathrm{B}\Z}) \to V(1)_*\TC(\ell)^{\mathrm{B}\Z}
\]
is given by  $\mathbb{F}_p[v_2] \otimes_{\mathbb{F}_p} (\ref{eqn:ellV2coassembly})$.

\begin{proof}[Proof of \Cref{thm:height2comptc}]
Examine the map of $v_2$-Bockstein spectral sequences

\[
\begin{tikzcd}
\mathrm{E}_1=V(2)_*\TC(\ell^{\mathrm{B}\ZZ})[v_2] \arrow[Rightarrow]{r}\arrow{d}
&  V(1)_*\TC(\ell^{\mathrm{B}\ZZ}) \arrow{d}\\
\mathrm{E}_1=V(2)_*\TC(\ell)^{\mathrm{B}\ZZ}[v_2] \arrow[Rightarrow]{r}
&  V(1)_*\TC(\ell)^{\mathrm{B}\ZZ}. 
\end{tikzcd}
\]

When analyzing these Bockstein spectral sequences, we emphasize again that the construction (see \Cref{rec:lambda1}) of elements $\lambda_1,\lambda_2 \in V(1)_*\TC(\ell)$ equips both $V(1)_*\TC(\ell^{\mathrm{B}\Z})$ and $V(1)_*\TC(\ell)^{\mathrm{B}\Z}$ with $\mathbb{F}_p[v_2]\langle \lambda_1,\lambda_2 \rangle$-module structure.  Both $v_2$-Bockstein spectral sequences, and the map between them, respect $\mathbb{F}_p \langle \lambda_1,\lambda_2 \rangle$-module structure.

Examining
$V(2)_*\TC(\ell^{\mathrm{B}\Z})$, as determined by \Cref{thm:v2tcell}, we see that this $\mathbb{F}_p \langle \lambda_1,\lambda_2 \rangle$-module is concentrated in degrees $-2 \le * \le 2p^2+2p-2$. It follows immediately that the only potential $v_2$-Bockstein differentials are $d_1$ differentials.  In fact, the $d_1$ differential is also trivial, as is forced by $\mathbb{F}_p\langle \lambda_1,\lambda_2 \rangle$-module structure.  Specifically, all classes of degrees larger than $2p^2-2$ are multiples of $\lambda_2$, and all classes of degrees larger than $2p^2-2p-1$ are multiples of either $\lambda_1$ or $\lambda_2$. On the other hand, no element of degree less than $2p-3$ is a multiple of $\lambda_1$, and no element of degree less than $2p^2-3$ is a multiple of $\lambda_2$.

The $v_2$-Bockstein spectral sequence for $V(1)_*\TC(\ell)^{\mathrm{B}\Z}$ similarly degenerates, because the Bockstein spectral sequence for $V(1)_*\TC(\ell)$ does by the direct computations of Ausoni and Rognes \cite[Theorem 0.3]{ausoni2002algebraic}.

Finally, when analyzing the map of Bockstein spectral sequences, all possible filtration jumps are ruled out by $\Lambda(\lambda_1,\lambda_2)$ module structure.
\end{proof}

To finish the subsection, we prove \Cref{thm:ht2disproof} below, which along with \Cref{thm:height2comptc} immediately implies \Cref{cor:height2comp}.

%
%

\begin{thm}\label{thm:ht2disproof}
    Let $p$ be a prime, and let $\Z$ act on the connective Adams summand $\ell$ and periodic Adams summand $L$ via $\Psi^{1+p}$.  Let $T(2)$ denote the telescope of a type $2$ $p$-local finite complex. Then, for any $k \ge 0$, there is a commuting diagram
      \[ \begin{tikzcd}
    L_{T(2)} K(L^{hp^k\ZZ}) \ar[d] &
    L_{T(2)} K(\ell^{hp^k\ZZ}) \ar[r, "\cong"] \ar[l, "\cong"'] \ar[d] &
    L_{T(2)} \TC(\ell^{hp^k\ZZ}) \ar[d] \\    
    L_{T(2)} K(L)^{hp^k\ZZ}&
    L_{T(2)} K(\ell)^{hp^k\ZZ} \ar[l, "\cong"'] \ar[r, "\cong"] &    
    L_{T(2)}\TC(\ell)^{hp^k\ZZ}
  \end{tikzcd} \]
 where the vertical maps are the coassembly maps. Moreover
 ,  
 the coassembly maps agree with both the cyclotomic completion and $K(2)$-localization maps.
\end{thm}

\begin{proof}
The diagram of equivalences above comes from applying \Cref{thm:TCequalK}, since $L = L_{T(1)}\ell$. By \Cref{lem:coassembly-descent}, to show the left map is an equivalence after cyclotomic completion, it is enough to show that the maps $f_k:L^{hp^k\Z} \to L$, are $L_{T(2)}K\cycl$-covers.

Recall that there is a $\ZZ_p^{\times}$-equivariant equivalence $\KU \cong L_{T(1)}\SP[\omega_{p^{\infty}}^{(1)}]$ \cite[Example 5.11]{carmeli2021chromatic}. This shows that $f_k$ is isomorphic to the map $\SP_{T(1)}[\omega_{p^k}^{(1)}]^{hT_p} \to \SP_{T(1)}[\omega_{p^\infty}^{(1)}]^{hT_p}$
. Since the source is a finite Galois extension of the $T(1)$-local sphere, this map admits a retraction after tensoring with $\KU$ by \Cref{lem:intermediateextn}. Then, by \Cref{cor:cycredshift} and \Cref{rmk:checklocally}, we learn $L^{hp^k\Z}\to L$ is a $L_{T(2)}K\cycl$-cover.

The target of the right coassembly map is $K(2)$-local by
\cite[Theorem 1.3.6]{hahn2022motivic}, so we conclude the result.





\end{proof}

\newpage

\includegraphics[height=23cm]{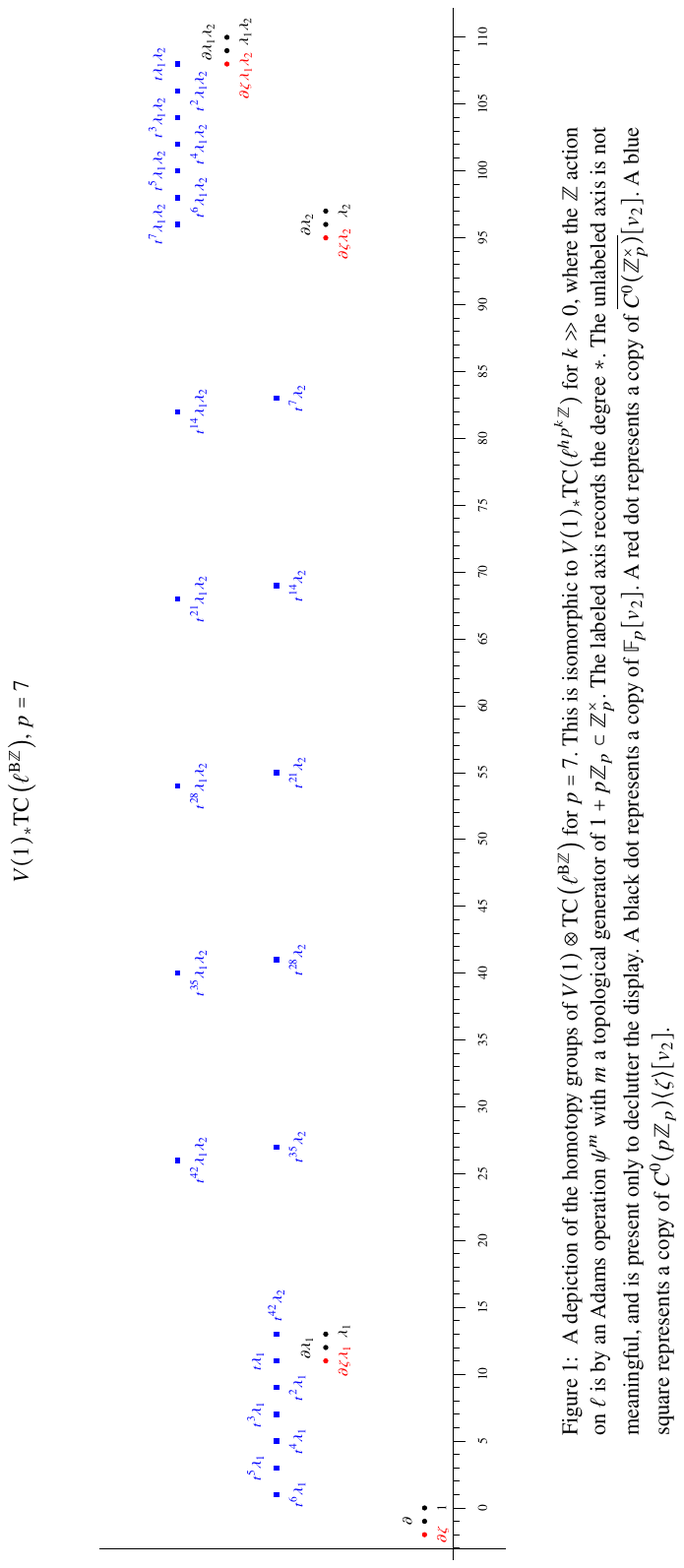}

\subsection{Lichtenbaum--Quillen for the \texorpdfstring{$K(1)$}{K1}-local sphere}
\label{sec:LQht1}

In this subsection, we apply the tools of \Cref{sec:tame} to prove the height $1$ LQ property for $\ell^{hp^k\ZZ}$ for each $k\geq0, p\geq5$, where $\ell$ is the connective Adams summand with $\ZZ$-action generated by the Adams operation $\Psi^{1+p}$. The rings $\ell^{hp^k\ZZ}$ were studied in work of Lee and the third author \cite{lee2023topological}, where a constancy result for their $\THH$ mod $(p,v_1)$ was proven. The Segal conjecture for their $\THH$ was also proven there, which we use here as an ingredient to prove the height $1$ LQ property.
By combining this with the work of the third author \cite{levy2022algebraic}, we show that for $p\geq5$ that the algebraic $K$-theory of the $K(1)$-local sphere is asymptotically $L_2^{f}$-local, in the sense of \Cref{dfn:assLnflocal} below.

The following lemma was communicated to us by Tristan Yang, and will appear in his work on effective LQ properties:
\todo{In this subsection, I would leave the indexing of $n$ and the use of $V$ as it is.}
\begin{lem}\label{lem:assLnflocal}
Let $X$ be a $p$-local spectrum and $n \ge 0$.  The following conditions are equivalent:
 \begin{enumerate}
        \item The map 
        \[X \to L_{n}^{f} X\]
        has bounded above fiber.
        \item For any type at least $n$ finite $p$-local spectrum $V$ with $v_{n+1}$-self map $v$, the map
        \[V \otimes X \to V[v^{-1}] \otimes X\]
        has bounded above fiber.
        \item For any type at least $n+1$ finite $p$-local spectrum $V$, $V\otimes X$ is bounded above.
    \end{enumerate}
\end{lem}

\begin{proof}
$(1)$ implies $(2)$, since $L_n^{f}$-localization is smashing and $L_n^{f} V \simeq V[v^{-1}]$. $(3)$ follows from $(2)$ since $V[v^{-1}]=0$ for $V$ of type at least $n+1$.

We now show that $(3)$ implies $(1)$. It will suffice to show by descending induction on $k$ that $V\otimes F$ is bounded above for each $V$ of type $k$, where $F$ is the fiber of $X \to L_n^fX$. By hypothesis we know the result for $k\geq n+1$. 

Suppose that we know the result for $k+1$ and let $V$ be a finite type $k$ spectrum with $v_k$-self map $v$.  $V[v^{-1}]\otimes F=0$ because $F$ is $L_n^{f}$-acyclic. Then it is enough to show the fiber of $V\otimes F \to V[v^{-1}]\otimes F$ is bounded above. But this fiber is built from extensions, negative suspensions, and filtered colimits from $V/v\otimes F$, which is bounded above by the inductive hypothesis, so we may conclude.
\end{proof}

    

\begin{dfn}\label{dfn:assLnflocal}
    We say that $X$ is \mdef{asymptotically $L_n^{f}$-local} if it satisfies the equivalent conditions of \Cref{lem:assLnflocal}.
\end{dfn}

Note that the above definition can be considered a weakening of Mahowald--Rezk's notion of an fp-type $n$ spectrum \cite{mahowald1999brown}.


We now prove the two theorems of this subsection, beginning with the height $1$ LQ property for the $K(1)$-local sphere:

\begin{thm}\label{thm:LQk1localsphere}
	For a fixed prime $p\geq5$ and type $3$ complex $V$, $V\otimes \ell^{hp^k\ZZ}$ is bounded in the cyclotomic $t$-structure, uniformly for all $k\geq0$.
\end{thm}

\begin{proof}
	
	Fix a prime $p\geq 5$. Let $V = \SP/(p,v_1)$, which is an $\A_2$-ring by \cite{OkaMult}. 

	
	To prove the theorem, it will suffice to verify the conditions of \Cref{prop:Rhz-cyc-bounded} with parameters uniform in $k$ for $X_k=V\otimes\THH(\ell^{hp^k\ZZ})$. Condition (3) follows from \cite[Theorem 8.3]{lee2023topological}, and condition (2) follows from \cite[Theorem 6.1, Remark 4.13]{lee2023topological}
	
	
	By \Cref{thm:AR}, $\ell$ satisfies the height $1$ LQ property. The $\ZZ$-action on $\pi_*\ell/p$ is trivial and in particular unipotent, so the action on $\ell$ is unipotent by \Cref{cor:check-unip-pi}. By applying \Cref{lem:coassemblygenTHH}, we learn that $(X_k)_{|0} \cong V\otimes \THH(\ell)^{hp^k\ZZ}$ verifying condition (1).
	
	
	%
\end{proof}

$\TC(\ell^{h\ZZ})$ for $p>2$ was shown in \cite{levy2022algebraic} to be closely related to the $K$-theory of the $K(1)$-local sphere. By combining this with \Cref{thm:LQk1localsphere}, we learn that the $K$-theory of the $K(1)$-local sphere is asymptotically $L_2^{f}$ local for $p\geq 5$, proving in particular \Cref{thm:kk1intro} in the introduction:

\begin{thm}\label{thm:kk1LQ}
	For $p\geq5$, $K(L_{K(1)}\SP)$ is asymptotically $L_2^{f}$-local.
\end{thm}

\begin{proof}
 The category of asymptotically $L_2^f$-local spectra is clearly a thick subcategory.
	
	By \cite{levy2022algebraic}, there is a cofiber sequence
	
	$$K(\FF_p) \to K(\ell^{h\ZZ}) \to K(L_{K(1)}\SP) $$
	
	and a pullback square
	\begin{center}
		\begin{tikzcd}
			K(\ell^{h\ZZ}) \ar[r]\ar[d] & \TC(\ell^{h\ZZ})\ar[d]\\
			K(\ZZ_p^{h\ZZ})\ar[r] & \TC(\ZZ_p^{h\ZZ})
		\end{tikzcd}
	\end{center}
	
	Since $K(\FF_p)$ $p$-adically is $\ZZ_p$, it is asymptotically $L_2^f$-local. It thus suffices to show that $\TC(\ell^{h\Z})$ as well as $\fib(K(\Z_p^{h\Z}) \to \TC(\Z_p^{h\Z}))$ are  too. 
	By \Cref{thm:LQk1localsphere}, $\TC(\ell^{h\ZZ})$ is asymptotically $L_2^f$-local.

	By \cite[Theorem 4.1]{land2023k}, to show that
	$\fib(K(\Z_p^{h\Z}) \to \TC(\Z_p^{h\Z}))$ is asymptotically $L_2^f$-local, it is equivalent to show that $\fib(K(\Z_p[x]) \to \TC(\Z_p[x]))$ is. But this now follows from \cite[Theorem 1.1,1.2]{clausen2021hyperdescent}.
\end{proof}



%% file: hring.tex
\begin{dfn}\label{dfn:ha2}
  The category of \textit{hcrings} in a symmetric monoidal presentable category $\CC$ is the category of commutative algebras in $h\CC$, the homotopy $1$-category of $\CC$. Similarly, an \textit{hring} is an associative algebra in $h\CC$, and a \textit{h$\A_2$ring} is an $\A_2$-algebra in $h\CC$.
\end{dfn}

\begin{dfn}\label{dfn:hcentral}
  Let $R$ be an h$\A_2$ring in $\CC$.
  A map $z: S \to R$ is \textit{hcentral} if the following diagrams in $\CC$ commute  
  \[ \begin{tikzcd}
    S\otimes R \ar[d, bend right, "z \cdot -"'] \ar[d, bend left, "- \cdot z"] &
    S \otimes R \otimes R \ar[d, "(z \cdot -) \otimes R"] \ar[r, "S\otimes\mu"] &
    S\otimes R \ar[d, "z \cdot -"] &
    S \otimes R \otimes R \ar[d, "R \otimes (- \cdot z)"] \ar[r, "S\otimes \mu"] &
    S\otimes R \ar[d, "- \cdot z"] \\
    R &
    R \otimes R \ar[r, "\mu"] &
    R &
    R \otimes R \ar[r, "\mu"] &
    R.
  \end{tikzcd} \]
  We say $z \in \pi_jR$ is hcentral if it is as a map $\Sigma^{j}\one \to R$.
\end{dfn}

\begin{rmk} \label{rmk:restrict-hcentral}
  Given an hcentral $X \to R$ and a map $Y \to X$
  the composite $Y \to X \to R$ is an hcentral as well.
  As a consequence of this,
  If $i : R_1 \to R_2$ is an hcentral ring map
  then the elements of $\pi_*R_1$ naturally land in the
  subring of hcentral elements of $\pi_*R_2$.
\end{rmk}

\begin{rmk}
  A map of hcrings (or of hrings) is the same data as a map of the underlying h$\A_2$rings.
  For this reason we will sometimes abuse notation, refer to these as simply ``ring maps''.
\end{rmk}

\begin{lem} \label{lem:moore-hcentral}
  Let $R$ be an h$\A_2$ring in a stable, exactly symmetric monoidal category $\CC$.
  If $p^k$ acts by zero on $R$ and $k \geq m_p^{\A_2}$, 
  then there is an hcentral ring map
  \[ \one_\CC/p^{2k} \to R. \]
\end{lem}

\begin{proof}
    Let $R'$ be a h$\A_2$-ring with $p^k=0$. Consider the diagram
  \[ \begin{tikzcd}
    R' \ar[r, "p^{2k}"] \ar[d, "p^k"] & R' \ar[d, equal] \ar[r] & R'/p^{2k} \ar[d] \\
    R' \ar[r, "p^k"] & R' \ar[r] & R'/p^k.
  \end{tikzcd} \]
  After picking a nullhomotopy of the endomorphism $p^k$ of $R'$
  we can use this nullhomotopy to give splittings
  $ R'/p^{2k} \cong R' \oplus \Sigma R'$ and $R'/p^k \cong R' \oplus \Sigma R'$
  which are compatible in the sense that the reduction map
  $R'/p^{2k} \to R'/p^k$ factors as
  \[ R'\oplus \Sigma R' \twoheadrightarrow R' \hookrightarrow R' \oplus \Sigma R'. \]
	
  Using the fact that $p^k =0$ in $R$ we construct a unital map
  $i : \one/p^k \to R$.
  In order to prove the lemma it will suffice to show that the composite map
  $\one/p^{2k} \to \one/p^k \to R $
  is hcentral and a ring map.
  For this we must check that each of the maps
  \begin{align*}
    \one/p^k \otimes R
    &\xrightarrow{x \otimes y \mapsto (i(x) \cdot y) - (y \cdot i(x))}
    R \\
    \one/p^k \otimes R \otimes R
    &\xrightarrow{x \otimes y \otimes z \mapsto ((i(x) \cdot y) \cdot z) - (i(x) \cdot (y \cdot z))} 
    R \\
    R \otimes R \otimes \one/p^k 
    &\xrightarrow{y \otimes z \otimes x \mapsto (y \cdot (z \cdot i(x))) - ((y \cdot z) \cdot i(x))} 
    R \\
    \one/p^k \otimes \one/p^k
    &\xrightarrow{x \otimes y \mapsto (i(x) \cdot i(y)) - i(x \cdot y)} 
    R
  \end{align*}
  becomes null upon precomposing with $r : \Ss/p^{2k} \to \Ss/p^k$.
  
  Using the splittings proven above, applied to $R'=R$ in the first three equations
  and applied to $R'=\Ss/p^k$ for the fourth (since $k\geq m_p^{\A_2}$), 
  and the fact that $i$ is unital we can rewrite these maps as follows
  \begin{align*}
    R &\oplus \Sigma R &
    &\hspace{-45pt}\xrightarrow{(0, a)} R & 
    (R \otimes R) &\oplus \Sigma (R \otimes R) &
    &\hspace{-45pt}\xrightarrow{(0,b)} R \\
    (R \otimes R) &\oplus \Sigma (R \otimes R) &
    &\hspace{-45pt}\xrightarrow{(0,c)} R &
    \one/p^k &\oplus \Sigma \one/p^k &
    &\hspace{-45pt}\xrightarrow{(0,d)} R 
  \end{align*}
  The factorization of the reduction map through the inclusion of the first summand
  now completes the proof of the lemma.
\end{proof}

%% file: acompact.tex
It will be useful for us to have a finiteness condition
weaker than compactness which is sensitive to the choice of a $t$-structure on a stable category. For this reason we introduce the notion of \emph{almost compact objects}.

\begin{dfn} \label{dfn:almost-compact}
  Let $\CC$ be a stable presentable category equipped with a $t$-structure.
  Given an $X \in \CC$ which is bounded below we say that $X$ is \emph{almost compact}
  if for any filtered diagram $F : K \to \CC$ such that
  the set of objects $\{ F(k) \}_{k\in K}$ are uniformly bounded in the range $[c,b]$,
  the assembly map
  \[ \colim_{k \in K} \Map_\CC(X, F(k)) \to \Map_{\CC}(X, \colim_{k \in K} F(k)) \]
  is an isomorphism.
\end{dfn}

\begin{exm}\label{exm:ac}
  A (not necessarily $p$-complete) spectrum $X$ is almost compact iff
  it is bounded below and $\pi_kX$ is finitely generated for all $k$. A $p$-complete spectrum $X$ is almost compact under the additional condition that $\pi_k$ is $p$-nil for each $k$.
\end{exm}

As the filtered colimits appearing in the definition of almost compact are a subset of those appearing in the definition of compact we have the following lemma:

\begin{lem}
  Let $\CC$ be a stable presentable category equipped with a $t$-structure.
  If an object $X \in \CC$ is compact, then it is almost compact.
\end{lem}

Almost compact objects enjoy stronger closure properties than compact objects.
Specifically, they are closed under geometric realizations which are uniformly bounded below.

\begin{lem}\label{lem:closure-of-ac}
  Let $\CC$ be a stable presentable category equipped with a $t$-structure.
  The full subcategory of almost compact objects is closed under finite (co)limits
  and geometric realizations of simplicial diagrams which are uniformly bounded below.
\end{lem}

\begin{proof}
  If $X$ almost compact, then by using suspensions of the diagram we map out to we can upgrade the isomorphism of mapping spaces in the definition of almost compact to an isomorphism of mapping spectra. Exactness of the mapping spectrum then makes it clear that the collection of almost compact objects is closed under finite (co)limits.

  For closure under geometric realizations that are uniformly bounded bounded below we follow the argument in \cite[Prop. C.6.4.4]{SAG}.\footnote{For a discussion of why we cannot simply cite this proposition see \Cref{rmk:Lurie-is-wrong} below.}
  Without loss of generality let $X_\bullet$ be a simplicial object in $\CC_{\geq 0}$
  and let $K \to \CC$ be a filtered diagram uniformly bounded in the range $[0,a]$.
  Then, we have isomorphisms
  \begin{align*}
    &\colim_{k \in K} \Map_{\CC} \left( \colim_{\bullet \in \Delta^\op} X_\bullet , k \right)
    \cong \colim_{k \in K} \lim_{\bullet \in \Delta} \Map_{\CC}( X_\bullet , k) 
    \cong \colim_{k \in K} \lim_{\bullet \in \Delta_{s, \leq a+1}} \Map_{\CC}( X_\bullet , k) \\
    &\cong \lim_{\bullet \in \Delta_{s, \leq a+1}} \colim_{k \in K} \Map_{\CC}( X_\bullet , k) 
    \cong \lim_{\bullet \in \Delta} \colim_{k \in K} \Map_{\CC}( X_\bullet , k)
    \cong \lim_{\bullet \in \Delta} \Map_{\CC} \left( X_\bullet ,\colim_{k \in K} k \right) \\
    &\cong \Map_{\CC} \left( \colim_{\Delta^\op} X_\bullet ,\colim_{k \in K} k \right).
  \end{align*}
  The first step pulls the colimit out.
  The second step uses the fact that $X_\bullet$ is $\geq 0$ and $k$ is $\leq a$ to
  conclude that $\Map_{\CC}( X_\bullet , k)$ is an $a$-truncated space
  and therefore we may replace the limit over $\Delta$ by a limit over $\Delta_{s, \leq a+1}$,
  which is a finite diagram.
  The third step uses the fact that $\Delta_{s, \leq a+1}$ is finite and $K$ is filtered to exchange the limit and colimit.
  The fourth step uses the fact that colimits in space preserve $a$-truncated objects
  and so that we may replace the limit over $\Delta_{s, \leq a+1}$ with a 
  limit over $\Delta$.
  The fifth step uses the hypothesis that each $X_\bullet$ is almost compact.
  The final step pulls the limit over $\Delta$ back inside.
\end{proof}

\begin{lem}\label{lem:thhac}
	Suppose that $R \in \Alg(\Sp_{\geq0})$ and $R/p$ is bounded below with finitely generated homotopy groups in each degree. Then $\THH(R)\otimes V$ is almost compact in $\Sp$ for any $V\in \Sp$ that is almost compact.
\end{lem}

\begin{proof}
    We first claim that the tensor product of $R$ with any almost compact object $V \in \Sp$ is almost compact. Our assumption on $R$ along with \Cref{exm:ac} implies that for all $k\geq 0$, $R/p^k$ is almost compact. By choosing a cell decomposition of $V$ made up of increasingly suspended copies $\SP/p^k$, we see that $V\otimes R$ is almost compact by \Cref{exm:ac} again. 
    
    Next, we observe that by induction, $V\otimes R^{\otimes n+1}$ is almost compact for all $n$. Since it is also uniformly bounded below (using connectivity of $R$), we learn by applying \Cref{lem:closure-of-ac} that $V\otimes \THH(R)$ is almost compact.
\end{proof}

\begin{exm} \label{exm:ac-zp-modules}
  Let $\Mod(\Z)_p$ be the category of $p$-complete $\Z$-modules
  with its standard $t$-structure.
  \begin{itemize}
  \item $\Z/p$ is compact and therefore almost compact.
  \item $\Z_p$ is neither compact nor almost compact.
  \item $\oplus_{i \geq 0} \Sigma^i \Z/p$ is almost compact, but not compact.\hfill\qedhere
  \end{itemize}
\end{exm}

\begin{rmk}\label{rmk:Lurie-is-wrong}
  In \cite[Dfn. C.6.4.1]{SAG}, Lurie defines an object $X$ in a presentable category $\DD$
  to be almost compact if $\tau_{\leq n}X$ is a compact object of $\tau_{\leq n}\DD$.
  In order to compare our definition with Lurie's it is useful to note that
  Lurie only applies this definition when $\DD = \CC_{\geq 0}$ is the connective objects of a
  $t$-structure on a stable category $\CC$
  and filtered colimits preserve coconnectivity in $\CC$.
  In this situation our definitions agree.

  Without the assumption that filtered colimits in $\CC$ preserve coconnectivity
  these definitions are not equivalent.
  As an example, note in \Cref{exm:ac-zp-modules}
  $\Z/p$ is compact and almost compact (according to \Cref{dfn:almost-compact}),
  but that $\Z/p \in (\Mod(\Z)_p)_{\geq 0}$ is not almost compact
  (according to Lurie's definition).
    
  This subtlety is not an idle curiosity and is indeed relevant to the present paper:
  Filtered colimits do not preserve coconnectivity in the cyclotomic $t$-structure.
\end{rmk}

\begin{dfn} \label{dfn:t-amplitude}
  Let $\CC$ and $\DD$ be stable categories equipped with $t$-structures
  and let $F: \CC \to \DD$ be an exact functor between them.

  We say that $F$ has \emph{$t$-amplitude $\geq a$} if
  $F(\CC_{\geq 0}) \subseteq \DD_{\geq a}$.
  We say that $F$ has \emph{$t$-amplitude $\leq b$} if
  $F(\CC_{\leq 0}) \subseteq \DD_{\leq b}$.
  If $F$ has $t$-amplitude $\geq a$ and $\leq b$, then we say 
  that $F$ has \emph{$t$-amplitude in the range $[a,b]$}. 
\end{dfn}

\begin{rmk} \label{rmk:spread-t-amplitude}
  Exactness implies that if a functor $F : \CC \to \DD$ has $t$-amplitude $[a_0,a_1]$,
  then for every interval $[b_0,b_1]$ we have
  $F(\CC_{[b_0,b_1]}) \subseteq \DD_{[b_0+a_0,b_1+a_1]}$.
\end{rmk}


\begin{lem} \label{lem:map-formula2} 
  Let $\CC$ be a stable, presentably symmetric monoidal category
  with a compatible $t$-structre and
  let $R_\bullet : I \to \CAlg(\CC)$ be a filtered diagram of commutative algebras in $\CC$
  with uniformly bounded $t$-amplitude and colimit $\on_{\CC}$.
  \begin{enumerate}
  \item Given an almost compact object $X \in \CC$ and a bounded object $Y \in \CC$
    there is an isomorphism
    \[ \colim_{\alpha \in I} \Map_{\Mod(\CC; R_\alpha)}(R_\alpha \otimes X, R_\alpha \otimes Y)
    \cong \Map_{\CC}(X,Y). \]
  \item Given a bounded, almost compact object $X \in \CC$ and an $\alpha \in I$,
    if $Y$ is a retract of $R_\alpha \otimes X$ in $\Mod(\CC;R_\alpha)$,
    then there is an arrow $\alpha \to \beta \in I$ and an isomorphism of
    $R_\beta$-modules
    \[ R_\beta \otimes_{R_\alpha} Y \cong R_\beta \otimes (\on \otimes_{R_\alpha} Y). \]
  \end{enumerate}  
\end{lem}

\begin{proof}  
  Part (1).
  Using the assumptions that $X$ is almost compact and
  the $R_\alpha \otimes Y$ are uniformly bounded we have isomorphisms
  \begin{align*}
    \colim_{\alpha \in I} &\Map_{\Mod(\CC; R_\alpha)}(R_\alpha \otimes X, R_\alpha \otimes Y)
    \cong \colim_{\alpha \in I} \Map_{\CC}(X, R_\alpha \otimes Y) \\
    &\cong \Map_{\CC}(X, \colim_{\alpha \in I}  R_\alpha \otimes Y)
    \cong \Map_{\CC}(X, Y). 
  \end{align*}

  Part (2).
  Let $e : R_\alpha \otimes X \to R_\alpha \otimes X$ be an idempotent map of $R_\alpha$-modules  such that $(R_\alpha \otimes X)[e^{-1}] \cong Y$. Then after base change to $\on$ we obtain an idempotent $e_\infty$ such that $X[e_\infty^{-1}] \cong \on \otimes_{R_\alpha} Y$. Our goal is now to compare the two idempotents $e$ and $1 \otimes e_\infty$ which agree after base change to $\on$. Using the assumption that $X$ is bounded and almost compact it follows from part (1) that $e$ and $1 \otimes e_\infty$ agree after base change to $R_\beta$ for some arrow $\alpha \to \beta \in I$.
\end{proof}

%% file: loc-unip-basic.tex
In this appendix we study \emph{locally unipotent} $\ZZ$-actions.
The material contained here is used heavily in \Cref{sec:tame}.
Unlike in the body of the paper we work integrally in this appendix.


\begin{ntn}
  Throughout this appendix if an object $X$ has a $\Z$-action
  we will write $\psi \colon X \to X$ for the automorphism
  associated to the generator $1 \in \Z$.
\end{ntn}

\begin{dfn}
  Let $\CC$ be a stable, presentable category.
  We define $\CC^{B\Z, u} \subseteq \CC^{B\Z}$ to be the localizing subcatgory
  of $\CC^{B\Z}$ generated by objects with trivial action.
  We also write $\iota : \CC^{B\Z, u} \to \CC^{B\Z}$ for the defining inclusion
  and $(-)^u$ for its right adjoint.  
\end{dfn}

\begin{lem} \label{lem:tensor-unip}
  Let $\CC$ be a stable, presentably symmetric monoidal category.
  The subcategory $\CC^{B\Z, u} \subseteq \CC^{B\Z}$ is closed under the tensor product.
\end{lem}

\begin{proof}
  As the tensor product on $\CC$ commutes with colimits seperately in each variable,
  it suffices to observe that if $X,Y \in \CC^{B\Z}$ have a trivial $\Z$-action,
  then $X \otimes Y \in \CC^{B\Z}$ also has a trivial $\Z$-action.
\end{proof}

\begin{rmk} \label{rmk:sm-unip}
  As a consequence of \Cref{lem:tensor-unip}, when $\CC$ is symmetric monoidal
  we may equip $\CC^{B\Z,u}$ and the defining inclusion $\iota$
  with symmetric monoidal structures by restriction.
  This in turn equips $(-)^u$ with a lax symmetric monoidal structure.
\end{rmk}



Before we proceed further with the general theory it will be profitable to study the case
$\CC = \Sp$ in more detail.

\begin{cnstr} \label{cnstr:sp-triv-action}
  Let $\Ss[t^{\pm1}]$ denote the spherical group ring of $\Z$,
  let $i$ denote the commutative algebra map $\Ss[t^{\pm1}] \to \Ss$
  obtained from the map $\Z \to *$ and
  let $j : \Ss[t^{\pm1}] \to \Ss[t^{\pm1}, (t-1)^{-1}]$
  denote the localization of $\Ss[t^{\pm1}]$ obtained by inverting $(t-1)$.  
  Using $j$ we construct a recollement
  \[ \begin{tikzcd}[column sep=huge]
    \Mod(\Ss[t^{\pm1}])^{\wedge}_{(t-1)}
    \ar[r, bend right]
    \ar[r, bend left, hook] &
    \Mod(\Ss[t^{\pm1}])
    \ar[r, bend right]
    \ar[r, bend left, "j^*"]
    \ar[l, "(-)^{\wedge}_{(t-1)}"] &
    \Mod(\Ss[t^{\pm1}][(t-1)^{-1}])
    \ar[l, , "j_*", hook]    
  \end{tikzcd} \]
  based on inverting and completing at $(t-1)$.
  
  Writing the $\Ss[t^{\pm1}]$-module $\Ss$ as $\Ss[t^{\pm1}]/(t-1)$
  we find that the functor $i_*$ (restriction of scalars along $i$)
  factors through the left adjoint to $(t-1)$-adic completion.
  We can then describe
  $ \Mod(\Ss[t^{\pm1}])^{\wedge}_{(t-1)} $
  as the localizing subcategory of 
  $ \Mod(\Ss[t^{\pm1}]) $
  generated under colimits by the image of $i_*$.
  
  If we identify $\Sp^{B\Z}$ with $\Ss[t^{\pm1}]$-modules,
  then we obtain the following identifications:
  \begin{itemize}
  \item The trivial action functor $(-)^{\triv} : \Sp \to \Sp^{B\Z}$
    can be identified with $i_*$.
  \item The right adjoint to $i_*$ can be identified with $(-)^{h\Z}$.
  \item The subcategory of $(t-1)$-nilpotent $\Ss[t^{\pm1}]$-modules
    (the image of the left adjoint to $(t-1)$-adic completion)
    can be identified with $\Sp^{B\Z,u}$.
  \end{itemize}  
  Note that since $\Sp^{B\Z,u}$ is generated under colimits by the image of
  the trivial action functor,
  its right adjiont $(-)^{h\Z} : \Sp^{B\Z,u} \to \Sp$ is conservative.
\end{cnstr}

\begin{lem} \label{lem:unip-affine} \label{lem:basechangepZ}
  There is a symmetric monoidal isomorphism
  \[ \Mod(\Ss^{B\Z}) \cong \Sp^{B\Z,u} \]
  under which
  $\iota$ is given by the formula $X \mapsto \Ss \otimes_{\Ss^{B\Z}} X$
  and $(-)^u$ is given by $Y \mapsto Y^{h\Z}$.
  This identification is natural in restriction along the maps
  $n\Z \subseteq \Z$.
\end{lem}

\begin{proof}
  We begin by considering the cocontinuous symmetric monoidal functor
  $ (-)^{\triv} : \Sp \to \Sp^{B\Z,u} $
  and its conservative right adjiont $(-)^{h\Z}$ from \Cref{cnstr:sp-triv-action}.
  As $(-)^{h\Z}$ is a finite limit, this right adjoint is cocontinuous.
  It now follows from \cite[Prop. A.4]{burklund2022galois} that
  the induced adjunction
  \[ \Mod(\Sp; \Ss^{B\Z}) \to \Sp^{B\Z,u} \]
  is the desired symmetric monoidal equivalence.
  Naturality in endomorphisms of $\Z$ follows by
  functoriality of the construction of the comparison functor.
\end{proof}

We now return to the general case.
The next lemma allows us to extend most results from the case
$\CC=\Sp$ to the general case using base-change.

\begin{lem} \label{lem:base-change-unip}
  Let $\CC$ be a stable presentable category.
  There is a natural identification of $\CC^{B\Z,u} \subseteq \CC^{B\Z}$ with
  $ \left( \Sp^{B\Z,u} \subseteq \Sp^{B\Z} \right) \otimes \CC $.  
\end{lem}

\begin{proof}
  Given a spectrum with a $\Z$-action and an object of $\CC$
  their tensor product naturally lives in $\CC^{B\Z}$ and so we obtain a comparison functor
  \[ \Sp^{B\Z} \otimes \CC \to \CC^{B\Z}. \]
  Using the fact that
  $\PrL$ is ambidextrous for all spaces-shaped colimits
  and the Lurie tensor product commutes with colimits seperately in each variable
  we can verify that the map above is an isomorphism
  via the chain of isomorphisms
  \[ \CC^{B\Z} \cong \CC_{B\Z} \cong (\Sp \otimes \CC)_{B\Z} \cong \Sp_{B\Z} \otimes \CC \cong \Sp^{B\Z} \otimes \CC. \]
  
  As fully faithful left adjoints are closed under basechange in $\PrL$
  (this follows for example from \cite[Proposition 4.8.1.17]{HA})
  it will now suffice for us to identify the two full subcategories
  $\CC^{B\Z,u}$ and $\Sp^{B\Z,u} \otimes \CC$ of $\CC$.
  The former is generated under colimits by objects of the form $X^{\triv}$ with $X \in \CC$.
  The latter is generated under colimits by objects of the form $X \otimes Y^{\triv}$
  with $X \in \CC$ and $Y \in \Sp$.
  The conclusion follows.      
\end{proof}

\begin{rmk} \label{rmk:unip-affine}
  In view of \Cref{lem:base-change-unip}
  tensoring \Cref{lem:unip-affine} with $\CC$ gives 
  a version of this identification for a general category.
\end{rmk}

\begin{prop} \label{prop:gluing-seq}
  Let $\CC$ be a stable presentable category.
  The $\iota \dashv (-)^u$ adjunction extends to a recollement
  \[ \begin{tikzcd}[column sep=huge]
    \CC^{B\Z,u}
    \ar[r, bend left, hook, "\iota"]
    \ar[r, bend right] &
    \CC^{B\Z}
    \ar[l, "(-)^u"]
    \ar[r, bend left, "j_!"]
    \ar[r, bend right, "j_*"] &
    \Mod(\Ss[t^{\pm1}][(t-1)^{-1}]) \otimes \CC 
    \ar[l, hook, "j^*"] 
  \end{tikzcd} \]
  whose gluing sequences take the form
  \[ \iota(X^u) \to X \to X[(\psi - 1)^{-1}]. \]
  In particular, $X^u = 0$ if and only if the map $\psi - 1 : X \to X$ is an isomorphism.
\end{prop}

\begin{proof}
  The case of $\Sp = \CC$ follows from \Cref{cnstr:sp-triv-action}.
  The general case now follows by tensoring the case of $\Sp$ with $\CC$ by
  \Cref{lem:base-change-unip}.
\end{proof}







We end the section with a collection of lemmas useful for checking whether
a $\Z$-action is locally unipotent, and examples.

\begin{exm} \label{exm:loc-nilp-heart}  
  Given an $M \in \Mod(\Z)^\heartsuit$ with a $\Z$-action
  we can read off from the gluing sequence in \Cref{prop:gluing-seq}
  that the action on $M$ is locally unipotent if and only if
  for all $m \in M$ we have $(\psi - 1)^{\circ N}(m) = 0$ for $N \gg 0$.
\end{exm}

\begin{lem} \label{lem:check-unip}
  Let $F : \CC \to \DD$ be a filtered colimit preserving, additive functor
  between stable presentable categories.
  $F$ preserves local unipotence and commutes with
  the functors $\iota$, $(-)^u$ and $(-)[(\psi-1)^{-1}]$.  
  If $F$ is conservative, then it detects local unipotence.
\end{lem}

\begin{proof}
  As $F$ is additive and preserves filtered colimits we have
  \[ F(X[(\psi-1)^{-1}]) \cong F(X)[(\psi-1)^{-1}]. \]
  The conclusion now follows from the fact the that gluing sequences in $\CC^{B\Z}$
  from \Cref{prop:gluing-seq} are sent to the corresponding gluing sequences in $\DD^{B\Z}$.
\end{proof}

\begin{cor} \label{cor:check-unip-pi}
  Let $\CC$ be a presentable stable category with a compact generator $V \in \CC$.
  An $X \in \CC^{B\Z}$ is locally unipotent
  if and only if the $\Z$-action on the homotopy groups
  $\pi_0(\Map_{\CC}(\Sigma^k V,X))$ is locally unipotent for all $k$.
\end{cor}
\todo{comment: You can apply this to modules over an E1 algebra, the p complete modules if you use unit mod a power of p and Tn local modules if you use a type n guy.}

\begin{proof}
  The functor $\pi_0(\Map_{\CC}(\Sigma^* V,-)) : \CC \to \Mod(\Z)^{\Gr}$
  is additive, filtered colimit preserving (since $V$ is compact)
  and conservative (since $V$ generates $\CC$).
\end{proof}



\begin{cor} \label{lem:check-loc-nilp}
  A $\Z$-action on an object $X \in \pgnsp{\langle p \rangle}$ is locally unipotent
  if and only if the associated $\Z$-actions on
  $X^{\Phi C_p}$ and $X^{\Phi e}$ are locally unipotent.  
\end{cor}

\begin{proof}
  The functor $(\Phi C_p, \Phi e) : \pgnsp{\langle p \rangle} \to \Sp \times \Sp$
  is conservative and colimit preserving.
\end{proof}

\begin{cor}\label{cor:cover-unip}
  Suppose $\CC$ is equipped with a $t$-structure compatible with filtered colimits.
  The truncation functors $\tau_{\geq 0}$ and $\tau_{\leq 0}$ preserve
  local unipotence.
\end{cor}

\begin{proof}
  $\tau_{\geq 0}$ and $\tau_{\leq 0}$ are additive and commute with filtered colimits (by hypothesis).
\end{proof}

\begin{lem} \label{lem:cover-unip-cplt}
  Suppose $\CC$ is equipped with a $t$-structure compatible with filtered colimits.
  If $X \in \CC^{B\Z,u}_p$, then $\tau_{\geq 0}X \in \CC^{B\Z,u}_p$ as well.
\end{lem}

\begin{proof}
  From \Cref{prop:gluing-seq} we know that
  an object in $X \in \CC^{B\Z}_p$ has a locally unipotent action
  iff the $p$-completion of $X[(\psi-1)^{-1}]$ is zero
  (here the colimit is evaluated in $\CC$).
  Equivalently, this is the same as saying that $X[(\psi-1)^{-1}]$ is a $\Q$-module.

  As the functor $\tau_{\geq0}(-)$ commutes with filtered colimits by assumption
  it suffices to observe that $\tau_{\geq0}(-)$ is additive
  and therefore sends $\Q$-modules to $\Q$-modules.  
\end{proof}

\begin{lem} \label{lem:res-u}
  Let $\CC$ be a stable, $p$-complete, presentable category.
  The functor $(-)^u$ commutes with restriction to $p\Z \subset \Z$.
  In particular, an $X \in \CC^{B\Z}$ has a locally unipotent action if and only if
  the action is locally unipotent after restricting to $p\Z \subseteq \Z$.
\end{lem}

\begin{proof}
  Using the $p$-completeness assumption we have an isomorphism
  \[ X[(\psi^{p} - 1)^{-1}] \cong X[(\psi - 1)^{-1}] \]
  since $(\psi - 1)^p = \psi^p - 1 + p\phi$ for some endomorphism $\phi$.
  The conclusion now follows from \Cref{prop:gluing-seq}.  
\end{proof}

\begin{lem} \label{lem:colim-unip}
  Let $\CC$ be a stable, $p$-complete, presentable category.
  The natural map $ \colim_k X^{hp^{k}\ZZ} \to X^u $ is an isomorphism.
  In particular, an $X \in \CC^{B\Z}$ has a locally unipotent action
  if and only if the natural map $ \colim_k X^{hp^{k}\ZZ} \to X $ is an isomorphism.
\end{lem}

\begin{proof}
  From \Cref{lem:res-u} we know that $(-)^u$ does not change upon restriction to
  the subgroups $p^k\Z$. The natural map is now obtained as the colimit of the maps
  $X^{hp^k\Z} \to X^u$.
  
  Recall that $(\colim_k \Ss^{Bp^k\Z})_p \cong \Ss_p$.
  Using the $p$-completeness assumption on $\CC$,
  the formula for $(-)^u$ from \Cref{rmk:unip-affine}
  and \Cref{lem:res-u} we obtain isomorphisms
  \begin{align*}
    \colim_k X^{hp^k\Z}
    &\cong \Ss_p \otimes_{\Ss_p} (\colim_k X^{hp^k\Z})
    \cong (\colim_k \Ss_p) \otimes_{\colim_k \Ss_p^{Bp^k\Z}} (\colim_k X^{hp^k\Z}) \\
    &\cong \colim_k \left( \Ss_p \otimes_{\Ss_p^{Bp^k\Z}} X^{hp^k\Z} \right)
    \cong \colim_k X^u 
    \cong X^u.\qedhere
  \end{align*}    
\end{proof}

\begin{lem} \label{lem:trivialize-action-cpt}
  Let $\CC$ be a stable, $p$-complete, presentable category.
  Suppose we are given a compact object $X \in \CC$
  with a locally unipotent $\Z$-action.
  For $k \gg 0$ the action of $p^k\Z$ on $X$ is trivializable.  
\end{lem}

\begin{proof}
  The action of $p^k\Z$ on $X$ is trivializable
  iff the natural map $X^{hp^k\Z} \to X$ admits a section.
  From \Cref{lem:colim-unip} we have a colimit diagram
  \[ X^{h\Z} \to X^{hp\Z} \to X^{hp^2\Z} \to \cdots \to X. \]
  Now the hypotheses imply that the identity map $X \to X$
  factors through a finite stage of this diagram.  
\end{proof}

\begin{lem} \label{lem:trivialize-action-ac}
  Let $\CC$ be a stable, $p$-complete, presentable category
  equipped with a $t$-structure.
  Suppose we are given an $X \in \CC$
  which is both bounded and almost compact
  and a locally unipotent $\Z$-action on $X$.    
  For $k \gg 0$ the action of $p^k\Z$ on $X$ is trivializable.  
\end{lem}

\begin{proof}
  The proof is the same as the proof of 
  \Cref{lem:trivialize-action-cpt},
  but using (uniform) boundedness of the $X^{hp^k\Z}$
  and almost compactness of $X$ in place of compactness of $X$.
\end{proof}

Finally, we discuss the unipotent algebra that corresponds to the ring map $\SP_p^{B\Z}\to \SP_p$ coming from the map $*\to B\Z$ via the equivalence $\Mod_{\Sp_p}(\SP_p^{B\Z})\cong\Sp_p^{B\Z,u}$ of \Cref{rmk:unip-affine}:

\begin{prop}\label{prop:unipotent_SW}
  Let $\W(\LCF{\overrightarrow{\ZZ_p}}) \in \CAlg(\Sp_p^{B\ZZ})$
  be the spherical Witt vectors of $\LCF{\Z_p}$
  endowed with the $\ZZ$-action coming from the map 
  $+1 : \ZZ_p \to \ZZ_p$.
  Let $\CC$ be a stable $p$-complete presentable category.
  \begin{enumerate} 
  \item The $\Z$-action on $\W(\LCF{\overrightarrow{\ZZ_p}})$ is locally unipotent.
  \item The unit map $\Ss_p \to \W(\LCF{\overrightarrow{\ZZ_p}})^{h\ZZ}$ is an isomorphism.
  \item There is an symmetric monoidal isomorphism
    \[ \CC \cong \Mod(\CC^{B\ZZ,u}; \W(\LCF{\overrightarrow{\ZZ_p}})) \]
    given by sending $X$ to $\W(\LCF{\overrightarrow{\ZZ_p}}) \otimes X $
    with inverse given by $Y \mapsto Y^{h\ZZ}$.
  \item For each $a \in \Z_p$ there is a isomorphism of symmetric monoidal functors between the functor
    $(-)^{h\Z}$ and the functor $(-)_{|a}$ given by the composite
    \[ \Mod(\Sp_p^{B\ZZ,u}; \W(\LCF{\overrightarrow{\ZZ_p}})) \to \Mod(\W(\LCF{\ZZ_p})) \xrightarrow{(-)_{|a}} \Sp_p. \]
  \end{enumerate}
  If $\CC$ is presentably symmetric monoidal,
  then the isomorphisms in (3) and (4) can be made symmetric monoidal.
\end{prop}

\begin{proof}
  Since $\W(-)$ and $\LCF-$ commute with colimits, we have isomorphisms
  $\W(\LCF{\overrightarrow{\ZZ_p}}) \simeq \colim_k \W(\LCF{\overrightarrow{\ZZ/p^k}})$.
  The $+1$ action on $\W(\LCF{\ZZ/p^k})$ is trivial
  after restricting to $p^k\Z \subseteq \Z$ is thus unipotent by \Cref{lem:res-u}. 
  Passing to the colimit along $k$ we obtain (1).
  
  For (2) note that since both the source and target are $p$-complete and bounded below
  it will suffice to verify the claim after tensoring with $\mathbb{F}_p$.
  The claim now follows from the short exact sequence of abelian groups
  \[0 \to \mathbb{F}_p \to  \LCF{{\ZZ_p}} \xrightarrow{\psi-\mathrm{id}}\LCF{{\ZZ_p}}\to 0.\]

  %

  For claims (3) and (4) the base-change formula from \Cref{lem:base-change-unip}
  and the recollement from \Cref{prop:gluing-seq} reduce to
  proving the symmetric monoidal versions of (3) and (4) for $\CC=\Sp_p$.

  Part (3) follows from (1), (2) and \Cref{rmk:unip-affine} upon taking modules
  over the commutative algebra $\W(\LCF{\overrightarrow{\ZZ_p}})$
  under the symmetric monoidal equivalence
  $\Mod(\SP_p^{B\ZZ}) \cong \Sp_p^{B\Z,u}$.
  Part (4) follows from (3) via the symmetric monoidal isomorphisms
  \[ (-)_{|a} \cong \left( \W(\LCF{\overrightarrow{\ZZ_p}}) \otimes (-)^{h\ZZ} \right)_{|a} \cong (\W(\LCF{\ZZ_p})_{|a} \otimes (-)^{h\ZZ} \cong (-)^{h\Z}. \qedhere\]
\end{proof}

%% file: O-alg.tex
In this subsection, let $\CC$ be a presentably symmetric monoidal stable category with $t$-structure compatible with the symmetric monoidal structure, and let $\cO$ be a unital operad in spaces. Here we prove some connectivity statements that are helpful in studying the deformation theory of $\cO$-algebras. The only result directly used in the main part of the paper is \Cref{lem:trivializeacalg}, which we use to trivialize powers of locally unipotent automorphisms of $\pi$-finite connective $\E_1\otimes \A_2$-algebras.

\begin{lem}\label{lem:tensorsquare}
	Let $A,B \in \CC^{\Delta^1\times \Delta^1}$ be two squares of connective objects such that the total fiber is $k_1+k_2$-connective, the horizontal morphisms are $k_1$-connective, the vertical morphisms are $k_2$-connective. Then $A\otimes B \in \CC^{\Delta^1\times\Delta^1}$ is also of this form.
\end{lem}

\begin{proof}
	We can expand $A\otimes B$ as the composite of squares
	
\[\begin{tikzcd}
	{A_{00}\otimes B_{00}} & {A_{01}\otimes B_{00}} & {A_{01}\otimes B_{01}} \\
	{A_{10}\otimes B_{00}} & {A_{11}\otimes B_{00}} & {A_{11}\otimes B_{01}} \\
	{A_{10}\otimes B_{10}} & {A_{11}\otimes B_{10}} & {A_{11}\otimes B_{11}}
	\arrow[from=1-3, to=2-3]
	\arrow[from=1-2, to=2-2]
	\arrow[from=1-1, to=2-1]
	\arrow[from=2-1, to=3-1]
	\arrow[from=2-2, to=3-2]
	\arrow[from=3-1, to=3-2]
	\arrow[from=2-1, to=2-2]
	\arrow[from=1-1, to=1-2]
	\arrow[from=1-2, to=1-3]
	\arrow[from=2-2, to=2-3]
	\arrow[from=3-2, to=3-3]
	\arrow[from=2-3, to=3-3]
\end{tikzcd}\]

From this it is easy to verify that the total fiber of each square is $k_1+k_2$-connective by the assumptions, and that the horizontal and vertical arrows above are $k_1,k_2$-connective respectively. These then imply the claims about $A\otimes B$.
\end{proof}

The following is a version of the Blakers--Massey theorem for $\cO$-algebras that establishes a stable range:

\begin{lem}\label{lem:oalgblakersmassey}
	Suppose we are given a pushout square	
	\begin{center}
		\begin{tikzcd}
			{A} \ar[r,"b"]\ar[d,"c"] & B \ar[d]\\
			C	\ar[r] & D
		\end{tikzcd}
	\end{center}
	of connective $\cO$-algebras in $\CC$ such that the maps $b,c$ are $k_1,k_2$ connective respectively. Then the natural map $A \to B\times_DC$ is $k_1+k_2$-connective and the map $B\coprod_AC \to D$ is $k_1+k_2+1$-connective, where $\coprod_A$ denotes the pushout along $A$ in $\CC$ (as opposed to $\Alg_{\cO}\CC$).
\end{lem}
\begin{proof}
	We first note that the free $\cO$-algebra functor takes $k$-connective maps of connective objects to $k$-connective maps, and geometric realization commutes with pullbacks and preserve connectedness. Thus, by taking the monadic resolution of the diagram and geometric realizing, we can assume that the diagram is obtained from applying the free functor to a pushout square
	
	\begin{center}
		\begin{tikzcd}
			{W} \ar[r,"x"]\ar[d,"y"] &X\ar[d]\\
			Y\ar[r] & Z
		\end{tikzcd}
	\end{center}  in $\CC_{\geq0}$ with $x,y$ $k_1,k_2$-connective respectively. 
	
	The total fiber of the square then breaks up into a direct sum over $i$ of the functor $\cO(i)\otimes_{\Sigma_i}(-)$ (which is exact and preserves connectivity) applied to the total fibers of the square 
	
	\begin{center}
		\begin{tikzcd}
			{W^{\otimes i}} \ar[r]\ar[d] & X^{\otimes i}\ar[d]\\
			Y^{\otimes i}\ar[r] & Z^{\otimes i}
		\end{tikzcd}
	\end{center}
	
	which are is $k_1+k_2$ by iteratively applying \Cref{lem:tensorsquare}.
	
%
%
%
	
	
	
	The last statement about $k_1+k_2+1$-connectivity follows from this since the cofiber of the map $B\coprod_AC \to D$ is the total cofiber of the square, which is the double suspension of the total fiber.
	\end{proof}
	
	\begin{cor}\label{lem:oalgstablerange}
Let $R$ be a connective $\cO$-algebra in $\CC$. Let $R' \to R$ be an $m$-connective map of $\cO$-algebras for $m\geq 0$.
Consider the pushout square in $\Alg_{\cO}(\CC)$
\begin{center}
	\begin{tikzcd}
		R'\ar[r]\ar[d] & R \ar[d]\\
		R	\ar[r] & SR'
	\end{tikzcd}
\end{center}
Then the induced map $R' \to \omega SR' \coloneqq R \times_{ SR'} R$ is $2m$-connective.

\end{cor}

We next prove a version of \cite[Theorem 7.4.1.23]{HA} that holds for $\cO$-algebras. 

Let $\Alg_{\cO}(\CC)_{/R,[a,b]}$ denote the full subcategory of $\Alg_{\cO}(\CC)_{/R}$ consisting of maps $f:R' \to R$ such that $\fib f$ is $[a,b]$ in the $t$-structure. 

\begin{lem}\label{lem:sqextnoalg}
	Let $m\geq 0$. Suppose $R= \tau_{\leq m}R$ is a connective $\cO$-algebra in $\CC$. For each $a,b$ with $m \leq a, b \leq 2a-1$, the trivial square-zero functor $\Omega^{\infty}:\Sp(\Alg_{\cO}(\CC)_{/R}) \to (\Alg_{\cO}(\CC)_{/R})_*$ restricts to an equivalence onto $(\Alg_{\cO}(\CC)_{/R,[a,b]})_*$ when restricted to those objects whose image lies in that subcategory.
	
	Moreover the functor $$\Alg_{\cO}(\CC)_{/R,[m,2m-1]} \to (\Alg_{\cO}(\CC)_{/R,[m+1,2m]})_{R\coprod R/}$$ sending $R'\to R$ to $\tau_{\leq2m}SR'$ equipped with the two sections from $R$ is an equivalence.
\end{lem}

\begin{proof}
	The functor $\Omega^{\infty}$ when restricted to objects whose image is in $(\Alg_{\cO}(\CC)_{/R,[a,b]})_*$, is the projection of the inverse limit
	
	$$\dots (\Alg_{\cO}(\CC)_{/R,[a+2,b+2]})_*\xrightarrow{\Omega} (\Alg_{\cO}(\CC)_{/R,[a+1,b+1]})_* \xrightarrow{\Omega} (\Alg_{\cO}(\CC)_{/R,[a,b]})_*$$
	
	Note that $[a+1,b+1]$ satisfies the inequalities in the hypothesis if $[a,b]$ does. We claim that the composite  $$\tau_{\leq b+1}\Sigma:(\Alg_{\cO}(\CC)_{/R,[a,b]})_* \to (\Alg_{\cO}(\CC)_{/R,[a,\infty]})_*\xrightarrow{\tau_{\leq b+1}\Sigma} (\Alg_{\cO}(\CC)_{/R,[a+1,b+1]})_*$$ is an inverse to $\Omega$. Indeed, the natural map $\id \to \Omega \tau_{\leq b+1}\Sigma$ is an equivalence by \Cref{lem:oalgstablerange}, so that $\tau_{\leq b+1}\Sigma$ is a fully faithful left adjoint. But its right adjoint $\Omega$ is conservative, so this is an adjoint equivalence.
	
	The second statement is similar, namely it follows from  \Cref{lem:oalgstablerange} that $\omega$ gives an inverse to the functor $\tau_{\leq 2m}S$.
\end{proof}

The following lemma gives us control on the Postnikov tower of a connective $\cO$-algebra.

\begin{lem}\label{lem:oalgfaithful}
	Let $m\geq 0$ and suppose $R= \tau_{\leq m}R$ is a connective $\cO$-algebra in $\CC$. Then the functor $(\Alg_{\cO}(\CC)_{/R,[m+1,m+1]})_* \to \Sigma^{m+1}\CC^{\heart}$ sending an algebra to the fiber of the augmentation is faithful.
\end{lem}

\begin{proof}
	First let $R$ be any connective $\cO$-algebra, and let $R \oplus M \to R$ be a trivial square-zero extension with $M \in \Sigma^{m+1}\CC^{\heart}$.
	
 We claim that the following square becomes a pushout square in $(\Alg_{\cO}(\CC)_{/R})_*$ after applying $\tau_{\leq m+1}$, and that the vertical maps are $m+1$-connective.
	
	\begin{center}
		\begin{tikzcd}
			R\coprod\on\{\fib \epsilon \}\ar[r]\ar[d] &R\coprod\on\{M\} \ar[d,"\epsilon"]\\
			R\ar[r] & R\oplus M
		\end{tikzcd}
	\end{center}
	All functors involved commute with sifted colimits, so by resolving $R$ by free augmented algebras via the monadic resolution, we can reduce to the case of $R=\on\{X_i\}$. 
	
	Now the right vertical map breaks into homogeneous components, which in degree $0,1$ are isomorphisms, and in degrees $j\geq2$ is the $m+1$-connective map $(X_i\oplus M)^{\otimes j} \to X_i^{\otimes j}$ composed with the connectivity preserving exact functor $\cO(j)\otimes_{\Sigma_j}$. It follows that the left vertical map is also $m+1$-connective since it is a sum of the maps $(X_i \oplus \fib \epsilon)^{\otimes j} \to (X_i)^{\otimes j}$ composed with $\cO(j)\otimes_{\Sigma_j}$. The pushout in $\cO$-algebras maps to the pushout in $\cC$, and this map is $2m+2$-connective using \Cref{lem:oalgblakersmassey}, so in particular induces an equivalence after applying $\tau_{\leq m+1}$.

Thus it suffices to see that the diagram is a pushout in $\CC$ after applying $\tau_{\leq m+1}$, which amounts to observing the following exact sequence:

$$\pi_{m+1}^{\heart}(R\coprod \on\{\fib \epsilon\} )\to \pi_{m+1}^{\heart}(R\coprod \on\{M\}) \to \pi_{m+1}^{\heart} M \to 0$$
Indeed, by construction, the composite of the first two maps is $0$ the second map is an epimorphism with kernel generated by $\fib \epsilon$, proving the exact sequence.

We finish the proof by observing that in the situation of the lemma, the pushout square above shows that the map $\tau_{\leq m+1}(R\coprod \on\{M\}) \to R\oplus M$ is an epimorphism in $(\Alg_{\cO}(\CC)_{/R,[m+1,m+1]})_*$, which means that maps out of $R\oplus M$ are determined by their maps on the underlying object in $\CC$.
\end{proof}

We have enough tools now to do basic deformation theory of $\cO$-algebras.

\begin{dfn}
	We define the category of $\cO$-modules of an $\cO$-algebra $B$ to be $\mdef{\Mod_{\cO}(B)}:= \Sp(\Alg_{\cO}(\CC)_{/B})$. We define $L_B$ to be the image of $B$ under $\Sigma^{\infty}_B: \Alg_{\cO}(\CC)_{/B} \to \Sp(\Alg_{\cO}(\CC)_{/B})$.
\end{dfn}

\begin{prop}\label{lem:pistartrivaction}
  Let $R$ be a connective unital $\cO$-algebra in $\CC$ with an automorphism $\phi$ that is the identity on $\pi_*^{\heart}$. Suppose that $p^k=0$ on $\pi_*^{\heart}R$ and that $R=\tau_{\leq m}R$. Then $\phi^{p^{km}}$ is equivalent to the identity.
\end{prop}

\begin{proof}
    We prove the result inductively on $m$. For $m=0$, the result is clear. For the inductive step, by replacing $\phi$ with $\phi^{p^{k(m-1)}}$, we can assume that $\phi$ is the identity on $\tau_{\leq{m-1}}R$, and it suffices to show that $\phi^{p^k}$ is equivalent to the identity on $R$.
    
    Let $B = \tau_{\leq m-1}R$. By applying \Cref{lem:sqextnoalg}, since $R \in \Alg_{\cO}(\CC)_{/B,[m,2m-1]}$, we learn that $R$ is canonically a square-zero extension of $B$ by $\Sigma^{n}\pi_nR \in \Mod_{\cO}(B) := \Sp(\Alg_{\cO}(\CC)_{/B})$. This means that its group of automorphisms can be computed in the category $\Mod_{\cO}(B)_{L_B/}$. By considering the conservative forgetful functor $\Mod_{\cO}(B)_{/L_B} \to \Mod_{\cO}(B)$, we obtain that the automorphism group fits into a fiber sequence of groups
    
    $$\Omega\Map_{\Mod_{\cO}(B)}(L_B, \Sigma^{m+1}\pi_mR) \to \Aut_{\Mod_{\cO}(B)}(B) \to \Aut_{\Mod_{\cO}(B)}(\Sigma^{m+1}\pi_mR)$$
    
    By \Cref{lem:oalgfaithful}, $\phi$ has trivial image under the first map since it acts trivially on $\pi_nR$. Thus it suffices to see that $p^k=0$ in $\Omega\Map_{\Mod_{\cO}(B)}(L_B, \Sigma^{m+1}\pi_nR)$. But this mapping spectrum is a module over the endomorphism ring of $\Sigma^{m+1}\pi_nR$ in $\Mod_{\cO}(B)$, which has $p^k=0$ using \Cref{lem:oalgfaithful} again and the fact that $p^k=0$ on $\pi_mR$.    
\end{proof}

\begin{lem}\label{lem:trivializeacalg}
    Let $R\in \Alg_{\cO}(\CC^{B\Z,u})$ be connective, almost compact, bounded, and $p$-nilpotent. Then the $p^k\Z$ action on $R$ is trivializable for all $k\gg0$.
\end{lem}

\begin{proof}
    Since $R$ is almost compact and bounded, applying \Cref{lem:trivialize-action-ac}, we learn that the $p^k\Z$-action on $R$ is trivializable as an action in $\CC$ for $k\gg0$. Since $R$ is bounded, connective, and $p$-nilpotent, by applying \Cref{lem:pistartrivaction}, we learn that after possibly increasing $k$, the action becomes trivializable as an $\cO$-algebra.
\end{proof}